\documentclass[10pt]{article}%
\usepackage{makeidx}
\usepackage{amsfonts}
\usepackage{amsmath}
\usepackage{amssymb}
\usepackage{graphicx}
\usepackage{palatino}
\usepackage{extarrows}
\usepackage[colorlinks]{hyperref}
\usepackage{mathtools}%
\providecommand{\U}[1]{\protect\rule{.1in}{.1in}}
\providecommand{\U}[1]{\protect\rule{.1in}{.1in}}
\providecommand{\U}[1]{\protect\rule{.1in}{.1in}}
\hypersetup{colorlinks=blue, linkcolor=cyan, citecolor=green,
filecolor=black, urlcolor=blue }
\newtheorem{theorem}{Theorem}[section]

\newtheorem{corollary}[theorem]{Corollary}

\newtheorem{definition}[theorem]{Definition}
\newtheorem{example}[theorem]{Example}

\newtheorem{lemma}[theorem]{Lemma}
\newtheorem{notation}[theorem]{Notation}

\newtheorem{proposition}[theorem]{Proposition}
\newtheorem{remark}[theorem]{Remark}

\newenvironment{proof}[1][Proof]{\noindent\textbf{#1.} }{\ \rule{0.5em}{0.5em}}

\renewcommand{\thefootnote}{\fnsymbol{footnote}}
\numberwithin{equation}{section}
\begin{document}

\title{$L^{p}$--Variational Solutions of Multivalued Backward Stochastic Differential Equations}
\author{Lucian Maticiuc$^{a}$, Aurel R\u{a}\c{s}canu$^{b}$\bigskip\\{\small $^{a}$ Faculty of Mathematics, \textquotedblleft Alexandru Ioan
Cuza\textquotedblright\ University,}\\{\small Carol I Blvd., no. 11, 700506, Ia\c{s}i, Romania}\bigskip\\{\small $^{b}$ \textquotedblleft Octav Mayer\textquotedblright\ Institute of
Mathematics of the Romanian Academy,}\\{\small Carol I Blvd., no. 8, 700506, Ia\c{s}i, Romania}}
\date{}
\maketitle

\begin{abstract}
Our aim is to study the existence and uniqueness of the $L^{p}$--variational
solution, with $p>1,$ of the following multivalued backward stochastic
differential equation with $p$--integrable data:
\[
\left\{
\begin{array}
[c]{l}%
-dY_{t}+\partial_{y}\Psi(t,Y_{t})dQ_{t}\ni H(t,Y_{t},Z_{t})dQ_{t}-Z_{t}%
dB_{t},\;0\leq t<\tau,\\[0.2cm]%
Y_{\tau}=\eta,
\end{array}
\right.
\]
where $\tau$ is a stopping time, $Q$ is a progresivelly measurable increasing
continuous stochastic process and $\partial_{y}\Psi$ is the subdifferential of
the convex lower semicontinuous function $y\mapsto\Psi(t,y).$

In the framework of \cite{ma-ra/15} (the case $p\geq2$), the strong solution
found it there is the unique variational solution, via the uniqueness property
proved in the present article.

\end{abstract}

AMS Classification subjects: 60H10, 60F25, 47J20, 49J40.\medskip


\footnotetext{{\scriptsize E--mail addresses: \texttt{lucian.maticiuc@uaic.ro}
(Lucian Maticiuc), \texttt{aurel.rascanu@uaic.ro} (Aurel R\u{a}\c{s}canu)}}

\renewcommand{\thefootnote}{\arabic{footnote}}

\section{Introduction}

\hspace{\parindent}The study of the standard backward stochastic differential
equations (BSDEs) was initiated by E. Pardoux and S. Peng in \cite{pa-pe/90}.
The authors have proved the existence and the uniqueness of the solution for
the BSDE on fixed time interval, under the assumption of Lipschitz continuity
of the generator $F$ with respect to $y$ and $z$ and square integrability of
$\eta$ and $F\left(  t,0,0\right)  $. The case of BSDEs on random time
interval have been treated by R.W.R. Darling and E. Pardoux in \cite{da-pa/97}%
, where it is obtained, as application, the existence of a continuous
viscosity solution to the elliptic partial differential equations (PDEs) with
Dirichlet boundary conditions. The more general case of reflected BSDEs was
considered for the first time by N. El Karoui et al. in \cite{ka-ka-pa/97}.

In the present paper, we prove the existence and uniqueness of a new type of
solution, called $L^{p}$--variational solution, in the case $p>1,$ of the
genera\-li\-zed backward stochastic variational inequality (BSVI for short)
with $p$--integrable data:%
\begin{equation}
\left\{
\begin{array}
[c]{r}%
\displaystyle Y_{t}+{\int\nolimits_{t\wedge\tau}^{\tau}}dK_{s}=\eta
+{\int\nolimits_{t\wedge\tau}^{\tau}}\left[  F\left(  s,Y_{s},Z_{s}\right)
ds+G\left(  s,Y_{s}\right)  dA_{s}\right]  -{\int\nolimits_{t\wedge\tau}%
^{\tau}}Z_{s}dB_{s}\,,\quad t\geq0,\medskip\\
\multicolumn{1}{l}{dK_{t}\in\partial\varphi\left(  Y_{t}\right)
dt+\partial\psi\left(  Y_{t}\right)  dA_{t}\,,\quad\text{on }\mathbb{R}_{+},}%
\end{array}
\right.  \label{GBSVI 1}%
\end{equation}
where $\partial\varphi$ and $\partial\psi$ are the subdifferentials of two
proper convex lower semicontinuous (l.s.c. for short) functions $\varphi$ and
$\psi$ and $\left\{  A_{t}:t\geq0\right\}  $ is a progressively measurable
increasing continuous stochastic process.

We prove the uniqueness property of the solution on a random time interval
$\left[  0,\tau\right]  ;$ the existence is obtained also on a random time
interval, but first in the case of a deterministic time interval, i.e.
$\tau=T>0,$ and it is made using the Moreau--Yosida regularization of
$\varphi$ and $\psi$ and a mollifier approximation of the generators $F$ and
$G.\medskip$

In fact, we will define and prove the existence and uniqueness of the $L^{p}%
$--variational solution for an equivalent form of (\ref{GBSVI 1}):%
\begin{equation}
\left\{
\begin{array}
[c]{r}%
\displaystyle Y_{t}+{\int\nolimits_{t\wedge\tau}^{\tau}}dK_{s}=\eta
+{\int\nolimits_{t\wedge\tau}^{\tau}}H\left(  s,Y_{s},Z_{s}\right)
dQ_{s}-{\int\nolimits_{t\wedge\tau}^{\tau}}Z_{s}dB_{s}\,,\;t\geq0\medskip\\
\multicolumn{1}{l}{dK_{t}\in\partial_{y}\Psi\left(  t,Y_{t}\right)
dQ_{t}\,,\;\text{on }\mathbb{R}_{+},}%
\end{array}
\,\right.  \label{GBSVI 2}%
\end{equation}
with $Q,$ $H$ and $\Psi$ adequately defined.

The second condition in (\ref{GBSVI 1}) says, among others, that the first
component $Y$ of the solution is forced to stay in the set $\mathrm{Dom}%
\left(  \partial\varphi\right)  \cap\mathrm{Dom}\left(  \partial\psi\right)
.$ The role of $K$ is to act in the evolution of the process $Y$ and also to
keep $Y$ in these domains.$\medskip$

We emphasize that, unlike the case $p\geq2,$ in the case $1\leq p<2$ it is not
possible to obtain the boundedness of the term $\mathbb{E}\Big(%
{\displaystyle\int_{0}^{T}}
e^{2V_{r}}|\nabla\Psi_{\varepsilon}(r,Y_{r}^{\varepsilon})|^{2}dQ_{r}%
\Big)^{p/2},$ which is essential in order to obtain a strong solution, where
$\Psi_{\varepsilon}$ is the Moreau-Yosida's regularization of $\Psi.$
Therefore we propose a generalization of the strong solution and we give the
definition of the solution using an inequality (instead of a equality)
involving only the function $\Psi$ and not the subdifferential operator
$\partial\Psi.$ However, under this kind of definition, we were able to prove
the uniqueness property (even if the solution $\left(  Y,Z\right)  $ satisfies
an inequality).$\medskip$

We mention that the presence of the process $A$ is justified by the possible
applications of equation (\ref{GBSVI 1}) in proving probabilistic proofs for
the existence of a solution of PDEs with Neumann boundary conditions on a
domain from $\mathbb{R}^{m}.$ The stochastic approach of the existence problem
for multivalued parabolic PDEs, was considered by L. Maticiuc and A.
R\u{a}\c{s}canu in \cite{ma-ra/10} and \cite{ma-ra/16}. We emphasize that if
the obstacles are fixed, the reflected BSDEs becomes a particular case of the
BSVI of type (\ref{GBSVI 1}), by taking $\varphi$ as convex indicator of the
interval defined by obstacles. In this case the solution of the BSVI belongs
to the domain of the multivalued operator $\partial\varphi$ and it is
reflected at the boundary of this domain.$\medskip$

The standard work on BSVI in the finite dimensional case is that of E. Pardoux
and A. R\u{a}\c{s}canu \cite{pa-ra/98}, where it is proved the existence and
uniqueness of the solution $\left(  Y,Z,K\right)  $ for the BSVI
(\ref{GBSVI 1}) with $A\equiv0$, under the following assumptions on $F$:
continuity with respect to $y$, monotonicity with respect to $y$ (in the sense
that $\langle y^{\prime}-y,F(t,y^{\prime},z)-F(t,y,z)\rangle\leq
\alpha|y^{\prime}-y|^{2}$), Lipschitzianity with respect to $z$ and a
sublinear growth for $F\left(  t,y,0\right)  $. Moreover, it was shown that,
unlike the forward case, the process $K$ is absolute continuous with respect
to $dt$. In \cite{pa-ra/99} the same authors extend these results to the
Hilbert spaces framework.

We mention that assumptions of Lipschitz continuity of the generator $F$ with
respect to $y$ and $z$ and the square integrability of the final condition and
$F\left(  t,0,0\right)  $ (as in articles El Karoui et al. \cite{ka-ka-pa/97}
and E. Pardoux and S. Peng \cite{pa-pe/90}) are sometimes too strong for
applications (see, e.g., D. Duffie and L. Epstein \cite{du-ep/92} and El
Karoui et al. \cite{ka-pe-qu/97} for the applications in mathematical finance
and P. Briand et al. \cite{br-ca/00} and A. Rozkosz and L. S\l omi\'{n}ski
\cite{ro-sl/12b} for the applications to PDEs). A possibility is to weaken the
integrability conditions imposed on $\eta$ and $F$ or to weaken the assumption
which concerns the Lipschitz continuity of the generators. In P. Briand and R.
Carmona \cite{br-ca/00} or E. Pardoux \cite{pa/99} it is considered the case
where the generators are Lipschitz continuous with respect to $z$, continuous
with respect to $y$ and satisfy a monotonicity condition and a growth
condition of the type $\left\vert F\left(  t,y,z\right)  \right\vert
\leq\left\vert F\left(  t,0,z\right)  \right\vert +\phi\left(  \left\vert
y\right\vert \right)  ,$ where $\phi$ is a polynomial or even an arbitrary
positive increasing continuous function.

We recall that the previous assumption was used in \cite{pa/99} in order to
prove the existence of a solution in $L^{2}$. This result was generalized by
P. Briand et al. in \cite{br-de-hu-pa/03}, where it is proved the existence
and uniqueness of $L^{p}$ solutions, with $p\in\lbrack1,2],$ for BSDEs
considered with a random terminal time $T:$ in the case $p\in(1,2],$ if
$\eta\in L^{p}$, $\displaystyle\int_{0}^{T}\left\vert F\left(  s,0,0\right)
\right\vert ds\in L^{p}$ and $\displaystyle\int_{0}^{T}\sup
\nolimits_{\left\vert y\right\vert \leq r}\left\vert F\left(  s,y,0\right)
-F\left(  s,0,0\right)  \right\vert ds\in L^{1},$ for any $r>0,$ and if $F$ is
Lipschitz continuous with respect to $z$, continuous with respect to $y$ and
satisfy a monotonicity condition, then there exists a unique $L^{p}$ solution.
In the case $p=1$ similar result is proved if $T$ is a fixed deterministic
terminal time and under additional assumptions.

We also note that the study of the reflected BSDEs was the subject, e.g., of
the papers: J.P. Lepeltier et al. \cite{le-ma-xu/05} (in the case of the
general growth condition with respect to $y$ and for $p=2$), S. Hamad\`{e}ne
and A. Popier \cite{ha-po/12} (in the case of Lipschitz continuity with
respect to $y$ the and for $p\in\left(  1,2\right)  $). Studies made, roughly
speaking, under the assumptions of \cite{br-de-hu-pa/03} are, e.g.: A. Aman
\cite{am/09} (in the case of a generalized reflected BSDE and for $p\in\left(
1,2\right)  \,$), A. Rozkosz and L. S\l omi\'{n}ski \cite{ro-sl/12a} (for
$p\in\left[  1,2\right]  $) and T. Klimsiak \cite{kl/13} (in the case of BSDE
with two irregular reflecting barriers and for $p\in\left[  1,2\right]
$).\medskip

Our paper generalizes the existence and uniqueness results from
\cite{pa-ra/98} by considering the $L^{p}$ solutions in the case $p\in\left(
1,2\right)  ,$ the Lebesgue--Stieltsjes integral terms, and by assuming a
weaker boundedness condition for the generator $F$ (instead of the sublinear
growth):%
\begin{equation}
\mathbb{E}\Big(\int_{0}^{T}F_{\rho}^{\#}(s)ds\Big)^{p}<\infty,\quad\text{where
}F_{\rho}^{\#}\left(  t\right)
\xlongequal{\hspace{-4pt}{\rm def}\hspace{-4pt}}\sup\nolimits_{\left\vert
y\right\vert \leq\rho}\left\vert F(t,y,0)\right\vert . \label{local bound 2}%
\end{equation}
We remark that article \cite{ma-ra/15} concerns the same type of backward
equation as in our study (and under the similar assumptions), but considered
in the infinite dimensional framework and in the case $p\geq2.\medskip$

In addition, it is worth pointing out that in the case $p\geq2,$ if we are in
the framework of \cite{ma-ra/15}, our variational solution is a strong one
since we have proved the uniqueness property of the variational
solution.$\medskip$

More precisely, Theorem \ref{t2exist} generalizes the results from
\cite{ma-ra/10} for $p=2$ and Theorem \ref{t3-random} ge\-ne\-ra\-li\-zes
(except the Hilbert spaces framework) the results from \cite{ma-ra/15} for
$p\geq2.$ We emphasize that the assumptions on $F$ and $G$ are weaker to those
adopted in \cite{ma-ra/15}; also the form of our hypothesis is more simplified
(and therefore more easily to be verified) with respect to those from
\cite{ma-ra/15}.$\medskip$

In this paper we use the following notation: $(\Omega,\mathcal{F},\mathbb{P})$
is a complete probability space, the set $\mathcal{N}=\{A\in\mathcal{F}%
:\mathbb{P}\left(  A\right)  =0\},$ $\left\{  \mathcal{F}_{t}\right\}
_{t\geq0}$ is a right continuous and complete filtration generated by a
standard $k$--dimensional Brownian motion $\left(  B_{t}\right)  _{t\geq
0}\,.\medskip$

$S_{m}^{p}\left[  0,T\right]  $ is the space of (equivalent classes of)
continuous progressively measurable stochastic processes (p.m.s.p. for short)
$X:\Omega\times\left[  0,T\right]  \rightarrow\mathbb{R}^{m}$ such that
$\mathbb{E}\sup_{t\in\left[  0,T\right]  }\left\vert X_{t}\right\vert
^{p}<\infty,$ if $p>0.$ The notation $S_{m}^{p}$ is the space of (equivalent
classes of) continuous p.m.s.p. $X:\Omega\times\lbrack0,\infty)\rightarrow
\mathbb{R}^{m}$ such that, for all $T>0,$ the restriction $X\big|_{\left[
0,T\right]  }$ belongs to $S_{m}^{p}\left[  0,T\right]  .$

$\Lambda_{m}^{p}\left(  0,T\right)  $ is the space of p.m.s.p. $X:\Omega
\times\left(  0,T\right)  \rightarrow\mathbb{R}^{m}$ such that such that
$\int_{0}^{T}\left\vert X_{t}\right\vert ^{2}dt<\infty$, $\mathbb{P}$--a.s. if
$p=0$ and $\displaystyle\mathbb{E}\bigg(\int_{0}^{T}\left\vert X_{t}%
\right\vert ^{2}dt\bigg)^{p/2}<\infty$, if $p>0.$ The notation $\Lambda
_{m}^{p}$ is the space of p.m.s.p. $X:\Omega\times(0,\infty)\rightarrow
\mathbb{R}^{m}$ such that, for all $T>0,$ the restriction $X\big|_{\left(
0,T\right)  }$ belongs to $\Lambda_{m}^{p}\left(  0,T\right)  .$\medskip

The article is organized as follows: next section is dedicated to the
presentation of the assumptions needed in our study. In the third section we
present a intuitive introduction and the definition of the notion of $L^{p}%
$--variational solution. The next section deals with proof of the uniqueness
and continuity properties. The fifth section is devoted to the proof of the
existence of our type of solution both in the case of a deterministic and
random time interval. The Appendix contains, mainly following \cite{pa-ra/14},
some results useful throughout the paper.

\section{Assumptions and definitions\label{assumptions}}

\hspace{\parindent}At the beginning of this subsection we introduce the
assumptions about equation (\ref{GBSVI 1}).

We consider throughout this paper that $p>1.$

\begin{itemize}
\item[\textrm{(A}$_{1}$\textrm{)}] The random variable $\tau:\Omega
\rightarrow\left[  0,\infty\right]  $\ is a stopping time;

\item[\textrm{(A}$_{2}$\textrm{)}] The random variable $\eta:\Omega
\rightarrow\mathbb{R}^{m}$\ is $\mathcal{F}_{\tau}$--measurable such that
$\mathbb{E}\left\vert \eta\right\vert ^{p}<\infty$\ and $\left(  \xi
,\zeta\right)  \in S_{m}^{p}\times\Lambda_{m\times k}^{p}\left(
0,\mathbb{\infty}\right)  $\ is the unique pair associated to $\eta$\ given by
the martingale representation formula (see \cite[Corollary 2.44]{pa-ra/14})%
\begin{equation}
\left\{
\begin{array}
[c]{l}%
\xi_{t}=\eta-\displaystyle{\int_{t}^{\infty}}\zeta_{s}dB_{s}\,,\quad
t\geq0,\;\mathbb{P}\text{--a.s.,}\smallskip\\
\xi_{t}=\mathbb{E}^{\mathcal{F}_{t}}\eta\quad\text{and}\quad\zeta
_{t}=_{\left[  0,\tau\right]  }\left(  t\right)  \zeta_{t}%
\end{array}
\right.  \label{mart repr th 1}%
\end{equation}
(or equivalently, $\xi_{t}=\eta-\displaystyle\int_{t\wedge\tau}^{\tau}%
\zeta_{s}dB_{s}\,,\;t\geq0,\;\mathbb{P}$--a.s.);

\item[\textrm{(A}$_{3}$\textrm{)}] The process $\left\{  A_{t}:t\geq0\right\}
$\ is a\ increasing and continuous p.m.s.p. such that $A_{0}=0$ and%
\begin{equation}
\mathbb{E}\left(  e^{\alpha A_{T}}\right)  <\infty,\quad\text{for any }%
\alpha,T>0; \label{exponential_moment}%
\end{equation}

\item[\textrm{(A}$_{4}$\textrm{)}] $\varphi,\psi:\mathbb{R}^{m}\rightarrow
\left[  0,+\infty\right]  $\ are proper l.s.c. functions, $\partial\varphi$
and $\partial\psi$\ denote their subdifferentials and we suppose that
$0\in\partial\varphi\left(  0\right)  \cap\partial\psi\left(  0\right)  $\ (or
equivalently $0=\varphi\left(  0\right)  \leq\varphi\left(  y\right)  $\ and
$0=\psi\left(  0\right)  \leq\psi\left(  y\right)  $\ for all $y\in
\mathbb{R}^{m}\,$);

In addition, we suppose that%
\[
\varphi\left(  \eta\right)  +\psi\left(  \eta\right)  <\infty,\quad
\mathbb{P}\text{--a.s.;}%
\]

\item[\textrm{(A}$_{5}$\textrm{)}] The functions $F:\Omega\times\mathbb{R}%
_{+}\times\mathbb{R}^{m}\times\mathbb{R}^{m\times k}\rightarrow\mathbb{R}^{m}%
$\ and $G:\Omega\times\mathbb{R}_{+}\times\mathbb{R}^{m}\rightarrow
\mathbb{R}^{m}$\ are such that $F\left(  \cdot,\cdot,y,z\right)  $, $G\left(
\cdot,\cdot,y\right)  $\ are p.m.s.p., for all $\left(  y,z\right)
\in\mathbb{R}^{m}\times\mathbb{R}^{m\times k}$, $F\left(  \omega,t,\cdot
,\cdot\right)  $, $G\left(  \omega,t,\cdot\right)  $ are continuous functions,
$d\mathbb{P}\otimes dt$-a.e. and, $\mathbb{P}$--a.s.,%
\begin{equation}
\int_{0}^{T}F_{\rho}^{\#}\left(  s\right)  ds+\int_{0}^{T}G_{\rho}^{\#}\left(
s\right)  dA_{s}<\infty,\quad\text{for all }\rho,T\geq0,
\label{F, G assumpt 1}%
\end{equation}
where%
\begin{equation}
F_{\rho}^{\#}\left(  \omega,s\right)
\xlongequal{\hspace{-4pt}{\rm def}\hspace{-4pt}}\sup\nolimits_{\left\vert
y\right\vert \leq\rho}\left\vert F\left(  \omega,s,y,0\right)  \right\vert
,\quad G_{\rho}^{\#}\left(  \omega,s\right)
\xlongequal{\hspace{-4pt}{\rm def}\hspace{-4pt}}\sup\nolimits_{\left\vert
y\right\vert \leq\rho}\left\vert G\left(  \omega,s,y\right)  \right\vert \,;
\label{def F sharp}%
\end{equation}

\item[\textrm{(A}$_{6}$\textrm{)}] Let%
\begin{equation}
n_{p}\xlongequal{\hspace{-4pt}{\rm def}\hspace{-4pt}}\left(  p-1\right)
\wedge1\quad\text{and}\quad\lambda\in\left(  0,1\right)  . \label{defnp}%
\end{equation}
Assume there exist three p.m.s.p. $\mu,\nu:\Omega\times\mathbb{R}%
_{+}\rightarrow\mathbb{R},$ $\ell:\Omega\times\mathbb{R}_{+}\rightarrow
\mathbb{R}_{+}\,,$ such that%
\begin{equation}
\mathbb{E}\exp\left(  p\int_{0}^{T}\left(  \mu_{s}^{+}+\frac{1}{2n_{p}\lambda
}\,\ell_{s}^{2}\right)  ds+p\int_{0}^{T}\nu_{s}^{+}dA_{s}\right)
<\infty,\quad\text{for all }\ T>0, \label{ip-mnl}%
\end{equation}
and for all $t\geq0$, $y,y^{\prime}\in\mathbb{R}^{m}$, $z,z^{\prime}%
\in\mathbb{R}^{m\times k},$ $\mathbb{P}$--a.s.%
\begin{equation}%
\begin{array}
[c]{l}%
\left\langle y^{\prime}-y,F(t,y^{\prime},z)-F(t,y,z)\right\rangle \leq\mu
_{t}\,\left\vert y^{\prime}-y\right\vert ^{2},\medskip\\
\left\langle y^{\prime}-y,G(t,y^{\prime})-G(t,y)\right\rangle \leq\nu
_{t}\,\left\vert y^{\prime}-y\right\vert ^{2},\medskip\\
\left\vert F(t,y,z^{\prime})-F(t,y,z)\right\vert \leq\ell_{t}\,\left\vert
z^{\prime}-z\right\vert .
\end{array}
\label{F, G assumpt 2}%
\end{equation}

\end{itemize}

\begin{remark}
Assumption (\ref{ip-mnl}) is necessary for some estimates throughout the
proofs. In Remark \ref{A6_example} here below we give a consistent example for
$\mu,\nu$ and $\ell.$
\end{remark}

We define%
\[
Q_{t}=t+A_{t}\,,
\]
and let $\left\{  \alpha_{t}:t\geq0\right\}  $\ be the real positive p.m.s.p.
such that $\alpha\in\left[  0,1\right]  $\ and%
\[
dt=\alpha_{t}dQ_{t}\quad\text{and}\quad dA_{t}=\left(  1-\alpha_{t}\right)
dQ_{t}\,.
\]
Let us introduce the functions%
\begin{equation}%
\begin{array}
[c]{l}%
\displaystyle H\left(  t,y,z\right)
\xlongequal{\hspace{-4pt}{\rm def}\hspace{-4pt}}\mathbf{1}_{\left[
0,\tau\right]  }\left(  t\right)  \left[  \alpha_{t}F\left(  t,y,z\right)
+\left(  1-\alpha_{t}\right)  G\left(  t,y\right)  \right]  ,\medskip\\
\displaystyle\Psi\left(  t,y\right)
\xlongequal{\hspace{-4pt}{\rm def}\hspace{-4pt}}\mathbf{1}_{\left[
0,\tau\right]  }\left(  t\right)  \left[  \alpha_{t}\varphi\left(  y\right)
+\left(  1-\alpha_{t}\right)  \psi\left(  y\right)  \right]  .
\end{array}
\label{def Phi}%
\end{equation}
Obviously, from (\ref{F, G assumpt 2}) we see that%
\begin{equation}%
\begin{array}
[c]{l}%
\left\langle y^{\prime}-y,H(t,y^{\prime},z)-H(t,y,z)\right\rangle
\leq\mathbf{1}_{\left[  0,\tau\right]  }\left(  t\right)  \left[  \mu
_{t}\alpha_{t}+\nu_{t}\left(  1-\alpha_{t}\right)  \right]  \left\vert
y^{\prime}-y\right\vert ^{2},\medskip\\
\left\vert H(t,y,z^{\prime})-H(t,y,z)\right\vert \leq\mathbf{1}_{\left[
0,\tau\right]  }\left(  t\right)  \,\alpha_{t}\,\ell_{t}\,\left\vert
z^{\prime}-z\right\vert .
\end{array}
\label{F, G assumpt 3}%
\end{equation}
Here and subsequently%
\begin{equation}
V_{t}\xlongequal{\hspace{-4pt}{\rm def}\hspace{-4pt}}%
{\displaystyle\int_{0}^{t}}
\mathbf{1}_{\left[  0,\tau\right]  }\left(  r\right)  \left(  \mu_{r}%
+\dfrac{1}{2n_{p}\lambda}\,\ell_{r}^{2}\right)  dr+%
{\displaystyle\int_{0}^{t}}
\mathbf{1}_{\left[  0,\tau\right]  }\left(  r\right)  \nu_{r}dA_{r}
\label{defV_1}%
\end{equation}
and%
\begin{equation}
V_{t}^{\left(  +\right)  }\xlongequal{\hspace{-4pt}{\rm def}\hspace{-4pt}}%
{\displaystyle\int_{0}^{t}}
\mathbf{1}_{\left[  0,\tau\right]  }\left(  r\right)  \left(  \mu_{r}%
^{+}+\dfrac{1}{2n_{p}\lambda}\,\ell_{r}^{2}\right)  dr+%
{\displaystyle\int_{0}^{t}}
\mathbf{1}_{\left[  0,\tau\right]  }\left(  r\right)  \nu_{r}^{+}dA_{r}\,.
\label{defV_3}%
\end{equation}
By assumption (\ref{ip-mnl}) we clearly have, for all $\ T>0,$%
\begin{equation}%
\begin{array}
[c]{l}%
\displaystyle\mathbb{E}\exp\left(  pV_{T}\right)  \leq\mathbb{E}\left(
\sup\nolimits_{r\in\left[  0,T\right]  }{e^{pV_{r}}}\right)  \leq
\mathbb{E}\exp\big(pV_{T}^{\left(  +\right)  }\big)\medskip\\
\displaystyle\leq\mathbb{E}\exp\bigg(p\int_{0}^{T}\left(  \mu_{s}^{+}+\frac
{1}{2n_{p}\lambda}\,\ell_{s}^{2}\right)  ds+p\int_{0}^{T}\nu_{s}^{+}%
\,dA_{s}\bigg)<\infty.
\end{array}
\label{exp-VT}%
\end{equation}

\begin{remark}
\label{A6_example}Usually, the monotonicity coefficients $\mu_{t},\nu_{t}$ and
the Lipschitz coefficient $\ell_{t}$ are considered deterministic constants.
But in many concrete cases the coefficients are stochastic processes (may
depend on $\omega$ and $t$); for instance, if we take, as a simple example, in
one dimensional case, a BSDE with the generator%
\[
F\left(  t,y,z\right)  =\frac{\tilde{a}\,B_{t}\left\vert B_{t}\right\vert
^{a}}{t^{b}}\left(  y-f_{1}\left(  B_{t}y\right)  \right)  +\frac{\tilde
{c}\left\vert B_{t}\right\vert ^{\left(  c+1\right)  /2}}{t^{d/2}}%
\,f_{2}\left(  z\right)  ,
\]
where $\tilde{a},\tilde{c}>0,$ $0\leq b,d<1,$ $0<a\leq1,$ $-1<c\leq1$ are some
suitable constants and $f_{1},f_{2}$ are two derivable and nondecreasing functions.

In this case we obtain the monotonicity coefficient function $\mu_{t}%
=\frac{\tilde{a}\,B_{t}\left\vert B_{t}\right\vert ^{a}}{t^{b}}$ and the
Lipschitz coefficient function $\ell_{t}=\frac{\tilde{c}\left\vert
B_{t}\right\vert ^{\left(  a+1\right)  /2}}{t^{b/2}}\,.\medskip$

Our assumption $\left(  \mathrm{A}_{6}\right)  $ is satisfied if
$\displaystyle\frac{p\tilde{a}}{1-b}\leq\frac{1}{2}\,T^{b-2}$ and
$\displaystyle\frac{p\tilde{c}}{1-d}\,\frac{1}{n_{p}\lambda}\leq T^{d-2},$
since%
\[
\mathbb{E}\left(  \exp\left(  \tilde{a}\sup\nolimits_{t\in\left[  0,T\right]
}\left\vert B_{t}\right\vert ^{a+1}\right)  \right)  <\infty\quad
\text{iff}\quad\tilde{a}<\frac{1}{2T}%
\]
(for the details see, for instance, \cite[Theorem 4.1]{do-gr-le/96}).
\end{remark}

\begin{definition}
The notation $dK_{t}\in\partial_{y}\Psi\left(  t,Y_{t}\right)  dQ_{t}$ means
that $K$ is $\mathbb{R}^{m}$--valued locally bounded variation stochastic
process, $Q$ is a real increasing stochastic process, $Y$ is $\mathbb{R}^{m}%
$-valued continuous stochastic process such that $\int_{0}^{T}\Psi\left(
t,Y_{t}\right)  dQ_{t}<\infty$, a.s. for all $T\geq0$ and, $\mathbb{P}$--a.s.,
for any $0\leq t\leq s$%
\[%
{\displaystyle\int_{t}^{s}}
\left\langle y\left(  r\right)  -Y_{r},dK_{r}\right\rangle +%
{\displaystyle\int_{t}^{s}}
\Psi\left(  r,Y_{r}\right)  dQ_{r}\leq%
{\displaystyle\int_{t}^{s}}
\Psi\left(  r,y\left(  r\right)  \right)  dQ_{r},\quad\text{for any }y\in
C\left(  \mathbb{R}_{+};\mathbb{R}^{m}\right)  .
\]

\end{definition}

\begin{remark}
The condition $0\in\partial\varphi\left(  0\right)  \cap\partial\psi\left(
0\right)  $ does not restrict the generality of the problem, since from
$Dom\left(  \partial\varphi\right)  \cap Dom\left(  \partial\psi\right)
\neq\emptyset$ it follows that there exists $u_{0}\in Dom\left(
\partial\varphi\right)  \cap Dom\left(  \partial\psi\right)  $ and $\hat
{u}_{01}\in\partial\varphi\left(  u_{0}\right)  $, $\hat{u}_{02}\in
\partial\psi\left(  u_{0}\right)  $. In this case equation (\ref{GBSVI 1}) is
equivalent to%
\[
\left\{
\begin{array}
[c]{r}%
\hat{Y}_{t}+\displaystyle{\int_{t\wedge\tau}^{\tau}}d\hat{K}_{s}=\eta
+{\int_{t\wedge\tau}^{\tau}}\big[\hat{F}(s,\hat{Y}_{s},\hat{Z}_{s})ds+\hat
{G}(s,\hat{Y}_{s})dA_{s}\big]-{\int_{t\wedge\tau}^{\tau}}\hat{Z}_{s}%
dB_{s},\text{\ a.s.,}\smallskip\\
\multicolumn{1}{l}{d\hat{K}_{t}\in\partial\hat{\varphi}(\hat{Y}_{t}%
)dt+\partial\hat{\psi}(\hat{Y}_{t})dA_{t},\;\text{for all }t\geq0,}%
\end{array}
\,\right.
\]
where%
\[
\hat{Y}_{t}:=Y_{t}-u_{0}\,,\quad\hat{Z}_{t}:=Z_{t}\,,\quad\hat{\eta}%
:=\eta-u_{0}%
\]
and%
\[%
\begin{array}
[c]{l}%
\hat{F}\left(  s,y,z\right)  =F\left(  t,y+u_{0},z\right)  -\hat{u}%
_{01}\,,\quad\hat{G}\left(  s,y,z\right)  =G\left(  t,y+u_{0}\right)  -\hat
{u}_{02}\,,\medskip\\
\hat{\varphi}\left(  y\right)  =\varphi\left(  y+u_{0}\right)  -\left\langle
\hat{u}_{01},y\right\rangle -\varphi\left(  u_{0}\right)  \,,\quad\hat{\psi
}\left(  y\right)  =\psi\left(  y+u_{0}\right)  -\left\langle \hat{u}%
_{02},y\right\rangle \,-\psi\left(  u_{0}\right)  ,\medskip\\
\partial\hat{\varphi}\left(  y\right)  =\partial\varphi\left(  y+u_{0}\right)
-\hat{u}_{01}\,,\quad\partial\hat{\psi}\left(  y\right)  =\partial\psi\left(
y+u_{0}\right)  -\hat{u}_{02}\medskip\\
\text{and}\medskip\\
d\hat{K}_{t}=dK_{t}-\hat{u}_{01}dt-\hat{u}_{02}dA_{t}\,.
\end{array}
\]

\end{remark}

Let $\varepsilon>0$ and the Moreau--Yosida regularization of $\varphi:$%
\begin{equation}
\varphi_{\varepsilon}\left(  y\right)
\xlongequal{\hspace{-4pt}{\rm def}\hspace{-4pt}}\inf\big\{\frac{1}%
{2\varepsilon}\left\vert y-v\right\vert ^{2}+\varphi\left(  v\right)
:v\in\mathbb{R}^{m}\big\}, \label{fi-MY}%
\end{equation}
which is a $C^{1}$--convex function.

The gradient $\nabla\varphi_{\varepsilon}(x)=\partial\varphi_{\varepsilon
}\left(  x\right)  \in\partial\varphi\left(  J_{\varepsilon}\left(  x\right)
\right)  ,$ where $J_{\varepsilon}\left(  x\right)
\xlongequal{\hspace{-4pt}{\rm def}\hspace{-4pt}}x-\varepsilon\nabla
\varphi_{\varepsilon}(x)$ and the next ine\-qua\-li\-ties are satisfied%
\begin{equation}%
\begin{array}
[c]{ll}%
\left(  a\right)  & \left\vert J_{\varepsilon}\left(  x\right)
-J_{\varepsilon}\left(  y\right)  \right\vert \leq\left\vert x-y\right\vert
,\medskip\\
\left(  b\right)  & \left\vert \nabla\varphi_{\varepsilon}\left(  x\right)
-\nabla\varphi_{\varepsilon}\left(  y\right)  \right\vert \leq\dfrac
{1}{\varepsilon}\,\left\vert x-y\right\vert ,\medskip\\
\left(  c\right)  & \varphi_{\varepsilon}\left(  y\right)  =\dfrac{\left\vert
y-J_{\varepsilon}\left(  y\right)  \right\vert ^{2}}{2\varepsilon}%
+\varphi\left(  J_{\varepsilon}\left(  y\right)  \right)
\end{array}
\label{fi-lip}%
\end{equation}
and%
\begin{equation}
-\left\langle u-v,\nabla\varphi_{\varepsilon}\left(  u\right)  -\nabla
\varphi_{\delta}\left(  v\right)  \right\rangle \leq(\varepsilon
+\delta)\left\langle \nabla\varphi_{\varepsilon}(u),\nabla\varphi_{\delta
}(v)\right\rangle \leq\dfrac{\varepsilon+\delta}{2}\Big[|\nabla\varphi
_{\varepsilon}(u)|^{2}+|\nabla\varphi_{\delta}(v)|^{2}\Big] \label{fi-Cauchy}%
\end{equation}
(for other useful inequalities see, e.g., \cite[inequalities $\left(
2.8\right)  $]{ma-ra/15}).

Since $0\in\partial\varphi\left(  0\right)  $ we deduce that%
\begin{equation}%
\begin{array}
[c]{l}%
0=\varphi\left(  0\right)  \leq\varphi\left(  J_{\varepsilon}\left(  u\right)
\right)  \leq\varphi_{\varepsilon}\left(  u\right)  \leq\varphi\left(
u\right)  ,\quad\text{for any }u\in\mathbb{R}^{m},\medskip\\
J_{\varepsilon}\left(  0\right)  =0,\quad\nabla\varphi_{\varepsilon
}(0)=0,\quad\text{and }\varphi_{\varepsilon}\left(  0\right)  =0.
\end{array}
\label{minimum point}%
\end{equation}
We introduce the compatibility conditions between $\varphi,\psi$ and $F,G$.

\begin{itemize}
\item[\textrm{(A}$_{7}$\textrm{)}] For all $\varepsilon>0$, $t\geq0$,
$y\in\mathbb{R}^{m}$, $z\in\mathbb{R}^{m\times k}$%
\begin{equation}%
\begin{array}
[c]{rl}%
\left(  i\right)  & \left\langle \nabla\varphi_{\varepsilon}\left(  y\right)
,\nabla\psi_{\varepsilon}\left(  y\right)  \right\rangle \geq0,\medskip\\
\left(  ii\right)  & \left\langle \nabla\varphi_{\varepsilon}\left(  y\right)
,G\left(  t,y\right)  \right\rangle \leq\left\vert \nabla\psi_{\varepsilon
}\left(  y\right)  \right\vert \left\vert G\left(  t,y\right)  \right\vert
,\quad\mathbb{P}\text{--a.s.,}\medskip\\
\left(  iii\right)  & \left\langle \nabla\psi_{\varepsilon}\left(  y\right)
,F\left(  t,y,z\right)  \right\rangle \leq\left\vert \nabla\varphi
_{\varepsilon}\left(  y\right)  \right\vert \left\vert F\left(  t,y,z\right)
\right\vert ,\quad\mathbb{P}\text{--a.s..}%
\end{array}
\label{compAssumpt}%
\end{equation}

\end{itemize}

\begin{example}
$\quad$

\noindent$\left(  a\right)  $ If $\varphi=\psi$ then the compatibility
assumptions (\ref{compAssumpt}) are clearly satisfied.$\medskip$

\noindent$\left(  b\right)  $ Let $m=1$. Since $\nabla\varphi_{\varepsilon}$
and $\nabla\psi_{\varepsilon}$ are increasing monotone functions on
$\mathbb{R}$, we see that, if $y\cdot G\left(  t,y\right)  \leq0$ and $y\cdot
F\left(  t,y,z\right)  \leq0$, for all $t,y,z,$ then the compatibility
assumptions (\ref{compAssumpt}) are satisfied.$\medskip$

\noindent$\left(  b\right)  $ Let $m=1$. If $\varphi,\psi:\mathbb{R}%
\rightarrow(-\infty,+\infty]$ are the convexity indicator functions
$\varphi\left(  y\right)  =\left\{
\begin{array}
[c]{rl}%
0, & \text{if\ }y\in\left[  a,b\right]  ,\smallskip\\
+\infty, & \text{if\ }y\notin\left[  a,b\right]  ,
\end{array}
\right.  $ and $\psi\left(  y\right)  =\left\{
\begin{array}
[c]{rl}%
0, & \text{if\ }y\in\left[  c,d\right]  ,\smallskip\\
+\infty, & \text{if\ }y\notin\left[  c,d\right]  ,
\end{array}
\right.  $ where $-\infty\leq a\leq b\leq\infty$ and $-\infty\leq c\leq
d\leq\infty$ are such that $0\in\left[  a,b\right]  \cap\left[  c,d\right]  $
(see (A$_{6}$)), then $\nabla\varphi_{\varepsilon}\left(  y\right)  =\dfrac
{1}{\varepsilon}[\left(  y-b\right)  ^{+}-\left(  a-y\right)  ^{+}]$, and
$\nabla\psi_{\varepsilon}\left(  y\right)  =\dfrac{1}{\varepsilon}[\left(
y-d\right)  ^{+}-\left(  c-y\right)  ^{+}].\medskip$

Assumption (A$_{7}-i$) is clearly fulfilled; the remaining compatibility
assumptions are satisfies if, for example, $G\left(  t,y\right)  \geq0$, for
$y\leq a$,$\quad G\left(  t,y\right)  \leq0$, for $y\geq b$, and,
respectively, $F\left(  t,y,z\right)  \geq0$, for $y\leq c$,$\quad F\left(
t,y,z\right)  \leq0$, for $y\geq d.$
\end{example}

\section{Intuitive introduction of $L^{p}$--variational
solutions\label{intuitive_intr}}

\hspace{\parindent}For $a\geq0,$ let us define the space $\mathcal{V}_{m}^{a}$
of the $m$-dimensional local continuous semimartingales $M$ such that for all
$T>0,$%
\begin{equation}
\mathbb{E}\left(  \sup\nolimits_{r\in\left[  0,T\right]  }{e^{aV_{r}}%
}\left\vert M_{r}\right\vert ^{a}\right)  <\infty,\text{\quad if }a>1
\label{def_M_2}%
\end{equation}
and given by%
\begin{equation}%
\begin{array}
[c]{l}%
\displaystyle M_{t}=\gamma-%
{\displaystyle\int_{0}^{t}}
N_{r}dQ_{r}+%
{\displaystyle\int_{0}^{t}}
R_{r}dB_{r}\quad\text{or equivalently}\medskip\\
\displaystyle M_{t}=M_{T}+\int_{t}^{T}N_{r}dQ_{r}-\int_{t}^{T}R_{r}%
dB_{r}\,,\quad M_{0}=\gamma,
\end{array}
\label{def_M}%
\end{equation}
where $\gamma\in\mathbb{R}^{m}$and $N:\Omega\times\mathbb{R}_{+}%
\rightarrow\mathbb{R}^{m},$ $R:\Omega\times\mathbb{R}_{+}\rightarrow
\mathbb{R}^{m\times k}$ are p.m.s.p. such that for all $T>0:$%
\begin{equation}
\mathbb{E}\left(
{\displaystyle\int_{0}^{T}}
{e^{V_{r}}}\left\vert N_{r}\right\vert dQ_{r}\right)  ^{a}+\mathbb{E}\left(
{\displaystyle\int_{0}^{T}}
{e^{2V_{r}}}\left\vert R_{r}\right\vert ^{2}dr\right)  ^{a/2}<\infty
,\quad\text{if }a>0 \label{def_M_3}%
\end{equation}
and
\[%
{\displaystyle\int_{0}^{T}}
\left\vert N_{r}\right\vert dQ_{r}+%
{\displaystyle\int_{0}^{T}}
\left\vert R_{r}\right\vert ^{2}dr<\infty,\;\mathbb{P}\text{--a.s.}%
,\quad\text{if }a=0.
\]
For a intuitive introduction let $\left(  Y,Z,U\right)  $ be a strong a
solution of (\ref{GBSVI 1}) or (\ref{GBSVI 2}) that is $Y,Z,$ and $U$ are
p.m.s.p., $Y$ has continuous trajectories,
\[%
{\displaystyle\int_{0}^{T}}
\left\vert Z_{r}\right\vert ^{2}dr+%
{\displaystyle\int_{0}^{T}}
\left\vert U_{r}\right\vert ^{2}dr<\infty,\quad\mathbb{P}-a.s.,\quad\text{for
all }T\geq0,
\]
and the following equation is satisfied, for all $T\geq0,$%
\[
\left\{
\begin{array}
[c]{r}%
\displaystyle Y_{t}+{\int_{t}^{T}}dK_{r}=Y_{T}+{\int_{t}^{T}}H\left(
r,Y_{r},Z_{r}\right)  dQ_{r}-{\int_{t}^{T}}Z_{r}dB_{r}\,,\quad\mathbb{P}%
\text{--a.s.,}\quad\text{for all }t\in\left[  0,T\right]  ,\medskip\\
\multicolumn{1}{l}{\displaystyle dK_{r}=U_{r}dQ_{r}\in\partial_{y}\Psi\left(
r,Y_{r}\right)  dQ_{r}\,,}%
\end{array}
\,\right.
\]
and%
\[
e^{V_{t}}\left\vert Y_{t}-\xi_{t}\right\vert +\int_{t}^{\infty}e^{2V_{r}%
}\left\vert Z_{r}-\zeta_{r}\right\vert ^{2}%
dr\xrightarrow[]{\;\;\;\;\mathbb{P}\;\;\;\;}0,\quad\text{as }t\rightarrow
\infty.
\]
For $\delta\in(0,1]$ we define%
\begin{equation}
\delta_{q}\xlongequal{\hspace{-4pt}{\rm def}\hspace{-4pt}}\delta
\,\mathbf{1}_{[1,2)}\left(  q\right)  =\left\{
\begin{array}
[c]{ll}%
\delta, & \text{if }1\leq q<2,\\
0, & \text{otherwise.}%
\end{array}
\right.  \label{defDelta}%
\end{equation}
Let $q\in\left[  1,2\right]  ,$ $n_{q}%
\xlongequal{\hspace{-4pt}{\rm def}\hspace{-4pt}}\left(  q-1\right)
\wedge1=q-1$, $M\in\mathcal{V}_{m}^{0}$ of the form (\ref{def_M}) and%
\begin{equation}
\Gamma_{t}\xlongequal{\hspace{-4pt}{\rm def}\hspace{-4pt}}\Big(\left\vert
M_{t}-Y_{t}\right\vert ^{2}+\delta_{q}\Big)^{1/2}. \label{def_gamma}%
\end{equation}
By It\^{o}'s formula applied to $\left(  \Gamma_{t}\right)  ^{q}$ we deduce,
using inequality (\ref{ito4}) from Remark \ref{r1-ito}, that, for all $0\leq
t\leq s$ and for all $\delta\in(0,1],$%
\begin{equation}%
\begin{array}
[c]{l}%
\left(  \Gamma_{t}\right)  ^{q}+\dfrac{q}{2}%
{\displaystyle\int_{t}^{s}}
{\left(  \Gamma_{r}\right)  ^{q-4}}\left(  n_{q}\left\vert M_{r}%
-Y_{r}\right\vert ^{2}+\delta_{q}\right)  \left\vert R_{r}-Z_{r}\right\vert
^{2}dr-q%
{\displaystyle\int_{t}^{s}}
{\left(  \Gamma_{r}\right)  ^{q-2}}\left\langle M_{r}-Y_{r},U_{r}%
\,dQ_{r}\right\rangle \medskip\\
\leq\left(  \Gamma_{s}\right)  ^{q}+q%
{\displaystyle\int_{t}^{s}}
{\left(  \Gamma_{r}\right)  ^{q-2}}\langle M_{r}-Y_{r},N_{r}-H\left(
r,Y_{r},Z_{r}\right)  \rangle dQ_{r}\medskip\\
\quad-q%
{\displaystyle\int_{t}^{s}}
{\left(  \Gamma_{r}\right)  ^{q-2}}\,\langle M_{r}-Y_{r},\left(  R_{r}%
-Z_{r}\right)  dB_{r}\rangle,
\end{array}
\label{def1a}%
\end{equation}
where $U_{t}dQ_{t}\in\partial_{y}\Psi\left(  t,Y_{t}\right)  dQ_{t}$ .

Using the subdifferential inequality%
\[
\left\langle M_{r}-Y_{r},U_{t}dQ_{t}\right\rangle +{\Psi}\left(
r,Y_{r}\right)  dQ_{r}\leq{\Psi}\left(  r,M_{r}\right)  dQ_{r}%
\]
we get, from (\ref{def1a}),%
\begin{equation}%
\begin{array}
[c]{l}%
\left(  \Gamma_{t}\right)  ^{q}+\dfrac{q}{2}%
{\displaystyle\int_{t}^{s}}
{\left(  \Gamma_{r}\right)  ^{q-4}}\left(  n_{q}\left\vert M_{r}%
-Y_{r}\right\vert ^{2}+\delta_{q}\right)  \left\vert R_{r}-Z_{r}\right\vert
^{2}dr+{q%
{\displaystyle\int_{t}^{s}}
\left(  \Gamma_{r}\right)  ^{q-2}\Psi}\left(  r,Y_{r}\right)  dQ_{r}\medskip\\
\leq\left(  \Gamma_{s}\right)  ^{q}+{q%
{\displaystyle\int_{t}^{s}}
\left(  \Gamma_{r}\right)  ^{q-2}\Psi}\left(  r,M_{r}\right)  dQ_{r}+q%
{\displaystyle\int_{t}^{s}}
{\left(  \Gamma_{r}\right)  ^{q-2}}\langle M_{r}-Y_{r},N_{r}-H\left(
r,Y_{r},Z_{r}\right)  \rangle dQ_{r}\medskip\\
\quad-q%
{\displaystyle\int_{t}^{s}}
{\left(  \Gamma_{r}\right)  ^{q-2}}\,\langle M_{r}-Y_{r},\left(  R_{r}%
-Z_{r}\right)  dB_{r}\rangle.
\end{array}
\label{def1b}%
\end{equation}

\begin{remark}
\label{r-nq}Let $\delta>0$, $p>1$ and $q\in\left\{  2,p\wedge2\right\}  .$ We
have%
\[
n_{p}\xlongequal{\hspace{-4pt}{\rm def}\hspace{-4pt}}\left(  p-1\right)
\wedge1\leq q-1=\left(  q-1\right)  \wedge
1\xlongequal{\hspace{-4pt}{\rm def}\hspace{-4pt}}n_{q}\leq1
\]
and moreover

\begin{itemize}
\item if $q=p\wedge2,$ then $n_{q}=q-1=\left(  p-1\right)  \wedge1=n_{p}$ and
$\delta_{q}=\delta\,\mathbf{1}_{p<2}\,;$

\item if $q=2,$ then $n_{p}\leq1=n_{q}$ and $\delta_{q}=0.$\medskip
\end{itemize}
\end{remark}

\subsection{Definition and preliminary estimates}

\hspace{\parindent}Following the approach used for the forward stochastic
variational inequalities from article \cite{ra/81}, we propose, starting from
(\ref{def1b}) and using $n_{q}{\left(  \Gamma_{r}\right)  ^{q-2}}\leq{\left(
\Gamma_{r}\right)  ^{q-4}}\left(  n_{q}\left\vert M_{r}-Y_{r}\right\vert
^{2}+\delta_{q}\right)  $, the next variational formulation for a solution of
the multivalued BSDE (\ref{GBSVI 2}).\medskip

\begin{definition}
\label{definition_weak solution}We say that $\left(  Y_{t},Z_{t}\right)
_{t\geq0}$ is a $L^{p}$--variational solution of (\ref{GBSVI 2}) if:

\begin{itemize}
\item $Y:\Omega\times\mathbb{R}_{+}\rightarrow\mathbb{R}^{m}$ and
$Z:\Omega\times\mathbb{R}_{+}\rightarrow\mathbb{R}^{m\times k}$ are two
p.m.s.p., $Y$ has continuous trajectories satisfying%
\begin{equation}
\mathbb{E}\left(  \sup\nolimits_{r\in\left[  0,\tau\right]  }{e^{pV_{r}}%
}\left\vert Y_{r}\right\vert ^{p}\right)  <\infty\label{def0-1}%
\end{equation}
and%
\begin{equation}
\mathbb{E}\left(
{\displaystyle\int_{0}^{\tau}}
{e^{2V_{r}}}\left\vert Z_{r}\right\vert ^{2}dr\right)  ^{p/2}+\mathbb{E}%
\left(
{\displaystyle\int_{0}^{\tau}}
e^{2V_{r}}{\Psi}\left(  r,Y_{r}\right)  dQ_{r}\right)  ^{p/2}<\infty,
\label{def0-2}%
\end{equation}
where $V$ is defined by (\ref{defV_1});

\item $\left(  Y_{t},Z_{t}\right)  =\left(  \xi_{t},\zeta_{t}\right)  =\left(
\eta,0\right)  ,$ for $t>\tau$ and%
\begin{equation}
e^{pV_{T}}\left\vert Y_{T}-\xi_{T}\right\vert ^{p}+\Big(\int_{T}^{\infty
}e^{2V_{s}}\left\vert Z_{s}-\zeta_{s}\right\vert ^{2}ds\Big)^{p/2}%
\xrightarrow[T\rightarrow\infty]{\;\;\;\;\mathbb{P}\;\;\;\;}0; \label{def2}%
\end{equation}

\item let $\Gamma_{t}$ be defined by (\ref{def_gamma}), i.e. $\Gamma
_{t}=\big(\left\vert M_{t}-Y_{t}\right\vert ^{2}+\delta_{q}\big)^{1/2};$ then,
for every $q\in\{2,p\wedge2\}$ and $\delta\in(0,1],$ it holds%
\begin{equation}%
\begin{array}
[c]{l}%
\displaystyle\left(  \Gamma_{t}\right)  ^{q}+\dfrac{q\left(  q-1\right)  }{2}%
{\displaystyle\int_{t}^{s}}
{\left(  \Gamma_{r}\right)  ^{q-2}}\left\vert R_{r}-Z_{r}\right\vert ^{2}dr+{q%
{\displaystyle\int_{t}^{s}}
\left(  \Gamma_{r}\right)  ^{q-2}\Psi}\left(  r,Y_{r}\right)  dQ_{r}\medskip\\
\displaystyle\leq\left(  \Gamma_{s}\right)  ^{q}+{q%
{\displaystyle\int_{t}^{s}}
\left(  \Gamma_{r}\right)  ^{q-2}\Psi}\left(  r,M_{r}\right)  dQ_{r}+q%
{\displaystyle\int_{t}^{s}}
{\left(  \Gamma_{r}\right)  ^{q-2}}\langle M_{r}-Y_{r},N_{r}-H\left(
r,Y_{r},Z_{r}\right)  \rangle dQ_{r}\medskip\\
\displaystyle\quad-q%
{\displaystyle\int_{t}^{s}}
{\left(  \Gamma_{r}\right)  ^{q-2}}\,\langle M_{r}-Y_{r},\left(  R_{r}%
-Z_{r}\right)  dB_{r}\rangle,
\end{array}
\label{def1}%
\end{equation}
for any $0\leq t\leq s<\infty$ and $M\in\mathcal{V}_{m}^{0}$ of the form
(\ref{def_M}), i.e. $M_{t}=M_{T}+\int_{t}^{T}N_{r}dQ_{r}-\int_{t}^{T}%
R_{r}dB_{r}\,.$
\end{itemize}
\end{definition}

\begin{remark}
For $q=2$ inequality (\ref{def1}) becomes%
\begin{equation}%
\begin{array}
[c]{l}%
\left\vert M_{t}-Y_{t}\right\vert ^{2}+{%
{\displaystyle\int_{t}^{s}}
}\left\vert R_{r}-Z_{r}\right\vert ^{2}dr+{2%
{\displaystyle\int_{t}^{s}}
\Psi}\left(  r,Y_{r}\right)  dQ_{r}\medskip\\
\leq\left\vert M_{s}-Y_{s}\right\vert ^{2}{+2%
{\displaystyle\int_{t}^{s}}
\Psi}\left(  r,M_{r}\right)  dQ_{r}+{2%
{\displaystyle\int_{t}^{s}}
}\langle M_{r}-Y_{r},N_{r}-H\left(  r,Y_{r},Z_{r}\right)  \rangle
dQ_{r}\medskip\\
\quad-2%
{\displaystyle\int_{t}^{s}}
\,\langle M_{r}-Y_{r},\left(  R_{r}-Z_{r}\right)  dB_{r}\rangle,\quad
\mathbb{P}\text{--a.s.,}%
\end{array}
\label{def1-a}%
\end{equation}
which is exactly the definition of the variational solution in the case
$p\geq2$ from article \cite{ma-ra/15} (since, in this case, $q\in
\{2,p\wedge2\}$ means $q=2\,$).
\end{remark}

\begin{remark}
As we see above, in the case $p\geq2,$ inequality (\ref{def1}) from Definition
(\ref{definition_weak solution}) should be satisfied only for $q=2.$

But in the case $1<p<2,$ we ask inequality (\ref{def1}) to be satisfied for
$q=p$ and also for $q=2.$ This is due to the fact that, without inequality
(\ref{def1}) accomplished for $q=2,$ we are not able to obtain the estimates
for the term $\mathbb{E}\Big(%
{\displaystyle\int_{0}^{T}}
{e^{2V_{r}}}\left\vert Z_{r}\right\vert ^{2}dr\Big)^{p/2}$ and therefore the
fact that $t\mapsto%
{\displaystyle\int_{0}^{t}}
{e^{pV_{r}}}\left(  \Gamma_{r}\right)  ^{p-2}\langle Y_{r},Z_{r}dB_{r}\rangle$
is a martingale (see Proposition \ref{prop2-exp-exp}).
\end{remark}

\begin{proposition}
\label{p1-exp}Let $\left\{  L_{t}:t\geq0\right\}  $ be arbitrary continuous
bounded variation p.m.s.p.. Then inequality (\ref{def1}) from Definition
\ref{definition_weak solution} is equivalent to%
\begin{equation}%
\begin{array}
[c]{l}%
\displaystyle e^{qL_{t}}\left(  \Gamma_{t}\right)  ^{q}+q%
{\displaystyle\int_{t}^{s}}
e^{qL_{r}}\left(  \Gamma_{r}\right)  ^{q}dL_{r}+\dfrac{q}{2}n_{q}%
{\displaystyle\int_{t}^{s}}
{e^{qL_{r}}\left(  \Gamma_{r}\right)  ^{q-2}}\left\vert R_{r}-Z_{r}\right\vert
^{2}dr\medskip\\
\displaystyle\quad+{q%
{\displaystyle\int_{t}^{s}}
e^{qL_{r}}\left(  \Gamma_{r}\right)  ^{q-2}\Psi}\left(  r,Y_{r}\right)
dQ_{r}\medskip\\
\displaystyle\leq e^{qL_{s}}\left(  \Gamma_{s}\right)  ^{q}+{q%
{\displaystyle\int_{t}^{s}}
e^{qL_{r}}\left(  \Gamma_{r}\right)  ^{q-2}\Psi}\left(  r,M_{r}\right)
dQ_{r}\medskip\\
\displaystyle\quad+{q%
{\displaystyle\int_{t}^{s}}
e^{qL_{r}}}\left(  \Gamma_{r}\right)  ^{q-2}\langle M_{r}-Y_{r},N_{r}-H\left(
r,Y_{r},Z_{r}\right)  \rangle dQ_{r}\medskip\\
\displaystyle\quad-q{%
{\displaystyle\int_{t}^{s}}
e^{qL_{r}}}\left(  \Gamma_{r}\right)  ^{q-2}\langle M_{r}-Y_{r},\left(
R_{r}-Z_{r}\right)  dB_{r}\rangle,
\end{array}
\label{def1-b}%
\end{equation}
for any $q\in\{2,p\wedge2\},$ $\delta\in(0,1],$ $0\leq t\leq s<\infty,$ and
$M\in\mathcal{V}_{m}^{0}$ of the form (\ref{def_M}).
\end{proposition}

\begin{proof}
Let $T>0$ be arbitrary and $0\leq t\leq s\leq T.$ Let $M\in\mathcal{V}_{m}%
^{0}$ of the form (\ref{def_M}) be such that
\[%
{\displaystyle\int_{0}^{T}}
{\Psi}\left(  r,M_{r}\right)  dQ_{r}<\infty,\quad\mathbb{P}\text{--a.s..}%
\]
\noindent(\ref{def1})\ $\Longrightarrow\;$(\ref{def1-b}):

We remark that the stochastic process%
\begin{align*}
\Lambda_{t}  &  \xlongequal{\hspace{-4pt}{\rm def}\hspace{-4pt}}\dfrac
{q\left(  q-1\right)  }{2}\int_{0}^{t}{\left(  \Gamma_{r}\right)  ^{q-2}%
}\left\vert R_{r}-Z_{r}\right\vert ^{2}dr+{q\int_{0}^{t}\left(  \Gamma
_{r}\right)  ^{q-2}\Psi}\left(  r,Y_{r}\right)  dQ_{r}\\
&  \quad-{q\int_{0}^{t}\left(  \Gamma_{r}\right)  ^{q-2}\Psi}\left(
r,M_{r}\right)  dQ_{r}-q\int_{0}^{t}{\left(  \Gamma_{r}\right)  ^{q-2}}\langle
M_{r}-Y_{r},N_{r}-H\left(  r,Y_{r},Z_{r}\right)  \rangle dQ_{r}\\
&  \quad+q\int_{0}^{t}{\left(  \Gamma_{r}\right)  ^{q-2}}\,\langle M_{r}%
-Y_{r},\left(  R_{r}-Z_{r}\right)  dB_{r}\rangle
\end{align*}
is a locally semimartingale, and from (\ref{def1}) it follows that%
\[
t\mapsto\left(  \Gamma_{t}\right)  ^{q}-\Lambda_{t}%
\]
is a continuous nondecreasing stochastic process.

Therefore, $\Gamma^{q}=\left[  \Gamma^{q}-\Lambda\right]  +\Lambda$ is a
locally semimartingale and, for all $0\leq t\leq s\leq T,$%
\begin{align*}
e^{qL_{s}}\left(  \Gamma_{s}\right)  ^{q}-e^{qL_{t}}\left(  \Gamma_{t}\right)
^{q}  &  =%
{\displaystyle\int_{t}^{s}}
d\left[  e^{qL_{r}}\left(  \Gamma_{r}\right)  ^{q}\right] \\
&  =q%
{\displaystyle\int_{t}^{s}}
e^{qL_{r}}\left(  \Gamma_{r}\right)  ^{q}dL_{r}+%
{\displaystyle\int_{t}^{s}}
e^{qL_{r}}d\left[  \left(  \Gamma_{r}\right)  ^{q}-\Lambda_{r}\right]  +%
{\displaystyle\int_{t}^{s}}
e^{qL_{r}}d\Lambda_{r}\\
&  \geq q%
{\displaystyle\int_{t}^{s}}
e^{qL_{r}}\left(  \Gamma_{r}\right)  ^{q}dL_{r}+%
{\displaystyle\int_{t}^{s}}
e^{qL_{r}}d\Lambda_{r}\,.
\end{align*}
which clearly yields (\ref{def1-b}).$\medskip$

The implication (\ref{def1-b})\ $\Longrightarrow\;$(\ref{def1}) is proved in
the same manner. Let%
\begin{align*}
\tilde{\Lambda}_{t}  &  =q\int_{0}^{t}e^{qL_{r}}\left(  \Gamma_{r}\right)
^{q}dL_{r}+\dfrac{q\left(  q-1\right)  }{2}\int_{0}^{t}e^{qL_{r}}{\left(
\Gamma_{r}\right)  ^{q-2}}\left\vert R_{r}-Z_{r}\right\vert ^{2}dr+{q\int
_{0}^{t}e^{qL_{r}}\left(  \Gamma_{r}\right)  ^{q-2}\Psi}\left(  r,Y_{r}%
\right)  dQ_{r}\\
&  \quad-{q\int_{0}^{t}e^{qL_{r}}\left(  \Gamma_{r}\right)  ^{q-2}\Psi}\left(
r,M_{r}\right)  dQ_{r}-q\int_{0}^{t}e^{qL_{r}}{\left(  \Gamma_{r}\right)
^{q-2}}\langle M_{r}-Y_{r},N_{r}-H\left(  r,Y_{r},Z_{r}\right)  \rangle
dQ_{r}\\
&  \quad+q\int_{0}^{t}e^{qL_{r}}{\left(  \Gamma_{r}\right)  ^{q-2}}\,\langle
M_{r}-Y_{r},\left(  R_{r}-Z_{r}\right)  dB_{r}\rangle.
\end{align*}
Then $t\mapsto\left(  e^{L_{t}}\Gamma_{t}\right)  ^{q}-\tilde{\Lambda}_{t}$ is
a continuous nondecreasing stochastic process and%
\begin{align*}
\left(  \Gamma_{s}\right)  ^{q}-\left(  \Gamma_{t}\right)  ^{q}  &  =%
{\displaystyle\int_{t}^{s}}
d\left[  e^{-qL_{r}}\left(  e^{L_{r}}\Gamma_{r}\right)  ^{q}\right] \\
&  =-q%
{\displaystyle\int_{t}^{s}}
e^{-qL_{r}}\left(  e^{L_{r}}\Gamma_{r}\right)  ^{q}dL_{r}+%
{\displaystyle\int_{t}^{s}}
e^{-qL_{r}}d\left[  \left(  e^{L_{r}}\Gamma_{r}\right)  ^{q}-\tilde{\Lambda
}_{r}\right]  +%
{\displaystyle\int_{t}^{s}}
e^{-qL_{r}}d\tilde{\Lambda}_{r}\\
&  \geq-q%
{\displaystyle\int_{t}^{s}}
\left(  \Gamma_{r}\right)  ^{q}dL_{r}+%
{\displaystyle\int_{t}^{s}}
e^{-qL_{r}}d\tilde{\Lambda}_{r}\,.
\end{align*}
\hfill
\end{proof}

\begin{remark}
In the following results we will often use the continuous bounded variation
p.m.s.p. $\left\{  V_{t}:t\geq0\right\}  $ given by (\ref{defV_1}) in the
place of $\left\{  L_{t}:t\geq0\right\}  .$
\end{remark}

\begin{proposition}
\label{prop2-exp-exp}Let $\left(  Y_{t},Z_{t}\right)  _{t\geq0}$ a $L^{p}%
$--variational solution in the sense of Definition
\ref{definition_weak solution} and $q=p\wedge2$, and $M\in\mathcal{V}_{m}^{q}$
of the form (\ref{def_M}). Then
\[
t\mapsto%
{\displaystyle\int_{0}^{t}}
{e^{qV_{r}}}\left(  \Gamma_{r}\right)  ^{q-2}\langle M_{r}-Y_{r},\left(
R_{r}-Z_{r}\right)  dB_{r}\rangle
\]
is a continuous martingale.

If moreover
\[
\mathbb{E}\left(
{\displaystyle\int_{0}^{T}}
{e^{V_{r}}}\left\vert H\left(  r,Y_{r},Z_{r}\right)  \right\vert
dQ_{r}\right)  ^{p\wedge2}<\infty,
\]
then, for all $T\geq0,$ $M\in\mathcal{V}_{m}^{q}$ of the form (\ref{def_M}),
and for all stopping times $0\leq\sigma\leq\theta\leq T:$%
\begin{equation}%
\begin{array}
[c]{l}%
\displaystyle e^{qV_{\sigma}}\left(  \Gamma_{\sigma}\right)  ^{q}%
+q\,\mathbb{E}^{\mathcal{F}_{\sigma}}%
{\displaystyle\int_{\sigma}^{\theta}}
e^{qV_{r}}\left(  \Gamma_{r}\right)  ^{q}dV_{r}+\dfrac{q\left(  q-1\right)
}{2}\,\mathbb{E}^{\mathcal{F}_{\sigma}}%
{\displaystyle\int_{\sigma}^{\theta}}
e^{qV_{r}}{\left(  \Gamma_{r}\right)  ^{q-2}}\left\vert R_{r}-Z_{r}\right\vert
^{2}dr\medskip\\
\displaystyle\quad+{q\,\mathbb{E}^{\mathcal{F}_{\sigma}}%
{\displaystyle\int_{\sigma}^{\theta}}
e^{pV_{r}}\left(  \Gamma_{r}\right)  ^{q-2}\Psi}\left(  r,Y_{r}\right)
dQ_{r}\medskip\\
\displaystyle\leq\mathbb{E}^{\mathcal{F}_{\sigma}}e^{qV_{\theta}}{\left(
\Gamma_{\theta}\right)  ^{q}}+{q\,\mathbb{E}^{\mathcal{F}_{\sigma}}%
{\displaystyle\int_{\sigma}^{\theta}}
e^{qV_{r}}\left(  \Gamma_{r}\right)  ^{q-2}\Psi}\left(  r,M_{r}\right)
dQ_{r}\medskip\\
\displaystyle\quad+{q\,\mathbb{E}^{\mathcal{F}_{\sigma}}%
{\displaystyle\int_{\sigma}^{\theta}}
e^{qV_{r}}\,\left(  \Gamma_{r}\right)  ^{q-2}}\langle M_{r}-Y_{r}%
,N_{r}-H\left(  r,Y_{r},Z_{r}\right)  \rangle dQ_{r},\quad\mathbb{P}%
\text{--a.s..}%
\end{array}
\label{def1d}%
\end{equation}

\end{proposition}

\begin{proof}
We have%
\[%
\begin{array}
[c]{l}%
\displaystyle\mathbb{E}\left[
{\displaystyle\int_{0}^{T}}
{e^{2qV_{r}}}\left(  \Gamma_{r}\right)  ^{2q-4}\,\left\vert M_{r}%
-Y_{r}\right\vert ^{2}\left\vert R_{r}-Z_{r}\right\vert ^{2}dr\right]
^{1/2}\leq\mathbb{E}\left[
{\displaystyle\int_{0}^{T}}
{e^{2qV_{r}}}\left(  \Gamma_{r}\right)  ^{2q-2}\,\left\vert R_{r}%
-Z_{r}\right\vert ^{2}dr\right]  ^{1/2}\medskip\\
\displaystyle\leq\mathbb{E}\left[  \sup\nolimits_{r\in\left[  0,T\right]
}\Big({e^{\left(  q-1\right)  V_{r}}}\big(\left\vert M_{r}-Y_{r}\right\vert
^{2}+\delta_{q}\big)^{\left(  q-1\right)  /2}\Big)\cdot\left(
{\displaystyle\int_{0}^{T}}
{e^{2V_{r}}}\left\vert R_{r}-Z_{r}\right\vert ^{2}dr\right)  ^{1/2}\right]
\medskip\\
\displaystyle\leq\left[  \mathbb{E}\left(  \sup\nolimits_{r\in\left[
0,T\right]  }{e^{qV_{r}}}\left(  \left\vert M_{r}-Y_{r}\right\vert ^{2}%
+\delta_{q}\right)  ^{q/2}\right)  \right]  ^{\left(  q-1\right)  /q}\left[
\mathbb{E}\left(
{\displaystyle\int_{0}^{T}}
{e^{2V_{r}}}\left\vert R_{r}-Z_{r}\right\vert ^{2}dr\right)  ^{q/2}\right]
^{1/q}\medskip\\
\displaystyle<\infty,
\end{array}
\]
since, from (\ref{exp-VT}),%
\[
\left(  \delta_{q}\right)  ^{q/2}\,\mathbb{E}\left(  \sup\nolimits_{r\in
\left[  0,T\right]  }{e^{qV_{r}}}\right)  <\infty,\quad\text{for all }T>0
\]
and inequalities (\ref{def_M_2}),(\ref{def_M_3}), (\ref{def0-1}) and
(\ref{def0-2}) hold.

Consequently, the stochastic integral $t\mapsto%
{\displaystyle\int_{0}^{t}}
{e^{qV_{r}}}\left(  \Gamma_{r}\right)  ^{q-2}\langle M_{r}-Y_{r},\left(
R_{r}-Z_{r}\right)  dB_{r}\rangle$ is a continuous martingale.

We also have%
\begin{align*}
&  \mathbb{E}%
{\displaystyle\int_{0}^{T}}
e^{qV_{r}}\left(  \Gamma_{r}\right)  ^{q-2}\left\vert \langle M_{r}%
-Y_{r},N_{r}-H\left(  r,Y_{r},Z_{r}\right)  \rangle\right\vert dQ_{r}\\
&  \leq\mathbb{E}%
{\displaystyle\int_{0}^{T}}
e^{qV_{r}}\left(  \Gamma_{r}\right)  ^{q-1}\left[  \left\vert N_{r}\right\vert
+\left\vert H\left(  r,Y_{r},Z_{r}\right)  \right\vert \right]  dQ_{r}\\
&  \leq\mathbb{E}\left[  \sup\nolimits_{r\in\left[  0,T\right]  }%
\Big({e^{\left(  q-1\right)  V_{r}}}\left(  \left\vert M_{r}-Y_{r}\right\vert
^{2}+\delta_{q}\right)  ^{\left(  q-1\right)  /2}\Big)\cdot\left(
{\displaystyle\int_{0}^{T}}
{e^{V_{r}}}\left[  \left\vert N_{r}\right\vert +\left\vert H\left(
r,Y_{r},Z_{r}\right)  \right\vert \right]  dQ_{r}\right)  \right] \\
&  \leq\left[  \mathbb{E}\left(  \sup\nolimits_{r\in\left[  0,T\right]
}{e^{qV_{r}}}\left(  \left\vert M_{r}-Y_{r}\right\vert ^{2}+\delta_{q}\right)
^{q/2}\right)  \right]  ^{\left(  q-1\right)  /q}\left[  \mathbb{E}\left(
{\displaystyle\int_{0}^{T}}
{e^{V_{r}}}\left[  \left\vert N_{r}\right\vert +\left\vert H\left(
r,Y_{r},Z_{r}\right)  \right\vert \right]  dQ_{r}dr\right)  ^{q}\right]
^{1/q}\\
&  <\infty.
\end{align*}
Hence, using inequality (\ref{def1-b}) with $L=V,$ inequality (\ref{def1d})
follows.\hfill
\end{proof}

\begin{remark}
\label{s-w}From Section \ref{intuitive_intr} (see the proof of inequality
(\ref{def1b})) we see that a strong solution $\left(  Y,Z\right)  \in
S_{m}^{0}\times\Lambda_{m\times k}^{0}$ of BSDE (\ref{GBSVI 2}), such that
(\ref{def0-1}), (\ref{def0-2}) and (\ref{def2}) are satisfied, is also an
$L^{p}$--variational solution. Conversely, we have:
\end{remark}

\begin{corollary}
\label{c1-weak sol to strong sol}If $\left(  Y,Z\right)  $ is an $L^{p}%
$--variational solution of BSDE (\ref{GBSVI 2}) with $\varphi=\psi=0$, $V$ is
a continuous nondecreasing process and
\[
\mathbb{E}\left(
{\displaystyle\int_{0}^{T}}
{e^{V_{r}}}\left\vert H\left(  r,Y_{r},Z_{r}\right)  \right\vert
dQ_{r}\right)  ^{p\wedge2}<\infty,\quad\text{for all }T>0,
\]
then $\left(  Y,Z\right)  $ is a strong solution of BSDE (\ref{GBSVI 2}).
\end{corollary}

\begin{proof}
By \cite[Theorem 2.42, Corollary 2.45]{pa-ra/14} there exists a unique
$\left(  M,R\right)  \in S_{m}^{q}\left[  0,T\right]  \times\Lambda_{m\times
k}^{q}\left(  0,T\right)  $ such that%
\[
M_{t}=Y_{T}+%
{\displaystyle\int_{t}^{T}}
H\left(  r,Y_{r},Z_{r}\right)  dQ_{r}-%
{\displaystyle\int_{t}^{T}}
R_{r}dB_{r}%
\]
and%
\[
\mathbb{E}\sup\nolimits_{r\in\left[  0,T\right]  }{e^{qV_{r}}}\left\vert
M_{r}\right\vert ^{q}+\mathbb{E}\left(
{\displaystyle\int_{0}^{T}}
e^{2V_{r}}\left\vert R_{r}\right\vert ^{2}dr\right)  ^{\frac{q}{2}}\leq
C_{q}\,\mathbb{E}\left[  e^{qV_{T}}\left\vert Y_{T}\right\vert ^{q}+\left(
{\displaystyle\int_{0}^{T}}
{e^{V_{r}}}\left\vert H\left(  r,Y_{r},Z_{r}\right)  \right\vert
dQ_{r}\right)  ^{q}\right]  .
\]
with $q=p\wedge2.$ With this $M$ inequality (\ref{def1-b}) becomes (since
${\Psi=0}$) $\mathbb{P}$--a.s.%
\[%
\begin{array}
[c]{l}%
e^{qV_{t}}\left(  \Gamma_{t}\right)  ^{q}+q\,%
{\displaystyle\int_{t}^{s}}
e^{qV_{r}}\left(  \Gamma_{r}\right)  ^{q}dV_{r}+\dfrac{q}{2}\,n_{q}\,%
{\displaystyle\int_{t}^{s}}
{e^{qV_{r}}\left(  \Gamma_{r}\right)  ^{q-2}}\left\vert R_{r}-Z_{r}\right\vert
^{2}dr\medskip\\
\leq e^{qV_{s}}\left(  \Gamma_{s}\right)  ^{q}-q\,{%
{\displaystyle\int_{t}^{s}}
e^{qV_{r}}}\left(  \Gamma_{r}\right)  ^{q-2}\langle M_{r}-Y_{r},\left(
R_{r}-Z_{r}\right)  dB_{r}\rangle,
\end{array}
\]
for any $\delta\in(0,1]$ and $0\leq t\leq s<\infty.$

By Proposition \ref{prop2-exp-exp}, the stochastic integral is a martingale
and therefore we obtain, using the last inequality with $s=T,$ for all $0\leq
t\leq T,$ $\mathbb{P}$--a.s.%
\begin{equation}
e^{qV_{t}}\left(  \Gamma_{t}\right)  ^{q}+\dfrac{q}{2}\,n_{q}\,\mathbb{E}%
^{\mathcal{F}_{t}}{%
{\displaystyle\int_{t}^{T}}
}e^{qV_{r}}\,\frac{1}{\big(\left\vert M_{r}-Y_{r}\right\vert ^{2}%
+1\big)^{\left(  2-q\right)  /2}}\,\left\vert R_{r}-Z_{r}\right\vert
^{2}dr\leq\left(  \delta_{q}\right)  ^{q/2}\,\mathbb{E}^{\mathcal{F}_{t}%
}e^{qV_{T}}, \label{w-to-s}%
\end{equation}
since $M_{T}=Y_{T}$ and $V$ is nondecreasing.

Passing to limit as $\delta\rightarrow0_{+}$ we obtain, by Fatou's Lemma, for
all $0\leq t\leq T,$ $\mathbb{P}$--a.s.%
\[
e^{qV_{t}}\left\vert M_{t}-Y_{t}\right\vert ^{q}+\dfrac{q\left(  q-1\right)
}{2}\,\mathbb{E}^{\mathcal{F}_{t}}{%
{\displaystyle\int_{t}^{T}}
}e^{qV_{r}}\,\frac{1}{\big(\left\vert M_{r}-Y_{r}\right\vert ^{2}%
+1\big)^{\left(  2-q\right)  /2}}\,\left\vert R_{r}-Z_{r}\right\vert
^{2}dr=0,
\]
which clearly yields $\left(  M,R\right)  =\left(  Y,Z\right)  $ in $S_{m}%
^{0}\left[  0,T\right]  \times\Lambda_{m\times k}^{q}\left(  0,T\right)  ,$
hence%
\[
Y_{t}=Y_{T}+%
{\displaystyle\int_{t}^{T}}
H\left(  r,Y_{r},Z_{r}\right)  dQ_{r}-%
{\displaystyle\int_{t}^{T}}
Z_{r}dB_{r}\,.
\]
\hfill
\end{proof}

\begin{proposition}
\label{p-estim}Let $M\in\mathcal{V}_{m}^{0}$ of the form (\ref{def_M}). Let
$Y:\Omega\times\mathbb{R}_{+}\rightarrow\mathbb{R}^{m}$ and $Z:\Omega
\times\mathbb{R}_{+}\rightarrow\mathbb{R}^{m\times k}$ be two p.m.s.p. such
that $Y$ has continuous trajectories and $\mathbb{P}$--a.s.,%
\[%
\begin{array}
[c]{rl}%
\left(  i\right)  &
{\displaystyle\int_{0}^{T}}
{e^{2V_{r}}}\left\vert R_{r}-Z_{r}\right\vert ^{2}dr+%
{\displaystyle\int_{0}^{T}}
e^{2V_{r}}{\Psi}\left(  r,Y_{r}\right)  dQ_{r}<\infty,\quad\text{for all
}T>0,\medskip\\
\left(  ii\right)  & {\Psi}\left(  r,M_{r}\right)  \leq\mathbf{1}_{q\geq
2}{\Psi}\left(  r,M_{r}\right)  ,\medskip\\
\left(  iii\right)  & \left\langle M_{r}-Y_{r},N_{r}\right\rangle dQ_{r}%
\leq\left\vert M_{r}-Y_{r}\right\vert dL_{r}\,,\quad\text{a.e. }r\in\left[
0,T\right]  ,
\end{array}
\]
with $L$ an increasing and continuous p.m.s.p. with $L_{0}=0.\medskip$

\noindent\textbf{I.} If inequality (\ref{def1}) holds for $q=2,$ then, for all
$a>0$ and for any stopping times $0\leq\sigma\leq\theta<\infty,$%
\begin{equation}%
\begin{array}
[c]{l}%
\mathbb{E}^{\mathcal{F}_{\sigma}}\bigg(%
{\displaystyle\int_{\sigma}^{\theta}}
{e^{2V_{r}}}\left\vert R_{r}-Z_{r}\right\vert ^{2}dr\bigg)^{a/2}%
+\mathbb{E}^{\mathcal{F}_{\sigma}}\bigg(%
{\displaystyle\int_{\sigma}^{\theta}}
e^{2V_{r}}{\Psi}\left(  r,Y_{r}\right)  dQ_{r}\bigg)^{a/2}\medskip\\
\leq C_{a,\lambda}\,\mathbb{E}^{\mathcal{F}_{\sigma}}\bigg[\sup\nolimits_{r\in
\left[  \sigma,\theta\right]  }e^{aV_{r}}\left\vert M_{r}-Y_{r}\right\vert
^{a}+\bigg(%
{\displaystyle\int_{\sigma}^{\theta}}
e^{V_{r}}{\Psi}\left(  r,M_{r}\right)  dQ_{r}\bigg)^{a/2}\medskip\\
\quad+\bigg(%
{\displaystyle\int_{\sigma}^{\theta}}
e^{V_{r}}\left\vert M_{r}-Y_{r}\right\vert \left[  dL_{r}+\left\vert H\left(
r,M_{r},R_{r}\right)  \right\vert dQ_{r}\right]  \bigg)^{a/2}\bigg]\medskip\\
\leq2C_{a,\lambda}\,\mathbb{E}^{\mathcal{F}_{\sigma}}\bigg[\sup\nolimits_{r\in
\left[  \sigma,\theta\right]  }e^{aV_{r}}\left\vert M_{r}-Y_{r}\right\vert
^{a}+\bigg(%
{\displaystyle\int_{\sigma}^{\theta}}
e^{V_{r}}{\Psi}\left(  r,M_{r}\right)  dQ_{r}\bigg)^{a/2}\medskip\\
\quad+\bigg(%
{\displaystyle\int_{\sigma}^{\theta}}
e^{V_{r}}\left[  dL_{r}+\left\vert H\left(  r,M_{r},R_{r}\right)  \right\vert
dQ_{r}\right]  \bigg)^{a}\bigg],\quad\mathbb{P}\text{--a.s..}%
\end{array}
\label{def-11}%
\end{equation}
In particular, for $\gamma=0,N=0,R=0$ (hence $M=0$) and $L=0,$ it follows%
\begin{equation}%
\begin{array}
[c]{l}%
\mathbb{E}^{\mathcal{F}_{\sigma}}\bigg(%
{\displaystyle\int_{\sigma}^{\theta}}
{e^{2V_{r}}}\left\vert Z_{r}\right\vert ^{2}dr\bigg)^{a/2}+\mathbb{E}%
^{\mathcal{F}_{\sigma}}\bigg(%
{\displaystyle\int_{\sigma}^{\theta}}
e^{2V_{r}}{\Psi}\left(  r,Y_{r}\right)  dQ_{r}\bigg)^{a/2}\medskip\\
\leq C_{a,\lambda}\,\mathbb{E}^{\mathcal{F}_{\sigma}}\left[  \sup
\nolimits_{r\in\left[  \sigma,\theta\right]  }e^{aV_{r}}\left\vert
Y_{r}\right\vert ^{a}+\bigg(%
{\displaystyle\int_{\sigma}^{\theta}}
e^{V_{r}}\left\vert Y_{r}\right\vert \left\vert H\left(  r,0,0\right)
\right\vert dQ_{r}\bigg)^{a/2}\right]  \medskip\\
\leq2C_{a,\lambda}\,\mathbb{E}^{\mathcal{F}_{\sigma}}\left[  \mathbb{E}%
^{\mathcal{F}_{\sigma}}\sup\nolimits_{r\in\left[  \sigma,\theta\right]
}e^{aV_{r}}\left\vert Y_{r}\right\vert ^{a}+\bigg(%
{\displaystyle\int_{\sigma}^{\theta}}
e^{V_{r}}\left\vert H\left(  r,0,0\right)  \right\vert dQ_{r}\bigg)^{a}%
\right]  ,\quad\mathbb{P}\text{--a.s..}%
\end{array}
\label{def-11aa}%
\end{equation}
\noindent\textbf{II.} If inequality (\ref{def1}) holds and for some fixed
stopping times $0\leq\sigma\leq\theta<\infty,$ $1<q\leq a$
\begin{equation}
\mathbb{E}\left(  \sup\nolimits_{r\in\left[  \sigma,\theta\right]  }%
{e^{aV_{r}}}\left\vert M_{r}-Y_{r}\right\vert ^{a}\right)  <\infty,
\label{def-11a}%
\end{equation}
then%
\begin{equation}%
\begin{array}
[c]{l}%
\mathbb{E}^{\mathcal{F}_{\sigma}}\sup\nolimits_{r\in\left[  \sigma
,\theta\right]  }e^{aV_{r}}\left\vert M_{r}-Y_{r}\right\vert ^{a}\\
\leq C_{\lambda,q,a}\,\mathbb{E}^{\mathcal{F}_{\sigma}}\bigg[e^{aV_{\theta}%
}\left\vert M_{\theta}-Y_{\theta}\right\vert ^{a}+\bigg(%
{\displaystyle\int_{\sigma}^{\theta}}
e^{V_{r}}\left\vert M_{r}-Y_{r}\right\vert ^{q-2}{\Large \,}\mathbf{1}%
_{q\geq2}\,{\Psi}\left(  r,M_{r}\right)  dQ_{r}\bigg)^{a/q}\medskip\\
\quad+\bigg(%
{\displaystyle\int_{\sigma}^{\theta}}
{e^{qV_{r}}}\left\vert M_{r}-Y_{r}\right\vert ^{q-1}\left[  dL_{r}+\left\vert
H\left(  r,M_{r},R_{r}\right)  \right\vert dQ_{r}\right]  \bigg)^{a/q}%
\bigg],\quad\mathbb{P}\text{--a.s.}%
\end{array}
\label{def-11b}%
\end{equation}
and%
\begin{equation}%
\begin{array}
[c]{l}%
\mathbb{E}^{\mathcal{F}_{\sigma}}\Big(\sup\nolimits_{r\in\left[  \sigma
,\theta\right]  }e^{aV_{r}}\left\vert M_{r}-Y_{r}\right\vert ^{a}%
\Big)+\mathbb{E}^{\mathcal{F}_{\sigma}}\bigg(%
{\displaystyle\int_{\sigma}^{\theta}}
{e^{qV_{r}}\left\vert M_{r}-Y_{r}\right\vert ^{q-2}}{\Large \,}\mathbf{1}%
_{M_{r}\neq Y_{r}}\,\left\vert R_{r}-Z_{r}\right\vert ^{2}dr\bigg)^{a/q}%
\medskip\\
\quad+\mathbb{E}^{\mathcal{F}_{\sigma}}\bigg(%
{\displaystyle\int_{\sigma}^{\theta}}
{e^{qV_{r}}\left\vert M_{r}-Y_{r}\right\vert ^{q-2}{\Large \,}\mathbf{1}%
_{M_{r}\neq Y_{r}}\,\Psi}\left(  r,Y_{r}\right)  dQ_{r}\bigg)^{a/q}\medskip\\
\leq C_{\lambda,q,a}\,\mathbb{E}^{\mathcal{F}_{\sigma}}\bigg[e^{aV_{\theta}%
}\left\vert M_{\theta}-Y_{\theta}\right\vert ^{a}+\bigg(%
{\displaystyle\int_{\sigma}^{\theta}}
e^{V_{r}}\,\mathbf{1}_{q\geq2}\,{\Psi}\left(  r,M_{r}\right)  dQ_{r}%
\bigg)^{a/2}\medskip\\
\quad+\bigg(%
{\displaystyle\int_{\sigma}^{\theta}}
e^{V_{r}}\left[  dL_{r}+\left\vert H\left(  r,M_{r},R_{r}\right)  \right\vert
dQ_{r}\right]  \bigg)^{a}\bigg],\quad\mathbb{P}\text{-a.s..}%
\end{array}
\label{def-11c}%
\end{equation}
In particular, for $\gamma=0,N=0,R=0$ (hence $M=0$) and $L=0,$ it follows%
\begin{equation}
\mathbb{E}^{\mathcal{F}_{\sigma}}\sup\nolimits_{r\in\left[  \sigma
,\theta\right]  }e^{aV_{r}}\left\vert Y_{r}\right\vert ^{a}\leq C_{\lambda
,q,a}\,\mathbb{E}^{\mathcal{F}_{\sigma}}\left[  e^{aV_{\theta}}\left\vert
Y_{\theta}\right\vert ^{a}+\bigg(%
{\displaystyle\int_{\sigma}^{\theta}}
{e^{qV_{r}}}\left\vert Y_{r}\right\vert ^{q-1}\left\vert H\left(
r,0,0\right)  \right\vert dQ_{r}\bigg)^{a/q}\right]  \label{def-11cc}%
\end{equation}
and%
\begin{equation}%
\begin{array}
[c]{l}%
\displaystyle\mathbb{E}^{\mathcal{F}_{\sigma}}\Big(\sup\nolimits_{r\in\left[
\sigma,\theta\right]  }e^{aV_{r}}\left\vert Y_{r}\right\vert ^{a}%
\Big)+\mathbb{E}^{\mathcal{F}_{\sigma}}\bigg(%
{\displaystyle\int_{\sigma}^{\theta}}
{e^{qV_{r}}\left\vert Y_{r}\right\vert ^{q-2}}{\Large \,}\mathbf{1}_{M_{r}\neq
Y_{r}}\,\left\vert Z_{r}\right\vert ^{2}dr\bigg)^{a/q}\medskip\\
\displaystyle\quad+\mathbb{E}^{\mathcal{F}_{\sigma}}\bigg(%
{\displaystyle\int_{\sigma}^{\theta}}
{e^{qV_{r}}\left\vert Y_{r}\right\vert ^{q-2}{\Large \,}\mathbf{1}_{M_{r}\neq
Y_{r}}\,\Psi}\left(  r,Y_{r}\right)  dQ_{r}\bigg)^{a/q}\medskip\\
\displaystyle\leq C_{\lambda,q,a}\,\mathbb{E}^{\mathcal{F}_{\sigma}}\left[
e^{aV_{\theta}}\left\vert Y_{\theta}\right\vert ^{a}+\bigg(%
{\displaystyle\int_{\sigma}^{\theta}}
e^{V_{r}}\left\vert H\left(  r,0,0\right)  \right\vert dQ_{r}\bigg)^{a}%
\right]  .
\end{array}
\label{def-11ccc}%
\end{equation}

\end{proposition}

\begin{proof}
Using the monotonicity of $H$ we have%
\[%
\begin{array}
[c]{l}%
\displaystyle\left\langle M_{r}-Y_{r},-H\left(  r,Y_{r},Z_{r}\right)
dQ_{r}\right\rangle \medskip\\
\displaystyle=\left\langle M_{r}-Y_{r},-H\left(  r,M_{r},R_{r}\right)
dQ_{r}\right\rangle +\left\langle M_{r}-Y_{r},H\left(  r,M_{r},R_{r}\right)
-H\left(  r,Y_{r},Z_{r}\right)  dQ_{r}\right\rangle \medskip\\
\displaystyle\leq|M_{r}-Y_{r}|\left\vert H\left(  r,M_{r},R_{r}\right)
\right\vert dQ_{r}+|M_{r}-Y_{r}|^{2}dV_{r}+\dfrac{n_{p}\lambda}{2}\,\left\vert
R_{r}-Z_{r}\right\vert ^{2}ds\medskip\\
\displaystyle=|M_{r}-Y_{r}|\left\vert H\left(  r,M_{r},R_{r}\right)
\right\vert dQ_{r}+\left(  \Gamma_{r}\right)  ^{2}dV_{r}-\delta_{q}%
\,dV_{r}+\dfrac{n_{p}\lambda}{2}\,\left\vert R_{r}-Z_{r}\right\vert ^{2}ds.
\end{array}
\]
If we suppose that (\ref{def1}) is satisfied, then%
\[%
\begin{array}
[c]{l}%
\displaystyle\left(  \Gamma_{t}\right)  ^{q}+\dfrac{q\left(  n_{q}%
-n_{p}\lambda\right)  }{2}%
{\displaystyle\int_{t}^{s}}
{\left(  \Gamma_{r}\right)  ^{q-2}}\left\vert R_{r}-Z_{r}\right\vert
^{2}dr+q\delta_{q}%
{\displaystyle\int_{t}^{s}}
\left(  \Gamma_{r}\right)  ^{q-2}dV_{r}+{q%
{\displaystyle\int_{t}^{s}}
\left(  \Gamma_{r}\right)  ^{q-2}\Psi}\left(  r,Y_{r}\right)  dQ_{r}\medskip\\
\displaystyle\leq\left(  \Gamma_{s}\right)  ^{q}+q%
{\displaystyle\int_{t}^{s}}
\left(  \Gamma_{r}\right)  ^{q}dV_{r}+{q%
{\displaystyle\int_{t}^{s}}
\left(  \Gamma_{r}\right)  ^{q-2}\mathbf{1}_{q\geq2}{\Psi}\left(
r,M_{r}\right)  }dQ_{r}\medskip\\
\displaystyle\quad+q%
{\displaystyle\int_{t}^{s}}
\left(  \Gamma_{r}\right)  ^{q-2}\left\vert M_{r}-Y_{r}\right\vert \left[
dL_{r}+\left\vert H\left(  r,M_{r},R_{r}\right)  \right\vert dQ_{r}\right]  -q%
{\displaystyle\int_{t}^{s}}
\left(  \Gamma_{r}\right)  ^{q-2}\,\langle M_{r}-Y_{r},\left(  R_{r}%
-Z_{r}\right)  dB_{r}\rangle.
\end{array}
\]
for all $0\leq t\leq s<\infty.$

Since $q\in\left\{  2,p\wedge2\right\}  ,$ we have (see Remark \ref{r-nq})%
\[
n_{q}-n_{p}\lambda\geq n_{q}\left(  1-\lambda\right)  =\left(  q-1\right)
\left(  1-\lambda\right)
\]
and therefore we deduce, using a Gronwall's type stochastic inequality for the
previous inequality (see, for instance, \cite[Lemma 12]{ma-ra/07}),%
\begin{equation}%
\begin{array}
[c]{l}%
\displaystyle e^{qV_{t}}\left(  \Gamma_{t}\right)  ^{q}+\dfrac{q}{2}%
\,n_{q}\left(  1-\lambda\right)
{\displaystyle\int_{t}^{s}}
e^{qV_{r}}\left(  \Gamma_{r}\right)  {^{q-2}}\left\vert R_{r}-Z_{r}\right\vert
^{2}dr+q\,\delta_{q}%
{\displaystyle\int_{t}^{s}}
e^{qV_{r}}\left(  \Gamma_{r}\right)  ^{q-2}dV_{r}\medskip\\
\displaystyle\quad+{q%
{\displaystyle\int_{t}^{s}}
e^{qV_{r}}\left(  \Gamma_{r}\right)  ^{q-2}\Psi}\left(  r,Y_{r}\right)
dQ_{r}\medskip\\
\displaystyle\leq e^{qV_{s}}\left(  \Gamma_{s}\right)  ^{q}+{q%
{\displaystyle\int_{t}^{s}}
e^{qV_{r}}\left(  \Gamma_{r}\right)  ^{q-2}\,\mathbf{1}_{q\geq2}\,{\Psi
}\left(  r,M_{r}\right)  }dQ_{r}\medskip\\
\displaystyle\quad+q%
{\displaystyle\int_{t}^{s}}
e^{qV_{r}}\left(  \Gamma_{r}\right)  ^{q-2}\left\vert M_{r}-Y_{r}\right\vert
\left[  dL_{r}+\left\vert H\left(  r,M_{r},R_{r}\right)  \right\vert
dQ_{r}\right]  \medskip\\
\displaystyle\quad-q%
{\displaystyle\int_{t}^{s}}
e^{qV_{r}}\left(  \Gamma_{r}\right)  ^{q-2}\,\langle M_{r}-Y_{r},\left(
R_{r}-Z_{r}\right)  dB_{r}\rangle.
\end{array}
\label{def1-e}%
\end{equation}
\noindent\textbf{I.} Writing (\ref{def1-e}) for $q=2$ we get%
\[%
\begin{array}
[c]{l}%
\displaystyle e^{2V_{t}}\left\vert M_{t}-Y_{t}\right\vert ^{2}+\left(
1-\lambda\right)  {%
{\displaystyle\int_{t}^{s}}
}e^{2V_{r}}\left\vert R_{r}-Z_{r}\right\vert ^{2}dr+{2%
{\displaystyle\int_{t}^{s}}
}e^{2V_{r}}{\Psi}\left(  r,Y_{r}\right)  dQ_{r}\medskip\\
\displaystyle\leq e^{2V_{s}}\left\vert M_{s}-Y_{s}\right\vert ^{2}+{2%
{\displaystyle\int_{t}^{s}}
e^{qV_{r}}{\Psi}\left(  r,M_{r}\right)  }dQ_{r}+{2%
{\displaystyle\int_{t}^{s}}
}e^{2V_{r}}\left\vert M_{r}-Y_{r}\right\vert \left[  dL_{r}+\left\vert
H\left(  r,M_{r},R_{r}\right)  \right\vert dQ_{r}\right]  \medskip\\
\displaystyle\quad-2%
{\displaystyle\int_{t}^{s}}
\,e^{2V_{r}}\langle M_{r}-Y_{r},\left(  R_{r}-Z_{r}\right)  dB_{r}%
\rangle,\;\mathbb{P}\text{--a.s.,}%
\end{array}
\]
for all $0\leq t\leq s<\infty,$ which yields (\ref{def-11}), by Proposition
\ref{an-prop-dz} from the Appendix.$\bigskip$

\noindent\textbf{II.} Using Fatou's Lemma, Lebesgue dominated convergence
theorem and the continuity in pro\-ba\-bi\-lity of the stochastic integral we
clearly deduce from (\ref{def1-e}), as $\delta\rightarrow0_{+}$, that:%
\begin{equation}%
\begin{array}
[c]{l}%
\displaystyle e^{qV_{t}}\left\vert M_{t}-Y_{t}\right\vert ^{q}+\dfrac{q}%
{2}\,\left(  q-1\right)  \left(  1-\lambda\right)  {\int\nolimits_{t}^{s}%
}e^{qV_{r}}\left\vert M_{r}-Y_{r}\right\vert ^{q-2}{\Large \,}\mathbf{1}%
_{M_{r}\neq Y_{r}}\,\left\vert R_{r}-Z_{r}\right\vert ^{2}dr\medskip\\
\displaystyle\quad+{q\int\nolimits_{t}^{s}e^{qV_{r}}\left\vert M_{r}%
-Y_{r}\right\vert ^{q-2}{\Large \,}\mathbf{1}_{M_{r}\neq Y_{r}}\,\Psi}\left(
r,Y_{r}\right)  dQ_{r}\medskip\\
\displaystyle\leq e^{qV_{s}}\left\vert M_{s}-Y_{s}\right\vert ^{q}%
+{q\int\nolimits_{t}^{s}e^{qV_{r}}\left\vert M_{r}-Y_{r}\right\vert
^{q-2}{\Large \,}\mathbf{1}_{q\geq2}\,{\Psi}\left(  r,M_{r}\right)  }%
dQ_{r}\medskip\\
\displaystyle\quad+q{\int\nolimits_{t}^{s}}e^{qV_{r}}\left\vert M_{r}%
-Y_{r}\right\vert ^{q-1}\left[  dL_{r}+\left\vert H\left(  r,M_{r}%
,R_{r}\right)  \right\vert dQ_{r}\right]  \medskip\\
\displaystyle\quad-q{\int\nolimits_{t}^{s}}e^{qV_{r}}\left\vert M_{r}%
-Y_{r}\right\vert ^{q-2}\langle M_{r}-Y_{r},\left(  R_{r}-Z_{r}\right)
dB_{r}\rangle,
\end{array}
\label{def1-f}%
\end{equation}
since $\Gamma_{r}\xlongrightarrow[\;\;\delta\rightarrow0_{+}\;\;]{}\left\vert
M_{r}-Y_{r}\right\vert {\Large \,}\mathbf{1}_{M_{r}\neq Y_{r}}\,,$
$\mathbb{P}$--a.s., for all $r\geq0.\medskip$

Using Proposition \ref{an-prop-ydz} and Remark \ref{an-prop-ydz_2} we get
(\ref{def-11b}) and (\ref{def-11c}).\hfill
\end{proof}

\section{Uniqueness and Continuity of $L^{p}$--variational solutions}

\begin{theorem}
[Continuity]\label{uniq}We suppose that assumptions $(\mathrm{A}%
_{1}-\mathrm{A}_{6})$ are satisfied. Let $(\hat{Y},\hat{Z}),(\tilde{Y}%
,\tilde{Z})$ be two $L^{p}-$variational solutions of (\ref{GBSVI 2})
corresponding to $(\hat{\eta},\hat{H})$ and $(\tilde{\eta},\tilde{H})$
respectively (it is sufficient to consider Definition
\ref{definition_weak solution} only for $q=p\wedge2\,$), where $\hat{H}$ and
$\tilde{H}$ have the same coefficients $\mu,\nu,\ell.$

Then, for any stopping time $0\leq\sigma\leq\tau,$ it holds, $\mathbb{P}%
$--a.s.,%
\begin{equation}%
\begin{array}
[c]{l}%
e^{qV_{\sigma}}|\hat{Y}_{\sigma}-\tilde{Y}_{\sigma}|^{q}+c_{q,\lambda
}\,{\mathbb{E}}^{\mathcal{F}_{\sigma}}{%
{\displaystyle\int_{\sigma}^{\tau}}
}e^{qV_{r}}|\hat{Y}_{r}-\tilde{Y}_{r}|^{q-2}{\Large \,}\mathbf{1}_{\hat{Y}%
_{r}\neq\tilde{Y}_{r}}\,|\hat{Z}_{r}-\tilde{Z}_{r}|^{2}dr\medskip\\
\leq{\mathbb{E}}^{\mathcal{F}_{\sigma}}e^{qV_{\tau}}|\hat{\eta}-\tilde{\eta
}|^{q}+q\,{\mathbb{E}}^{\mathcal{F}_{\sigma}}{%
{\displaystyle\int_{\sigma}^{\tau}}
}e^{qV_{r}}|\hat{Y}_{r}-\tilde{Y}_{r}|^{q-1}\big|\hat{H}(r,\hat{Y}_{r},\hat
{Z}_{r})-\tilde{H}(r,\hat{Y}_{r},\hat{Z}_{r})\big|dQ_{r}\medskip\\
\leq\mathbb{E}^{\mathcal{F}_{\sigma}}e^{qV_{\tau}}|\hat{\eta}-\tilde{\eta
}|^{q}\medskip\\
\quad+C_{q,\lambda}\,\big[{\mathbb{E}}^{\mathcal{F}_{\sigma}}\left(
\Lambda_{\sigma,\tau}\right)  \big]^{\left(  q-1\right)  /q}\cdot\left[
\mathbb{E}^{\mathcal{F}_{\sigma}}\left(  {%
{\displaystyle\int_{\sigma}^{\tau}}
}e^{V_{r}}\big|\hat{H}(r,\hat{Y}_{r},\hat{Z}_{r})-\tilde{H}(r,\hat{Y}_{r}%
,\hat{Z}_{r})\big|dQ_{r}\right)  ^{q}\right]  ^{1/q},
\end{array}
\label{cont-1}%
\end{equation}
where%
\begin{equation}
\Lambda_{\sigma,\tau}=e^{qV_{\tau}}|\hat{\eta}|^{q}+e^{qV_{\tau}}|\tilde{\eta
}|^{q}+\Big(%
{\displaystyle\int_{\sigma}^{\tau}}
e^{V_{r}}\big|\hat{H}\left(  r,0,0\right)  \big|dQ_{r}\Big)^{q}+\Big(%
{\displaystyle\int_{\sigma}^{\tau}}
e^{V_{r}}\big|\tilde{H}\left(  r,0,0\right)  \big|dQ_{r}\Big)^{q}
\label{def_M_4}%
\end{equation}
and $c_{q,\lambda},C_{q,\lambda}>0.\medskip$

Moreover, for all $0<\alpha<1,$%
\begin{equation}%
\begin{array}
[c]{l}%
\displaystyle\mathbb{E}\sup\nolimits_{t\in\left[  0,\tau\right]  }e^{\alpha
qV_{t}}|\hat{Y}_{t}-\tilde{Y}_{t}|^{\alpha q}+\left(  q-1\right)
\mathbb{E}\bigg(%
{\displaystyle\int_{0}^{\tau}}
\frac{1}{\big(e^{V_{r}}|\hat{Y}_{r}-\tilde{Y}_{r}|+1\big)^{2-q}}\,e^{2V_{r}%
}|\hat{Z}_{r}-\tilde{Z}_{r}|^{2}dr\bigg)^{\alpha}\medskip\\
\displaystyle\leq\mathbb{E}\sup\nolimits_{t\in\left[  0,\tau\right]
}e^{\alpha qV_{t}}|\hat{Y}_{t}-\tilde{Y}_{t}|^{\alpha q}+\left(  q-1\right)
\mathbb{E}\bigg(%
{\displaystyle\int_{0}^{\tau}}
e^{qV_{r}}|\hat{Y}_{r}-\tilde{Y}_{r}|^{q-2}\,|\hat{Z}_{r}-\tilde{Z}_{r}%
|^{2}dr\bigg)^{\alpha}\medskip\\
\displaystyle\leq C_{\alpha,q,\lambda}\,\big(\mathbb{E}e^{qV_{\tau}}|\hat
{\eta}-\tilde{\eta}|^{q}\big)^{\alpha}\medskip\\
\displaystyle\quad+C_{\alpha,q,\lambda}\,\big[{\mathbb{E}}\left(
\Lambda_{\sigma,\tau}\right)  \big]^{\alpha\left(  q-1\right)  /q}\cdot\left(
{\mathbb{E}}\bigg(%
{\displaystyle\int_{0}^{\tau}}
e^{V_{r}}\big|\hat{H}(r,\hat{Y}_{r},\hat{Z}_{r})-\tilde{H}(r,\hat{Y}_{r}%
,\hat{Z}_{r})\big|dQ_{r}\bigg)^{q}\right)  ^{\alpha/q},
\end{array}
\label{cont-2}%
\end{equation}
where $\Lambda_{\sigma,\tau}$ is given by (\ref{def_M_4}) and $C_{\alpha
,q,\lambda}>0.$
\end{theorem}

\begin{proof}
Let $M\in\mathcal{V}_{m}^{q}$ of the form (\ref{def_M}) and, following
(\ref{def_gamma}), we define%
\[
\hat{\Gamma}_{t}=\big(|M_{t}-\hat{Y}_{t}|^{2}+\delta_{q}\big)^{1/2}%
\quad\text{and}\quad\tilde{\Gamma}_{t}=\big(|M_{t}-\tilde{Y}_{t}|^{2}%
+\delta_{q}\big)^{1/2}.
\]
Let $T>0$ and the stopping times $0\leq\sigma\leq\theta\leq T\wedge\tau$ such
that
\[
\mathbb{E}\left(
{\displaystyle\int_{\sigma}^{\theta}}
{e^{V_{r}}}\left[  |\hat{H}(r,\hat{Y}_{r},\hat{Z}_{r})|+|\tilde{H}(r,\tilde
{Y}_{r},\tilde{Z}_{r})|\right]  dQ_{r}\right)  ^{q}<\infty
\]
From (\ref{def1d}) we deduce that%
\begin{equation}%
\begin{array}
[c]{l}%
\big[e^{qV_{\sigma}}(\hat{\Gamma}_{\sigma})^{q}+e^{qV_{\sigma}}(\tilde{\Gamma
}_{\sigma})^{q}\big]+q\,\mathbb{E}^{\mathcal{F}_{\sigma}}%
{\displaystyle\int_{\sigma}^{\theta}}
e^{qV_{r}}\big[(\hat{\Gamma}_{r})^{q}+(\tilde{\Gamma}_{r})^{q}\big]dV_{r}\\
\quad+\dfrac{q}{2}\,n_{q}\,\mathbb{E}^{\mathcal{F}_{\sigma}}{%
{\displaystyle\int_{\sigma}^{\theta}}
}e^{qV_{r}}\left[  (\hat{\Gamma}_{r})^{q-2}\big|R_{r}-\hat{Z}_{r}%
\big|^{2}+(\tilde{\Gamma}_{r})^{q-2}\big|R_{r}-\tilde{Z}_{r}\big|^{2}\right]
dr\medskip\\
\quad+{q\,\mathbb{E}^{\mathcal{F}_{\sigma}}%
{\displaystyle\int_{\sigma}^{\theta}}
}e^{qV_{r}}\big[(\hat{\Gamma}_{r})^{q-2}{\Psi}\big(r,\hat{Y}_{r}%
\big)+(\tilde{\Gamma}_{r})^{q-2}{\Psi}\big(r,\tilde{Y}_{r}\big)\big]dQ_{r}%
\medskip\\
\leq\mathbb{E}^{\mathcal{F}_{\sigma}}\big[e^{qV_{\theta}}(\hat{\Gamma}%
_{\theta})^{q}+e^{qV_{\theta}}(\tilde{\Gamma}_{\theta})^{q}%
\big]+{q\,\mathbb{E}^{\mathcal{F}_{\sigma}}%
{\displaystyle\int_{\sigma}^{\theta}}
}e^{qV_{r}}\big[(\hat{\Gamma}_{r})^{q-2}+(\tilde{\Gamma}_{r})^{q-2}\big]{\Psi
}\left(  r,M_{r}\right)  dQ_{r}\medskip\\
\quad+{q\,\mathbb{E}^{\mathcal{F}_{\sigma}}%
{\displaystyle\int_{\sigma}^{\theta}}
}e^{qV_{r}}(\hat{\Gamma}_{r})^{q-2}\langle M_{r}-\hat{Y}_{r},N_{r}-\hat
{H}(r,\hat{Y}_{r},\hat{Z}_{r})\rangle dQ_{r}\\
\quad+{q\,\mathbb{E}^{\mathcal{F}_{\sigma}}%
{\displaystyle\int_{\sigma}^{\theta}}
}e^{qV_{r}}(\tilde{\Gamma}_{r})^{q-2}\langle M_{r}-\tilde{Y}_{r},N_{r}%
-\tilde{H}(r,\tilde{Y}_{r},\tilde{Z}_{r})\rangle dQ_{r}\,,\quad\mathbb{P}%
\text{--a.s..}%
\end{array}
\label{uniq1}%
\end{equation}
Let%
\[
Y_{r}\xlongequal{\hspace{-4pt}{\rm def}\hspace{-4pt}}\dfrac{1}{2}\big(\hat
{Y}_{r}+\tilde{Y}_{r}\big).
\]
We have, for all $\beta>0,$ by Young's inequality,%
\begin{align*}
|M-\hat{Y}|^{2}  &  \leq\frac{1+\beta}{\beta}\,|M-Y|^{2}+\frac{1+\beta}%
{4}\,|\hat{Y}-\tilde{Y}|^{2},\\
|M-\tilde{Y}|^{2}  &  \leq\frac{1+\beta}{\beta}\,|M-Y|^{2}+\frac{1+\beta}%
{4}\,|\hat{Y}-\tilde{Y}|^{2}%
\end{align*}
and therefore%
\[%
\begin{array}
[c]{l}%
\displaystyle(\hat{\Gamma}_{r})^{q-2}\big|R_{r}-\hat{Z}_{r}\big|^{2}%
+(\tilde{\Gamma}_{r})^{q-2}\big|R_{r}-\tilde{Z}_{r}\big|^{2}\medskip\\
\displaystyle=\big(|M-\hat{Y}|^{2}+\delta_{q}\big)^{\left(  q-2\right)
/2}{\Large \,}\big|R-\hat{Z}\big|^{2}+\big(|M-\tilde{Y}|^{2}+\delta
_{q}\big)^{\left(  q-2\right)  /2}{\Large \,}\big|R-\tilde{Z}\big|^{2}%
\medskip\\
\displaystyle\geq\left[  \dfrac{1+\beta}{\beta}\,|M-Y|^{2}+\dfrac{1+\beta}%
{4}\,|\hat{Y}-\tilde{Y}|^{2}+\delta_{q}\right]  ^{\left(  q-2\right)
/2}\left[  \big|R-\hat{Z}\big|^{2}+\big|R-\tilde{Z}\big|^{2}\right]  ,
\end{array}
\]
since $1<q\leq2.$

Hence%
\begin{equation}%
\begin{array}
[c]{l}%
(\hat{\Gamma}_{r})^{q-2}\big|R_{r}-\hat{Z}_{r}\big|^{2}+(\tilde{\Gamma}%
_{r})^{q-2}\big|R_{r}-\tilde{Z}_{r}\big|^{2}\medskip\\
\geq\dfrac{1}{2}\left[  \dfrac{1+\beta}{\beta}\,|M-Y|^{2}+\dfrac{1+\beta}%
{4}\,|\hat{Y}-\tilde{Y}|^{2}+\delta_{q}\right]  ^{\left(  q-2\right)
/2}\big|\hat{Z}-\tilde{Z}\big|^{2}\medskip
\end{array}
\label{uniq1b}%
\end{equation}
Let $0<\varepsilon\leq1$ and%
\[
M_{t}^{\varepsilon}=\dfrac{1}{Q_{\varepsilon}}\,\mathbb{E}^{\mathcal{F}_{t}}%
{\displaystyle\int_{t\vee\varepsilon}^{\infty}}
e^{-\frac{Q_{r}-Q_{t\vee\varepsilon}}{Q_{\varepsilon}}}\,Y_{r}\,dQ_{r}\,,\quad
t\geq0.
\]
Then, by Proposition \ref{p1-approx}, $\left(  M^{\varepsilon},R^{\varepsilon
}\right)  \in S_{m}^{p}\times\Lambda_{m\times k}^{p}$ is the unique solution
of the BSDE:
\[
\left\{
\begin{array}
[c]{l}%
\displaystyle M_{t}^{\varepsilon}=M_{T}^{\varepsilon}+\frac{1}{Q_{\varepsilon
}}\,{\int_{t}^{T}}\mathbf{1}_{[\varepsilon,\infty)}\left(  r\right)  \left(
Y_{r}-M_{r}^{\varepsilon}\right)  dQ_{r}-{%
{\displaystyle\int_{t}^{T}}
}R_{r}^{\varepsilon}\,dB_{r}\,,\quad\text{for any }T>0,\quad t\in\left[
0,T\right]  ,\medskip\\
\displaystyle\lim\nolimits_{T\rightarrow\infty}\mathbb{E}\left\vert
M_{T}^{\varepsilon}-\xi_{T}\right\vert ^{p}=0
\end{array}
\right.
\]
and%
\[%
\begin{array}
[c]{ll}%
\left(  a\right)  & \left\vert M_{t}^{\varepsilon}\right\vert \leq
\mathbb{E}^{\mathcal{F}_{t}}\sup\nolimits_{r\geq0}\left\vert Y_{r}\right\vert
,\quad\mathbb{P}\text{--a.s., for all }t\geq0,\medskip\\
\left(  b\right)  & \lim\nolimits_{\varepsilon\rightarrow0}M_{t}^{\varepsilon
}=Y_{t}\,,\quad\mathbb{P}\text{--a.s.},\;\text{for all }t\geq0,\medskip\\
\left(  c\right)  & \lim\nolimits_{\varepsilon\rightarrow0}\mathbb{E}%
\sup\nolimits_{t\in\left[  0,T\right]  }\left\vert M_{t}^{\varepsilon}%
-Y_{t}\right\vert ^{p}=0,\quad\text{for all }T>0.
\end{array}
\]
We will replace now in (\ref{uniq1}) $M$ by $M^{\varepsilon},$ $R$ by
$R^{\varepsilon}$ and $N$ by $N^{\varepsilon}=\frac{1}{Q_{\varepsilon}%
}\,\mathbf{1}_{[\varepsilon,\infty)}\left(  r\right)  \left(  Y_{r}%
-M_{r}^{\varepsilon}\right)  .$

We see first that%
\[%
\begin{array}
[c]{l}%
\langle M_{r}^{\varepsilon}-\hat{Y}_{r},N_{r}^{\varepsilon}\rangle=\langle
M_{r}^{\varepsilon}-\hat{Y}_{r},\dfrac{1}{Q_{\varepsilon}}\left(  Y_{r}%
-M_{r}^{\varepsilon}\right)  \rangle=\dfrac{1}{2Q_{\varepsilon}}\,\langle
M_{r}^{\varepsilon}-\hat{Y}_{r},(\hat{Y}_{r}-M_{r}^{\varepsilon})+(\tilde
{Y}_{r}-M_{r}^{\varepsilon})\rangle\medskip\\
\leq\dfrac{1}{2Q_{\varepsilon}}\,\left[  -|M_{r}^{\varepsilon}-\hat{Y}%
_{r}|^{2}+|M_{r}^{\varepsilon}-\hat{Y}_{r}||M_{r}^{\varepsilon}-\tilde{Y}%
_{r}|\right]  =\dfrac{1}{2Q_{\varepsilon}}\,\left[  |M_{r}^{\varepsilon
}-\tilde{Y}_{r}|-|M_{r}^{\varepsilon}-\hat{Y}_{r}|\right]  |M_{r}%
^{\varepsilon}-\hat{Y}_{r}|
\end{array}
\]
and similarly%
\[
\langle M_{r}^{\varepsilon}-\tilde{Y}_{r},N_{r}^{\varepsilon}\rangle\leq
\frac{1}{2Q_{\varepsilon}}\,\left[  |M_{r}^{\varepsilon}-\hat{Y}_{r}%
|-|M_{r}^{\varepsilon}-\tilde{Y}_{r}|\right]  |M_{r}^{\varepsilon}-\tilde
{Y}_{r}|.
\]
Therefore%
\begin{equation}%
\begin{array}
[c]{l}%
(\hat{\Gamma}_{r}^{\varepsilon})^{q-2}\langle M_{r}^{\varepsilon}-\hat{Y}%
_{r},N_{r}^{\varepsilon}\rangle+(\tilde{\Gamma}_{r}^{\varepsilon}%
)^{q-2}\langle M_{r}^{\varepsilon}-\tilde{Y}_{r},N_{r}^{\varepsilon}%
\rangle\medskip\\
=\big(|M_{r}^{\varepsilon}-\hat{Y}_{r}|^{2}+\delta_{q}\big)^{\left(
q-2\right)  /2}\langle M_{r}^{\varepsilon}-\hat{Y}_{r},N_{r}^{\varepsilon
}\rangle+\big(|M_{r}^{\varepsilon}-\tilde{Y}_{r}|^{2}+\delta_{q}\big)^{\left(
q-2\right)  /2}\langle M_{r}^{\varepsilon}-\tilde{Y}_{r},N_{r}^{\varepsilon
}\rangle\medskip\\
\leq\dfrac{-1}{2Q_{\varepsilon}}\,\left[  \big(|M_{r}^{\varepsilon}-\hat
{Y}_{r}|^{2}+\delta_{q}\big)^{\left(  q-2\right)  /2}|M_{r}^{\varepsilon}%
-\hat{Y}_{r}|-\big(|M_{r}^{\varepsilon}-\tilde{Y}_{r}|^{2}+\delta
_{q}\big)^{\left(  q-2\right)  /2}|M_{r}^{\varepsilon}-\tilde{Y}_{r}|\right]
\medskip\\
\quad\cdot\left[  |M_{r}^{\varepsilon}-\hat{Y}_{r}|-|M_{r}^{\varepsilon
}-\tilde{Y}_{r}|\right]  \leq0,
\end{array}
\label{uniq1c}%
\end{equation}
since we have, for all $a,b\geq0$, $\delta\geq0$ and $\beta\geq-1/2,$%
\[
\left[  \left(  a^{2}+\delta\right)  ^{\beta}a-\left(  b^{2}+\delta\right)
^{\beta}b\right]  \left(  a-b\right)  \geq0.
\]
We use inequalities (\ref{uniq1b}) and (\ref{uniq1c}) and inequality
(\ref{uniq1}) becomes:%
\begin{equation}%
\begin{array}
[c]{l}%
\left[  e^{qV_{\sigma}}(\hat{\Gamma}_{\sigma}^{\varepsilon})^{q}%
+e^{qV_{\sigma}}(\tilde{\Gamma}_{\sigma}^{\varepsilon})^{q}\right]
+q\,\mathbb{E}^{\mathcal{F}_{\sigma}}%
{\displaystyle\int_{\sigma}^{\theta}}
e^{qV_{r}}\left[  (\hat{\Gamma}_{r}^{\varepsilon})^{q}+(\tilde{\Gamma}%
_{r}^{\varepsilon})^{q}\right]  dV_{r}\medskip\\
\quad+\dfrac{q}{4}\,n_{q}\,\mathbb{E}^{\mathcal{F}_{\sigma}}{%
{\displaystyle\int_{\sigma}^{\theta}}
}e^{qV_{r}}\left[  \frac{1+\beta}{\beta}\,|M^{\varepsilon}-Y|^{2}%
+\frac{1+\beta}{4}\,|\hat{Y}-\tilde{Y}|^{2}+\delta_{q}\right]  ^{\left(
q-2\right)  /2}|\hat{Z}-\tilde{Z}|^{2}dr\medskip\\
\quad+{q\,\mathbb{E}^{\mathcal{F}_{\sigma}}%
{\displaystyle\int_{\sigma}^{\theta}}
}e^{qV_{r}}\left[  (\hat{\Gamma}_{r}^{\varepsilon})^{q-2}{\Psi(}r,\hat{Y}%
_{r})+(\tilde{\Gamma}_{r}^{\varepsilon})^{q-2}{\Psi(}r,\tilde{Y}_{r})\right]
dQ_{r}\medskip\\
\leq\mathbb{E}^{\mathcal{F}_{\sigma}}\left[  e^{qV_{\theta}}(\hat{\Gamma
}_{\theta}^{\varepsilon})^{q}+e^{qV_{\theta}}(\tilde{\Gamma}_{\theta
}^{\varepsilon})^{q}\right]  +{q\,\mathbb{E}^{\mathcal{F}_{\sigma}}%
{\displaystyle\int_{\sigma}^{\theta}}
}e^{qV_{r}}\left[  (\hat{\Gamma}_{r}^{\varepsilon})^{q-2}+(\tilde{\Gamma}%
_{r}^{\varepsilon})^{q-2}\right]  {\Psi}\left(  r,M_{r}^{\varepsilon}\right)
dQ_{r}\medskip\\
+{q\,\mathbb{E}^{\mathcal{F}_{\sigma}}%
{\displaystyle\int_{\sigma}^{\theta}}
}e^{qV_{r}}\left[  (\hat{\Gamma}_{r}^{\varepsilon})^{q-2}\langle
M_{r}^{\varepsilon}-\hat{Y}_{r},-\hat{H}(r,\hat{Y}_{r},\hat{Z}_{r}%
)\rangle+(\tilde{\Gamma}_{r}^{\varepsilon})^{q-2}\langle M_{r}^{\varepsilon
}-\tilde{Y}_{r},-\tilde{H}(r,\tilde{Y}_{r},\tilde{Z}_{r})\rangle\right]
dQ_{r}\,.
\end{array}
\label{uniq1d}%
\end{equation}
Let $0\leq u\leq v$ and the stopping times $v^{\ast}=Q_{v}^{-1},$ $u^{\ast
}=Q_{u}^{-1}\,,$ where $Q_{\cdot}^{-1}$ is the inverse of the function
$r\mapsto Q_{r}:[0,\infty)\rightarrow\lbrack0,\infty).$

Let, for each $k,i\in\mathbb{N}^{\ast},$ the stopping times%
\[%
\begin{array}
[c]{l}%
\alpha_{k}=\inf\Big\{u\geq0:\left\updownarrow V\right\updownarrow _{u}%
+\sup\nolimits_{r\in\left[  0,u\right]  }|e^{V_{r}}\hat{Y}_{r}-\hat{Y}%
_{0}|+\sup\nolimits_{r\in\left[  0,u\right]  }|e^{V_{r}}\tilde{Y}_{r}%
-\tilde{Y}_{0}|+%
{\displaystyle\int_{0}^{u}}
e^{2V_{r}}|\hat{Z}_{r}|^{2}dr\medskip\\
\quad\quad\quad\quad+%
{\displaystyle\int_{0}^{u}}
e^{2V_{r}}|\tilde{Z}_{r}|^{2}dr+%
{\displaystyle\int_{0}^{u}}
e^{V_{r}}|\hat{H}(r,\hat{Y}_{r},\hat{Z}_{r})|dQ_{r}+%
{\displaystyle\int_{0}^{u}}
e^{V_{r}}|\tilde{H}(r,\tilde{Y}_{r},\tilde{Z}_{r})|dQ_{r}\medskip\\
\quad\quad\quad\quad+%
{\displaystyle\int_{0}^{u}}
e^{2V_{r}}{\Psi(}r,\hat{Y}_{r})dQ_{r}+%
{\displaystyle\int_{0}^{u}}
e^{2V_{r}}{\Psi(}r,\tilde{Y}_{r})dQ_{r}\geq k\Big\}.
\end{array}
\]
and define%
\[
u_{k}^{\ast}=\sigma\wedge u^{\ast}\wedge\alpha_{k}\quad\text{and}\quad
v_{k+i}^{\ast}=\theta\wedge v^{\ast}\wedge\alpha_{k+i}\,.
\]
We consider in (\ref{uniq1d})%
\[
\sigma=u_{k}^{\ast}\quad\text{and}\quad\theta=v_{k+i}^{\ast}%
\]
and passing to $\liminf_{\varepsilon\searrow0}$ we obtain (using Proposition
\ref{p1-approx}, Fatou's Lemma and Lebesgue dominated convergence theorem):%
\[%
\begin{array}
[c]{l}%
\displaystyle2\,e^{qV_{u_{k}^{\ast}}}\left(  \dfrac{1}{4}\,\big|\hat{Y}%
_{u_{k}^{\ast}}-\tilde{Y}_{u_{k}^{\ast}}\big|^{2}+\delta_{q}\right)
^{q/2}+2q\,{\mathbb{E}}^{\mathcal{F}_{u_{k}^{\ast}}}{%
{\displaystyle\int_{u_{k}^{\ast}}^{v_{k+i}^{\ast}}}
e^{qV_{r}}}\left(  \dfrac{1}{4}\,\big|\hat{Y}_{r}-\tilde{Y}_{r}\big|^{2}%
+\delta_{q}\right)  ^{q/2}dV_{r}\medskip\\
\displaystyle\quad+\dfrac{q}{4}\,n_{q}\,{\mathbb{E}}^{\mathcal{F}_{u_{k}%
^{\ast}}}{%
{\displaystyle\int_{u_{k}^{\ast}}^{v_{k+i}^{\ast}}}
}e^{qV_{r}}\left(  \dfrac{1+\beta}{4}\,|\hat{Y}_{r}-\tilde{Y}_{r}|^{2}%
+\delta_{q}\right)  ^{\left(  q-2\right)  /2}|\hat{Z}_{r}-\tilde{Z}_{r}%
|^{2}dr\medskip\\
\displaystyle\quad+q\,{\mathbb{E}}^{\mathcal{F}_{u_{k}^{\ast}}}{%
{\displaystyle\int_{u_{k}^{\ast}}^{v_{k+i}^{\ast}}}
}e^{qV_{r}}\left(  \dfrac{1}{4}\,|\hat{Y}_{r}-\tilde{Y}_{r}|^{2}+\delta
_{q}\right)  ^{\left(  q-2\right)  /2}\left[  \Psi(r,\hat{Y}_{r}%
)+\Psi(r,\tilde{Y}_{r})\right]  {dQ_{r}}\medskip\\
\displaystyle\leq2\,{\mathbb{E}}^{\mathcal{F}_{u_{k}^{\ast}}}e^{qV_{v_{k+i}%
^{\ast}}}\left(  \dfrac{1}{4}\,\big|\hat{Y}_{v_{k+i}^{\ast}}-\tilde
{Y}_{v_{k+i}^{\ast}}\big|^{2}+\delta_{q}\right)  ^{q/2}\medskip\\
\displaystyle\quad+q\,{\mathbb{E}}^{\mathcal{F}_{u_{k}^{\ast}}}{%
{\displaystyle\int_{u_{k}^{\ast}}^{v_{k+i}^{\ast}}}
}e^{qV_{r}}\left(  \dfrac{1}{4}\,|\hat{Y}_{r}-\tilde{Y}_{r}|^{2}+\delta
_{q}\right)  ^{\left(  q-2\right)  /2}\cdot2{\Psi}\left(  r,Y_{r}\right)
{dQ_{r}}\medskip\\
\displaystyle\quad+\dfrac{{q}}{2}\,{\mathbb{E}}^{\mathcal{F}_{u_{k}^{\ast}}}{%
{\displaystyle\int_{u_{k}^{\ast}}^{v_{k+i}^{\ast}}}
}e^{qV_{r}}\left(  \dfrac{1}{4}\,|\hat{Y}_{r}-\tilde{Y}_{r}|^{2}+\delta
_{q}\right)  ^{\left(  q-2\right)  /2}\langle\hat{Y}_{r}-\tilde{Y}_{r},\hat
{H}(r,\hat{Y}_{r},\hat{Z}_{r})-\tilde{H}(r,\tilde{Y}_{r},\tilde{Z}_{r})\rangle
dQ_{r}\,.
\end{array}
\]
By Fatou's Lemma, as as $\beta\rightarrow0,$ we deduce that the previous
inequality holds true also for $\beta=0.$

We remark now that%
\[
2{\Psi}\left(  r,Y_{r}\right)  =2{\Psi}\Big(r,\frac{1}{2}\hat{Y}_{r}+\frac
{1}{2}\tilde{Y}_{r}\Big)\leq\Psi(r,\hat{Y}_{r})+\Psi(r,\tilde{Y}_{r}).
\]
and%
\begin{align*}
&  \langle\hat{Y}_{r}-\tilde{Y}_{r},\hat{H}(r,\hat{Y}_{r},\hat{Z}_{r}%
)-\tilde{H}(r,\hat{Y}_{r},\hat{Z}_{r})+\tilde{H}(r,\hat{Y}_{r},\hat{Z}%
_{r})-\tilde{H}(r,\tilde{Y}_{r},\tilde{Z}_{r})\rangle dQ_{r}\\
&  \leq\langle\hat{Y}_{r}-\tilde{Y}_{r},\hat{H}(r,\hat{Y}_{r},\hat{Z}%
_{r})-\tilde{H}(r,\hat{Y}_{r},\hat{Z}_{r})\rangle dQ_{r}+|\hat{Y}_{r}%
-\tilde{Y}_{r}|^{2}dV_{r}+\dfrac{n_{p}\lambda}{2}\,|\hat{Z}_{r}-\tilde{Z}%
_{r}|^{2}dr.
\end{align*}
Hence%
\[%
\begin{array}
[c]{l}%
2\,e^{qV_{u_{k}^{\ast}}}\left(  \dfrac{1}{4}\,\big|\hat{Y}_{u_{k}^{\ast}%
}-\tilde{Y}_{u_{k}^{\ast}}\big|^{2}+\delta_{q}\right)  ^{q/2}+2q\delta
_{q}\,{\mathbb{E}}^{\mathcal{F}_{u_{k}^{\ast}}}{%
{\displaystyle\int_{u_{k}^{\ast}}^{v_{k+i}^{\ast}}}
e^{qV_{r}}}\left(  \dfrac{1}{4}\,|\hat{Y}_{r}-\tilde{Y}_{r}|^{2}+\delta
_{q}\right)  ^{\left(  q-2\right)  /2}dV_{r}\medskip\\
\quad+\dfrac{q}{4}\left(  n_{q}-n_{p}\lambda\right)  \,{\mathbb{E}%
}^{\mathcal{F}_{u_{k}^{\ast}}}{%
{\displaystyle\int_{u_{k}^{\ast}}^{v_{k+i}^{\ast}}}
}e^{qV_{r}}\left(  \dfrac{1}{4}\,|\hat{Y}_{r}-\tilde{Y}_{r}|^{2}+\delta
_{q}\right)  ^{\left(  q-2\right)  /2}|\hat{Z}_{r}-\tilde{Z}_{r}%
|^{2}dr\medskip\\
\leq2\,{\mathbb{E}}^{\mathcal{F}_{u_{k}^{\ast}}}e^{qV_{v_{k+i}^{\ast}}}\left(
\dfrac{1}{4}\,\big|\hat{Y}_{v_{k+i}^{\ast}}-\tilde{Y}_{v_{k+i}^{\ast}%
}\big|^{2}+\delta_{q}\right)  ^{q/2}\medskip\\
\quad+\dfrac{{q}}{2}\,{\mathbb{E}}^{\mathcal{F}_{u_{k}^{\ast}}}{%
{\displaystyle\int_{u_{k}^{\ast}}^{v_{k+i}^{\ast}}}
}e^{qV_{r}}\left(  \dfrac{1}{4}\,|\hat{Y}_{r}-\tilde{Y}_{r}|^{2}+\delta
_{q}\right)  ^{\left(  q-2\right)  /2}\langle\hat{Y}_{r}-\tilde{Y}_{r},\hat
{H}(r,\hat{Y}_{r},\hat{Z}_{r})-\tilde{H}(r,\hat{Y}_{r},\hat{Z}_{r})\rangle
dQ_{r}\,.
\end{array}
\]
Since $q\in\left\{  2,p\wedge2\right\}  ,$ we have (see Remark \ref{r-nq})%
\[
n_{q}-n_{p}\lambda\geq n_{q}\left(  1-\lambda\right)  =\left(  q-1\right)
\left(  1-\lambda\right)  ,
\]
passing to $\lim_{u,v\rightarrow\infty}\,,$ by Fatou's Lemma and Lebesgue's
dominated convergence theorem and using the continuity of the natural
filtration $\left\{  \mathcal{F}_{r}:r\geq0\right\}  ,$ it follows that%
\begin{equation}%
\begin{array}
[c]{l}%
2\,e^{qV_{\sigma\wedge\alpha_{k}}}\left(  \dfrac{1}{4}\,\big|\hat{Y}%
_{\sigma\wedge\alpha_{k}}-\tilde{Y}_{\sigma\wedge\alpha_{k}}\big|^{2}%
+\delta_{q}\right)  ^{q/2}\medskip\\
\quad+2q\delta_{q}\,{\mathbb{E}}^{\mathcal{F}_{\sigma\wedge\alpha_{k}}}{%
{\displaystyle\int_{\sigma\wedge\alpha_{k}}^{\theta\wedge\alpha_{k+i}}}
e^{qV_{r}}}\left(  \dfrac{1}{4}\,|\hat{Y}_{r}-\tilde{Y}_{r}|^{2}+\delta
_{q}\right)  ^{\left(  q-2\right)  /2}dV_{r}\medskip\\
\quad+\dfrac{q}{4}\left(  q-1\right)  \left(  1-\lambda\right)  \,{\mathbb{E}%
}^{\mathcal{F}_{\sigma\wedge\alpha_{k}}}{%
{\displaystyle\int_{\sigma\wedge\alpha_{k}}^{\theta\wedge\alpha_{k+i}}}
}e^{qV_{r}}\left(  \dfrac{1}{4}\,|\hat{Y}_{r}-\tilde{Y}_{r}|^{2}+\delta
_{q}\right)  ^{\left(  q-2\right)  /2}|\hat{Z}_{r}-\tilde{Z}_{r}%
|^{2}dr\medskip\\
\leq2\,{\mathbb{E}}^{\mathcal{F}_{\sigma\wedge\alpha_{k}}}e^{qV_{\theta
\wedge\alpha_{k+i}}}\left(  \dfrac{1}{4}\,\big|\hat{Y}_{\theta\wedge
\alpha_{k+i}}-\tilde{Y}_{\theta\wedge\alpha_{k+i}}\big|^{2}+\delta_{q}\right)
^{q/2}\medskip\\
\quad+q\,{\mathbb{E}}^{\mathcal{F}_{\sigma\wedge\alpha_{k}}}{%
{\displaystyle\int_{\sigma\wedge\alpha_{k}}^{\theta\wedge\alpha_{k+i}}}
}e^{qV_{r}}\left(  \dfrac{1}{4}\,|\hat{Y}_{r}-\tilde{Y}_{r}|^{2}+\delta
_{q}\right)  ^{\left(  q-1\right)  /2}\big|\hat{H}(r,\hat{Y}_{r},\hat{Z}%
_{r})-\tilde{H}(r,\hat{Y}_{r},\hat{Z}_{r})\big|dQ_{r}\,.
\end{array}
\label{uniq1e}%
\end{equation}
Passing to the limit, as $\delta\rightarrow0_{+}\,,$ by Fatou's Lemma for the
left-hand side and inequality%
\[
\delta_{q}\left(  \dfrac{1}{4}\,|\hat{Y}_{r}-\tilde{Y}_{r}|^{2}+\delta
_{q}\right)  ^{\left(  q-2\right)  /2}\leq\delta_{q}{}^{q/2},
\]
and by Lebesgue's dominated convergence theorem for the right-hand side we get%
\begin{equation}%
\begin{array}
[c]{l}%
\displaystyle e^{qV_{\sigma\wedge\alpha_{k}}}\big|\hat{Y}_{\sigma\wedge
\alpha_{k}}-\tilde{Y}_{\sigma\wedge\alpha_{k}}\big|^{q}\medskip\\
\displaystyle\quad+\dfrac{q}{2}\left(  q-1\right)  \left(  1-\lambda\right)
\,{\mathbb{E}}^{\mathcal{F}_{\sigma\wedge\alpha_{k}}}{%
{\displaystyle\int_{\sigma\wedge\alpha_{k}}^{\theta\wedge\alpha_{k+i}}}
}e^{qV_{r}}|\hat{Y}_{r}-\tilde{Y}_{r}|^{q-2}\,|\hat{Z}_{r}-\tilde{Z}_{r}%
|^{2}dr\medskip\\
\displaystyle\leq{\mathbb{E}}^{\mathcal{F}_{\sigma\wedge\alpha_{k}}%
}e^{qV_{\theta\wedge\alpha_{k+i}}}\big|\hat{Y}_{\theta\wedge\alpha_{k+i}%
}-\tilde{Y}_{\theta\wedge\alpha_{k+i}}\big|^{q}\medskip\\
\displaystyle\quad+q\,{\mathbb{E}}^{\mathcal{F}_{\sigma\wedge\alpha_{k}}}{%
{\displaystyle\int_{\sigma\wedge\alpha_{k}}^{\theta\wedge\alpha_{k+i}}}
}e^{qV_{r}}|\hat{Y}_{r}-\tilde{Y}_{r}|^{q-1}\big|\hat{H}(r,\hat{Y}_{r},\hat
{Z}_{r})-\tilde{H}(r,\hat{Y}_{r},\hat{Z}_{r})\big|dQ_{r}\,.
\end{array}
\label{uniq-33}%
\end{equation}
Passing to the limit successively, $\lim_{i\rightarrow\infty}\,,$
$\lim_{k\rightarrow\infty}$ and $\lim_{T\rightarrow\infty}\,,$ in
(\ref{uniq-33}), we obtain (using Fatou's Lemma and Lebesgue dominated
convergence theorem via condition (\ref{def0-1})), for any stopping times
$0\leq\sigma\leq\theta\leq\tau,$%
\begin{equation}%
\begin{array}
[c]{l}%
e^{qV_{\sigma}}\big|\hat{Y}_{\sigma}-\tilde{Y}_{\sigma}\big|^{q}+\dfrac{q}%
{2}\,\left(  q-1\right)  \left(  1-\lambda\right)  \,{\mathbb{E}}%
^{\mathcal{F}_{\sigma}}{%
{\displaystyle\int_{\sigma}^{\theta}}
}e^{qV_{r}}|\hat{Y}_{r}-\tilde{Y}_{r}|^{q-2}\,|\hat{Z}_{r}-\tilde{Z}_{r}%
|^{2}dr\medskip\\
\leq{\mathbb{E}}^{\mathcal{F}_{\sigma}}e^{qV_{\theta}}\big|\hat{Y}_{\theta
}-\tilde{Y}_{\theta}\big|^{q}+q\,{\mathbb{E}}^{\mathcal{F}_{\sigma}}{%
{\displaystyle\int_{\sigma}^{\theta}}
}e^{qV_{r}}|\hat{Y}_{r}-\tilde{Y}_{r}|^{q-1}\big|\hat{H}(r,\hat{Y}_{r},\hat
{Z}_{r})-\tilde{H}(r,\hat{Y}_{r},\hat{Z}_{r})\big|dQ_{r}\,,\quad
\mathbb{P}\text{--a.s.}%
\end{array}
\label{uniq-44}%
\end{equation}
(we also used the continuity of the natural filtration $\left\{
\mathcal{F}_{r}:r\geq0\right\}  \,$).$\medskip$

From (\ref{uniq-44}) we obtain conclusion (\ref{cont-1}) if we use Holder's
inequality and (\ref{def-11ccc}) and ine\-qua\-lity $\big|\hat{Y}_{r}%
-\tilde{Y}_{r}\big|^{q}\leq2^{q-1}\,\big(|\hat{Y}_{r}|^{q}+|\tilde{Y}_{r}%
|^{q}\big).\medskip$

Applying now Proposition \ref{an-prop-yd}, we infer from (\ref{uniq-44}) that,
for all $0<\alpha<1,$%
\begin{equation}%
\begin{array}
[c]{l}%
\mathbb{E}\sup\nolimits_{r\in\left[  \sigma,\theta\right]  }e^{\alpha qV_{r}%
}\big|\hat{Y}_{r}-\tilde{Y}_{r}\big|^{\alpha q}+\left(  q-1\right)  ^{\alpha
}\,\mathbb{E}\Big(%
{\displaystyle\int_{\sigma}^{\theta}}
e^{qV_{r}}|\hat{Y}_{r}-\tilde{Y}_{r}|^{q-2}\,|\hat{Z}_{r}-\tilde{Z}_{r}%
|^{2}dr\Big)^{\alpha}\medskip\\
\leq C_{q,\alpha,\lambda}\left[  \big(\mathbb{E}e^{qV_{\theta}}\big|\hat
{Y}_{\theta}-\tilde{Y}_{\theta}\big|^{q}\big)^{\alpha}+L^{\frac{\alpha\left(
q-1\right)  }{q}}\,\bigg(\mathbb{E}\Big(%
{\displaystyle\int_{\sigma}^{\theta}}
e^{V_{r}}\big|\hat{H}(r,\hat{Y}_{r},\hat{Z}_{r})-\tilde{H}(r,\hat{Y}_{r}%
,\hat{Z}_{r})\big|dQ_{r}\Big)^{q}\bigg)^{\frac{\alpha}{q}}\right]  ,
\end{array}
\label{uniq-55}%
\end{equation}
where $L$ is a constant such that, see inequality (\ref{def-11ccc}),%
\[
\mathbb{E}\sup\nolimits_{r\in\left[  0,T\right]  }e^{qV_{r}}\big|\hat{Y}%
_{r}-\tilde{Y}_{r}\big|^{q}\leq2^{q-1}\,\mathbb{E}\sup\nolimits_{r\in\left[
0,T\right]  }e^{qV_{r}}\big(|\hat{Y}_{r}|^{q}+|\tilde{Y}_{r}|^{q}\big)\leq L.
\]
Hence conclusion (\ref{cont-2}) follows.$\medskip$

The uniqueness property is a consequence, since, if $\hat{\eta}=\tilde{\eta}$
and $\hat{H}=\tilde{H},$ we conclude, from (\ref{cont-2}), in the case $q>1,$
that $\hat{Y}=\tilde{Y}$ in $S_{m}^{0}$ and $\hat{Z}=\tilde{Z}$ in
$\Lambda_{m\times k}^{0}\,.$\hfill
\end{proof}

\begin{remark}
We notice that if we have the uniqueness of $Y$ and if we multiply
(\ref{uniq1e}) with $\delta^{1-q/2},$ then the uniqueness of $Z$ follows also
from (\ref{uniq1e}) by taking $\delta\rightarrow0_{+}\,.$
\end{remark}

\section{Existence of the solution}

\subsection{Existence on a deterministic interval time $\left[  0,T\right]  $}

\hspace{\parindent}The existence of a $L^{p}$--variational solution will be
proved firstly in the case of a deterministic time interval, i.e. $\tau=T>0.$

\begin{lemma}
[Strong solution]\label{l1-strong sol}We suppose that assumptions $\left(
\mathrm{A}_{1}-\mathrm{A}_{7}\right)  $ are satisfied. Let $V^{\left(
+\right)  }$ be given by definition (\ref{defV_3}). In addition we assume
that:$\medskip$

\noindent$\left(  i\right)  $ there exists $L>0$ such that%
\begin{equation}
\left\vert \eta\right\vert +\varphi\left(  \eta\right)  +\psi\left(
\eta\right)  +\ell_{t}+F_{1}^{\#}\left(  t\right)  +G_{1}^{\#}\left(
t\right)  \leq L,\quad\text{a.e. }t\in\left[  0,T\right]  ,\;\mathbb{P}%
\text{--a.s.;} \label{ap-eq-1b}%
\end{equation}
\noindent$\left(  ii\right)  $ there exists $\delta>0$ such that%
\begin{equation}
\mathbb{E}\exp\left[  \left(  2+\delta\right)  V_{T}^{\left(  +\right)
}\right]  <\infty; \label{ap-eq-1bb}%
\end{equation}
\noindent$\left(  iii\right)  $ there exists $\tilde{L}>0$ such that%
\begin{equation}
\big|e^{V_{T}^{\left(  +\right)  }}\eta\big|^{2}+\bigg(%
{\displaystyle\int_{0}^{T}}
e^{V_{r}^{\left(  +\right)  }}\big(F_{1}^{\#}\left(  r\right)  dr+G_{1}%
^{\#}\left(  r\right)  dA_{r}\big)\bigg)^{2}\leq\tilde{L},\quad\mathbb{P}%
-\text{a.s.;} \label{ap-eq-1c}%
\end{equation}
\noindent$\left(  iv\right)  $ for\footnote{The constant $C_{\lambda
}:=C_{2,\lambda}\,,$ where $C_{2,\lambda}$ is given by (\ref{an3a}).}
$\tilde{\rho}_{0}\xlongequal{\hspace{-4pt}{\rm def}\hspace{-4pt}}(C_{\lambda
}\tilde{L})^{1/2}>0$ it holds%
\begin{equation}
\mathbb{E}%
{\displaystyle\int_{0}^{T}}
e^{2V_{r}^{\left(  +\right)  }}\left[  \big(F_{1+\tilde{\rho}_{0}}^{\#}\left(
r\right)  \big)^{2}dr+\big(G_{1+\tilde{\rho}_{0}}^{\#}\left(  r\right)
\big)^{2}dA_{r}\right]  <\infty. \label{ap-eq-6}%
\end{equation}
Then the multivalued BSDE%
\[
\left\{
\begin{array}
[c]{l}%
\displaystyle Y_{t}+{\int_{t}^{T}}dK_{r}=Y_{T}+{\int_{t}^{T}}H\left(
r,Y_{r},Z_{r}\right)  dQ_{r}-{\int_{t}^{T}}Z_{r}dB_{r}\,,\medskip\\
\displaystyle dK_{t}=U_{t}^{\left(  1\right)  }dt+U_{t}^{\left(  2\right)
}dA_{t}\,,\medskip\\
\displaystyle U_{t}^{\left(  1\right)  }dt\in\partial\varphi\left(
Y_{t}\right)  dt\quad\text{and}\quad U_{t}^{\left(  2\right)  }dA_{t}%
\in\partial\psi\left(  Y_{t}\right)  dA_{t}\,,\quad t\in\left[  0,T\right]  ,
\end{array}
\,\right.
\]
has a strong a solution $\left(  Y,Z,U^{\left(  1\right)  },U^{\left(
2\right)  }\right)  \in S_{m}^{0}\times\Lambda_{m\times k}^{0}\times
\Lambda_{m}^{0}\times\Lambda_{m}^{0}$ such that
\[
\mathbb{E}\sup\nolimits_{t\in\left[  0,T\right]  }{e^{2V_{r}^{\left(
+\right)  }}}\left\vert Y_{r}\right\vert ^{2}+\mathbb{E}%
{\displaystyle\int_{0}^{T}}
{e^{2V_{r}^{\left(  +\right)  }}}\left\vert Z_{r}\right\vert ^{2}dr+\mathbb{E}%
{\displaystyle\int_{0}^{T}}
{e^{2V_{r}^{\left(  +\right)  }}}\big|U_{r}^{\left(  1\right)  }%
\big|^{2}dr+\mathbb{E}%
{\displaystyle\int_{0}^{T}}
{e^{2V_{r}^{\left(  +\right)  }}}\big|U_{r}^{\left(  2\right)  }%
\big|^{2}dA_{r}<\infty.
\]

\end{lemma}

\begin{proof}
Let $0<\varepsilon\leq1.$ We consider the approximating BSDE%
\begin{equation}
Y_{t}^{\varepsilon}+{\int_{t}^{T}}\nabla_{y}\Psi^{\varepsilon}(r,Y_{r}%
^{\varepsilon})dQ_{r}=\eta+{\int_{t}^{T}}H_{\varepsilon}(r,Y_{r}^{\varepsilon
},Z_{r}^{\varepsilon})dQ_{r}-{\int_{t}^{T}}Z_{r}^{\varepsilon}\,dB_{r}%
\,,\quad\mathbb{P}\text{--a.s., }t\in\left[  0,T\right]  , \label{ap-eq -1}%
\end{equation}
where%
\begin{equation}%
\begin{array}
[c]{l}%
\displaystyle\Psi^{\varepsilon}\left(  r,y\right)
\xlongequal{\hspace{-4pt}{\rm def}\hspace{-4pt}}\alpha_{r}\,\varphi
_{\varepsilon}\left(  y\right)  +\left(  1-\alpha_{r}\right)  \,\psi
_{\varepsilon}\left(  y\right)  \,\mathbf{1}_{[0,\frac{1}{\varepsilon}%
]}\left(  A_{r}\right)  ,\medskip\\
\displaystyle\nabla_{y}\Psi^{\varepsilon}\left(  r,y\right)  =\left[
\alpha_{r}\,\nabla_{y}\varphi_{\varepsilon}\left(  y\right)  +\left(
1-\alpha_{r}\right)  \,\nabla_{y}\psi_{\varepsilon}\left(  y\right)  \right]
\,\mathbf{1}_{[0,\frac{1}{\varepsilon}]}\left(  A_{r}\right)  ,\medskip\\
\displaystyle H_{\varepsilon}\left(  r,y,z\right)
\xlongequal{\hspace{-4pt}{\rm def}\hspace{-4pt}}\left[  \alpha_{r}%
\,F_{\varepsilon}\left(  r,y,z\right)  +\left(  1-\alpha_{r}\right)
\,G_{\varepsilon}\left(  r,y\right)  \right]  \,\mathbf{1}_{[0,\frac
{1}{\varepsilon}]}\left(  A_{r}\right)  ,
\end{array}
\label{ap-eq-1a}%
\end{equation}
where $\varphi_{\varepsilon}$ and $\psi_{\varepsilon}$ are the Moreau-Yosida's
regularization given by (\ref{fi-MY}) and $F_{\varepsilon},G_{\varepsilon}$
are the mollifier approximations defined by Section \ref{Annex-MA}.

By (\ref{ma-1}) and (\ref{fi-lip}) we see that the function%
\[
\Phi_{\varepsilon}\left(  r,y,z\right)
\xlongequal{\hspace{-4pt}{\rm def}\hspace{-4pt}}H_{\varepsilon}\left(
r,y,z\right)  -\nabla_{y}\Psi^{\varepsilon}\left(  r,y\right)
\]
is a Lipschitz function:%
\[%
\begin{array}
[c]{l}%
\displaystyle\left\vert \Phi_{\varepsilon}\left(  r,y,z\right)  -\Phi
_{\varepsilon}\left(  r,\hat{y},\hat{z}\right)  \right\vert \medskip\\
\displaystyle\leq\bigg[\alpha_{r}\left(  \ell_{r}\,\left\vert z-\hat
{z}\right\vert +\dfrac{\kappa\left(  1+\ell_{t}\right)  }{\varepsilon^{2}%
}\,\left\vert y-\hat{y}\right\vert \right)  +\left(  1-\alpha_{r}\right)
\,\dfrac{\kappa}{\varepsilon^{2}}\,\left\vert y-\hat{y}\right\vert \medskip\\
\displaystyle\quad+\dfrac{1}{\varepsilon}\,\alpha_{r}\,\left\vert y-\hat
{y}\right\vert +\dfrac{1}{\varepsilon}\left(  1-\alpha_{r}\right)
\,\left\vert y-\hat{y}\right\vert \bigg]\,\mathbf{1}_{[0,\frac{1}{\varepsilon
}]}\left(  A_{r}\right)  \medskip\\
\displaystyle\leq\left[  \alpha_{r}\,\dfrac{\kappa L+\kappa+1}{\varepsilon
^{2}}\,\left\vert y-\hat{y}\right\vert +\left(  1-\alpha_{r}\right)
\,\dfrac{\kappa+1}{\varepsilon^{2}}\,\left\vert y-\hat{y}\right\vert
+L\,\alpha_{r}\left\vert z-\hat{z}\right\vert \right]  \,\mathbf{1}%
_{[0,\frac{1}{\varepsilon}]}\left(  A_{r}\right)  \medskip\\
\displaystyle\leq\left[  \dfrac{\kappa L+\kappa+1}{\varepsilon^{2}%
}\,\left\vert y-\hat{y}\right\vert +L\,\alpha_{r}\left\vert z-\hat
{z}\right\vert \right]  \,\mathbf{1}_{[0,\frac{1}{\varepsilon}]}\left(
A_{r}\right)  .
\end{array}
\]
If we use the above properties of $\Phi_{\varepsilon}\,,$ the presence of the
indicator $\mathbf{1}_{[0,\frac{1}{\varepsilon}]}\left(  A_{r}\right)  $ in
the definition of $\Phi_{\varepsilon}\,,$ assumption (\ref{ap-eq-1b}) and
properties (\ref{ma-4}) and (\ref{ma-1}--$a$) we see that the assumptions of
\cite[Lemma 5.20]{pa-ra/14} are satisfied for any $p^{\prime}\geq2$ arbitrary fixed.

Therefore equation (\ref{ap-eq -1}) has a unique solution $\left(
Y^{\varepsilon},Z^{\varepsilon}\right)  \in S_{m}^{p^{\prime}}\left[
0,T\right]  \times\Lambda_{m\times k}^{p^{\prime}}\left(  0,T\right)  $ and
consequently, for any $p^{\prime}\geq2,$%
\begin{equation}
\mathbb{E}\sup\nolimits_{t\in\left[  0,T\right]  }\left\vert Y_{t}%
^{\varepsilon}\right\vert ^{p^{\prime}}<\infty. \label{ap-eq-2}%
\end{equation}
Remark that, by (\ref{ma-2}), since $n_{p}\leq1,$%
\begin{align*}
&  \big\langle Y_{t}^{\varepsilon},\Phi_{\varepsilon}\left(  t,Y_{t}%
^{\varepsilon},Z_{t}^{\varepsilon}\right)  \big\rangle dQ_{t}\medskip\\[3pt]
&  =\big\langle Y_{t}^{\varepsilon},F_{\varepsilon}\left(  t,Y_{t}%
^{\varepsilon},Z_{t}^{\varepsilon}\right)  \big\rangle\,\mathbf{1}%
_{[0,\frac{1}{\varepsilon}]}\left(  A_{r}\right)  dt+\big\langle Y_{t}%
^{\varepsilon},G_{\varepsilon}\left(  t,Y_{t}^{\varepsilon}\right)
\big\rangle\,\mathbf{1}_{[0,\frac{1}{\varepsilon}]}\left(  A_{r}\right)
dA_{t}\medskip\\[3pt]
&  \quad-\big\langle Y_{t}^{\varepsilon},\nabla\varphi_{\varepsilon}\left(
t,Y_{t}^{\varepsilon}\right)  \big\rangle\,\mathbf{1}_{[0,\frac{1}%
{\varepsilon}]}\left(  A_{r}\right)  dt-\big\langle Y_{t}^{\varepsilon}%
,\nabla\psi_{\varepsilon}\left(  t,Y_{t}^{\varepsilon}\right)
\big\rangle\,\mathbf{1}_{[0,\frac{1}{\varepsilon}]}\left(  A_{r}\right)
dA_{t}\medskip\\[3pt]
&  \leq\Big[\left\vert Y_{t}^{\varepsilon}\right\vert \,F_{1}^{\#}\left(
t\right)  +\Big(\mu_{t}+\frac{1}{2n_{p}\lambda}\,\ell_{t}^{2}\Big)^{+}%
\left\vert Y_{t}^{\varepsilon}\right\vert ^{2}+\dfrac{n_{p}\lambda}%
{2}\,\left\vert Z_{t}^{\varepsilon}\right\vert ^{2}\Big]dt+\left[  \left\vert
Y_{t}^{\varepsilon}\right\vert G_{1}^{\#}\left(  t\right)  +\nu_{t}%
^{+}\left\vert Y_{t}^{\varepsilon}\right\vert ^{2}\right]  dA_{t}%
\medskip\\[3pt]
&  \leq\left\vert Y_{t}^{\varepsilon}\right\vert \,\bar{H}_{1}^{\#}\left(
t\right)  dQ_{t}+\left\vert Y_{t}^{\varepsilon}\right\vert ^{2}dV_{t}^{\left(
+\right)  }+\dfrac{\lambda}{2}\,\left\vert Z_{t}^{\varepsilon}\right\vert
^{2}dt,
\end{align*}
where%
\[
\bar{H}_{1}^{\#}\left(  t\right)
\xlongequal{\hspace{-4pt}{\rm def}\hspace{-4pt}}\alpha_{t}F_{1}^{\#}\left(
t\right)  +\left(  1-\alpha_{t}\right)  G_{1}^{\#}\left(  t\right)  .
\]
By Young's inequality and assumption (\ref{ap-eq-1bb}) and (\ref{ap-eq-2}) we
have%
\begin{align}
\mathbb{E}\sup\nolimits_{t\in\left[  0,T\right]  }e^{2V_{t}^{\left(  +\right)
}}\left\vert Y_{t}^{\varepsilon}\right\vert ^{2}  &  \leq\mathbb{E}\left[
\big(\exp2V_{T}^{\left(  +\right)  }\big)\,\sup\nolimits_{t\in\left[
0,T\right]  }\left\vert Y_{t}^{\varepsilon}\right\vert ^{2}\right]
\medskip\label{ap-eq-9}\\
&  \leq\left[  \frac{2}{2+\delta}\,\mathbb{E}\exp\left(  2+\delta\right)
V_{T}^{\left(  +\right)  }+\frac{\delta}{2+\delta}\,\mathbb{E}\sup
\nolimits_{t\in\left[  0,T\right]  }\left\vert Y_{t}^{\varepsilon}\right\vert
^{\left(  4+2\delta\right)  /\delta}\right]  <\infty,\nonumber
\end{align}
therefore, by Proposition \ref{Appendix_result 1}, we have%
\[
\mathbb{E}^{\mathcal{F}_{t}}\bigg[\sup\nolimits_{r\in\left[  t,T\right]
}\big|e^{V_{r}^{\left(  +\right)  }}Y_{r}^{\varepsilon}\big|^{2}+%
{\displaystyle\int_{t}^{T}}
e^{2V_{r}^{\left(  +\right)  }}\left\vert Z_{r}^{\varepsilon}\right\vert
^{2}dr\bigg]\leq C_{\lambda}\,\mathbb{E}^{\mathcal{F}_{t}}\bigg[\big|e^{V_{T}%
^{\left(  +\right)  }}\eta\big|^{2}+\Big(%
{\displaystyle\int_{t}^{T}}
e^{V_{r}^{\left(  +\right)  }}\bar{H}_{1}^{\#}\left(  r\right)  dQ_{r}%
\Big)^{2}\bigg]
\]
($C_{\lambda}=C_{2,\lambda}\,,$ where $C_{2,\lambda}$ is given by
(\ref{an3a})).$\medskip$

From the above inequality we get, using (\ref{ap-eq-1c}) and (\ref{ma-4}),
that, $\mathbb{P}$--a.s., for all $t\in\left[  0,T\right]  ,$%
\begin{equation}%
\begin{array}
[c]{ll}%
\left(  a\right)  & \left\vert Y_{t}^{\varepsilon}\right\vert \leq
e^{V_{t}^{\left(  +\right)  }}\left\vert Y_{t}^{\varepsilon}\right\vert
\leq\left[  \mathbb{E}^{\mathcal{F}_{t}}\Big(\sup\nolimits_{r\in\left[
t,T\right]  }\big|e^{V_{r}^{\left(  +\right)  }}Y_{r}^{\varepsilon}%
\big|^{2}\Big)\right]  ^{1/2}\leq(C_{\lambda}\tilde{L})^{1/2}=\tilde{\rho}%
_{0}\,,\medskip\\
\left(  b\right)  & \mathbb{E}\Big(%
{\displaystyle\int_{0}^{T}}
e^{2V_{r}^{\left(  +\right)  }}\left\vert Z_{r}^{\varepsilon}\right\vert
^{2}dr\Big)\leq\tilde{\rho}_{0}^{2}\,,\medskip\\
\left(  c\right)  & \left\vert F_{\varepsilon}\left(  t,Y_{t}^{\varepsilon
},Z_{t}^{\varepsilon}\right)  \right\vert \leq\ell_{r}\left\vert
Z_{r}^{\varepsilon}\right\vert +F_{1+\tilde{\rho}_{0}}^{\#}\left(  r\right)
,\quad\left\vert G_{\varepsilon}\left(  t,Y_{t}^{\varepsilon}\right)
\right\vert \leq G_{1+\tilde{\rho}_{0}}^{\#}\left(  r\right)  \medskip\\
\left(  d\right)  & \left\vert H_{\varepsilon}(r,Y_{r}^{\varepsilon}%
,Z_{r}^{\varepsilon})\right\vert \leq\left[  \alpha_{r}\left(  \ell
_{r}\left\vert Z_{r}^{\varepsilon}\right\vert +F_{1+\tilde{\rho}_{0}}%
^{\#}\left(  r\right)  \right)  +\left(  1-\alpha_{r}\right)  G_{1+\tilde
{\rho}_{0}}^{\#}\left(  r\right)  \right]  \,\mathbf{1}_{[0,\frac
{1}{\varepsilon}]}\left(  A_{r}\right)  .
\end{array}
\label{ap-eq-3}%
\end{equation}
Using the stochastic subdifferential inequality (see \cite[Lemma 2.38, Remark
2.39]{pa-ra/14})%
\[%
\begin{array}
[c]{r}%
e^{2V_{t}^{\left(  +\right)  }}\varphi_{\varepsilon}(Y_{t}^{\varepsilon})\leq
e^{2V_{s}^{\left(  +\right)  }}\varphi_{\varepsilon}(Y_{s}^{\varepsilon})+%
{\displaystyle\int_{t}^{s}}
e^{2V_{r}^{\left(  +\right)  }}\left\langle \nabla\varphi_{\varepsilon}%
(Y_{r}^{\varepsilon}),\Phi_{\varepsilon}(r,Y_{r}^{\varepsilon},Z_{r}%
^{\varepsilon})\right\rangle dQ_{r}\medskip\\
-%
{\displaystyle\int_{t}^{s}}
e^{2V_{r}^{\left(  +\right)  }}\left\langle \nabla\varphi_{\varepsilon}%
(Y_{r}^{\varepsilon}),Z_{r}^{\varepsilon}dB_{r}\right\rangle ,\quad0\leq t\leq
s\leq T
\end{array}
\]
(and similar inequality for $\psi_{\varepsilon}$) and following the ideas from
\cite{ma-ra/10}, \cite{ma-ra/15} and \cite[Section 5.6.2]{pa-ra/14}, we deduce
that:%
\begin{equation}%
\begin{array}
[c]{l}%
e^{2V_{t}^{\left(  +\right)  }}\left[  \varphi_{\varepsilon}(Y_{t}%
^{\varepsilon})+\psi_{\varepsilon}(Y_{t}^{\varepsilon})\right]  \medskip\\
\quad+%
{\displaystyle\int_{t}^{s}}
e^{2V_{r}^{\left(  +\right)  }}\Big[\alpha_{r}|\nabla\varphi_{\varepsilon
}(Y_{r}^{\varepsilon})|^{2}+\left\langle \nabla\varphi_{\varepsilon}%
(Y_{r}^{\varepsilon}),\nabla\psi_{\varepsilon}(Y_{r}^{\varepsilon
})\right\rangle +\left(  1-\alpha_{r}\right)  |\nabla\psi_{\varepsilon}%
(Y_{r}^{\varepsilon})|^{2}\Big]dQ_{r}\medskip\\
\leq e^{2V_{s}^{\left(  +\right)  }}\left[  \varphi_{\varepsilon}%
(Y_{s}^{\varepsilon})+\psi_{\varepsilon}(Y_{s}^{\varepsilon})\right]  +%
{\displaystyle\int_{t}^{s}}
e^{2V_{r}^{\left(  +\right)  }}\left\langle \nabla\varphi_{\varepsilon}%
(Y_{r}^{\varepsilon})+\nabla\psi_{\varepsilon}(Y_{r}^{\varepsilon
}),H_{\varepsilon}(r,Y_{r}^{\varepsilon},Z_{r}^{\varepsilon})\right\rangle
dQ_{r}\medskip\\
\quad-%
{\displaystyle\int_{t}^{s}}
e^{2V_{r}^{\left(  +\right)  }}\left\langle \nabla\varphi_{\varepsilon}%
(Y_{r}^{\varepsilon})+\nabla\psi_{\varepsilon}(Y_{r}^{\varepsilon}%
),Z_{r}^{\varepsilon}dB_{r}\right\rangle .
\end{array}
\label{ap-eq-4}%
\end{equation}
The compatibility assumptions (\ref{compAssumpt}) and inequality (\ref{ma-4})
yield for $\left\vert y\right\vert \leq\tilde{\rho}_{0}:$%
\[%
\begin{array}
[c]{l}%
\displaystyle\left\langle \nabla\psi_{\varepsilon}(y),F_{\varepsilon
}(t,y,z)\right\rangle \medskip\\
\displaystyle=%
{\displaystyle\int_{\overline{B\left(  0,1\right)  }}}
\left\langle \nabla\psi_{\varepsilon}(y)-\nabla\psi_{\varepsilon
}(y-\varepsilon u),F\left(  t,y-\varepsilon u,\beta_{\varepsilon}\left(
z\right)  \right)  \right\rangle \,\mathbf{1}_{\left[  0,1\right]  }\left(
\varepsilon\left\vert F\left(  t,y-\varepsilon u\right)  ,0\right\vert
\right)  \rho\left(  u\right)  du\medskip\\
\displaystyle\quad+%
{\displaystyle\int_{\overline{B\left(  0,1\right)  }}}
\left\langle \nabla\psi_{\varepsilon}(y-\varepsilon u),F\left(
t,y-\varepsilon u,\beta_{\varepsilon}\left(  z\right)  \right)  \right\rangle
\,\mathbf{1}_{\left[  0,1\right]  }\left(  \varepsilon\left\vert F\left(
t,y-\varepsilon u\right)  ,0\right\vert \right)  \rho\left(  u\right)
du\medskip\\
\displaystyle\leq\frac{1}{\varepsilon}\,%
{\displaystyle\int_{\overline{B\left(  0,1\right)  }}}
\left\vert \varepsilon u\right\vert \,\left\vert F\left(  t,y-\varepsilon
u,\beta_{\varepsilon}\left(  z\right)  \right)  \right\vert \,\mathbf{1}%
_{\left[  0,1\right]  }\left(  \varepsilon\left\vert F\left(  t,y-\varepsilon
u\right)  ,0\right\vert \right)  \rho\left(  u\right)  du\medskip\\
\displaystyle\quad+%
{\displaystyle\int_{\overline{B\left(  0,1\right)  }}}
\left\vert \nabla\psi_{\varepsilon}(y-\varepsilon u)\right\vert \left\vert
F\left(  t,y-\varepsilon u,\beta_{\varepsilon}\left(  z\right)  \right)
\right\vert \,\mathbf{1}_{\left[  0,1\right]  }\left(  \varepsilon\left\vert
F\left(  t,y-\varepsilon u\right)  ,0\right\vert \right)  \rho\left(
u\right)  du\medskip\\
\displaystyle\leq\left\vert F\right\vert _{\varepsilon}\left(  t,y,z\right)  +%
{\displaystyle\int_{\overline{B\left(  0,1\right)  }}}
\left[  \left\vert \nabla\varphi_{\varepsilon}(y-\varepsilon u)-\nabla
\varphi_{\varepsilon}(y)\right\vert +\left\vert \nabla\varphi_{\varepsilon
}(y)\right\vert \right]  \medskip\\
\displaystyle\quad\quad\quad\quad\quad\quad\quad\quad\cdot\left\vert F\left(
t,y-\varepsilon u,\beta_{\varepsilon}\left(  z\right)  \right)  \right\vert
\,\mathbf{1}_{\left[  0,1\right]  }\left(  \varepsilon\left\vert F\left(
t,y-\varepsilon u\right)  ,0\right\vert \right)  \rho\left(  u\right)
du\medskip\\
\displaystyle\leq\left\vert F\right\vert _{\varepsilon}\left(  t,y,z\right)
+\left(  1+\left\vert \nabla\varphi_{\varepsilon}(y)\right\vert \right)
\,\left\vert F\right\vert _{\varepsilon}\left(  t,y,z\right)  =\left(
2+\left\vert \nabla\varphi_{\varepsilon}(y)\right\vert \right)  \,\left\vert
F\right\vert _{\varepsilon}\left(  t,y,z\right)  \medskip\\
\displaystyle\leq2L\left\vert z\right\vert +2F_{1+\tilde{\rho}_{0}}%
^{\#}\left(  t\right)  +\left\vert \nabla\varphi_{\varepsilon}(y)\right\vert
\left(  L\left\vert z\right\vert +F_{1+\tilde{\rho}_{0}}^{\#}\left(  t\right)
\right)
\end{array}
\]
and similarly%
\[
\left\langle \nabla\varphi_{\varepsilon}(y),G_{\varepsilon}(t,y)\right\rangle
=\left(  2+\left\vert \nabla\psi_{\varepsilon}(y)\right\vert \right)
\,\left\vert G\right\vert _{\varepsilon}\left(  t,y\right)  \leq
2G_{1+\tilde{\rho}_{0}}^{\#}\left(  t\right)  +\left\vert \nabla
\psi_{\varepsilon}(y)\right\vert G_{1+\tilde{\rho}_{0}}^{\#}\left(  t\right)
.\medskip
\]
Hence, using the above inequalities and the definition of $H_{\varepsilon
}(t,y,z),$ we have, for any $\left\vert y\right\vert \leq\tilde{\rho}_{0}\,,$%
\[%
\begin{array}
[c]{l}%
\left\langle \nabla\varphi_{\varepsilon}(y)+\nabla\psi_{\varepsilon
}(y),H_{\varepsilon}(s,y,z)\right\rangle \medskip\\
=\left\langle \nabla\varphi_{\varepsilon}(y)+\nabla\psi_{\varepsilon
}(y),\alpha_{s}F_{\varepsilon}\left(  s,y,z\right)  +\left(  1-\alpha
_{s}\right)  G_{\varepsilon}\left(  s,y\right)  \right\rangle \mathbf{1}%
_{[0,\frac{1}{\varepsilon}]}\left(  A_{r}\right)  \medskip\\
\leq\alpha_{t}\,\left(  2+2\left\vert \nabla\varphi_{\varepsilon
}(y)\right\vert \right)  \,\left\vert F\right\vert _{\varepsilon}\left(
t,y,z\right)  +\left(  1-\alpha_{t}\right)  \,\left(  2+2\left\vert \nabla
\psi_{\varepsilon}(y)\right\vert \right)  \,\left\vert G\right\vert
_{\varepsilon}\left(  t,y\right)  \medskip\\
\leq\alpha_{t}\Big[\dfrac{1}{2}\left\vert \nabla\varphi_{\varepsilon
}(y)\right\vert ^{2}+1+3\big(\left\vert F\right\vert _{\varepsilon}\left(
t,y,z\right)  \big)^{2}\Big]+\left(  1-\alpha_{t}\right)  \Big[\dfrac{1}%
{2}\left\vert \nabla\psi_{\varepsilon}(y)\right\vert ^{2}+1+3\big(\left\vert
G\right\vert _{\varepsilon}\left(  t,y\right)  \big)^{2}\Big].
\end{array}
\]
Using (\ref{minimum point}) we deduce inequality%
\begin{equation}
\mathbb{E}\big(e^{2V_{T}^{\left(  +\right)  }}\left(  \varphi_{\varepsilon
}(Y_{T}^{\varepsilon})+\psi_{\varepsilon}(Y_{T}^{\varepsilon})\right)
\big)\leq\mathbb{E}\big(e^{2V_{T}^{\left(  +\right)  }}\left(  \varphi
(\eta)+\psi(\eta)\right)  \big). \label{ap-eq-8}%
\end{equation}
On the other hand,%
\begin{equation}
M_{t}^{\varepsilon}=\int_{0}^{t}e^{2V_{r}^{\left(  +\right)  }}\left\langle
\nabla\varphi_{\varepsilon}(Y_{r}^{\varepsilon})+\nabla\psi_{\varepsilon
}(Y_{r}^{\varepsilon}),Z_{r}^{\varepsilon}dB_{r}\right\rangle \quad\text{is a
martingale.} \label{ap-eq_10}%
\end{equation}
By Young's inequality and assumption (\ref{ap-eq-1bb}) we have%
\begin{equation}
\mathbb{E}\int_{0}^{T}e^{2V_{r}^{\left(  +\right)  }}dQ_{r}\leq\mathbb{E}%
\big(e^{2V_{T}^{\left(  +\right)  }}Q_{T}\big)\leq\left[  \frac{2}{2+\delta
}\,\mathbb{E}\exp\left(  2+\delta\right)  V_{T}^{\left(  +\right)  }%
+\frac{\delta}{2+\delta}\,\mathbb{E}Q_{T}^{\left(  2+\delta\right)  /\delta
}\right]  <\infty\label{ap-eq_11}%
\end{equation}
and%
\begin{equation}
\mathbb{E}\big(e^{2V_{T}^{\left(  +\right)  }}\left(  \varphi(\eta)+\psi
(\eta)\right)  \big)\leq2L\,\mathbb{E}\big(e^{2V_{T}^{\left(  +\right)  }%
}\big)<\infty. \label{ap-eq_12}%
\end{equation}
Therefore, using inequalities (\ref{ma-4}), (\ref{ap-eq-3}--$a$) and
(\ref{ap-eq-8}--\ref{ap-eq_12}), we deduce from (\ref{ap-eq-4}) that, for all
$0\leq t\leq s\leq T,$%
\begin{equation}%
\begin{array}
[c]{l}%
\mathbb{E}e^{2V_{t}^{\left(  +\right)  }}\varphi_{\varepsilon}(Y_{t}%
^{\varepsilon})+\mathbb{E}e^{2V_{t}^{\left(  +\right)  }}\psi_{\varepsilon
}(Y_{t}^{\varepsilon})+\dfrac{1}{2}\,\mathbb{E}%
{\displaystyle\int_{t}^{T}}
e^{2V_{r}^{\left(  +\right)  }}\Big[|\nabla\varphi_{\varepsilon}%
(Y_{r}^{\varepsilon})|^{2}dr+|\nabla\psi_{\varepsilon}(Y_{r}^{\varepsilon
})|^{2}dA_{r}\Big]\medskip\\
\leq\mathbb{E}\left[  e^{2V_{T}^{\left(  +\right)  }}\left(  \varphi
(\eta)+\psi(\eta)\right)  \right]  \medskip\\
\quad+\mathbb{E}%
{\displaystyle\int_{t}^{T}}
e^{2V_{r}^{\left(  +\right)  }}\left(  1+3\big(\left\vert F\right\vert
_{\varepsilon}\left(  r,Y_{r}^{\varepsilon},Z_{r}^{\varepsilon}\right)
\big)^{2}\right)  dr+\mathbb{E}%
{\displaystyle\int_{t}^{T}}
e^{2V_{r}^{\left(  +\right)  }}\left(  1+3\big(\left\vert G\right\vert
_{\varepsilon}\left(  r,Y_{r}^{\varepsilon}\right)  \big)^{2}\right)
dA_{r}\medskip\\
\leq\mathbb{E}\left[  e^{2V_{T}^{\left(  +\right)  }}\left(  \varphi
(\eta)+\psi(\eta)\right)  \right]  \medskip\\
\quad+\mathbb{E}%
{\displaystyle\int_{t}^{T}}
e^{2V_{r}^{\left(  +\right)  }}\left[  1+6L^{2}\left\vert Z_{r}^{\varepsilon
}\right\vert ^{2}+6\big(F_{1+\tilde{\rho}_{0}}^{\#}\left(  r\right)
\big)^{2}\right]  dr+\mathbb{E}%
{\displaystyle\int_{t}^{T}}
e^{2V_{r}^{\left(  +\right)  }}\left[  1+3\big(G_{1+\tilde{\rho}_{0}}%
^{\#}\left(  r\right)  \big)^{2}\right]  dA_{r}\,.
\end{array}
\label{ap-eq-5a}%
\end{equation}
Therefore, by assumption (\ref{ap-eq-6}) and (\ref{ap-eq-3}$-b$),%
\begin{equation}%
\begin{array}
[c]{ll}%
\left(  a\right)  & \sup_{t\in\left[  0,T\right]  }\left[  \mathbb{E}%
e^{2V_{t}^{\left(  +\right)  }}\varphi_{\varepsilon}(Y_{t}^{\varepsilon
})+\mathbb{E}e^{2V_{t}^{\left(  +\right)  }}\psi_{\varepsilon}(Y_{t}%
^{\varepsilon})\right]  \leq C_{\tilde{\rho}_{0},L,T,\lambda}\medskip\\
\left(  b\right)  & \mathbb{E}%
{\displaystyle\int_{0}^{T}}
e^{2V_{r}^{\left(  +\right)  }}\Big[|\nabla\varphi_{\varepsilon}%
(Y_{r}^{\varepsilon})|^{2}dr+|\nabla\psi_{\varepsilon}(Y_{r}^{\varepsilon
})|^{2}dA_{r}\Big]\leq C_{\tilde{\rho}_{0},L,T,\lambda}%
\end{array}
\label{ap-eq-7}%
\end{equation}
($C_{\tilde{\rho}_{0},L,T,\lambda}$ is independent of $\varepsilon$%
).$\medskip$

Let $\varepsilon,\delta\in(0,1].$ We have%
\[
Y_{t}^{\varepsilon}-Y_{t}^{\delta}=\int_{t}^{T}dK_{r}^{\varepsilon,\delta
}-\int_{t}^{T}(Z_{r}^{\varepsilon}-Z_{r}^{\delta})dB_{r}\,,\quad
\mathbb{P}\text{--a.s.,}%
\]
with%
\begin{align*}
dK_{r}^{\varepsilon,\delta}  &  =\Big[H_{\varepsilon}(r,Y_{r}^{\varepsilon
},Z_{r}^{\varepsilon})-H_{\delta}(r,Y_{r}^{\delta},Z_{r}^{\delta})-\left[
\nabla_{y}\Psi^{\varepsilon}(r,Y_{r}^{\varepsilon})-\nabla_{y}\Psi^{\delta
}(r,Y_{r}^{\delta})\right]  \Big]dQ_{r}\\[2pt]
&  =\alpha_{r}\left[  F_{\varepsilon}(r,Y_{r}^{\varepsilon},Z_{r}%
^{\varepsilon})-F_{\delta}(r,Y_{r}^{\delta},Z_{r}^{\delta})\right]
\mathbf{1}_{\left[  0,\frac{1}{\varepsilon}\right]  }\left(  A_{r}\right)
dr\\[2pt]
&  \quad+\left(  1-\alpha_{r}\right)  \left[  G_{\varepsilon}(r,Y_{r}%
^{\varepsilon})-G_{\delta}(r,Y_{r}^{\delta})\right]  \mathbf{1}_{\left[
0,\frac{1}{\varepsilon}\right]  }\left(  A_{r}\right)  dA_{r}\\[2pt]
&  \quad+\alpha_{r}\,F_{\delta}(r,Y_{r}^{\delta},Z_{r}^{\delta}%
)\Big(\mathbf{1}_{\left[  0,\frac{1}{\varepsilon}\right]  }\left(
A_{r}\right)  -\mathbf{1}_{\left[  0,\frac{1}{\delta}\right]  }\left(
A_{r}\right)  \Big)dr\\[2pt]
&  \quad+\left(  1-\alpha_{r}\right)  G_{\delta}(r,Y_{r}^{\delta
})\Big(\mathbf{1}_{\left[  0,\frac{1}{\varepsilon}\right]  }\left(
A_{r}\right)  -\mathbf{1}_{\left[  0,\frac{1}{\delta}\right]  }\left(
A_{r}\right)  \Big)dA_{r}\\[2pt]
&  \quad-\alpha_{r}\left[  \nabla\varphi_{\varepsilon}(Y_{r}^{\varepsilon
})-\nabla\varphi_{\delta}(Y_{r}^{\delta})\right]  dr-\left(  1-\alpha
_{r}\right)  \left[  \nabla\psi_{\varepsilon}(Y_{r}^{\varepsilon})-\nabla
\psi_{\delta}(Y_{r}^{\delta})\right]  dA_{r}\,.
\end{align*}
By (\ref{fi-Cauchy}) and (\ref{ma-3}$-c$) and (\ref{ap-eq-3}--$a$) we have
(since $n_{p}\leq1$)%
\begin{align*}
\big\langle Y_{r}^{\varepsilon}-Y_{r}^{\delta},dK_{r}^{\varepsilon,\delta
}\big\rangle  &  \leq dR_{r}^{\varepsilon,\delta}+|Y_{r}^{\varepsilon}%
-Y_{r}^{\delta}|dN_{r}^{\varepsilon,\delta}+|Y_{r}^{\varepsilon}-Y_{r}%
^{\delta}|^{2}dV_{r}^{\left(  +\right)  }+\dfrac{\lambda}{2}\,|Z_{r}%
^{\varepsilon}-Z_{r}^{\delta}|^{2}dr\\[2pt]
&  \leq\left(  1+2\tilde{\rho}_{0}\right)  d\left(  R_{r}^{\varepsilon,\delta
}+N_{r}^{\varepsilon,\delta}\right)  +|Y_{r}^{\varepsilon}-Y_{r}^{\delta}%
|^{2}dV_{r}^{\left(  +\right)  }+\dfrac{\lambda}{2}\,|Z_{r}^{\varepsilon
}-Z_{r}^{\delta}|^{2}dr,
\end{align*}
where%
\begin{align*}
dR_{r}^{\varepsilon,\delta}  &  =\left\vert \varepsilon-\delta\right\vert
\left[  \mu_{r}^{+}\,\left\vert \varepsilon-\delta\right\vert +2F_{1+\tilde
{\rho}_{0}}^{\#}\left(  r\right)  +2\ell_{r}\left\vert Z_{r}^{\varepsilon
}\right\vert \right]  dr+\left\vert \varepsilon-\delta\right\vert \left[
\nu_{r}^{+}\left\vert \varepsilon-\delta\right\vert +2G_{1+\tilde{\rho}_{0}%
}^{\#}\left(  r\right)  \right]  dA_{r}\\[2pt]
&  \quad+\dfrac{\varepsilon+\delta}{2}\big(|\nabla\varphi_{\varepsilon}%
(Y_{r}^{\varepsilon})|^{2}+|\nabla\varphi_{\delta}(Y_{r}^{\delta}%
)|^{2}\big)dr+\dfrac{\varepsilon+\delta}{2}\big(|\nabla\psi_{\varepsilon
}(Y_{r}^{\varepsilon})|^{2}+|\nabla\psi_{\delta}(Y_{r}^{\delta})|^{2}%
\big)dA_{r}%
\end{align*}
and%
\begin{align*}
dN_{r}^{\varepsilon,\delta}  &  =\Big[2\mu_{r}^{+}\,\left\vert \varepsilon
-\delta\right\vert +\ell_{r}\left\vert Z_{r}^{\delta}\right\vert
\,\mathbf{1}_{[\frac{1}{\varepsilon}\wedge\frac{1}{\delta},\infty)}\left(
\left\vert Z_{r}^{\delta}\right\vert +A_{r}\right)  \,\mathbf{1}%
_{\varepsilon\neq\delta}\\[2pt]
&  \quad+\big(F_{1+\tilde{\rho}_{0}}^{\#}\left(  r\right)  +\ell_{r}\left\vert
Z_{r}^{\delta}\right\vert \big)\,\mathbf{1}_{[\frac{1}{\varepsilon}\wedge
\frac{1}{\delta},\infty)}\big(F_{1+\tilde{\rho}_{0}}^{\#}\left(  r\right)
+A_{r}\big)\Big]dr\\[2pt]
&  \quad+\left[  2\nu_{r}^{+}\,\left\vert \varepsilon-\delta\right\vert
+G_{1+\tilde{\rho}_{0}}^{\#}\left(  r\right)  \,\mathbf{1}_{[\frac
{1}{\varepsilon}\wedge\frac{1}{\delta},\infty)}\big(G_{1+\tilde{\rho}_{0}%
}^{\#}\left(  r\right)  +A_{r}\big)\right]  dA_{r}\,.
\end{align*}
By (\ref{ap-eq-9}) and Proposition \ref{Appendix_result 1} we get%
\[
\mathbb{E}\sup\nolimits_{r\in\left[  0,T\right]  }e^{2V_{r}^{\left(  +\right)
}}|Y_{r}^{\varepsilon}-Y_{r}^{\delta}|^{2}+\mathbb{E}{%
{\displaystyle\int_{0}^{T}}
}e^{2V_{r}^{\left(  +\right)  }}|Z_{r}^{\varepsilon}-Z_{r}^{\delta}|^{2}dr\leq
C_{\lambda}\,\mathbb{E}%
{\displaystyle\int_{0}^{T}}
e^{2V_{r}^{\left(  +\right)  }}d\left(  R_{r}^{\varepsilon,\delta}%
+N_{r}^{\varepsilon,\delta}\right)  .
\]
Boundedness assumptions (\ref{exponential_moment}), (\ref{ip-mnl}),
(\ref{ap-eq-1b}), (\ref{ap-eq-1bb}), (\ref{ap-eq-6}) and (\ref{ap-eq-3}$-b$),
(\ref{ap-eq-7}$-b$) yield%
\begin{equation}
\lim\nolimits_{\varepsilon,\delta\rightarrow0}\mathbb{E}%
{\displaystyle\int_{0}^{T}}
e^{2V_{r}^{\left(  +\right)  }}d\left(  R_{r}^{\varepsilon,\delta}%
+N_{r}^{\varepsilon,\delta}\right)  =0 \label{Cauchy_sequence}%
\end{equation}
(also the calculus for obtaining (\ref{ap-eq-9}) is useful).

For instance, if we denote%
\[
\bar{H}_{1+\tilde{\rho}_{0}}^{\#}\left(  t\right)
\xlongequal{\hspace{-4pt}{\rm def}\hspace{-4pt}}\alpha_{t}F_{1+\tilde{\rho
}_{0}}^{\#}\left(  t\right)  +\left(  1-\alpha_{t}\right)  G_{1+\tilde{\rho
}_{0}}^{\#}\left(  t\right)  ,
\]
we deduce%
\[%
\begin{array}
[c]{l}%
\displaystyle\mathbb{E}%
{\displaystyle\int_{0}^{T}}
e^{2V_{r}^{\left(  +\right)  }}\left[  F_{1+\tilde{\rho}_{0}}^{\#}\left(
r\right)  dr+G_{1+\tilde{\rho}_{0}}^{\#}\left(  r\right)  dA_{r}\right]
=\mathbb{E}%
{\displaystyle\int_{0}^{T}}
e^{2V_{r}^{\left(  +\right)  }}\bar{H}_{1+\tilde{\rho}_{0}}^{\#}\left(
r\right)  dQ_{r}\medskip\\
\displaystyle\leq\bigg(\mathbb{E}%
{\displaystyle\int_{0}^{T}}
e^{2V_{r}^{\left(  +\right)  }}\big(\bar{H}_{1+\tilde{\rho}_{0}}^{\#}\left(
r\right)  \big)^{2}dQ_{r}\bigg)^{1/2}\bigg(\mathbb{E}%
{\displaystyle\int_{0}^{T}}
e^{2V_{r}^{\left(  +\right)  }}dQ_{r}\bigg)^{1/2}\medskip\\
\displaystyle\leq\sqrt{2}\bigg(\mathbb{E}%
{\displaystyle\int_{0}^{T}}
e^{2V_{r}^{\left(  +\right)  }}\left[  \big(F_{1+\tilde{\rho}_{0}}^{\#}\left(
r\right)  \big)^{2}dr+\big(G_{1+\tilde{\rho}_{0}}^{\#}\left(  r\right)
\big)^{2}dA_{r}\right]  \bigg)^{1/2}\left(  \mathbb{E}e^{2V_{T}^{\left(
+\right)  }}Q_{T}\right)  ^{1/2}%
\end{array}
\]
which is finite, using assumption (\ref{ap-eq-6}) and inequality
(\ref{ap-eq_11}).

For instance, for any $a>0,$%
\[%
\begin{array}
[c]{l}%
\displaystyle\mathbb{E}%
{\displaystyle\int_{0}^{T}}
e^{2V_{r}^{\left(  +\right)  }}\left(  \mu_{r}^{+}\,dr+\nu_{r}^{+}%
\,dA_{r}\right)  \leq\mathbb{E}\bigg(e^{2V_{T}^{\left(  +\right)  }}%
{\displaystyle\int_{0}^{T}}
\left(  \mu_{r}^{+}\,dr+\nu_{r}^{+}\,dA_{r}\right)  \bigg)\medskip\\
\displaystyle\leq\frac{2}{2+a}\,\mathbb{E}\exp\left(  2+a\right)
V_{T}^{\left(  +\right)  }+\frac{a}{2+a}\,\mathbb{E}\bigg(%
{\displaystyle\int_{0}^{T}}
\left(  \mu_{r}^{+}\,dr+\nu_{r}^{+}\,dA_{r}\right)  \bigg)^{\left(
2+a\right)  /a}.
\end{array}
\]
On the other hand, by Holder's inequality,%
\[%
\begin{array}
[c]{l}%
\displaystyle\mathbb{E}\bigg(\int_{0}^{T}\left(  \mu_{r}^{+}\,dr+\nu_{r}%
^{+}\,dA_{r}\right)  \bigg)^{\left(  2+a\right)  /a}\leq\bigg(\mathbb{E}%
\bigg(\int_{0}^{T}\left(  \mu_{r}^{+}\,dr+\nu_{r}^{+}\,dA_{r}\right)
\bigg)^{k}\bigg)^{\frac{2+a}{ak}}\medskip\\
\displaystyle\leq\bigg(\frac{k!}{p^{k}}\bigg)^{\frac{2+a}{ak}}\bigg(\mathbb{E}%
\Big(e^{p\int_{0}^{T}\left(  \mu_{r}^{+}\,dr+\nu_{r}^{+}\,dA_{r}\right)
}\Big)\bigg)^{\frac{2+a}{ak}},
\end{array}
\]
where $\mathbb{N}^{\ast}\ni k=\left[  \frac{2+a}{a}\right]  +1>\frac{2+a}%
{a}\,,$ since%
\[
\alpha^{k}\leq\frac{k!}{p^{k}}\,e^{p\alpha},\quad\text{for any }\alpha>0,p>1.
\]
Hence $\mathbb{E}%
{\displaystyle\int_{0}^{T}}
e^{2V_{r}^{\left(  +\right)  }}\left(  \mu_{r}^{+}\,dr+\nu_{r}^{+}%
\,dA_{r}\right)  <\infty.\medskip$

For instance,%
\[%
\begin{array}
[c]{l}%
\displaystyle\mathbb{E}%
{\displaystyle\int_{0}^{T}}
e^{2V_{r}^{\left(  +\right)  }}\ell_{r}\,\left\vert Z_{r}^{\delta}\right\vert
\,\mathbf{1}_{[\frac{1}{\varepsilon}\wedge\frac{1}{\delta},\infty)}\left(
\left\vert Z_{r}^{\delta}\right\vert +A_{r}\right)  dr\leq L\left(
\varepsilon+\delta\right)  \,\mathbb{E}%
{\displaystyle\int_{0}^{T}}
e^{2V_{r}^{\left(  +\right)  }}\left\vert Z_{r}^{\delta}\right\vert \,\left(
\left\vert Z_{r}^{\delta}\right\vert +A_{r}\right)  dr\medskip\\
\displaystyle=L\left(  \varepsilon+\delta\right)  \,\bigg(\mathbb{E}%
{\displaystyle\int_{0}^{T}}
e^{2V_{r}^{\left(  +\right)  }}\left\vert Z_{r}^{\delta}\right\vert
^{2}dr+\mathbb{E}%
{\displaystyle\int_{0}^{T}}
e^{2V_{r}^{\left(  +\right)  }}\left\vert Z_{r}^{\delta}\right\vert
\,A_{r}\,dr\bigg),
\end{array}
\]
since%
\[
\mathbf{1}_{[\frac{1}{\varepsilon}\wedge\frac{1}{\delta},\infty)}\left(
\left\vert Z_{r}^{\delta}\right\vert +A_{r}\right)  \leq\displaystyle\frac
{\left\vert Z_{r}^{\delta}\right\vert +A_{r}}{\frac{1}{\varepsilon}\wedge
\frac{1}{\delta}}\leq\left(  \varepsilon+\delta\right)  \left(  \left\vert
Z_{r}^{\delta}\right\vert +A_{r}\right)  \,,\quad\text{for any }r.
\]
On the other hand,%
\[%
\begin{array}
[c]{l}%
\displaystyle\mathbb{E}%
{\displaystyle\int_{0}^{T}}
e^{2V_{r}^{\left(  +\right)  }}\left\vert Z_{r}^{\delta}\right\vert
\,A_{r}\,dr\leq\frac{1}{2}\mathbb{E}%
{\displaystyle\int_{0}^{T}}
e^{2V_{r}^{\left(  +\right)  }}\left\vert Z_{r}^{\delta}\right\vert
^{2}dr+\frac{1}{2}\mathbb{E}%
{\displaystyle\int_{0}^{T}}
e^{2V_{r}^{\left(  +\right)  }}A_{r}^{2}\,dr\medskip\\
\displaystyle\leq\frac{1}{2}\mathbb{E}%
{\displaystyle\int_{0}^{T}}
e^{2V_{r}^{\left(  +\right)  }}\left\vert Z_{r}^{\delta}\right\vert
^{2}dr+\frac{T}{2}\,\mathbb{E}\left(  e^{2V_{T}^{\left(  +\right)  }}A_{T}%
^{2}\right)  \medskip\\
\displaystyle\leq\frac{1}{2}\mathbb{E}%
{\displaystyle\int_{0}^{T}}
e^{2V_{r}^{\left(  +\right)  }}\left\vert Z_{r}^{\delta}\right\vert
^{2}dr+\frac{2}{2+\delta}\,\mathbb{E}\exp\left(  2+\delta\right)
V_{T}^{\left(  +\right)  }+\frac{\delta}{2+\delta}\,\mathbb{E}A_{T}^{\left(
4+2\delta\right)  /\delta}\medskip\\
\displaystyle\leq\frac{1}{2}\mathbb{E}%
{\displaystyle\int_{0}^{T}}
e^{2V_{r}^{\left(  +\right)  }}\left\vert Z_{r}^{\delta}\right\vert
^{2}dr+\frac{2}{2+\delta}\,\mathbb{E}\exp\left(  2+\delta\right)
V_{T}^{\left(  +\right)  }+\frac{\delta}{2+\delta}\,\mathbb{E}e^{\frac
{4+2\delta}{\delta}A_{T}}.
\end{array}
\]
Hence $\mathbb{E}%
{\displaystyle\int_{0}^{T}}
e^{2V_{r}^{\left(  +\right)  }}\ell_{r}\,\left\vert Z_{r}^{\delta}\right\vert
\,\mathbf{1}_{[\frac{1}{\varepsilon}\wedge\frac{1}{\delta},\infty)}\left(
\left\vert Z_{r}^{\delta}\right\vert +A_{r}\right)  dr\rightarrow0,$ as
$\varepsilon,\delta\rightarrow0.\medskip$

Using the similar calculus for the other quantities, conclusion
(\ref{Cauchy_sequence}) is completely proved.$\medskip$

Consequently there exists $(Y,Z)\in S_{m}^{0}\times\Lambda_{m\times k}^{0}$
such that%
\[
\mathbb{E}\sup\nolimits_{r\in\left[  0,T\right]  }e^{2V_{r}^{\left(  +\right)
}}|Y_{r}^{\varepsilon}-Y_{r}|^{2}+\mathbb{E}\int_{0}^{T}e^{2V_{r}^{\left(
+\right)  }}|Z_{r}^{\varepsilon}-Z_{r}|^{2}dr\rightarrow0,\quad\text{as
}\varepsilon\rightarrow0.
\]
From (\ref{ap-eq-7}) there exist two p.m.s.p. $U^{\left(  1\right)  }$ and
$U^{\left(  2\right)  }$, such that along a sequence $\varepsilon
_{n}\rightarrow0$, we have%
\begin{align*}
e^{V^{\left(  +\right)  }}\nabla\varphi_{\varepsilon_{n}}(Y^{\varepsilon_{n}%
})  &  \xrightharpoonup[]{\;\;\;\;}e^{V^{\left(  +\right)  }}U^{\left(
1\right)  },\quad\text{weakly in }L^{2}\left(  \Omega\times\left[  0,T\right]
,d\mathbb{P}\otimes dt;\mathbb{R}^{m}\right)  ,\\[3pt]
e^{V^{\left(  +\right)  }}\nabla\psi_{\varepsilon_{n}}(Y^{\varepsilon_{n}})
&  \xrightharpoonup[]{\;\;\;\;}e^{V^{\left(  +\right)  }}U^{\left(  2\right)
},\quad\text{weakly in }L^{2}\left(  \Omega\times\left[  0,T\right]
,d\mathbb{P}\otimes dA_{t};\mathbb{R}^{m}\right)  .
\end{align*}
Passing to limit in the approximating equation (\ref{ap-eq -1}), for
$\varepsilon=\varepsilon_{n}\rightarrow0,$ we infer%
\[
Y_{t}+\int_{t}^{T}U_{r}dQ_{r}=\eta+\int_{t}^{T}H(r,Y_{r},Z_{r})dQ_{r}-\int
_{t}^{T}Z_{r}dB_{r}\,,\;\mathbb{P}\text{--a.s.,}%
\]
where%
\[
U_{r}\xlongequal{\hspace{-4pt}{\rm def}\hspace{-4pt}}\left[  \alpha_{r}%
U_{r}^{1}+\left(  1-\alpha_{r}\right)  U_{r}^{2}\right]  ,\;\text{for }%
r\in\left[  0,T\right]  .
\]
Since $\nabla\varphi_{\varepsilon}(y)\in\partial\varphi\left(  y-\varepsilon
\nabla\varphi_{\varepsilon}(y)\right)  $ then for all $E\in\mathcal{F},$
$0\leq t\leq s\leq T$ and $X\in S_{m}^{2}\left[  0,T\right]  $%
\[%
\begin{array}
[c]{r}%
\mathbb{E}%
{\displaystyle\int_{t}^{s}}
\big\langle e^{2V_{r}^{\left(  +\right)  }}\nabla\varphi_{\varepsilon_{n}%
}(Y_{r}^{\varepsilon_{n}}),X_{r}-Y_{r}^{\varepsilon_{n}}%
\big\rangle\,\mathbf{1}_{E}\,dr+\mathbb{E}%
{\displaystyle\int_{t}^{s}}
e^{2V_{r}^{\left(  +\right)  }}\varphi(Y_{r}^{\varepsilon_{n}}-\varepsilon
\nabla\varphi_{\varepsilon}(Y_{r}^{\varepsilon_{n}}))\,\mathbf{1}%
_{E}\,dr\medskip\\
\leq\mathbb{E}%
{\displaystyle\int_{t}^{s}}
e^{2V_{r}^{\left(  +\right)  }}\varphi(X_{r})\,\mathbf{1}_{E}\,dr.
\end{array}
\]
Passing to $\liminf_{n\rightarrow\infty}$ in the above inequality we obtain
$U_{s}^{\left(  1\right)  }\in\partial\varphi(Y_{s}),\quad d\mathbb{P}\otimes
ds-$a.e. and, with similar arguments, $U_{s}^{\left(  2\right)  }\in
\partial\psi(Y_{s}),\quad d\mathbb{P}\otimes dA_{s}-$a.e..$\medskip$

Summarizing the above conclusions we conclude that $(Y,Z,U)\in S_{m}%
^{0}\left[  0,T\right]  \times\Lambda_{m\times k}^{0}\left[  0,T\right]
\times\Lambda_{m\times k}^{0}\left[  0,T\right]  $ is a strong solution of%
\begin{equation}
\left\{
\begin{array}
[c]{r}%
Y_{t}+%
{\displaystyle\int_{t}^{T}}
\big(U_{s}^{\left(  1\right)  }ds+U_{s}^{\left(  2\right)  }dA_{s}\big)=\eta+%
{\displaystyle\int_{t}^{T}}
\left[  F\left(  s,Y_{s},Z_{s}\right)  ds+G\left(  s,Y_{s}\right)
dA_{s}\right]  -%
{\displaystyle\int_{t}^{T}}
Z_{s}dB_{s}\,,\medskip\\
\text{for any }t\in\left[  0,T\right]  ,\medskip\\
\multicolumn{1}{l}{U_{s}^{\left(  1\right)  }\in\partial\varphi(Y_{s}),\quad
d\mathbb{P}\otimes ds-\text{a.e.}\quad\text{and }\quad U_{s}^{\left(
2\right)  }\in\partial\psi(Y_{s}),\quad d\mathbb{P}\otimes dA_{s}%
-\text{a.e.,}\quad\text{on }\left[  0,T\right]  .}%
\end{array}
\right.  \label{BSVI-0T}%
\end{equation}
Moreover, from (\ref{ap-eq-3}),%
\begin{equation}%
\begin{array}
[c]{ll}%
\left(  a\right)  & \left\vert Y_{t}\right\vert \leq e^{V_{t}^{\left(
+\right)  }}\left\vert Y_{t}\right\vert \leq\left[  \mathbb{E}^{\mathcal{F}%
_{t}}\Big(\sup\nolimits_{r\in\left[  t,T\right]  }\big|e^{V_{r}^{\left(
+\right)  }}Y_{r}\big|^{2}\Big)\right]  ^{1/2}\leq\tilde{\rho}_{0}%
\,,\medskip\\
\left(  b\right)  & \mathbb{E}\Big(%
{\displaystyle\int_{0}^{T}}
e^{2V_{r}^{\left(  +\right)  }}\left\vert Z_{r}\right\vert ^{2}dr\Big)\leq
\tilde{\rho}_{0}^{2}\,.
\end{array}
\label{b-1}%
\end{equation}
From inequalities (\ref{ma-4}) and (\ref{ap-eq-3}--$a$) we have%
\[%
\begin{array}
[c]{l}%
\displaystyle\big(\left\vert F\right\vert _{\varepsilon}\left(  r,Y_{r}%
^{\varepsilon},Z_{r}^{\varepsilon}\right)  \big)^{2}\leq2L^{2}\left\vert
Z_{r}^{\varepsilon}\right\vert ^{2}+2\big(F_{1+\tilde{\rho}_{0}}^{\#}\left(
r\right)  \big)^{2}\quad\text{and}\medskip\\
\displaystyle\big(\left\vert G\right\vert _{\varepsilon}\left(  r,Y_{r}%
^{\varepsilon}\right)  \big)^{2}\leq\big(G_{1+\tilde{\rho}_{0}}^{\#}\left(
r\right)  \big)^{2}%
\end{array}
\]
and therefore, passing to $\liminf_{\varepsilon\rightarrow0}$ in the first
inequality of (\ref{ap-eq-5a}) and using (\ref{ma-1}$-b$) and Fatou's Lemma
and Lebesgue's dominated convergence theorem, we have%
\begin{equation}%
\begin{array}
[c]{l}%
\dfrac{1}{2}\,\mathbb{E}%
{\displaystyle\int_{0}^{T}}
e^{2V_{r}^{\left(  +\right)  }}\Big[\big|U_{r}^{\left(  1\right)  }%
\big|^{2}dr+\big|U_{r}^{\left(  2\right)  }\big|^{2}dA_{r}\Big]\medskip\\
\leq\mathbb{E}\left[  e^{2V_{T}^{\left(  +\right)  }}\left(  \varphi
(\eta)+\psi(\eta)\right)  \right]  +\mathbb{E}%
{\displaystyle\int_{0}^{T}}
e^{2V_{r}^{\left(  +\right)  }}\left(  1+6L^{2}\left\vert Z_{r}\right\vert
^{2}+6\left\vert F\left(  r,Y_{r},0\right)  \right\vert ^{2}\right)
dr\medskip\\
\quad+\mathbb{E}%
{\displaystyle\int_{0}^{T}}
e^{2V_{r}^{\left(  +\right)  }}\left(  1+3\left\vert G\left(  r,Y_{r}\right)
\right\vert ^{2}\right)  dA_{r}\,.
\end{array}
\label{b-1a}%
\end{equation}
\hfill
\end{proof}

\begin{proposition}
[$L^{p}$-- variational solution]\label{p1-wvs}We suppose that assumptions
$\left(  \mathrm{A}_{1}-\mathrm{A}_{7}\right)  $ are satisfied. Let
$V^{\left(  +\right)  }$ be given by definition (\ref{defV_3}). In addition we
assume that:$\medskip$

\noindent$\left(  i\right)  $ there exists $\hat{L}>0$ such that%
\begin{equation}
\big|e^{V_{T}^{\left(  +\right)  }}\eta\big|^{2}+\Big(%
{\displaystyle\int_{0}^{T}}
e^{V_{r}^{\left(  +\right)  }}\left(  \left\vert F\left(  r,0,0\right)
\right\vert dr+\left\vert G\left(  r,0\right)  \right\vert dA_{r}\right)
\Big)^{2}\leq\hat{L}; \label{assump-H(t,0,0)}%
\end{equation}
\noindent$\left(  ii\right)  $ there exists $a\in(1+n_{p}\lambda,p\wedge2)$
such that%
\begin{equation}%
\begin{array}
[c]{rl}%
\left(  a\right)  & \mathbb{E}\bigg(%
{\displaystyle\int_{0}^{T}}
\ell_{s}^{2}ds\bigg)^{\frac{a}{2-a}}<\infty,\medskip\\
\left(  b\right)  & \mathbb{E}\bigg[%
{\displaystyle\int_{0}^{T}}
e^{V_{s}^{\left(  +\right)  }}\left(  F_{1+\hat{\rho}_{0}}^{\#}\left(
s\right)  ds+G_{1+\hat{\rho}_{0}}^{\#}\left(  s\right)  dA_{s}\right)
\bigg]^{a}<\infty,
\end{array}
\label{ip-1a}%
\end{equation}
where\footnote{The constant $C_{\lambda}:=C_{2,\lambda}\,,$ where
$C_{2,\lambda}$ is given by (\ref{an3a}).} $\hat{\rho}_{0}%
\xlongequal{\hspace{-4pt}{\rm def}\hspace{-4pt}}(C_{\lambda}\hat{L}%
)^{1/2};\medskip$

\noindent$\left(  iii\right)  $ there exists a p.m.s.p. $\left(  \Theta
_{t}\right)  _{t\geq0}$ and, for each $\rho\geq0,$ there exist an
non-decreasing function $K_{\rho}:\mathbb{R}_{+}\rightarrow\mathbb{R}_{+}$
such that%
\begin{equation}
F_{\rho}^{\#}\left(  t\right)  +G_{\rho}^{\#}\left(  t\right)  \leq K_{\rho
}\left(  \Theta_{t}\right)  ,\quad\text{a.e. }t\in\left[  0,T\right]
,\;\mathbb{P}\text{--a.s..} \label{ip-1b}%
\end{equation}
Then the multivalued BSDE
\[
\left\{
\begin{array}
[c]{r}%
\displaystyle Y_{t}+{\int_{t}^{T}}dK_{r}=\eta+{\int_{t}^{T}}H\left(
r,Y_{r},Z_{r}\right)  dQ_{r}-{\int_{t}^{T}}Z_{r}dB_{r}\,,\quad\mathbb{P}%
\text{--a.s., for all }t\in\left[  0,T\right]  ,\medskip\\
\multicolumn{1}{l}{\displaystyle dK_{r}=U_{r}dQ_{r}\in\partial_{y}\Psi\left(
r,Y_{r}\right)  dQ_{r}}%
\end{array}
\,\right.
\]
has a unique $L^{p}$--variational solution, in the sense of Definition
\ref{definition_weak solution}.$\medskip$

Moreover, for all $t\in\left[  0,T\right]  ,$ $\mathbb{P}$--a.s.,%
\begin{equation}%
\begin{array}
[c]{r}%
\displaystyle\mathbb{E}^{\mathcal{F}_{t}}\Big(\sup\nolimits_{s\in\left[
t,T\right]  }\big|e^{V_{s}^{\left(  +\right)  }}Y_{s}\big|^{p}\Big)+\mathbb{E}%
^{\mathcal{F}_{t}}\Big(%
{\displaystyle\int_{t}^{T}}
e^{2V_{s}^{\left(  +\right)  }}\left(  \varphi\left(  Y_{s}\right)
ds+\psi\left(  Y_{s}\right)  dA_{s}\right)  \Big)^{p/2}\medskip\\
\displaystyle\quad+\mathbb{E}^{\mathcal{F}_{t}}\Big(%
{\displaystyle\int_{0}^{T}}
e^{2V_{s}^{\left(  +\right)  }}\left\vert Z_{s}\right\vert ^{2}ds\Big)^{p/2}%
\medskip\\
\displaystyle\leq C_{p,\lambda}\,\mathbb{E}^{\mathcal{F}_{t}}\left[
e^{pV_{T}^{\left(  +\right)  }}\left\vert \eta\right\vert ^{p}+\Big(%
{\displaystyle\int_{t}^{T}}
e^{V_{s}^{\left(  +\right)  }}\left(  \left\vert F\left(  r,0,0\right)
\right\vert dr+\left\vert G\left(  t,0\right)  \right\vert dA_{r}\right)
\Big)^{p}\right]  .
\end{array}
\label{es-p}%
\end{equation}

\end{proposition}

\begin{proof}
Let $t\in\left[  0,T\right]  ,$ $n\in\mathbb{N}^{\ast}$ and%
\[
\beta_{t}=t+A_{t}+\left\vert \mu_{t}\right\vert +\left\vert \nu_{t}\right\vert
+\ell_{t}+V_{t}^{\left(  +\right)  }+F_{1+\hat{\rho}_{0}}^{\#}\left(
t\right)  +G_{1+\hat{\rho}_{0}}^{\#}\left(  t\right)  +\Theta_{t}\,.
\]
Consider the BSDE%
\begin{equation}
\left\{
\begin{array}
[c]{l}%
Y_{t}^{\left(  n\right)  }+%
{\displaystyle\int_{t}^{T}}
U_{s}^{\left(  n\right)  }dQ_{s}=\eta^{\left(  n\right)  }+%
{\displaystyle\int_{t}^{T}}
H^{\left(  n\right)  }\big(s,Y_{s}^{\left(  n\right)  },Z_{s}^{\left(
n\right)  }\big)dQ_{s}-%
{\displaystyle\int_{t}^{T}}
Z_{s}^{\left(  n\right)  }dB_{s}\,,\quad t\in\left[  0,T\right]  ,\medskip\\
U_{s}^{\left(  n\right)  }=\alpha_{r}U_{r}^{\left(  1,n\right)  }+\left(
1-\alpha_{r}\right)  U_{r}^{\left(  2,n\right)  }\medskip\\
U_{s}^{\left(  1,n\right)  }\in\partial\varphi(Y_{s}^{\left(  n\right)
}),\quad d\mathbb{P}\otimes ds\text{--a.e.}\quad\text{and }\quad
U_{s}^{\left(  2,n\right)  }\in\partial\psi(Y_{s}^{\left(  n\right)  }),\quad
d\mathbb{P}\otimes dA_{s}\text{--a.e. on }\left[  0,T\right]  ,
\end{array}
\right.  \label{st2-1}%
\end{equation}
where%
\[%
\begin{array}
[c]{l}%
\displaystyle\eta^{\left(  n\right)  }%
\xlongequal{\hspace{-4pt}{\rm def}\hspace{-4pt}}\eta\,\mathbf{1}_{\left[
0,n\right]  }\big(\left\vert \eta\right\vert +\varphi\left(  \eta\right)
+\psi\left(  \eta\right)  +V_{T}^{\left(  +\right)  }\big),\medskip\\
\displaystyle F^{\left(  n\right)  }\left(  t,y,z\right)
\xlongequal{\hspace{-4pt}{\rm def}\hspace{-4pt}}F\left(  t,y,z\right)
\mathbf{1}_{\left[  0,n\right]  }\left(  \beta_{t}\right)  \quad
\text{and}\quad G^{\left(  n\right)  }\left(  t,y,z\right)
\xlongequal{\hspace{-4pt}{\rm def}\hspace{-4pt}}G\left(  t,y\right)
\mathbf{1}_{\left[  0,n\right]  }\left(  \beta_{t}\right)  ,\medskip\\
\displaystyle H^{\left(  n\right)  }\left(  s,y,z\right)
\xlongequal{\hspace{-4pt}{\rm def}\hspace{-4pt}}\alpha_{s}F^{\left(  n\right)
}\left(  s,y,z\right)  +\left(  1-\alpha_{s}\right)  G^{\left(  n\right)
}\left(  s,y\right)  .
\end{array}
\]
If we denote%
\[
\displaystyle\mu_{t}^{\left(  n\right)  }%
\xlongequal{\hspace{-4pt}{\rm def}\hspace{-4pt}}\mathbf{1}_{\left[
0,n\right]  }\left(  \beta_{t}\right)  \,\mu_{t}\,,\quad\quad\nu_{t}^{\left(
n\right)  }\xlongequal{\hspace{-4pt}{\rm def}\hspace{-4pt}}\mathbf{1}_{\left[
0,n\right]  }\left(  \beta_{t}\right)  \,\nu_{t}\,,\quad\quad\ell_{t}^{\left(
n\right)  }\xlongequal{\hspace{-4pt}{\rm def}\hspace{-4pt}}\mathbf{1}_{\left[
0,n\right]  }\left(  \beta_{t}\right)  \,\ell_{t}\,,
\]
then we have%
\[%
\begin{array}
[c]{l}%
\displaystyle\big\langle y-\hat{y},H^{\left(  n\right)  }(t,y,z)-H^{\left(
n\right)  }(t,\hat{y},z)\big\rangle\leq\big(\alpha_{t}\mu_{t}^{\left(
n\right)  }+\left(  1-\alpha_{t}\right)  \nu_{t}^{\left(  n\right)
}\big)\left\vert y-\hat{y}\right\vert ^{2},\medskip\\
\displaystyle\big|H^{\left(  n\right)  }(t,y,z)-H^{\left(  n\right)
}(t,y,\hat{z})\big|\leq\alpha_{t}\ell_{t}^{\left(  n\right)  }\,\left\vert
z-\hat{z}\right\vert .
\end{array}
\]
Of course,%
\[%
\begin{array}
[c]{l}%
\displaystyle\big|\eta^{\left(  n\right)  }\big|\leq n\,\mathbf{1}_{\left[
0,n\right]  }\big(\left\vert \eta\right\vert +V_{T}^{\left(  +\right)
}\big),\medskip\\
\displaystyle\big|\mu_{t}^{\left(  n\right)  }\big|\leq n\,\mathbf{1}_{\left[
0,n\right]  }\left(  \beta_{t}\right)  ,\quad\big|\nu_{t}^{\left(  n\right)
}\big|\leq n\,\mathbf{1}_{\left[  0,n\right]  }\left(  \beta_{t}\right)
,\quad\big|\ell_{t}^{\left(  n\right)  }\big|\leq n\,\mathbf{1}_{\left[
0,n\right]  }\left(  \beta_{t}\right)  ,\medskip\\
\displaystyle F_{1}^{\left(  n\right)  \#}\left(  t\right)  =\sup
\nolimits_{\left\vert u\right\vert \leq1}\big|F^{\left(  n\right)  }\left(
t,u,0\right)  \big|\leq n\mathbf{1}_{\left[  0,n\right]  }\left(  \beta
_{t}\right)  ,\medskip\\
\displaystyle G_{1}^{\left(  n\right)  \#}\left(  t\right)  =\sup
\nolimits_{\left\vert u\right\vert \leq1}\big|G^{\left(  n\right)  }\left(
t,u\right)  \big|\leq n\mathbf{1}_{\left[  0,n\right]  }\left(  \beta
_{t}\right)  .\medskip
\end{array}
\]
Let%
\[
\theta_{n}:=\inf\big\{r\geq0:r+A_{r}+V_{r}^{\left(  +\right)  }>n\big\}.
\]
We have $\mathbf{1}_{\left[  0,n\right]  }\left(  \beta_{r}\right)
\leq\mathbf{1}_{\left[  0,\theta_{n}\right]  }\left(  r\right)  $ and
therefore%
\begin{align}
V_{t}^{\left(  n,+\right)  }  &
\xlongequal{\hspace{-4pt}{\rm def}\hspace{-4pt}}\int_{0}^{t}\left[
\Big(\mu_{r}^{\left(  n\right)  }+\frac{1}{2n_{p}\lambda}\,\big(\ell
_{r}^{\left(  n\right)  }\big)^{2}\Big)^{+}dr+\big(\nu_{r}^{\left(  n\right)
}\big)^{+}dA_{r}\right] \label{defV_4}\\[2pt]
&  =\int_{0}^{t}\mathbf{1}_{\left[  0,n\right]  }\left(  \beta_{r}\right)
\left[  \Big(\mu_{r}+\frac{1}{2n_{p}\lambda}\left(  \ell_{r}\right)
^{2}\Big)^{+}dr+\nu_{r}^{+}dA_{r}\right]  \medskip\nonumber\\[2pt]
&  \leq\int_{0}^{t}\mathbf{1}_{\left[  0,\theta_{n}\right]  }\left(  \beta
_{r}\right)  \left[  \Big(\mu_{r}+\frac{1}{2n_{p}\lambda}\left(  \ell
_{r}\right)  ^{2}\Big)^{+}dr+\nu_{r}^{+}dA_{r}\right]  =V_{t\wedge\theta_{n}%
}^{\left(  +\right)  }\leq V_{\theta_{n}}^{\left(  +\right)  }\leq n\nonumber
\end{align}
and%
\begin{align*}
&  \big|e^{V_{T}^{\left(  n,+\right)  }}\eta^{\left(  n\right)  }%
\big|^{2}+\Big(%
{\displaystyle\int_{0}^{T}}
e^{V_{r}^{\left(  n,+\right)  }}\big(F_{1}^{\left(  n\right)  \#}\left(
r\right)  dr+G_{1}^{\left(  n\right)  \#}\left(  r\right)  dA_{r}%
\big)\Big)^{2}\\
&  \leq n^{2}e^{2n}+n^{2}e^{2n}\bigg(%
{\displaystyle\int_{0}^{T\wedge\theta_{n}}}
\left(  dr+dA_{r}\right)  \bigg)^{2}\leq n^{2}e^{2n}\left(  1+n^{2}\right)
=\tilde{L}^{\left(  n\right)  }%
\end{align*}
and, for every $\rho\geq0,$%
\[
F_{\rho}^{\left(  n\right)  \#}\left(  t\right)  +G_{\rho}^{\left(  n\right)
\#}\left(  t\right)  \leq\left[  F_{\rho}^{\#}\left(  t\right)  +G_{\rho}%
^{\#}\left(  t\right)  \right]  \,\mathbf{1}_{\left[  0,n\right]  }\left(
\beta_{t}\right)  \leq K_{\rho}\left(  \Theta_{t}\right)  \,\mathbf{1}%
_{\left[  0,n\right]  }\left(  \beta_{t}\right)  \leq K_{\rho}\left(
n\right)  .
\]
Therefore assumptions (\ref{ap-eq-1b}--\ref{ap-eq-6}) are satisfied.$\medskip$

Hence, by Lemma \ref{l1-strong sol}, there exists a unique (strong) solution
$\left(  Y^{\left(  n\right)  },Z^{\left(  n\right)  },U^{\left(  n\right)
}\right)  \in S_{m}^{0}\left[  0,T\right]  \times\Lambda_{m\times k}%
^{0}\left(  0,T\right)  \times\Lambda_{m}^{0}\left(  0,T\right)  $ of BSDE
(\ref{st2-1}).$\medskip$

We have%
\begin{equation}%
\begin{array}
[c]{l}%
\big\langle Y_{t}^{\left(  n\right)  },H^{\left(  n\right)  }(t,Y_{t}^{\left(
n\right)  },Z_{t}^{\left(  n\right)  })-U_{t}^{\left(  n\right)
}\big\rangle dQ_{t}\medskip\\
\leq\Big[\Big(\alpha_{t}\mu_{t}+\left(  1-\alpha_{t}\right)  \nu_{t}%
+\alpha_{t}\,\dfrac{1}{2n_{p}\lambda}\,\ell_{t}^{2}\Big)\mathbf{1}_{\left[
0,n\right]  }\left(  \beta_{t}\right)  \big|Y_{t}^{\left(  n\right)
}\big|^{2}\medskip\\
\quad+\alpha_{t}\mathbf{1}_{\left[  0,n\right]  }\left(  \beta_{t}\right)
\,\dfrac{n_{p}\lambda}{2}\,\big|Z_{t}^{\left(  n\right)  }\big|^{2}%
+|H^{\left(  n\right)  }\left(  t,0,0\right)  |\big|Y_{t}^{\left(  n\right)
}\big|\Big]dQ_{t}\medskip\\
\leq\big|Y_{t}^{\left(  n\right)  }\big|d\bar{N}_{t}+\big|Y_{t}^{\left(
n\right)  }\big|^{2}dV_{t}^{\left(  +\right)  }+\dfrac{\lambda}{2}%
\,\big|Z_{t}^{\left(  n\right)  }\big|^{2}dt,
\end{array}
\label{yn-ineq}%
\end{equation}
where%
\begin{equation}
\bar{N}_{t}:=\int_{0}^{t}\left[  \left\vert F\left(  r,0,0\right)  \right\vert
dr+\left\vert G\left(  r,0\right)  \right\vert dA_{r}\right]  . \label{def_N}%
\end{equation}
Since by (\ref{b-1}), for all $t\in\left[  0,T\right]  $, $\mathbb{P}$--a.s.,%
\[
\big|Y_{t}^{\left(  n\right)  }\big|^{2}\leq\big|e^{V_{t}^{\left(  n,+\right)
}}Y_{t}^{\left(  n\right)  }\big|^{2}\leq\mathbb{E}^{\mathcal{F}_{t}}%
\Big(\sup\nolimits_{r\in\left[  t,T\right]  }\big|e^{V_{r}^{\left(
n,+\right)  }}Y_{r}^{\left(  n\right)  }\big|^{2}\Big)\leq C_{\lambda}%
\tilde{L}^{\left(  n\right)  }=:\left(  \tilde{\rho}_{0}^{n}\right)  ^{2}%
\]
and $\left\vert \eta^{\left(  n\right)  }\right\vert \leq\left\vert
\eta\right\vert $, we deduce, by Proposition \ref{Appendix_result 1} applied
to (\ref{yn-ineq}), that, for all $t\in\left[  0,T\right]  ,$%
\[
\mathbb{E}^{\mathcal{F}_{t}}\Big(\sup\nolimits_{r\in\left[  t,T\right]
}\big|e^{V_{r}^{\left(  +\right)  }}Y_{r}^{\left(  n\right)  }\big|^{2}%
\Big)+\mathbb{E}^{\mathcal{F}_{t}}\Big(%
{\displaystyle\int_{t}^{T}}
e^{2V_{r}^{\left(  +\right)  }}\big|Z_{r}^{\left(  n\right)  }\big|^{2}%
dr\Big)\leq C_{\lambda}\,\mathbb{E}^{\mathcal{F}_{t}}\bigg[\big|e^{V_{T}%
^{\left(  +\right)  }}\eta\big|^{2}+\Big(%
{\displaystyle\int_{t}^{T}}
e^{V_{r}^{\left(  +\right)  }}d\bar{N}_{r}\Big)^{2}\bigg]
\]
(the constant $C_{\lambda}:=C_{2,\lambda}\,,$ where $C_{2,\lambda}$ is given
by (\ref{an3a}).$\medskip$

By assumption (\ref{assump-H(t,0,0)}) we have, for all $t\in\left[
0,T\right]  ,$ $\mathbb{P}$--a.s.,%
\begin{equation}%
\begin{array}
[c]{l}%
\displaystyle\big|Y_{t}^{\left(  n\right)  }\big|\leq\big|e^{V_{t}^{\left(
+\right)  }}Y_{t}^{\left(  n\right)  }\big|\leq\Big[\mathbb{E}^{\mathcal{F}%
_{t}}\Big(\sup\nolimits_{r\in\left[  t,T\right]  }\big|e^{V_{r}^{\left(
+\right)  }}Y_{r}^{\left(  n\right)  }\big|^{2}\Big)\Big]^{1/2}\leq
(C_{\lambda}\hat{L})^{1/2}=\hat{\rho}_{0}\,,\medskip\\
\displaystyle\mathbb{E}%
{\displaystyle\int_{0}^{T}}
e^{2V_{r}^{\left(  +\right)  }}\big|Z_{r}^{\left(  n\right)  }\big|^{2}%
dr\leq\hat{\rho}_{0}^{2}\,.
\end{array}
\label{yn-ineq_2}%
\end{equation}
Let $n,i\in\mathbb{N}^{\ast}$ arbitrary fixed. We have%
\[%
\begin{array}
[c]{l}%
\displaystyle Y_{t}^{\left(  n+i\right)  }-Y_{t}^{\left(  n\right)  }+\int
_{t}^{T}\big(U_{s}^{\left(  n+i\right)  }-U_{s}^{\left(  n\right)
}\big)dQ_{s}\medskip\\
\multicolumn{1}{r}{\displaystyle=\eta^{\left(  n+i\right)  }-\eta^{\left(
n\right)  }+\int_{t}^{T}\Big(H^{\left(  n+i\right)  }\big(s,Y_{s}^{\left(
n+i\right)  },Z_{s}^{\left(  n+i\right)  }\big)-H^{\left(  n\right)
}\big(s,Y_{s}^{\left(  n\right)  },Z_{s}^{\left(  n\right)  }\big)\Big)dQ_{s}%
}\\
\multicolumn{1}{r}{\displaystyle-\int_{t}^{T}\big(Z_{s}^{\left(  n+i\right)
}-Z_{s}^{\left(  n\right)  }\big)dB_{s}\,.}%
\end{array}
\]
We recall the property%
\[
\big\langle Y_{s}^{\left(  n+i\right)  }-Y_{s}^{\left(  n\right)  }%
,U_{s}^{\left(  n+i\right)  }-U_{s}^{\left(  n\right)  }\big\rangle dQ_{s}%
\geq0.
\]
On the other hand%
\[%
\begin{array}
[c]{l}%
\displaystyle\big\langle Y_{s}^{\left(  n+i\right)  }-Y_{s}^{\left(  n\right)
},H^{\left(  n+i\right)  }\big(s,Y_{s}^{\left(  n+i\right)  },Z_{s}^{\left(
n+i\right)  }\big)-H^{\left(  n\right)  }\big(s,Y_{s}^{\left(  n\right)
},Z_{s}^{\left(  n\right)  }\big)\big\rangle dQ_{s}\medskip\\
\displaystyle=\big\langle Y_{s}^{\left(  n+i\right)  }-Y_{s}^{\left(
n\right)  },H^{\left(  n+i\right)  }\big(s,Y_{s}^{\left(  n+i\right)  }%
,Z_{s}^{\left(  n+i\right)  }\big)-H^{\left(  n+i\right)  }\big(s,Y_{s}%
^{\left(  n\right)  },Z_{s}^{\left(  n\right)  }\big)\big\rangle
dQ_{s}\medskip\\
\displaystyle\quad+\big\langle Y_{s}^{\left(  n+i\right)  }-Y_{s}^{\left(
n\right)  },H^{\left(  n+i\right)  }\big(s,Y_{s}^{\left(  n\right)  }%
,Z_{s}^{\left(  n\right)  }\big)-H^{\left(  n\right)  }\big(s,Y_{s}^{\left(
n\right)  },Z_{s}^{\left(  n\right)  }\big)\big\rangle dQ_{s}\medskip\\
\displaystyle\leq\mathbf{1}_{\left[  0,n+i\right]  }\left(  \beta_{s}\right)
\Big(\mu_{s}ds+\nu_{s}dA_{s}+\dfrac{1}{2n_{p}\lambda}\,\ell_{s}^{2}%
ds\Big)\big|Y_{s}^{\left(  n+i\right)  }-Y_{s}^{\left(  n\right)  }%
\big|^{2}+\dfrac{n_{p}\lambda}{2}\,\big|Z_{s}^{\left(  n+i\right)  }%
-Z_{s}^{\left(  n\right)  }\big|^{2}ds\medskip\\
\displaystyle\quad+\big|Y_{s}^{\left(  n+i\right)  }-Y_{s}^{\left(  n\right)
}\big|\,\left\vert \mathbf{1}_{\left[  0,n+i\right]  }\left(  \beta
_{s}\right)  -\mathbf{1}_{\left[  0,n\right]  }\left(  \beta_{s}\right)
\right\vert \,\big(\ell_{s}\big|Z_{s}^{\left(  n\right)  }\big|+F_{\hat{\rho
}_{0}}^{\#}\left(  s\right)  ds+G_{\hat{\rho}_{0}}^{\#}\left(  s\right)
dA_{s}\big)\medskip\\
\displaystyle\leq\big|Y_{s}^{\left(  n+i\right)  }-Y_{s}^{\left(  n\right)
}\big|\,\mathbf{1}_{\left(  n,\infty\right)  }\left(  \beta_{s}\right)
\,\big(\ell_{s}\big|Z_{s}^{\left(  n\right)  }\big|+F_{\hat{\rho}_{0}}%
^{\#}\left(  s\right)  ds+G_{\hat{\rho}_{0}}^{\#}\left(  s\right)
dA_{s}\big)\medskip\\
\displaystyle\quad+\big|Y_{s}^{\left(  n+i\right)  }-Y_{s}^{\left(  n\right)
}\big|^{2}dV_{s}^{\left(  +\right)  }+\dfrac{n_{a}\lambda^{\prime}}%
{2}\big|Z_{s}^{\left(  n+i\right)  }-Z_{s}^{\left(  n\right)  }\big|^{2}ds,
\end{array}
\]
since%
\begin{equation}
\left\vert \mathbf{1}_{\left[  0,n+i\right]  }\left(  \beta_{s}\right)
-\mathbf{1}_{\left[  0,n\right]  }\left(  \beta_{s}\right)  \right\vert
\leq\mathbf{1}_{\left(  n,\infty\right)  }\left(  \beta_{s}\right)
\label{indicator_ineq}%
\end{equation}
and%
\[
n_{p}\lambda\leq n_{a}\lambda^{\prime},\quad\text{with }\lambda^{\prime
}:=\frac{n_{p}\lambda+a-1}{2\left(  a-1\right)  }\in\left(  0,1\right)
\]
(where $a$ is given by assumption (\ref{ip-1a}) and $n_{a}:=\left(
a-1\right)  \wedge1=a-1\,$).$\medskip$

By (\ref{yn-ineq_2}) and Proposition \ref{Appendix_result 1} and H\"{o}lder's
inequality we obtain:%
\[%
\begin{array}
[c]{l}%
\displaystyle\mathbb{E}\sup\nolimits_{s\in\left[  0,T\right]  }e^{aV_{s}%
^{\left(  +\right)  }}\big|Y_{s}^{\left(  n+i\right)  }-Y_{s}^{\left(
n\right)  }\big|^{a}+\mathbb{E}\left(
{\displaystyle\int_{0}^{T}}
e^{aV_{s}^{\left(  +\right)  }}\big|Z_{s}^{\left(  n+i\right)  }%
-Z_{s}^{\left(  n\right)  }\big|^{2}ds\right)  ^{a/2}\medskip\\
\displaystyle\leq C_{a,\lambda}\,\mathbb{E}\left[  e^{aV_{T}^{\left(
+\right)  }}\left\vert \eta\right\vert ^{a}\mathbf{1}_{\left(  n,\infty
\right)  }\big(\left\vert \eta\right\vert +V_{T}^{\left(  +\right)
}\big)\right]  \medskip\\
\displaystyle\quad\mathbb{+}C_{a,\lambda}\,\mathbb{E}\left(
{\displaystyle\int_{0}^{T}}
e^{V_{s}^{\left(  +\right)  }}\mathbf{1}_{\left(  n,\infty\right)  }\left(
\beta_{s}\right)  \left[  \ell_{s}\big|Z_{s}^{\left(  n\right)  }%
\big|+F_{\hat{\rho}_{0}}^{\#}\left(  s\right)  ds+G_{\hat{\rho}_{0}}%
^{\#}\left(  s\right)  dA_{s}\right]  \right)  ^{a}\medskip\\
\displaystyle\leq C_{a,\lambda}\,\mathbb{E}\left[  e^{aV_{T}^{\left(
+\right)  }}\left\vert \eta\right\vert ^{a}\mathbf{1}_{\left(  n,\infty
\right)  }\big(\left\vert \eta\right\vert +V_{T}^{\left(  +\right)
}\big)\right]  \mathbb{+}2^{a-1}C_{a,\lambda}\,\mathbb{E}\left(
{\displaystyle\int_{0}^{T}}
e^{V_{s}^{\left(  +\right)  }}\mathbf{1}_{\left(  n,\infty\right)  }\left(
\beta_{s}\right)  \,\ell_{s}\,\big|Z_{s}^{\left(  n\right)  }\big|ds\right)
^{a}\medskip\\
\displaystyle\quad+2^{a-1}C_{a,\lambda}\,\mathbb{E}\left(
{\displaystyle\int_{0}^{T}}
e^{V_{s}^{\left(  +\right)  }}\mathbf{1}_{\left(  n,\infty\right)  }\left(
\beta_{s}\right)  \left(  F_{\hat{\rho}_{0}}^{\#}\left(  s\right)
ds+G_{\hat{\rho}_{0}}^{\#}\left(  s\right)  dA_{s}\right)  \right)
^{a}\medskip\\
\displaystyle\leq C_{a,\lambda}\,\mathbb{E}\left[  e^{aV_{T}^{\left(
+\right)  }}\left\vert \eta\right\vert ^{a}\mathbf{1}_{\left(  n,\infty
\right)  }\big(\left\vert \eta\right\vert +V_{T}^{\left(  +\right)
}\big)\right]  +C_{a,\lambda}^{\prime}\mathbb{E}\left[  \Big(%
{\displaystyle\int_{0}^{T}}
\ell_{s}^{2}\mathbf{1}_{(n,\infty)}\left(  \beta_{s}\right)  ds\Big)^{\frac
{a}{2}}\Big(%
{\displaystyle\int_{0}^{T}}
e^{2V_{s}^{\left(  +\right)  }}\left\vert Z_{s}^{n}\right\vert ^{2}%
ds\Big)^{\frac{a}{2}}\right]  \medskip\\
\displaystyle\quad+C_{a,\lambda}^{\prime}\,\mathbb{E}\left(
{\displaystyle\int_{0}^{T}}
e^{V_{s}^{\left(  +\right)  }}\mathbf{1}_{\left(  n,\infty\right)  }\left(
\beta_{s}\right)  \left(  F_{\hat{\rho}_{0}}^{\#}\left(  s\right)
ds+G_{\hat{\rho}_{0}}^{\#}\left(  s\right)  dA_{s}\right)  \right)
^{a}\medskip\\
\displaystyle\leq C_{a,\lambda}\,\hat{L}^{\frac{a}{2}}\,\mathbb{E}\left[
\mathbf{1}_{\left(  n,\infty\right)  }\big(\left\vert \eta\right\vert
+V_{T}^{\left(  +\right)  }\big)\right]  +C_{a,\lambda}^{\prime}\left[
\mathbb{E}\Big(%
{\displaystyle\int_{0}^{T}}
\ell_{s}^{2}\mathbf{1}_{(n,\infty)}\left(  \beta_{s}\right)  ds\Big)^{\frac
{a}{2-a}}\right]  ^{\frac{2-a}{2}}\left[  \mathbb{E}\Big(%
{\displaystyle\int_{0}^{T}}
e^{2V_{s}^{\left(  +\right)  }}\left\vert Z_{s}^{n}\right\vert ^{2}%
ds\Big)\right]  ^{\frac{a}{2}}\medskip\\
\displaystyle\quad+C_{a,\lambda}^{\prime}\,\mathbb{E}\left(
{\displaystyle\int_{0}^{T}}
e^{V_{s}^{\left(  +\right)  }}\mathbf{1}_{\left(  n,\infty\right)  }\left(
\beta_{s}\right)  \left(  F_{\hat{\rho}_{0}}^{\#}\left(  s\right)
ds+G_{\hat{\rho}_{0}}^{\#}\left(  s\right)  dA_{s}\right)  \right)  ^{a}.
\end{array}
\]
Hence there exists $\left(  Y,Z\right)  \in S_{m}^{0}\left[  0,T\right]
\times\Lambda_{m\times k}^{0}\left(  0,T\right)  $ such that%
\begin{equation}%
\begin{array}
[c]{rl}%
\left(  j\right)  & \left\vert Y_{t}\right\vert \leq e^{V_{t}^{\left(
+\right)  }}\left\vert Y_{t}\right\vert \leq(C_{\lambda}\hat{L})^{1/2}%
=\hat{\rho}_{0}\,,\quad\text{for all }t\in\left[  0,T\right]  ,\;\mathbb{P}%
\text{--a.s.},\medskip\\
\left(  jj\right)  & \mathbb{E}%
{\displaystyle\int_{0}^{T}}
e^{2V_{r}^{\left(  +\right)  }}\left\vert Z_{r}\right\vert ^{2}dr\leq\hat
{\rho}_{0}^{2}\,,\medskip\\
\left(  jjj\right)  & \lim\nolimits_{n\rightarrow\infty}\bigg[\mathbb{E}%
\sup\nolimits_{s\in\left[  0,T\right]  }e^{aV_{s}^{\left(  +\right)  }%
}\left\vert Y_{s}^{n}-Y_{s}\right\vert ^{a}+\mathbb{E}\Big(%
{\displaystyle\int_{0}^{T}}
e^{aV_{s}^{\left(  +\right)  }}\left\vert Z_{s}^{n}-Z_{s}\right\vert
^{2}ds\Big)^{a/2}\bigg]=0,\medskip\\
\left(  jv\right)  & \left(  Y_{t},Z_{t}\right)  =\left(  \eta,0\right)
,\quad\text{for all }t>T.
\end{array}
\label{c-1}%
\end{equation}
Using (\ref{st2-1}) and assumption $\left(  \mathrm{A}_{4}\right)  $ we deduce%
\[
\varphi\big(Y_{t}^{\left(  n\right)  }\big)dt+\psi\big(Y_{t}^{\left(
n\right)  }\big)dA_{t}\leq\big\langle Y_{t}^{\left(  n\right)  }%
,U_{t}^{\left(  1,n\right)  }\big\rangle dt+\big\langle Y_{t}^{\left(
n\right)  },U_{t}^{\left(  2,n\right)  }\big\rangle dA_{t}%
\]
and therefore%
\[%
\begin{array}
[c]{l}%
\displaystyle\varphi\big(Y_{t}^{\left(  n\right)  }\big)dt+\psi\big(Y_{t}%
^{\left(  n\right)  }\big)dA_{t}+\big\langle Y_{t}^{\left(  n\right)
},H^{\left(  n\right)  }(t,Y_{t}^{\left(  n\right)  },Z_{t}^{\left(  n\right)
})-U_{t}^{\left(  n\right)  }\big\rangle dQ_{t}\leq\big\langle Y_{t}^{\left(
n\right)  },H^{\left(  n\right)  }(t,Y_{t}^{\left(  n\right)  },Z_{t}^{\left(
n\right)  })\big\rangle dQ_{t}\medskip\\
\displaystyle\leq\Big[\Big(\alpha_{t}\mu_{t}+\left(  1-\alpha_{t}\right)
\nu_{t}+\alpha_{t}\,\dfrac{1}{2n_{p}\lambda}\,\ell_{t}^{2}\Big)\,\mathbf{1}%
_{\left[  0,n\right]  }\left(  \beta_{t}\right)  \big|Y_{t}^{\left(  n\right)
}\big|^{2}\medskip\\
\displaystyle\quad+\alpha_{t}\mathbf{1}_{\left[  0,n\right]  }\left(
\beta_{t}\right)  \,\dfrac{n_{p}\lambda}{2}\,\big|Z_{t}^{\left(  n\right)
}\big|^{2}+\big|H^{\left(  n\right)  }\left(  t,0,0\right)  \big|\big|Y_{t}%
^{\left(  n\right)  }\big|\Big]dQ_{t}\medskip\\
\displaystyle\leq\big|Y_{t}^{\left(  n\right)  }\big|d\bar{N}_{t}%
+\big|Y_{t}^{\left(  n\right)  }\big|^{2}dV_{t}^{\left(  +\right)  }%
+\dfrac{n_{p}\lambda}{2}\,\big|Z_{t}^{\left(  n\right)  }\big|^{2}dr,
\end{array}
\]
where $\bar{N}$ is defined by (\ref{def_N}).

Also by (\ref{yn-ineq_2}) and assumption (\ref{ip-mnl}) we have%
\[
\mathbb{E}\sup\nolimits_{t\in\left[  0,T\right]  }e^{pV_{t}^{\left(  +\right)
}}\big|Y_{t}^{\left(  n\right)  }\big|^{p}\leq\hat{\rho}_{0}^{p}%
\,\mathbb{E}\exp\left(  p\int_{0}^{T}\left(  \left\vert \mu_{s}\right\vert
+\frac{1}{2n_{p}\lambda}\,\ell_{s}^{2}\right)  ds+p\int_{0}^{T}\left\vert
\nu_{s}\right\vert dA_{s}\right)  <\infty.
\]
Hence, by Proposition \ref{Appendix_result 1}, we deduce that, for all
$t\in\left[  0,T\right]  ,$%
\begin{equation}%
\begin{array}
[c]{l}%
\displaystyle\mathbb{E}^{\mathcal{F}_{t}}\Big(\sup\nolimits_{s\in\left[
t,T\right]  }\big|e^{V_{s}^{\left(  +\right)  }}Y_{s}^{\left(  n\right)
}\big|^{p}\Big)+\mathbb{E}^{\mathcal{F}_{t}}\Big(%
{\displaystyle\int_{t}^{T}}
e^{2V_{s}^{\left(  +\right)  }}\Big(\varphi\big(Y_{s}^{\left(  n\right)
}\big)ds+\psi\big(Y_{s}^{\left(  n\right)  }\big)dA_{s}\Big)\Big)^{p/2}%
\medskip\\
\displaystyle\quad+\mathbb{E}^{\mathcal{F}_{t}}\Big(%
{\displaystyle\int_{t}^{T}}
e^{2V_{s}^{\left(  +\right)  }}\big|Z_{s}^{\left(  n\right)  }\big|^{2}%
ds\Big)^{p/2}\medskip\\
\displaystyle\leq C_{p,\lambda}\,\mathbb{E}^{\mathcal{F}_{t}}\left[
e^{pV_{T}^{\left(  +\right)  }}\left\vert \eta\right\vert ^{p}+\Big(%
{\displaystyle\int_{t}^{T}}
e^{V_{s}^{\left(  +\right)  }}d\bar{N}_{s}\Big)^{p}\right]  ,\quad
\mathbb{P}\text{--a.s..}%
\end{array}
\label{es-np}%
\end{equation}
By Remark \ref{s-w} and inequality $V_{t}\leq V_{s}^{\left(  +\right)  },$ we
see that $\left(  Y^{\left(  n\right)  },Z^{\left(  n\right)  }\right)  ,$ as
a strong solution of (\ref{st2-1}), is also an $L^{p}-$variational solution on
$\left[  0,T\right]  $ for (\ref{st2-1}).\medskip

Hence, for%
\[
q\in\{2,p\wedge2\},\quad\delta_{q}=\delta\mathbf{1}_{[1,2)}\left(  q\right)
\quad\text{and}\quad\Gamma_{t}^{\left(  n\right)  }=\big(\big|M_{t}%
-Y_{t}^{\left(  n\right)  }\big|^{2}+\delta_{q}\big)^{1/2},
\]
it holds%
\begin{equation}%
\begin{array}
[c]{l}%
\big(\Gamma_{t}^{\left(  n\right)  }\big)^{q}+\dfrac{q\left(  q-1\right)  }%
{2}\,%
{\displaystyle\int_{t}^{s}}
{\big(\Gamma_{r}^{\left(  n\right)  }\big)^{q-2}}{\Large \,}\big|R_{r}%
-Z_{r}^{\left(  n\right)  }\big|^{2}dr+{q%
{\displaystyle\int_{t}^{s}}
}\big({\Gamma_{r}^{\left(  n\right)  }\big)^{q-2}\Psi}\big(r,Y_{r}^{\left(
n\right)  }\big)dQ_{r}\medskip\\
\leq\big(\Gamma_{s}^{\left(  n\right)  }\big)^{q}+{q%
{\displaystyle\int_{t}^{s}}
}\big({\Gamma_{r}^{\left(  n\right)  }\big)^{q-2}\Psi}\left(  r,M_{r}\right)
dQ_{r}\medskip\\
\quad+q%
{\displaystyle\int_{t}^{s}}
{\big(\Gamma_{r}^{\left(  n\right)  }\big)^{q-2}}\langle M_{r}-Y_{r}^{\left(
n\right)  },N_{r}-H\big(r,Y_{r}^{\left(  n\right)  },Z_{r}^{\left(  n\right)
}\big)\rangle dQ_{r}\medskip\\
\quad-q%
{\displaystyle\int_{t}^{s}}
\big({\Gamma_{r}^{\left(  n\right)  }\big)^{q-2}}\,\langle M_{r}%
-Y_{r}^{\left(  n\right)  },{\big(}R_{r}-Z_{r}^{\left(  n\right)  }%
{\big)}dB_{r}\rangle;
\end{array}
\label{an-vws}%
\end{equation}
for any $0\leq t\leq s<\infty$ and $M\in\mathcal{V}_{m}^{0}$ of the form
(\ref{def_M}), i.e. $M_{t}=M_{T}+\int_{t}^{T}N_{r}dQ_{r}-\int_{t}^{T}%
R_{r}dB_{r}\,.\medskip$

By convergence result (\ref{c-1}$-jjj$) and assumptions $(\mathrm{A}%
_{4}-\mathrm{A}_{6})$ we can pass to $\liminf_{n\rightarrow\infty}$ (on a
subsequence) in (\ref{es-np}) and (\ref{an-vws}) to conclude that $\left(
Y,Z\right)  $ is also an $L^{p}-$variational solution on $\left[  0,T\right]
$ and inequality (\ref{es-p}) holds.\hfill
\end{proof}

\begin{corollary}
Let the assumptions of Proposition \ref{p1-wvs} be satisfied. If $\varphi
=\psi=0,$ then BSDE%
\begin{equation}
Y_{t}=\eta+{\int_{t}^{T}}H\left(  r,Y_{r},Z_{r}\right)  dQ_{r}-{\int_{t}^{T}%
}Z_{r}dB_{r}\,,\quad\mathbb{P}\text{--a.s., }t\in\left[  0,T\right]  ,
\label{bsde-cl}%
\end{equation}
has a unique strong solution $\left(  Y,Z\right)  \in S_{m}^{p}\left[
0,T\right]  \times\Lambda_{m\times k}^{p}\left(  0,T\right)  .$
\end{corollary}

\begin{proof}
Based on the results from (\ref{c-1}) and assumptions $(\mathrm{A}%
_{4}-\mathrm{A}_{6})$ we can pass to limit $\lim_{n\rightarrow\infty}$ in the
approximating equation (\ref{st2-1}) with $\varphi=\psi=0$ and $U^{\left(
1\right)  }=U^{\left(  2\right)  }=0$ to infer that $\left(  Y,Z\right)  $
satisfies (\ref{bsde-cl}). From (\ref{es-p}) and assumption
(\ref{assump-H(t,0,0)}) we get $\left(  Y,Z\right)  \in S_{m}^{p}\left[
0,T\right]  \times\Lambda_{m\times k}^{p}\left(  0,T\right)  .$ Moreover by
(\ref{c-1}$-j$)
\[
\left\vert Y_{t}\right\vert \leq\hat{\rho}_{0}\,,\quad\text{for all }%
t\in\left[  0,T\right]  ,\;\mathbb{P}\text{--a.s..}%
\]
\hfill
\end{proof}

\begin{corollary}
Let the assumptions of Proposition \ref{p1-wvs} be satisfied. In addition we
assume that:%
\begin{equation}%
\begin{array}
[c]{rl}%
\left(  i\right)  & \mathbb{E}\left[  e^{2V_{T}^{\left(  +\right)  }}\left(
\varphi(\eta)+\psi(\eta)\right)  \right]  <\infty,\medskip\\
\left(  ii\right)  & \mathbb{E}%
{\displaystyle\int_{0}^{T}}
e^{2V_{r}^{\left(  +\right)  }}dQ_{r}<\infty,\medskip\\
\left(  iii\right)  & \mathbb{E}%
{\displaystyle\int_{0}^{T}}
e^{2V_{r}^{\left(  +\right)  }}\left(  \big|F_{\hat{\rho}_{0}}^{\#}\left(
r\right)  \big|^{2}dr+\big|G_{\hat{\rho}_{0}}^{\#}\left(  r\right)
\big|^{2}dA_{r}\right)  <\infty.
\end{array}
\label{ip-diez}%
\end{equation}
Then the BSDE%
\[
\left\{
\begin{array}
[c]{l}%
\displaystyle Y_{t}+{\int_{t}^{T}}dK_{r}=Y_{T}+{\int_{t}^{T}}H\left(
r,Y_{r},Z_{r}\right)  dQ_{r}-{\int_{t}^{T}}Z_{r}dB_{r},\quad\mathbb{P}%
\text{--a.s., for all }t\in\left[  0,T\right]  ,\medskip\\
\displaystyle dK_{r}=U_{r}^{\left(  1\right)  }dr+U_{r}^{\left(  2\right)
}dA_{r}\,,\medskip\\
\displaystyle U^{\left(  1\right)  }dr\in\partial\varphi\left(  Y_{r}\right)
dr\quad\text{and}\quad U^{\left(  2\right)  }dA_{r}\in\partial\psi\left(
Y_{r}\right)  dA_{r}%
\end{array}
\,\right.
\]
has a unique strong a solution $\left(  Y,Z,U^{\left(  1\right)  },U^{\left(
2\right)  }\right)  \in S_{m}^{0}\times\Lambda_{m\times k}^{0}\times
\Lambda_{m}^{0}$ $\times\Lambda_{m}^{0}$ such that
\begin{equation}%
\begin{array}
[c]{r}%
\displaystyle\mathbb{E}\sup\nolimits_{t\in\left[  0,T\right]  }e^{2V_{t}%
}\left\vert Y_{t}\right\vert ^{2}+\mathbb{E}\bigg(%
{\displaystyle\int_{0}^{T}}
{e^{2V_{r}}}\left\vert Z_{r}\right\vert ^{2}dr\bigg)+\mathbb{E}\bigg(%
{\displaystyle\int_{0}^{T}}
{e^{2V_{r}}}\big|U_{r}^{\left(  1\right)  }\big|^{2}dr\bigg)\medskip\\
\displaystyle+\mathbb{E}\bigg(%
{\displaystyle\int_{0}^{T}}
{e^{2V_{r}}}\big|U_{r}^{\left(  2\right)  }\big|^{2}dA_{r}\bigg)<\infty.
\end{array}
\label{0}%
\end{equation}
Moreover
\[
\left\vert Y_{t}\right\vert \leq e^{V_{t}^{\left(  +\right)  }}\left\vert
Y_{t}\right\vert \leq\hat{\rho}_{0}\,,\quad\text{for all }t\in\left[
0,T\right]  ,\;\mathbb{P}\text{--a.s..}%
\]

\end{corollary}

\begin{proof}
We are in the framework of the proof of Proposition \ref{p1-wvs}. We can
deduce again inequality (\ref{ap-eq-5a}) and therefore, using (\ref{yn-ineq_2}%
),%
\[%
\begin{array}
[c]{l}%
\displaystyle\dfrac{1}{2}\,\mathbb{E}%
{\displaystyle\int_{0}^{T}}
e^{2V_{r}^{\left(  n,+\right)  }}\Big[\big|U_{r}^{\left(  1,n\right)
}\big|^{2}dr+\big|U_{r}^{\left(  2,n\right)  }\big|^{2}dA_{r}\Big]\medskip\\
\displaystyle\leq\mathbb{E}\left[  e^{2V_{T}^{\left(  n,+\right)  }%
}\big(\varphi(\eta^{\left(  n\right)  })+\psi(\eta^{\left(  n\right)
})\big)\right]  +\mathbb{E}%
{\displaystyle\int_{0}^{T}}
e^{2V_{r}^{\left(  n,+\right)  }}\left(  1+6L^{2}\big|Z_{r}^{\left(  n\right)
}\big|^{2}+6\big|F^{\left(  n\right)  }(r,Y_{r}^{\left(  n\right)
},0)\big|^{2}\right)  dr\medskip\\
\displaystyle\quad+\mathbb{E}%
{\displaystyle\int_{0}^{T}}
e^{2V_{r}^{\left(  n,+\right)  }}\left(  1+3\big|G^{\left(  n\right)
}(r,Y_{r}^{\left(  n\right)  })\big|^{2}\right)  dA_{r}\medskip\\
\displaystyle\leq\mathbb{E}\left[  e^{2V_{T}^{\left(  +\right)  }}%
\big(\varphi(\eta)+\psi(\eta)\big)\right]  +\mathbb{E}%
{\displaystyle\int_{0}^{T}}
e^{2V_{r}^{\left(  +\right)  }}\left(  dr+dA_{r}\right)  +6L^{2}\hat{\rho}%
_{0}^{2}\medskip\\
\displaystyle\quad+6\,\mathbb{E}%
{\displaystyle\int_{0}^{T}}
e^{2V_{r}^{\left(  +\right)  }}\big|F_{\hat{\rho}_{0}}^{\#}\left(  r\right)
\big|^{2}dr+3\,\mathbb{E}%
{\displaystyle\int_{0}^{T}}
e^{2V_{r}^{\left(  +\right)  }}\big|G_{\hat{\rho}_{0}}^{\#}\left(  r\right)
\big|^{2}dA_{r}\,,
\end{array}
\]
where $V^{\left(  n,+\right)  }$ is defined by (\ref{defV_4}).$\medskip$

Hence, using assumptions (\ref{ip-diez}), there exists $\big(\hat{U}^{\left(
1\right)  },\hat{U}^{\left(  2\right)  }\big)\in\Lambda_{m}^{0}\left(
0,T\right)  \times\Lambda_{m}^{0}\left(  0,T\right)  $ such that, on a
subsequence still denoted by $\left\{  U^{\left(  1,n\right)  },U^{\left(
2,n\right)  };\;n\in\mathbb{N}^{\ast}\right\}  ,$ we have, if we denote
$\left(  U^{\left(  1\right)  },U^{\left(  2\right)  }\right)
\xlongequal{\hspace{-4pt}{\rm def}\hspace{-4pt}}\big(e^{-V^{\left(  +\right)
}}\hat{U}^{\left(  1\right)  },e^{-V^{\left(  +\right)  }}\hat{U}^{\left(
2\right)  }\big)\in\Lambda_{m}^{0}\left(  0,T\right)  \times\Lambda_{m}%
^{0}\left(  0,T\right)  :$%
\[%
\begin{array}
[c]{l}%
e^{V_{r}^{\left(  n,+\right)  }}U^{\left(  1,n\right)  }%
\xrightharpoonup[]{\;\;\;\;}e^{V^{\left(  +\right)  }}U^{\left(  1\right)
},\quad\text{weakly in }L^{2}\left(  \Omega\times\left[  0,T\right]
,d\mathbb{P}\otimes dt;\mathbb{R}^{m}\right)  ,\medskip\\
e^{V^{\left(  n,+\right)  }}U^{\left(  2,n\right)  }%
\xrightharpoonup[]{\;\;\;\;}e^{V^{\left(  +\right)  }}U^{\left(  2\right)
},\quad\text{weakly in }L^{2}\left(  \Omega\times\left[  0,T\right]
,d\mathbb{P}\otimes dA_{t};\mathbb{R}^{m}\right)  .
\end{array}
\]
and, passing to the limit,%
\[%
\begin{array}
[c]{l}%
\dfrac{1}{2}\,\mathbb{E}%
{\displaystyle\int_{0}^{T}}
e^{2V_{r}^{\left(  +\right)  }}\Big[\big|U_{r}^{\left(  1\right)  }%
\big|^{2}dr+\big|U_{r}^{\left(  2\right)  }\big|^{2}dA_{r}\Big]\medskip\\
\leq\mathbb{E}\left[  e^{2V_{T}^{\left(  +\right)  }}\left(  \varphi
(\eta)+\psi(\eta)\right)  \right]  +\mathbb{E}%
{\displaystyle\int_{0}^{T}}
e^{2V_{r}^{\left(  +\right)  }}\left(  1+6L^{2}\left\vert Z_{r}\right\vert
^{2}+6\left\vert F\left(  r,Y_{r},0\right)  \right\vert ^{2}\right)
dr\medskip\\
\quad+\mathbb{E}%
{\displaystyle\int_{0}^{T}}
e^{2V_{r}^{\left(  +\right)  }}\left(  1+3\left\vert G\left(  r,Y_{r}\right)
\right\vert ^{2}\right)  dA_{r}\,.
\end{array}
\]
Passing to $\lim_{n\rightarrow\infty}$ in the approximating equation
(\ref{st2-1}) and using the results from the proof of Proposition \ref{p1-wvs}
we infer%
\begin{equation}
\left\{
\begin{array}
[c]{l}%
Y_{t}+%
{\displaystyle\int_{t}^{T}}
U_{s}dQ_{s}=\eta+%
{\displaystyle\int_{t}^{T}}
H\left(  s,Y_{s},Z_{s}\right)  dQ_{s}-%
{\displaystyle\int_{t}^{T}}
Z_{s}dB_{s}\,,\;t\in\left[  0,T\right]  ,\medskip\\
U_{s}=\alpha_{r}U_{r}^{\left(  1\right)  }+\left(  1-\alpha_{r}\right)
U_{r}^{\left(  2\right)  }\medskip\\
U_{s}^{\left(  1\right)  }\in\partial\varphi(Y_{s})~,\quad d\mathbb{P}\otimes
ds-a.e.\quad\text{and }\quad U_{s}^{\left(  2\right)  }\in\partial\psi
(Y_{s}),\quad d\mathbb{P}\otimes dA_{s}-a.e.\quad\text{on }\left[  0,T\right]
,
\end{array}
\right.  \label{st2-11}%
\end{equation}
and the conclusion follows.\hfill
\end{proof}

\begin{theorem}
[$L^{p}$-- variational solution]\label{t2exist}We suppose that assumptions
$\left(  \mathrm{A}_{1}-\mathrm{A}_{7}\right)  $ are satisfied. In addition we
assume that:$\medskip$

\noindent$\left(  i\right)  $%
\begin{equation}
\mathbb{E}\left[  e^{pV_{T}}\left\vert \eta\right\vert ^{p}+\Big(%
{\displaystyle\int_{0}^{T}}
e^{V_{s}}\left(  \left\vert F\left(  r,0,0\right)  \right\vert dr+\left\vert
G\left(  t,0\right)  \right\vert dA_{r}\right)  \Big)^{p}\right]  <\infty,
\label{ip-t2-2}%
\end{equation}
where $V$ is defined by (\ref{defV_1});$\medskip$

\noindent$\left(  ii\right)  $ there exists $a\in\left(  1+n_{p}%
\lambda,p\wedge2\right)  $ such that%
\begin{equation}%
\begin{array}
[c]{rl}%
\left(  a\right)  & \displaystyle\mathbb{E}\bigg(\int_{0}^{T}\ell_{s}%
^{2}ds\bigg)^{\frac{a}{2-a}}<\infty,\medskip\\
\left(  b\right)  & \displaystyle\mathbb{E}\left[
{\displaystyle\int_{0}^{T}}
e^{V_{s}^{\left(  +\right)  }}\left(  F_{\rho}^{\#}\left(  s\right)
ds+G_{\rho}^{\#}\left(  s\right)  dA_{s}\right)  \right]  ^{a}<\infty
,\quad\text{for all }\rho>0,
\end{array}
\label{ip-t2-1}%
\end{equation}
where $V^{\left(  +\right)  }$ is defined by (\ref{defV_3}) and $F_{\rho}%
^{\#}\,,G_{\rho}^{\#}$ are defined by (\ref{def F sharp});$\medskip$

\noindent$\left(  iii\right)  $ there exists a positive p.m.s.p. $\left(
\Theta_{t}\right)  _{t\geq0}$ and for each $\rho\geq0$ there exist an
non-decreasing function $K_{\rho}:\mathbb{R}_{+}\rightarrow\mathbb{R}_{+}$
such that%
\begin{equation}
F_{\rho}^{\#}\left(  t\right)  +G_{\rho}^{\#}\left(  t\right)  \leq K_{\rho
}\left(  \Theta_{t}\right)  ,\quad\text{a.e. }t\in\left[  0,T\right]  .
\label{ip-t2-1a}%
\end{equation}
Then the multivalued BSDE
\[
\left\{
\begin{array}
[c]{r}%
\displaystyle Y_{t}+{\int_{t}^{T}}dK_{r}=\eta+{\int_{t}^{T}}H\left(
r,Y_{r},Z_{r}\right)  dQ_{r}-{\int_{t}^{T}}Z_{r}dB_{r}\,,\quad\text{a.s., for
all }t\in\left[  0,T\right]  ,\medskip\\
\multicolumn{1}{l}{\displaystyle dK_{r}=U_{r}dQ_{r}\in\partial_{y}\Psi\left(
r,Y_{r}\right)  dQ_{r}}%
\end{array}
\,\right.
\]
has a unique $L^{p}$--variational solution, in the sense of Definition
\ref{definition_weak solution}.

Moreover this solution satisfies%
\begin{equation}%
\begin{array}
[c]{l}%
\mathbb{E}\Big(\sup\nolimits_{t\in\left[  0,T\right]  }e^{pV_{t}}\left\vert
Y_{t}\right\vert ^{p}\Big)+\mathbb{E}\left(
{\displaystyle\int_{0}^{T}}
{e^{2V_{r}}}\left\vert Z_{r}\right\vert ^{2}dr\right)  ^{p/2}+\mathbb{E}%
\left(
{\displaystyle\int_{0}^{T}}
e^{2V_{r}}{\Psi}\left(  r,Y_{r}\right)  dQ_{r}\right)  ^{p/2}\medskip\\
\quad+\mathbb{E}\left(
{\displaystyle\int_{0}^{T}}
{e^{qV_{r}}\left\vert Y_{r}\right\vert ^{q-2}}\left\vert Z_{r}\right\vert
^{2}dr\right)  ^{p/q}+\mathbb{E}\left(
{\displaystyle\int_{0}^{T}}
{e^{qV_{r}}\left\vert Y_{r}\right\vert ^{q-2}\Psi}\left(  r,Y_{r}\right)
dQ_{r}\right)  ^{p/q}\medskip\\
\leq C_{p,\lambda}\,\mathbb{E}\left[  e^{pV_{T}}\left\vert \eta\right\vert
^{p}+\Big(%
{\displaystyle\int_{0}^{T}}
e^{V_{r}}\left\vert H\left(  r,0,0\right)  \right\vert dQ_{r}\Big)^{p}\right]
,
\end{array}
\label{t2-bound}%
\end{equation}
where $q\in\{2,p\wedge2\}.$
\end{theorem}

\begin{proof}
Let $t\in\left[  0,T\right]  $ and%
\[
\beta_{t}=t+A_{t}+\left\vert \mu_{t}\right\vert +\left\vert \nu_{t}\right\vert
+\ell_{t}+V_{t}^{\left(  +\right)  }+\left\vert F\left(  t,0,0\right)
\right\vert +\left\vert G\left(  t,0\right)  \right\vert +\Theta_{t},
\]
Define, for $n\in\mathbb{N}^{\ast},$%
\[%
\begin{array}
[c]{l}%
\displaystyle\eta^{\left(  n\right)  }%
\xlongequal{\hspace{-4pt}{\rm def}\hspace{-4pt}}\eta\,\mathbf{1}_{\left[
0,n\right]  }\big(\left\vert \eta\right\vert +V_{T}^{\left(  +\right)
}\big),\medskip\\
\displaystyle F^{\left(  n\right)  }\left(  t,y,z\right)
\xlongequal{\hspace{-4pt}{\rm def}\hspace{-4pt}}F\left(  t,y,z\right)
-F\left(  t,0,0\right)  \,\mathbf{1}_{\left(  n,\infty\right)  }\left(
\beta_{t}\right)  ,\medskip\\
\displaystyle G^{\left(  n\right)  }\left(  t,y\right)
\xlongequal{\hspace{-4pt}{\rm def}\hspace{-4pt}}G\left(  t,y\right)  -G\left(
t,0\right)  \,\mathbf{1}_{\left(  n,\infty\right)  }\left(  \beta_{t}\right)
,\medskip\\
\displaystyle H^{\left(  n\right)  }\left(  t,y,z\right)
\xlongequal{\hspace{-4pt}{\rm def}\hspace{-4pt}}\alpha_{t}F^{\left(  n\right)
}\left(  t,y,z\right)  +\left(  1-\alpha_{t}\right)  G^{\left(  n\right)
}\left(  t,y\right)  .
\end{array}
\]
We highlight the following the following properties of the function
$H^{\left(  n\right)  }:$%
\begin{equation}%
\begin{array}
[c]{l}%
\left\langle y^{\prime}-y,H^{\left(  n\right)  }(t,y^{\prime},z)-H^{\left(
n\right)  }(t,y,z)\right\rangle \leq\left[  \mu_{t}\alpha_{t}+\nu_{t}\left(
1-\alpha_{t}\right)  \right]  \left\vert y^{\prime}-y\right\vert ^{2}%
,\medskip\\
\left\vert H^{\left(  n\right)  }(t,y,z^{\prime})-H^{\left(  n\right)
}(t,y,z)\right\vert \leq\alpha_{t}\,\ell_{t}\,\left\vert z^{\prime
}-z\right\vert ,\medskip\\
\left\vert H^{\left(  n+i\right)  }(t,y,z)-H^{\left(  n\right)  }%
(t,y,z)\right\vert \leq\left[  \alpha_{t}\left\vert F\left(  t,0,0\right)
\right\vert +\left(  1-\alpha_{t}\right)  \left\vert G\left(  t,0\right)
\right\vert \right]  \,\mathbf{1}_{\left(  n,\infty\right)  }\left(  \beta
_{t}\right)
\end{array}
\label{t2-hn}%
\end{equation}
and, using the previous, the monotonicity properties%
\begin{align*}
\big\langle y,H^{\left(  n\right)  }\left(  t,y,z\right)  \big\rangle  &
\leq|y|\left[  \alpha_{t}\left\vert F\left(  t,0,0\right)  \right\vert
+\left(  1-\alpha_{t}\right)  \left\vert G\left(  t,0\right)  \right\vert
\right]  \,\mathbf{1}_{\left[  0,n\right]  }\left(  \beta_{t}\right)
+|y|^{2}dV_{s}+\alpha_{t}\,\dfrac{n_{p}\lambda}{2}\,\left\vert z\right\vert
^{2}\\[2pt]
&  \leq|y|\left[  \alpha_{t}\left\vert F\left(  t,0,0\right)  \right\vert
+\left(  1-\alpha_{t}\right)  \left\vert G\left(  t,0\right)  \right\vert
\right]  \,\mathbf{1}_{\left[  0,n\right]  }\left(  \beta_{t}\right)
+|y|^{2}dV_{s}^{\left(  +\right)  }+\alpha_{t}\,\dfrac{n_{p}\lambda}%
{2}\,\left\vert z\right\vert ^{2}%
\end{align*}
and%
\begin{align*}
\big\langle Y_{t}^{\prime}-Y_{t},H^{\left(  n\right)  }(t,Y_{t}^{\prime}%
,Z_{t}^{\prime})-H^{\left(  n\right)  }(t,Y_{t},Z_{t})\big\rangle dQ_{t}  &
\leq\left\vert Y_{t}^{\prime}-Y_{t}\right\vert ^{2}dV_{t}+\dfrac{n_{p}\lambda
}{2}\,\left\vert Z_{t}^{\prime}-Z_{t}\right\vert ^{2}dt\\[2pt]
&  \leq\left\vert Y_{t}^{\prime}-Y_{t}\right\vert ^{2}dV_{t}^{\left(
+\right)  }+\dfrac{n_{p}\lambda}{2}\,\left\vert Z_{t}^{\prime}-Z_{t}%
\right\vert ^{2}dt.
\end{align*}
Clearly, the assumptions of Proposition \ref{p1-wvs} are satisfied for the
approximating BSDE%
\begin{equation}
\left\{
\begin{array}
[c]{l}%
Y_{t}^{\left(  n\right)  }+%
{\displaystyle\int_{t}^{T}}
dK_{s}=\eta^{\left(  n\right)  }+%
{\displaystyle\int_{t}^{T}}
H^{\left(  n\right)  }\big(s,Y_{s}^{\left(  n\right)  },Z_{s}^{\left(
n\right)  }\big)dQ_{s}-%
{\displaystyle\int_{t}^{T}}
Z_{s}^{\left(  n\right)  }dB_{s}\,,\quad t\in\left[  0,T\right]  ,\medskip\\
dK_{s}^{\left(  n\right)  }\in\partial_{y}\Psi\big(r,Y_{r}^{\left(  n\right)
}\big)dQ_{r}=\alpha_{r}\partial\varphi\big(Y_{r}^{\left(  n\right)
}\big)dr+\left(  1-\alpha_{r}\right)  \partial\psi\big(Y_{r}^{\left(
n\right)  }\big)dA_{r}%
\end{array}
\right.  \label{t2-ae}%
\end{equation}
and therefore there exists a unique $L^{p}$--variational solution $\left(
Y^{\left(  n\right)  },Z^{\left(  n\right)  }\right)  $ of (\ref{t2-ae}%
).$\medskip$

Since $V\leq V^{\left(  +\right)  },$ we obtain from (\ref{es-p})%
\[
\mathbb{E}\Big(\sup\nolimits_{r\in\left[  0,T\right]  }{e^{pV_{r}}}%
\big|Y_{r}^{\left(  n\right)  }\big|^{p}\Big)+\mathbb{E}\bigg(%
{\displaystyle\int_{0}^{T}}
{e^{2V_{r}}}\big|Z_{r}^{\left(  n\right)  }\big|^{2}dr\bigg)^{p/2}%
+\mathbb{E}\bigg(%
{\displaystyle\int_{0}^{T}}
e^{2V_{r}}{\Psi}\big(r,Y_{r}^{\left(  n\right)  }\big)dQ_{r}\bigg)^{p/2}%
<\infty
\]
and for any $q\in\{2,p\wedge2\},$ $\delta_{q}=\delta\mathbf{1}_{[1,2)}\left(
q\right)  $ and $\Gamma_{t}^{\left(  n\right)  }%
\xlongequal{\hspace{-4pt}{\rm def}\hspace{-4pt}}\Big(\big|M_{t}-Y_{t}^{\left(
n\right)  }\big|^{2}+\delta_{q}\Big)^{1/2}$ it holds%
\begin{equation}%
\begin{array}
[c]{l}%
\big(\Gamma_{t}^{\left(  n\right)  }\big)^{q}+\dfrac{q\left(  q-1\right)  }{2}%
{\displaystyle\int_{t}^{s}}
\big({\Gamma_{r}^{\left(  n\right)  }\big)^{q-2}}{\Large \,}\big|R_{r}%
-Z_{r}^{\left(  n\right)  }\big|^{2}dr+{q%
{\displaystyle\int_{t}^{s}}
}\big({\Gamma_{r}^{\left(  n\right)  }\big)^{q-2}\Psi\big(}r,Y_{r}^{\left(
n\right)  }\big)dQ_{r}\medskip\\
\leq\big(\Gamma_{s}^{\left(  n\right)  }\big)^{q}+{q%
{\displaystyle\int_{t}^{s}}
}\big({\Gamma_{r}^{\left(  n\right)  }\big)^{q-2}\Psi}\left(  r,M_{r}\right)
dQ_{r}\medskip\\
\quad+q%
{\displaystyle\int_{t}^{s}}
{\big(\Gamma_{r}^{\left(  n\right)  }\big)^{q-2}}\langle M_{r}-Y_{r}^{\left(
n\right)  },N_{r}-H^{\left(  n\right)  }{\big(}r,Y_{r}^{\left(  n\right)
},Z_{r}^{\left(  n\right)  }\big)\rangle dQ_{r}\medskip\\
\quad-q%
{\displaystyle\int_{t}^{s}}
\big({\Gamma_{r}^{\left(  n\right)  }\big)^{q-2}}\,\langle M_{r}%
-Y_{r}^{\left(  n\right)  },{\big(}R_{r}-Z_{r}^{\left(  n\right)  }%
\big)dB_{r}\rangle
\end{array}
\label{vw-ap2}%
\end{equation}
for any $0\leq t\leq s<\infty$ and for any $M\in\mathcal{V}_{m}^{0}$ of the
form (\ref{def_M}), i.e. $M_{t}=M_{T}+\int_{t}^{T}N_{r}dQ_{r}-\int_{t}%
^{T}R_{r}dB_{r}\,.\medskip$

Since $\mathbb{E}\big(\sup\nolimits_{r\in\left[  0,T\right]  }{e^{pV_{r}}%
}\big|Y_{r}^{\left(  n\right)  }\big|^{p}\big|<\infty$ and inequality
(\ref{vw-ap2}) holds for $1<q=p\wedge2\leq p,$ inequalities (\ref{def-11ccc})
and (\ref{def-11aa}) yield%
\begin{equation}%
\begin{array}
[c]{l}%
\mathbb{E}\Big(\sup\nolimits_{t\in\left[  0,T\right]  }e^{pV_{t}}%
\big|Y_{t}^{\left(  n\right)  }\big|^{p}\Big)+\mathbb{E}\bigg(%
{\displaystyle\int_{0}^{T}}
{e^{2V_{r}}}\big|Z_{r}^{\left(  n\right)  }\big|^{2}dr\bigg)^{p/2}%
+\mathbb{E}\bigg(%
{\displaystyle\int_{0}^{T}}
e^{2V_{r}}{\Psi\big(}r,Y_{r}^{\left(  n\right)  }\big)dQ_{r}\bigg)^{p/2}%
\medskip\\
\quad+\mathbb{E}\bigg(%
{\displaystyle\int_{0}^{T}}
{e^{qV_{r}}\big|Y_{r}^{\left(  n\right)  }\big|^{q-2}}\big|Z_{r}^{\left(
n\right)  }\big|^{2}dr\bigg)^{p/q}+\mathbb{E}\bigg(%
{\displaystyle\int_{0}^{T}}
{e^{qV_{r}}\big|Y_{r}^{\left(  n\right)  }\big|^{q-2}\Psi\big(}r,Y_{r}%
^{\left(  n\right)  }\big)dQ_{r}\bigg)^{p/q}\medskip\\
\leq C_{p,\lambda}\,\mathbb{E}\left[  e^{pV_{T}}\left\vert \eta\right\vert
^{p}+\Big(%
{\displaystyle\int_{0}^{T}}
e^{V_{r}}\left\vert H\left(  r,0,0\right)  \right\vert dQ_{r}\Big)^{p}\right]
.
\end{array}
\label{vw-ap3}%
\end{equation}
From (\ref{cont-2}), with $q=p\wedge2,$ we have for all $0<\alpha<1:$%
\begin{equation}%
\begin{array}
[c]{l}%
\mathbb{E}\sup\nolimits_{t\in\left[  0,T\right]  }e^{\alpha qV_{t}}%
\big|Y_{t}^{\left(  n+i\right)  }-Y_{t}^{\left(  n\right)  }\big|^{\alpha
q}+\bigg(\mathbb{E}%
{\displaystyle\int_{0}^{T}}
e^{2V_{r}}\dfrac{\big|Z_{r}^{\left(  n+i\right)  }-Z_{r}^{\left(  n\right)
}\big|^{2}}{\big(e^{V_{r}}\big|Y_{t}^{\left(  n+i\right)  }-Y_{t}^{\left(
n\right)  }\big|+1\big)^{2-q}}\,dr\bigg)^{\alpha}\medskip\\
\leq C_{\alpha,q,\lambda}\Bigg[\,\mathbb{E}e^{qV_{T}}\left\vert \eta^{\left(
n+i\right)  }-\eta^{\left(  n\right)  }\right\vert ^{q}\medskip\\
\quad+K\,\bigg(\mathbb{E}\bigg(%
{\displaystyle\int_{0}^{T}}
e^{V_{r}}\big|H^{\left(  n+i\right)  }(t,Y_{t}^{\left(  n\right)  }%
,Z_{t}^{\left(  n\right)  })-H^{\left(  n\right)  }(t,Y_{t}^{\left(  n\right)
},Z_{t}^{\left(  n\right)  })\big|dQ_{r}\bigg)^{q}\bigg)^{1/q}\Bigg]^{\alpha},
\end{array}
\label{Cauchy-3}%
\end{equation}
where%
\[%
\begin{array}
[c]{l}%
\displaystyle K=\bigg[{\mathbb{E}}e^{qV_{T}}\big|\eta^{\left(  n+i\right)
}\big|^{q}+{\mathbb{E}}\bigg(%
{\displaystyle\int_{0}^{T}}
e^{V_{r}}\big|H^{\left(  n+i\right)  }\left(  r,0,0\right)  \big|dQ_{r}%
^{q}\bigg)+{\mathbb{E}}e^{qV_{T}}\big|\eta^{\left(  n\right)  }\big|^{q}%
\medskip\\
\hfill\hfill\displaystyle\quad+{\mathbb{E}}\bigg(%
{\displaystyle\int_{0}^{T}}
e^{V_{r}}\big|H^{\left(  n\right)  }\left(  r,0,0\right)  \big|dQ_{r}%
^{q}\bigg)\bigg]^{\left(  q-1\right)  /q}\medskip\\
\displaystyle\leq2^{\left(  q-1\right)  /q}\bigg[{\mathbb{E}}\left(
e^{qV_{T}}\left\vert \eta\right\vert ^{q}+\bigg(%
{\displaystyle\int_{0}^{T}}
e^{V_{r}}\left\vert F\left(  r,0,0\right)  \right\vert dr+\left\vert G\left(
r,0\right)  \right\vert dA_{r}^{q}\bigg)\right)  \bigg]^{\left(  q-1\right)
/q}%
\end{array}
\]
and $C_{\alpha,q,\lambda}$ is a positive constant depending only $\alpha,q$
and $\lambda.\medskip$

First we remark, see also (\ref{indicator_ineq}),%
\[
\mathbb{E}e^{qV_{T}}\big|\eta^{\left(  n+i\right)  }-\eta^{\left(  n\right)
}\big|^{q}\leq\mathbb{E}e^{qV_{T}}\left\vert \eta\right\vert ^{q}%
\,\mathbf{1}_{\left(  n,\infty\right)  }\big(\left\vert \eta\right\vert
+V_{T}^{\left(  +\right)  }\big)\rightarrow0,\quad\mathbb{P}\text{--a.s.,}%
\;\text{for }n\rightarrow\infty,
\]
since, by $1<q\leq p$ and assumption (\ref{ip-t2-2}), we have%
\[
\mathbb{E}e^{qV_{T}}\left\vert \eta\right\vert ^{q}\leq\left(  \mathbb{E}%
~e^{pV_{T}}\left\vert \eta\right\vert ^{p}\right)  ^{q/p}<\infty.
\]
Again by (\ref{indicator_ineq}) and assumption (\ref{ip-t2-2}) we deduce:%
\[%
\begin{array}
[c]{l}%
\displaystyle\mathbb{E}\bigg(%
{\displaystyle\int_{0}^{T}}
e^{V_{r}}\big|H^{\left(  n+i\right)  }(t,Y_{t}^{\left(  n\right)  }%
,Z_{t}^{\left(  n\right)  })-H^{\left(  n\right)  }(t,Y_{t}^{\left(  n\right)
},Z_{t}^{\left(  n\right)  })\big|dQ_{r}\bigg)^{q}\medskip\\
\displaystyle\leq\mathbb{E}\bigg(%
{\displaystyle\int_{0}^{T}}
e^{V_{r}}\left[  \left\vert F\left(  r,0,0\right)  \right\vert \mathbf{1}%
_{\left(  n,\infty\right)  }\left(  \beta_{r}\right)  dr+\left\vert G\left(
r,0\right)  \right\vert \mathbf{1}_{\left(  n,\infty\right)  }\left(
\beta_{r}\right)  dA_{r}\right]  \bigg)^{q}\medskip\\
\displaystyle\leq2^{q-1}\,\bigg[\mathbb{E}\bigg(%
{\displaystyle\int_{0}^{T}}
e^{V_{r}}\left\vert F\left(  r,0,0\right)  \right\vert \mathbf{1}_{\left(
n,\infty\right)  }\left(  \beta_{r}\right)  dr\bigg)^{q}+\mathbb{E}\bigg(%
{\displaystyle\int_{0}^{T}}
e^{V_{r}}\left\vert G\left(  r,0\right)  \right\vert \mathbf{1}_{\left(
n,\infty\right)  }\left(  \beta_{r}\right)  dA_{r}\bigg)^{q}\bigg]\medskip\\
\displaystyle\rightarrow0,\quad\mathbb{P}\text{--a.s., for }n\rightarrow
\infty.
\end{array}
\]
From (\ref{Cauchy-3}) we conclude that there exists $\left(  Y,Z\right)  \in
S_{m}^{0}\times\Lambda_{m\times k}^{0}$ such that (on a subsequence)%
\[
\sup\nolimits_{t\in\left[  0,T\right]  }\big|Y_{t}^{\left(  n\right)  }%
-Y_{t}\big|+%
{\displaystyle\int_{0}^{T}}
\big|Z_{r}^{\left(  n\right)  }-Z_{r}\big|^{2}dr\rightarrow0,\quad
\mathbb{P}\text{--a.s.,}\;\text{for }n\rightarrow\infty.
\]
Passing to $\liminf_{n\rightarrow\infty}$ in (\ref{vw-ap2}) and (\ref{vw-ap3})
we infer that $\left(  Y,Z\right)  $ is an $L^{p}$--variational
solution.\hfill
\end{proof}

\subsection{Existence on a random interval time $\left[  0,\tau\right]  $}

\begin{theorem}
\label{t3-random}We suppose that assumptions $\left(  \mathrm{A}%
_{1}-\mathrm{A}_{7}\right)  $ are satisfied. Let $V$ and $V^{\left(  +\right)
}$ be given by definitions (\ref{defV_1}) and (\ref{defV_3}) respectively. In
addition we assume that:$\medskip$

\noindent$\left(  i\right)  $%
\[
0\leq{\Psi}\left(  r,\eta\right)  \leq\mathbf{1}_{q\geq2}\,{\Psi}\left(
r,\eta\right)  ,\quad\text{a.e. }r\geq0;
\]
\noindent$\left(  ii\right)  $%
\begin{equation}%
\begin{array}
[c]{l}%
\mathbb{E}\big(\sup\nolimits_{t\geq0}e^{pV_{t}}\left\vert \xi_{t}\right\vert
^{p}\big)+\mathbb{E}\bigg(%
{\displaystyle\int_{0}^{\tau}}
{e^{2V_{r}}}\left\vert \zeta_{r}\right\vert ^{2}dr\bigg)^{p/2}+\mathbb{E}%
\bigg(%
{\displaystyle\int_{0}^{\tau}}
e^{2V_{r}}{\Psi}\left(  r,\xi_{r}\right)  dQ_{r}\bigg)^{p/2}\medskip\\
\hfill+\mathbb{E}\bigg(%
{\displaystyle\int_{0}^{\tau}}
e^{V_{r}}\left(  \left\vert F\left(  r,0,0\right)  \right\vert dr+\left\vert
G\left(  r,0\right)  \right\vert dA_{r}\right)  \bigg)^{p}%
\xlongequal{\hspace{-4pt}{\rm def}\hspace{-4pt}}L<\infty;
\end{array}
\label{t3-ip1}%
\end{equation}
\noindent$\left(  iii\right)  $%
\begin{equation}
\lim\nolimits_{t\rightarrow\infty}\mathbb{E}\left(  \Lambda_{t}\right)
=0,\smallskip\label{t3-ip4}%
\end{equation}
where $\Lambda_{t}\xlongequal{\hspace{-4pt}{\rm def}\hspace{-4pt}}\bigg(%
{\displaystyle\int_{t}^{\infty}}
e^{V_{r}}\mathbf{1}_{q\geq2}{\Psi}\left(  r,\xi_{r}\right)  dQ_{r}%
\bigg)^{p/2}+\bigg(%
{\displaystyle\int_{t}^{\infty}}
e^{V_{r}}\left\vert H\left(  r,\xi_{r},\zeta_{r}\right)  \right\vert
dQ_{r}\bigg)^{p};$

\noindent$\left(  iv\right)  $ there exists $a\in\left(  1+n_{p}%
\lambda,p\wedge2\right)  $ such that for every $T\geq0:$%
\begin{equation}%
\begin{array}
[c]{rl}%
\left(  a\right)  & \mathbb{E}\bigg(%
{\displaystyle\int_{0}^{T}}
\ell_{s}^{2}ds\bigg)^{\frac{a}{2-a}}<\infty,\medskip\\
\left(  b\right)  & \mathbb{E}\bigg(%
{\displaystyle\int_{0}^{T}}
e^{V_{s}^{\left(  +\right)  }}\left(  F_{\rho}^{\#}\left(  s\right)
ds+G_{\rho}^{\#}\left(  s\right)  dA_{s}\right)  \bigg)^{a}<\infty
,\quad\text{for all }\rho>0;
\end{array}
\label{t3-ip2}%
\end{equation}
\noindent$\left(  v\right)  $ there exists a positive p.m.s.p. $\left(
\Theta_{t}\right)  _{t\geq0}$ and, for each $\rho\geq0,$ there exist an
non-decreasing function $K_{\rho}:\mathbb{R}_{+}\rightarrow\mathbb{R}_{+}$
such that%
\begin{equation}
F_{\rho}^{\#}\left(  t\right)  +G_{\rho}^{\#}\left(  t\right)  \leq K_{\rho
}\left(  \Theta_{t}\right)  ,\quad\text{a.e. }t\geq0,\text{ }\mathbb{P}%
\text{--a.s..} \label{t3-ip3}%
\end{equation}
Then the multivalued BSDE%
\[
\left\{
\begin{array}
[c]{r}%
\displaystyle Y_{t}+{\int_{t\wedge\tau}^{\tau}}dK_{r}=\eta+{\int_{t\wedge\tau
}^{\tau}}H\left(  r,Y_{r},Z_{r}\right)  dQ_{r}-{\int_{t\wedge\tau}^{\tau}%
}Z_{r}dB_{r}\,,\quad\mathbb{P}\text{--a.s., for all }t\geq0,\medskip\\
\multicolumn{1}{l}{\displaystyle dK_{r}=U_{r}dQ_{r}\in\partial_{y}\Psi\left(
r,Y_{r}\right)  dQ_{r}}%
\end{array}
\,\right.
\]
has a unique $L^{p}$--variational solution, in the sense of Definition
\ref{definition_weak solution}.

Moreover this solution satisfies%
\begin{equation}%
\begin{array}
[c]{l}%
\mathbb{E}\Big(\sup\nolimits_{t\in\left[  0,\tau\right]  }e^{pV_{t}}\left\vert
Y_{t}\right\vert ^{p}\Big)+\mathbb{E}\left(
{\displaystyle\int_{0}^{\tau}}
{e^{2V_{r}}}\left\vert Z_{r}\right\vert ^{2}dr\right)  ^{p/2}+\mathbb{E}%
\left(
{\displaystyle\int_{0}^{\tau}}
e^{2V_{r}}{\Psi}\left(  r,Y_{r}\right)  dQ_{r}\right)  ^{p/2}\medskip\\
\quad+\mathbb{E}\left(
{\displaystyle\int_{0}^{\tau}}
{e^{qV_{r}}\left\vert Y_{r}\right\vert ^{q-2}}\left\vert Z_{r}\right\vert
^{2}dr\right)  ^{p/q}+\mathbb{E}\left(
{\displaystyle\int_{0}^{\tau}}
{e^{qV_{r}}\left\vert Y_{r}\right\vert ^{q-2}\Psi}\left(  r,Y_{r}\right)
dQ_{r}\right)  ^{p/q}\medskip\\
\leq C_{p,\lambda}\,\mathbb{E}\left[  e^{pV_{\tau}}\left\vert \eta\right\vert
^{p}+\bigg(%
{\displaystyle\int_{0}^{\tau}}
e^{V_{r}}\left\vert H\left(  r,0,0\right)  \right\vert dQ_{r}\bigg)^{p}%
\right]  .
\end{array}
\label{t3-c1}%
\end{equation}

\end{theorem}

\begin{remark}
Our initial assumptions are about $\eta$ but the first three terms from
(\ref{t3-ip1}) involves the processes $\xi$ and $\zeta$ (associated to
$\eta\,$). However, we remark that in order to obtain that the first three
terms in (\ref{t3-ip1}) are bounded it is sufficient to impose%
\[
\mathbb{E}\big(\sup\nolimits_{t\geq0}e^{pV_{t}}\left\vert \eta\right\vert
^{p}\big)<\infty
\]
(for the proof we can apply \cite[Corollary 2.45]{pa-ra/14} for the process
$\sup_{t\geq0}V_{t}\,$) and respectively%
\[
\mathbb{E}\bigg(%
{\displaystyle\int_{0}^{\tau}}
e^{2V_{r}}{\Psi}\left(  r,\eta\right)  dQ_{r}\bigg)^{p/2}<\infty.
\]

\end{remark}

\begin{proof}
By Theorem \ref{t2exist} there exists a unique pair $\left(  Y^{\left(
n\right)  },Z^{\left(  n\right)  }\right)  $ as $L^{p}$--variational solution
on $\left[  0,n\right]  $ of the BSDE, with the final data $\mathbb{\xi}%
_{n}=\mathbb{E}^{\mathcal{F}_{n}}\eta,$%
\[
\left\{
\begin{array}
[c]{l}%
Y_{t}^{\left(  n\right)  }+%
{\displaystyle\int_{t}^{n}}
dK_{s}^{\left(  n\right)  }=\mathbb{\xi}_{n}+%
{\displaystyle\int_{t}^{n}}
H\big(s,Y_{s}^{\left(  n\right)  },Z_{s}^{\left(  n\right)  }\big)dQ_{s}-%
{\displaystyle\int_{t}^{n}}
Z_{s}^{\left(  n\right)  }dB_{s}\,,\quad t\in\left[  0,n\right]  ,\medskip\\
dK_{s}^{\left(  n\right)  }\in\partial_{y}\Psi(s,Y_{s}^{\left(  n\right)
})dQ_{s}\,.
\end{array}
\right.
\]
Hence%
\begin{equation}
\big(Y_{t}^{\left(  n\right)  },Z_{t}^{\left(  n\right)  }\big)=\left(
\xi_{t},\zeta_{t}\right)  ,\quad\text{for all }t\geq n\quad\text{and}%
\quad\big(Y_{t}^{\left(  n\right)  },Z_{t}^{\left(  n\right)  }\big)=\left(
\eta,0\right)  ,\quad\text{for all }t\geq\tau. \label{propr_yn_zn}%
\end{equation}
By (\ref{t2-bound}) we have%
\begin{equation}%
\begin{array}
[c]{l}%
\mathbb{E}\Big(\sup\nolimits_{t\geq0}e^{pV_{t}}\big|Y_{t}^{\left(  n\right)
}\big|^{p}\Big)+\mathbb{E}\left(
{\displaystyle\int_{0}^{\infty}}
{e^{2V_{r}}}\big|Z_{r}^{\left(  n\right)  }\big|^{2}dr\right)  ^{p/2}%
+\mathbb{E}\left(
{\displaystyle\int_{0}^{\infty}}
e^{2V_{r}}{\Psi}\big(r,Y_{r}^{\left(  n\right)  }\big)dQ_{r}\right)
^{p/2}\medskip\\
\leq\mathbb{E}\Big(\sup\nolimits_{t\in\left[  0,n\right]  }e^{pV_{t}%
}\big|Y_{t}^{\left(  n\right)  }\big|^{p}\Big)+\mathbb{E}\left(
{\displaystyle\int_{0}^{n}}
{e^{2V_{r}}}\big|Z_{r}^{\left(  n\right)  }\big|^{2}dr\right)  ^{p/2}%
+\mathbb{E}\left(
{\displaystyle\int_{0}^{n}}
e^{2V_{r}}{\Psi}\big(r,Y_{r}^{\left(  n\right)  }\big)dQ_{r}\right)
^{p/2}\medskip\\
\quad+\mathbb{E}\Big(\sup\nolimits_{t\geq0}e^{pV_{t}}\left\vert \xi
_{t}\right\vert ^{p}\Big)+\mathbb{E}\left(
{\displaystyle\int_{0}^{\tau}}
{e^{2V_{r}}}\left\vert \zeta_{r}\right\vert ^{2}dr\right)  ^{p/2}%
+\mathbb{E}\left(
{\displaystyle\int_{0}^{\tau}}
e^{2V_{r}}{\Psi}\left(  r,\xi_{r}\right)  dQ_{r}\right)  ^{p/2}\medskip\\
\leq C_{p,\lambda}\,\mathbb{E}\left[  e^{pV_{n}}\left\vert \mathbb{\xi}%
_{n}\right\vert ^{p}+\Big(%
{\displaystyle\int_{0}^{n}}
e^{V_{r}}\left\vert H\left(  r,0,0\right)  \right\vert dQ_{r}\Big)^{p}\right]
+L\medskip\\
\leq L\cdot C_{p,\lambda}%
+L\xlongequal{\hspace{-4pt}{\rm def}\hspace{-4pt}}\tilde{L}.
\end{array}
\label{bound-4}%
\end{equation}
On the other hand, let us take in Proposition \ref{p-estim}%
\[
M_{t}=\xi_{t}=\mathbb{E}^{\mathcal{F}_{t}}\eta,\quad R_{t}=\zeta_{t}\,,\quad
N_{t}=0,\quad\text{and }L_{t}=0
\]
and therefore%
\[
M_{t}=M_{T}+%
{\displaystyle\int_{t}^{T}}
0\,dr-%
{\displaystyle\int_{t}^{T}}
R_{r}dB_{r}\,,\quad\text{for all }0\leq t\leq T<\infty.
\]
Since, for all $k\in\mathbb{N}^{\ast},$%
\[
\mathbb{E}\left(  \sup\nolimits_{t\in\left[  0,k\right]  }{e^{pV_{t}}%
}\big|M_{t}-Y_{t}^{\left(  k\right)  }\big|^{p}\right)  \leq2^{p-1}%
\,\mathbb{E}\sup\nolimits_{t\in\left[  0,k\right]  }{e^{pV_{t}}}\left\vert
\xi_{t}\right\vert ^{p}+2^{p-1}\,\mathbb{E}\sup\nolimits_{t\in\left[
0,k\right]  }{e^{pV_{t}}}\big|Y_{t}^{\left(  k\right)  }\big|^{p}<\infty,
\]
we can apply (\ref{def-11c}) and (\ref{def-11}) in order to deduce that, for
all $0\leq t\leq s\leq k,$%
\[%
\begin{array}
[c]{l}%
\displaystyle\mathbb{E}^{\mathcal{F}_{t}}\Big(\sup\nolimits_{r\in\left[
t,s\right]  }e^{pV_{r}}\big|\mathbb{\xi}_{r}-Y_{r}^{\left(  k\right)
}\big|^{p}\Big)+\mathbb{E}^{\mathcal{F}_{t}}\left(
{\displaystyle\int_{t}^{s}}
{e^{2V_{r}}}\left\vert \zeta_{r}-Z_{r}^{\left(  k\right)  }\right\vert
^{2}dr\right)  ^{p/2}\medskip\\
\displaystyle\leq C_{p,\lambda}\,\mathbb{E}^{\mathcal{F}_{t}}\bigg[e^{pV_{s}%
}\big|\mathbb{\xi}_{s}-Y_{s}^{\left(  k\right)  }\big|^{p}+\left(
{\displaystyle\int_{t}^{s}}
e^{V_{r}}\mathbf{1}_{q\geq2}{\Psi}\left(  r,\xi_{r}\right)  dQ_{r}\right)
^{p/2}+\Big(%
{\displaystyle\int_{t}^{s}}
e^{V_{r}}\left\vert H\left(  r,\xi_{r},\zeta_{r}\right)  \right\vert
dQ_{r}\Big)^{p}\bigg].
\end{array}
\]
Hence%
\[%
\begin{array}
[c]{l}%
\displaystyle\mathbb{E}^{\mathcal{F}_{t}}e^{pV_{s}}\big|\mathbb{\xi}_{s}%
-Y_{s}^{\left(  k\right)  }\big|^{p}\leq\mathbb{E}^{\mathcal{F}_{t}}%
\Big(\sup\nolimits_{s\in\left[  t,k\right]  }e^{pV_{s}}\big|\mathbb{\xi}%
_{s}-Y_{s}^{\left(  k\right)  }\big|^{p}\Big)\medskip\\
\displaystyle\leq C_{p,\lambda}\,\mathbb{E}^{\mathcal{F}_{t}}\bigg[e^{pV_{k}%
}\big|\mathbb{\xi}_{k}-Y_{k}^{\left(  k\right)  }\big|^{p}+\left(
{\displaystyle\int_{t}^{k}}
e^{V_{r}}\mathbf{1}_{q\geq2}{\Psi}\left(  r,\xi_{r}\right)  dQ_{r}\right)
^{p/2}+\Big(%
{\displaystyle\int_{t}^{k}}
e^{V_{r}}\left\vert H\left(  r,\xi_{r},\zeta_{r}\right)  \right\vert
dQ_{r}\Big)^{p}\bigg]\medskip\\
\displaystyle=C_{p,\lambda}\,\mathbb{E}^{\mathcal{F}_{t}}\bigg[\left(
{\displaystyle\int_{t}^{k}}
e^{V_{r}}\mathbf{1}_{q\geq2}{\Psi}\left(  r,\xi_{r}\right)  dQ_{r}\right)
^{p/2}+\Big(%
{\displaystyle\int_{t}^{k}}
e^{V_{r}}\left\vert H\left(  r,\xi_{r},\zeta_{r}\right)  \right\vert
dQ_{r}\Big)^{p}\bigg].
\end{array}
\]
Using the above two inequalities we deduce%
\[%
\begin{array}
[c]{l}%
\displaystyle\mathbb{E}^{\mathcal{F}_{t}}\Big(\sup\nolimits_{r\in\left[
t,s\right]  }e^{pV_{r}}\big|\mathbb{\xi}_{r}-Y_{r}^{\left(  k\right)
}\big|^{p}\Big)+\mathbb{E}^{\mathcal{F}_{t}}\left(
{\displaystyle\int_{t}^{s}}
{e^{2V_{r}}}\big|\zeta_{r}-Z_{r}^{\left(  k\right)  }\big|^{2}dr\right)
^{p/2}\medskip\\
\displaystyle\leq C_{p,\lambda}\,\mathbb{E}^{\mathcal{F}_{t}}\bigg[\bigg(%
{\displaystyle\int_{t}^{k}}
e^{V_{r}}\mathbf{1}_{q\geq2}{\Psi}\left(  r,\xi_{r}\right)  dQ_{r}%
\bigg)^{p/2}+\bigg(%
{\displaystyle\int_{t}^{k}}
e^{V_{r}}\left\vert H\left(  r,\xi_{r},\zeta_{r}\right)  \right\vert
dQ_{r}^{p}\bigg)\bigg]\medskip\\
\displaystyle\leq C_{p,\lambda}\,\mathbb{E}^{\mathcal{F}_{t}}\bigg[\bigg(%
{\displaystyle\int_{t}^{\infty}}
e^{V_{r}}\mathbf{1}_{q\geq2}{\Psi}\left(  r,\xi_{r}\right)  dQ_{r}%
\bigg)^{p/2}+\bigg(%
{\displaystyle\int_{t}^{\infty}}
e^{V_{r}}\left\vert H\left(  r,\xi_{r},\zeta_{r}\right)  \right\vert
dQ_{r}\bigg)^{p}\bigg]\medskip\\
\displaystyle=C_{p,\lambda}\cdot\mathbb{E}^{\mathcal{F}_{t}}\left(
\Lambda_{t}\right)  .
\end{array}
\]
In particular, since%
\[
\big(Y_{r}^{\left(  n\right)  },Z_{r}^{\left(  n\right)  }\big)=\left(
\xi_{t},\zeta_{t}\right)  ,\quad\text{for all }r\geq n\quad\text{and}%
\quad\big(Y_{r}^{\left(  n+i\right)  },Z_{r}^{\left(  n+i\right)
}\big)=\left(  \xi_{t},\zeta_{t}\right)  ,\quad\text{for all }r\geq n+i,
\]
we obtain%
\[%
\begin{array}
[c]{l}%
\displaystyle\mathbb{E}^{\mathcal{F}_{n}}\Big(\sup\nolimits_{r\geq n}%
e^{pV_{r}}\big|Y_{r}^{\left(  n\right)  }-Y_{r}^{\left(  n+i\right)
}\big|^{p}\Big)+\mathbb{E}^{\mathcal{F}_{n}}\left(
{\displaystyle\int_{n}^{\infty}}
{e^{2V_{r}}}\big|Z_{r}^{\left(  n\right)  }-Z_{r}^{\left(  n+i\right)
}\big|^{2}dr\right)  ^{p/2}\medskip\\
\displaystyle=\mathbb{E}^{\mathcal{F}_{n}}\Big(\sup\nolimits_{r\in\left[
n,n+i\right]  }e^{pV_{r}}\big|Y_{r}^{\left(  n\right)  }-Y_{r}^{\left(
n+i\right)  }\big|^{p}\Big)+\mathbb{E}^{\mathcal{F}_{n}}\left(
{\displaystyle\int_{n}^{n+i}}
{e^{2V_{r}}}\big|Z_{r}^{\left(  n\right)  }-Z_{r}^{\left(  n+i\right)
}\big|^{2}dr\right)  ^{p/2}\medskip\\
\displaystyle=\mathbb{E}^{\mathcal{F}_{n}}\Big(\sup\nolimits_{r\in\left[
n,n+i\right]  }e^{pV_{r}}\big|\mathbb{\xi}_{r}-Y_{r}^{\left(  n+i\right)
}\big|^{p}\Big)+\mathbb{E}^{\mathcal{F}_{n}}\left(
{\displaystyle\int_{n}^{n+i}}
{e^{2V_{r}}}\big|\zeta_{r}-Z_{r}^{\left(  n+i\right)  }\big|^{2}dr\right)
^{p/2}\medskip\\
\displaystyle\leq C_{p,\lambda}\,\mathbb{E}^{\mathcal{F}_{n}}\bigg[\left(
{\displaystyle\int_{n}^{\infty}}
e^{V_{r}}\mathbf{1}_{q\geq2}{\Psi}\left(  r,\xi_{r}\right)  dQ_{r}\right)
^{p/2}+\Big(%
{\displaystyle\int_{n}^{\infty}}
e^{V_{r}}\left\vert H\left(  r,\xi_{r},\zeta_{r}\right)  \right\vert
dQ_{r}\Big)^{p}\bigg]\medskip\\
\displaystyle=C_{p,\lambda}\cdot\mathbb{E}^{\mathcal{F}_{n}}\left(
\Lambda_{n}\right)
\end{array}
\]
and therefore, using (\ref{t3-ip4}),%
\[
\lim\nolimits_{n\rightarrow\infty}\mathbb{E}\Big(\sup\nolimits_{r\geq
n}e^{pV_{r}}\big|\mathbb{\xi}_{r}-Y_{r}^{\left(  n\right)  }\big|^{p}%
\Big)+\mathbb{E}\left(
{\displaystyle\int_{n}^{\infty}}
{e^{2V_{r}}}\big|\zeta_{r}-Z_{r}^{\left(  n\right)  }\big|^{2}dr\right)
^{p/2}=0.
\]
By the continuity property (\ref{cont-2}) on $\left[  0,n\right]  ,$ with
$0<\alpha<1,$%
\begin{equation}%
\begin{array}
[c]{l}%
\displaystyle\mathbb{E}\sup\nolimits_{t\in\left[  0,n\right]  }e^{\alpha
qV_{t}}|Y_{t}^{\left(  n+i\right)  }-Y_{t}^{\left(  n\right)  }|^{\alpha
q}+\bigg(\mathbb{E}\int_{0}^{n}e^{2V_{r}}\frac{\big|Z_{r}^{\left(  n+i\right)
}-Z_{r}^{\left(  n\right)  }\big|^{2}}{\big(e^{V_{r}}\big|Y_{r}^{\left(
n+i\right)  }-Y_{r}^{\left(  n\right)  }\big|+1\big)^{2-q}}dr\bigg)^{\alpha
}\medskip\\
\displaystyle\leq C_{\alpha,q,\lambda}\left[  \mathbb{E}e^{qV_{n}}%
|Y_{n}^{\left(  n+i\right)  }-Y_{n}^{\left(  n\right)  }|^{q}\right]
^{\alpha}\leq C_{\alpha,q,\lambda}\left[  \mathbb{E}e^{pV_{n}}|Y_{n}^{\left(
n+i\right)  }-\xi_{n}|^{p}\right]  ^{\alpha q/p}\medskip\\
\displaystyle\leq C_{\alpha,p,\lambda}\cdot\left[  \mathbb{E}\left(
\Lambda_{n}\right)  \right]  ^{\alpha q/p}\rightarrow0,\quad\text{as
}n\rightarrow\infty.
\end{array}
\label{0n}%
\end{equation}
Hence, using again H\"{o}lder's inequality,%
\[%
\begin{array}
[c]{l}%
\displaystyle\mathbb{E}\sup\nolimits_{t\geq0}e^{\alpha qV_{t}}|Y_{t}^{\left(
n+i\right)  }-Y_{t}^{\left(  n\right)  }|^{\alpha q}\medskip\\
\displaystyle\leq\mathbb{E}\sup\nolimits_{t\in\left[  0,n\right]  }e^{\alpha
qV_{t}}|Y_{t}^{\left(  n+i\right)  }-Y_{t}^{\left(  n\right)  }|^{\alpha
q}+\left(  \mathbb{E}\Big(\sup\nolimits_{r\geq n}e^{pV_{r}}\big|Y_{r}^{\left(
n\right)  }-Y_{r}^{\left(  n+i\right)  }\big|^{p}\Big)\right)  ^{\alpha
q/p}\medskip\\
\displaystyle\leq C_{\alpha,p,\lambda}\cdot\left[  \mathbb{E}\left(
\Lambda_{n}\right)  \right]  ^{\alpha q/p}\rightarrow0,\quad\text{as
}n\rightarrow\infty.
\end{array}
\]
Hence there exists $Y\in S_{m}^{0}$ such that%
\[
\mathbb{E}\sup\nolimits_{t\geq0}e^{\alpha qV_{t}}|Y_{t}-Y_{t}^{\left(
n\right)  }|^{\alpha q}\rightarrow0,\quad\text{as }n\rightarrow\infty
\]
and, on a subsequence denoted also by $Y^{\left(  n\right)  },$%
\begin{equation}
\sup\nolimits_{t\geq0}e^{\alpha qV_{t}}|Y_{t}-Y_{t}^{\left(  n\right)
}|^{\alpha q}\rightarrow0,\quad\mathbb{P}\text{--a.s., as }n\rightarrow\infty.
\label{conv1}%
\end{equation}
Since $\big(Y_{t}^{\left(  n\right)  },Z_{t}^{\left(  n\right)  }\big)=\left(
\eta,0\right)  ,$ for all $t\geq\tau$ and $n\in\mathbb{N}^{\ast},$ it clearly
follows that $Y_{t}=\eta,$ for all $t\geq\tau.$

By Fatou's Lemma applied to inequality (\ref{bound-4}) we deduce%
\[
\mathbb{E}\Big(\sup\nolimits_{t\geq0}e^{pV_{t}}\left\vert Y_{t}\right\vert
^{p}\Big)+\mathbb{E}\left(
{\displaystyle\int_{0}^{\infty}}
e^{2V_{r}}{\Psi}\left(  r,Y_{r}\right)  dQ_{r}\right)  ^{p/2}\leq\tilde{L}.
\]
From (\ref{bound-4}) we also infer that there exists $Z\in\Lambda_{m\times
k}^{0}$ such that%
\[
Z^{\left(  n\right)  }\xrightharpoonup[]{\;\;\;\;}Z,\quad\text{weakly in
}L^{p}\left(  \Omega;L^{2}\left(  \mathbb{R}_{+};\mathbb{R}^{m\times
k}\right)  \right)  ,\;\text{as }n\rightarrow\infty
\]
and%
\[
\mathbb{E}\left(
{\displaystyle\int_{0}^{\infty}}
{e^{2V_{r}}}\left\vert Z_{r}\right\vert ^{2}dr\right)  ^{p/2}\leq
\liminf\nolimits_{n\rightarrow\infty}\mathbb{E}\left(
{\displaystyle\int_{0}^{\infty}}
{e^{2V_{r}}}\big|Z_{r}^{\left(  n\right)  }\big|^{2}dr\right)  ^{p/2}%
\leq\tilde{L}.
\]
Now, by the continuity property (\ref{cont-1}) on $\left[  0,n\right]  ,$ we
have, $\mathbb{P}$--a.s., for all $t\in\left[  0,n\right]  ,$%
\begin{equation}%
\begin{array}
[c]{l}%
\displaystyle e^{qV_{t}}\big|Y_{t}^{\left(  n+i\right)  }-Y_{t}^{\left(
n\right)  }\big|^{q}+c_{q,\lambda}\,{\mathbb{E}}^{\mathcal{F}_{t}}{\int
_{t}^{n}}e^{2V_{r}}\frac{\big|Z_{r}^{\left(  n+i\right)  }-Z_{r}^{\left(
n\right)  }\big|^{2}}{\big(e^{V_{r}}\big|Y_{r}^{\left(  n+i\right)  }%
-Y_{r}^{\left(  n\right)  }\big|+1\big)^{2-q}}\,dr\medskip\\
\displaystyle e^{qV_{t}}\big|Y_{t}^{\left(  n+i\right)  }-Y_{t}^{\left(
n\right)  }\big|^{q}+c_{q,\lambda}\,{\mathbb{E}}^{\mathcal{F}_{t}}{\int
_{t}^{n}}e^{qV_{r}}\big|Y_{r}^{\left(  n+i\right)  }-Y_{r}^{\left(  n\right)
}\big|^{q-2}\big|Z_{r}^{\left(  n+i\right)  }-Z_{r}^{\left(  n\right)
}\big|^{2}dr\medskip\\
\displaystyle\leq{\mathbb{E}}^{\mathcal{F}_{t}}e^{qV_{n}}\big|Y_{n}^{\left(
n+i\right)  }-Y_{n}^{\left(  n\right)  }\big|^{q}\leq\left(  \mathbb{E}%
^{\mathcal{F}_{t}}e^{pV_{n}}\big|Y_{n}^{\left(  n+i\right)  }-\xi_{n}%
\big|^{p}\right)  ^{q/p}\medskip\\
\displaystyle\leq\left[  \mathbb{E}^{\mathcal{F}_{t}}\left(  C_{\lambda
,\alpha,p}\,\mathbb{E}^{\mathcal{F}_{n}}\left(  \Lambda_{n}\right)  \right)
\right]  ^{q/p}=C_{\lambda,\alpha,p}\left[  \mathbb{E}^{\mathcal{F}_{t}%
}\left(  \Lambda_{n}\right)  \right]  ^{q/p}.
\end{array}
\label{bound-5}%
\end{equation}
Therefore, if we denote $\Delta_{t}^{\left(  n\right)  }%
\xlongequal{\hspace{-4pt}{\rm def}\hspace{-4pt}}\sup_{i\in\mathbb{N}^{\ast}%
}e^{V_{t}}\big|Y_{t}^{\left(  n+i\right)  }-Y_{t}^{\left(  n\right)  }\big|,$
then%
\[
\big(\Delta_{t}^{\left(  n\right)  }\big)^{p}\leq C_{\lambda,\alpha
,p}\,\mathbb{E}^{\mathcal{F}_{t}}\left(  \Lambda_{n}\right)  ,\quad
\mathbb{P}\text{--a.s., }\quad\text{for all }t\in\left[  0,n\right]  .
\]
From \cite[Proposition 1.56]{pa-ra/14} we infer%
\[
\mathbb{E}\sup\nolimits_{t\in\left[  0,T\right]  }\big(\Delta_{t}^{\left(
n\right)  }\big)^{\alpha p}\leq\frac{1}{1-\alpha}\,\left(  C_{\lambda
,\alpha,p}\,\mathbb{E}\left(  \Lambda_{n}\right)  \right)  ^{\alpha}%
,\quad\text{for all }0<\alpha<1.
\]
Consequently, by Beppo Levi monotone convergence theorem for $T\rightarrow
\infty,$ it follows%
\[
\mathbb{E}\sup\nolimits_{t\geq0}\big(\Delta_{t}^{\left(  n\right)
}\big)^{\alpha p}\leq\frac{1}{1-\alpha}\,\left(  C_{\lambda,\alpha
,p}\,\mathbb{E}\left(  \Lambda_{n}\right)  \right)  ^{\alpha},\rightarrow
0,\quad\text{as }n\rightarrow\infty.
\]
Let $T>0$ be arbitrary and $T\leq n.$ Then, from (\ref{bound-5}),%
\[%
\begin{array}
[c]{l}%
\displaystyle c_{q,\lambda}\,{\mathbb{E}}\bigg(\,\frac{1}{\big(\sup
\nolimits_{t\geq0}\Delta_{t}^{\left(  n\right)  }+1\big)^{2-q}}\,{\int_{0}%
^{T}}e^{2V_{r}}\big|Z_{r}^{\left(  n+i\right)  }-Z_{r}^{\left(  n\right)
}\big|^{2}dr\bigg)\medskip\\
\displaystyle\leq c_{q,\lambda}\,{\mathbb{E}\int_{0}^{n}}e^{2V_{r}}%
\frac{\big|Z_{r}^{\left(  n+i\right)  }-Z_{r}^{\left(  n\right)  }\big|^{2}%
}{\big(e^{V_{r}}\big|Y_{r}^{\left(  n+i\right)  }-Y_{r}^{\left(  n\right)
}\big|+1\big)^{2-q}}\,dr\leq\left[  C_{\lambda,\alpha,p}\,\mathbb{E}\left(
\Lambda_{n}\right)  \right]  ^{q/p}.
\end{array}
\]
In this inequality we pass to $\liminf_{i\rightarrow\infty}$ and it follows%
\[
c_{q,\lambda}\,{\mathbb{E}}\bigg(\,\frac{1}{\big(\sup\nolimits_{t\geq0}%
\Delta_{t}^{\left(  n\right)  }+1\big)^{2-q}}\,{\int_{0}^{T}}e^{2V_{r}%
}\big|Z_{r}-Z_{r}^{\left(  n\right)  }\big|^{2}dr\bigg)\leq\left[
C_{\lambda,\alpha,p}\,\mathbb{E}\left(  \Lambda_{n}\right)  \right]  ^{q/p}.
\]
Now, by Beppo Levi monotone convergence theorem for $T\rightarrow\infty,$ we
obtain%
\[
c_{q,\lambda}\,{\mathbb{E}}\bigg(\,\frac{1}{\big(\sup\nolimits_{t\geq0}%
\Delta_{t}^{\left(  n\right)  }+1\big)^{2-q}}\,{\int_{0}^{\infty}}e^{2V_{r}%
}\big|Z_{r}-Z_{r}^{\left(  n\right)  }\big|^{2}dr\bigg)\leq\left[
C_{\lambda,\alpha,p}\,\mathbb{E}\left(  \Lambda_{n}\right)  \right]
^{q/p}\rightarrow0,\quad\text{as }n\rightarrow\infty.
\]
Hence, on a subsequence denoted also by $Z^{\left(  n\right)  },$ we get
\begin{equation}
{%
{\displaystyle\int_{0}^{\infty}}
}e^{2V_{r}}\big|Z_{r}-Z_{r}^{\left(  n\right)  }\big|^{2}dr\rightarrow
0,\quad\mathbb{P}\text{--a.s., as }n\rightarrow\infty. \label{conv2}%
\end{equation}
Since $Z_{r}^{\left(  n\right)  }=0$ for all $r>\tau,$ we clearly deduce
$Z_{r}=0,$ for all $r>\tau.\medskip$

We shall verify that (\ref{def2}) is satisfied. For $1\leq n<T$ we have%
\[%
\begin{array}
[c]{l}%
\displaystyle\limsup\nolimits_{T\rightarrow\infty}\left[  e^{pV_{T}}\left\vert
Y_{T}-\xi_{T}\right\vert ^{p}+\bigg(\int_{T}^{\infty}e^{2V_{s}}\left\vert
Z_{s}-\zeta_{s}\right\vert ^{2}ds\bigg)^{p/2}\right]  \medskip\\
\displaystyle=\limsup\nolimits_{T\rightarrow\infty}\left[  e^{pV_{T}%
}\big|Y_{T}-Y_{T}^{\left(  n\right)  }\big|^{p}+\bigg(\int_{T}^{\infty
}e^{2V_{s}}\big|Z_{s}-Z_{s}^{\left(  n\right)  }\big|^{2}ds\bigg)^{p/2}%
\right]  \medskip\\
\displaystyle\leq\big(\sup\nolimits_{t\geq0}e^{pV_{t}}\big|Y_{t}%
-Y_{t}^{\left(  n\right)  }\big|^{p}\big)^{\alpha q/p}+\bigg(\int_{0}^{\infty
}e^{2V_{s}}\big|Z_{s}-Z_{s}^{\left(  n\right)  }\big|^{2}ds\bigg)^{p/2}.
\end{array}
\]
From (\ref{conv1}) and (\ref{conv2}) we get%
\[
e^{pV_{T}}\left\vert Y_{T}-\xi_{T}\right\vert ^{p}+\bigg(\int_{T}^{\infty
}e^{2V_{s}}\left\vert Z_{s}-\zeta_{s}\right\vert ^{2}ds\bigg)^{p/2}%
\rightarrow0,\quad\mathbb{P}\text{--a.s., \ as }T\rightarrow\infty,
\]
and consequently the convergence holds true in $L^{0}\left(  \Omega
,\mathcal{F},\mathbb{P};\mathbb{R}^{m}\right)  $ (which is the convergence in
pro\-ba\-bi\-lity).$\medskip$

In order to verify (\ref{def1}), let $0\leq t\leq s<\infty$ be arbitrary
chosen and let $n>s.$

Then, for $q\in\{2,p\wedge2\},$ $\delta_{q}=\delta\mathbf{1}_{[1,2)}\left(
q\right)  $ and $\Gamma_{t}^{\left(  n\right)  }=\big(\big|M_{t}%
-Y_{t}^{\left(  n\right)  }\big|^{2}+\delta_{q}\big)^{1/2},$ it holds%
\begin{equation}%
\begin{array}
[c]{l}%
\displaystyle\big(\Gamma_{t}^{\left(  n\right)  }\big)^{q}+\dfrac{q\left(
q-1\right)  }{2}%
{\displaystyle\int_{t}^{s}}
\big({\Gamma_{r}^{\left(  n\right)  }\big)^{q-2}}{\Large \,}\big|R_{r}%
-Z_{r}^{\left(  n\right)  }\big|^{2}dr+{q%
{\displaystyle\int_{t}^{s}}
}\big({\Gamma_{r}^{\left(  n\right)  }\big)^{q-2}\Psi\big(}r,Y_{r}^{\left(
n\right)  }\big)dQ_{r}\medskip\\
\displaystyle\leq\big(\Gamma_{s}^{\left(  n\right)  }\big)^{q}+{q%
{\displaystyle\int_{t}^{s}}
}\big({\Gamma_{r}^{\left(  n\right)  }\big)^{q-2}\Psi}\left(  r,M_{r}\right)
dQ_{r}-q%
{\displaystyle\int_{t}^{s}}
\big({\Gamma_{r}^{\left(  n\right)  }\big)^{q-2}}\,\langle M_{r}%
-Y_{r}^{\left(  n\right)  },{\big(}R_{r}-Z_{r}^{\left(  n\right)  }%
\big)dB_{r}\rangle\medskip\\
\displaystyle\quad+q%
{\displaystyle\int_{t}^{s}}
{\big(\Gamma_{r}^{\left(  n\right)  }\big)^{q-2}}\langle M_{r}-Y_{r}^{\left(
n\right)  },N_{r}-H^{\left(  n\right)  }{\big(}r,Y_{r}^{\left(  n\right)
},Z_{r}^{\left(  n\right)  }\big)\rangle dQ_{r}\,,
\end{array}
\label{def1n}%
\end{equation}
for all $M\in\mathcal{V}_{m}^{0}$ of the form%
\[
M_{t}=M_{T}+%
{\displaystyle\int_{t}^{T}}
N_{r}dQ_{r}-%
{\displaystyle\int_{t}^{T}}
R_{r}dB_{r}\,.
\]
Passing to the limit, for $n\rightarrow\infty,$ in (\ref{def1n}) we infer
(using for the left-hand side the Fatou's Lemma and, for the right-hand side,
the Lebesgue's dominated convergence theorem and the continuity of the
stochastic integral with respect to the convergence in probability) that the
pair $\left(  Y,Z\right)  $ satisfies inequality (\ref{def1}).\hfill
\end{proof}

\section{Appendix}

\hspace{\parindent}In this section, mainly based on some results from
\cite{pa-ra/14} and their proofs, we recall and we obtain new inequalities and
properties useful in our framework and frequently used in our paper. These
results concern mainly inequalities for BSDEs and are interesting by
themselves. For more details the interested readers are referred to the
monograph of Pardoux and R\u{a}\c{s}canu \cite{pa-ra/14}.

Let $\left\{  B_{t}:t\geq0\right\}  $ be a $k$--dimensional Brownian motion
with respect to a given stochastic basis $\left(  \Omega,\mathcal{F}%
,\mathbb{P},\{\mathcal{F}_{t}\}_{t\geq0}\right)  $, where $\left(
\mathcal{F}_{t}\right)  _{t\geq0}$ is the natural filtration associated to
$\left\{  B_{t}:t\geq0\right\}  $ augmented with $\mathcal{N}$ (the set of
$\mathbb{P}$--null events of $\mathcal{F}\,$).

\begin{notation}
If $p\geq1$ we denote $n_{p}%
\xlongequal{\hspace{-4pt}{\rm def}\hspace{-4pt}}\left(  p-1\right)  \wedge1.$
\end{notation}

\subsection{An It\^{o}'s formula and some backward stochastic inequalities}

\hspace{\parindent}For the proof of the next result see equality $\left(
2.24\right)  $ from the proof of \cite[Proposition 2.26]{pa-ra/14} and
\cite[Corollary 2.29]{pa-ra/14}.

\begin{proposition}
\label{p1-ito}Let $p\in\mathbb{R}$, $\rho\geq0$ and $\delta$ such that
$\delta\geq0,$ if $p\geq2$ and $\delta>0,$ if $p<2.$

Let $Y\in S_{d}^{0}$ be a local \textit{semimartingale of the form}%
\begin{equation}%
\begin{array}
[c]{l}%
\displaystyle Y_{t}=Y_{0}-\int_{0}^{t}dK_{r}+\int_{0}^{t}R_{r}dB_{r},\quad
t\geq0\quad\text{or equivalently}\medskip\\
\displaystyle Y_{t}=Y_{T}+\int_{t}^{T}dK_{r}-\int_{t}^{T}R_{r}dB_{r}%
,\quad\text{for all }0\leq t\leq T,
\end{array}
\label{ito1}%
\end{equation}
\textit{where }$R\in\Lambda_{m\times k}^{0}\,,$ $K\in S_{m}^{0}\,,$ $K_{\cdot
}\in\mathrm{BV}_{\mathrm{loc}}\left(  \mathbb{R}_{+};\mathbb{R}^{m}\right)  ,$
$\mathbb{P}$--a.s..

Let $\varphi_{\rho,\delta}:\mathbb{R}^{d}\rightarrow\left(  0,\infty\right)
,$%
\[
\varphi_{\rho,\delta}\left(  x\right)  =\bigg(\dfrac{\left\vert x\right\vert
^{2}}{1+\rho\left\vert x\right\vert ^{2}}+\delta\bigg)^{1/2}.
\]
By It\^{o}'s formula, applied to $\varphi_{\rho,\delta}^{p}\left(
Y_{t}\right)  ,$ with $p\in\mathbb{R},$ we have, for all $0\leq t\leq s\leq
T,$%
\begin{equation}%
\begin{array}
[c]{l}%
\displaystyle\varphi_{\rho,\delta}^{p}\left(  Y_{t}\right)  +\dfrac{p}{2}%
{\displaystyle\int_{t}^{s}}
R_{r}^{\left(  p,\rho,\delta\right)  }dr+\dfrac{p}{2}\left[  L_{s}^{\left(
p,\rho,\delta\right)  }-L_{t}^{\left(  p,\rho,\delta\right)  }\right]
\medskip\\
\displaystyle=\varphi_{\rho,\delta}^{p}\left(  Y_{s}\right)  +\dfrac{p}{2}%
{\displaystyle\int_{t}^{s}}
Q_{r}^{\left(  p,\rho,\delta\right)  }dr+p%
{\displaystyle\int_{t}^{s}}
\big\langle U_{r}^{\left(  p,\rho,\delta\right)  },dK_{r}\big\rangle-p%
{\displaystyle\int_{t}^{s}}
\big\langle U_{r}^{\left(  p,\rho,\delta\right)  },R_{r}dB_{r}%
\big\rangle,\quad\mathbb{P}\text{--a.s.,}%
\end{array}
\label{ito2}%
\end{equation}
where

$U_{r}^{\left(  p,\rho,\delta\right)  }=\varphi_{\rho,\delta}^{p-2}\left(
Y_{r}\right)  \,\dfrac{1}{(1+\rho|Y_{r}|^{2})^{2}}\,Y_{r}\,,\medskip$

$R_{r}^{\left(  p,\rho,\delta\right)  }=\varphi_{\rho,\delta}^{p-4}\left(
Y_{r}\right)  \,\dfrac{1}{(1+\rho|Y_{r}|^{2})^{3}}\,\bigg[\dfrac{p-1}%
{1+\rho\left\vert Y_{r}\right\vert ^{2}}\,\left\vert R_{r}^{\ast}%
Y_{r}\right\vert ^{2}+\left(  \left\vert R_{r}\right\vert ^{2}\left\vert
Y_{r}\right\vert ^{2}-\left\vert R_{r}^{\ast}Y_{r}\right\vert ^{2}\right)
\bigg],\medskip$

$L_{t}^{\left(  p,\rho,\delta\right)  }=\delta\,%
{\displaystyle\int_{0}^{t}}
\varphi_{\rho,\delta}^{p-4}\left(  Y_{r}\right)  \,\dfrac{1}{(1+\rho
|Y_{r}|^{2})^{3}}\,\left[  \left\vert R_{r}\right\vert ^{2}+\rho\left(
\left\vert R_{r}\right\vert ^{2}\left\vert Y_{r}\right\vert ^{2}-\left\vert
R_{r}^{\ast}Y_{r}\right\vert ^{2}\right)  \right]  dr,\medskip$

and$\medskip$

$Q_{r}^{\left(  p,\rho,\delta\right)  }=\varphi_{\rho,\delta}^{p-2}\left(
Y_{r}\right)  \,\dfrac{3\rho}{(1+\rho|Y_{r}|^{2})^{3}}\,\left\vert R_{r}%
^{\ast}Y_{r}\right\vert ^{2}.\medskip$

In the case $\rho=0$ we have%
\begin{equation}%
\begin{array}
[c]{l}%
\displaystyle\big(\left\vert Y_{t}\right\vert ^{2}+\delta\big)^{p/2}+\dfrac
{p}{2}\,%
{\displaystyle\int_{t}^{s}}
\big(\left\vert Y_{r}\right\vert ^{2}+\delta\big)^{\left(  p-4\right)
/2}\left[  \left(  p-1\right)  \left\vert R_{r}^{\ast}Y_{r}\right\vert
^{2}+\big(\left\vert R_{r}\right\vert ^{2}\left\vert Y_{r}\right\vert
^{2}-\left\vert R_{r}^{\ast}Y_{r}\right\vert ^{2}\big)\right]  dr\medskip\\
\displaystyle\quad+\dfrac{p}{2}\,%
{\displaystyle\int_{t}^{s}}
\delta\,\big(\left\vert Y_{r}\right\vert ^{2}+\delta\big)^{\left(  p-4\right)
/2}\left\vert R_{r}\right\vert ^{2}dr\medskip\\
\displaystyle=\big(\left\vert Y_{s}\right\vert ^{2}+\delta\big)^{p/2}+p%
{\displaystyle\int_{t}^{s}}
\big(\left\vert Y_{r}\right\vert ^{2}+\delta\big)^{\left(  p-2\right)
/2}\left\langle Y_{r},dK_{r}\right\rangle \medskip\\
\displaystyle\quad-p%
{\displaystyle\int_{t}^{s}}
\big(\left\vert Y_{r}\right\vert ^{2}+\delta\big)^{\left(  p-2\right)
/2}\left\langle Y_{r},R_{r}dB_{r}\right\rangle .
\end{array}
\label{ito3}%
\end{equation}

\end{proposition}

\begin{remark}
\label{r1-ito}If $p\geq1,$ then%
\[%
\begin{array}
[c]{l}%
\displaystyle\left(  p-1\right)  \left\vert R_{r}^{\ast}Y_{r}\right\vert
^{2}+\big(\left\vert R_{r}\right\vert ^{2}\left\vert Y_{r}\right\vert
^{2}-\left\vert R_{r}^{\ast}Y_{r}\right\vert ^{2}\big)+\delta\left\vert
R_{r}\right\vert ^{2}\medskip\\
\displaystyle\geq n_{p}\left[  \left\vert R_{r}^{\ast}Y_{r}\right\vert
^{2}+\big(\left\vert R_{r}\right\vert ^{2}\left\vert Y_{r}\right\vert
^{2}-\left\vert R_{r}^{\ast}Y_{r}\right\vert ^{2}\big)\right]  +\delta
\left\vert R_{r}\right\vert ^{2}\medskip\\
\displaystyle=\big(n_{p}\left\vert Y_{r}\right\vert ^{2}+\delta\big)\left\vert
R_{r}\right\vert ^{2}\geq n_{p}\big(\left\vert Y_{r}\right\vert ^{2}%
+\delta\big)\left\vert R_{r}\right\vert ^{2}%
\end{array}
\]
and from (\ref{ito3}) we infer%
\begin{equation}%
\begin{array}
[c]{l}%
\displaystyle\big(\left\vert Y_{t}\right\vert ^{2}+\delta\big)^{p/2}+\dfrac
{p}{2}%
{\displaystyle\int_{t}^{s}}
\big(\left\vert Y_{r}\right\vert ^{2}+\delta\big)^{\left(  p-4\right)
/2}\big(n_{p}\left\vert Y_{r}\right\vert ^{2}+\delta\big)\left\vert
R_{r}\right\vert ^{2}dr\medskip\\
\displaystyle\leq\big(\left\vert Y_{s}\right\vert ^{2}+\delta\big)^{p/2}+p%
{\displaystyle\int_{t}^{s}}
\big(\left\vert Y_{r}\right\vert ^{2}+\delta\big)^{\left(  p-2\right)
/2}\left\langle Y_{r},dK_{r}\right\rangle \medskip\\
\displaystyle\quad-p%
{\displaystyle\int_{t}^{s}}
\big(\left\vert Y_{r}\right\vert ^{2}+\delta\big)^{\left(  p-2\right)
/2}\left\langle Y_{r},R_{r}dB_{r}\right\rangle ,\quad\mathbb{P}\text{--a.s.,}%
\end{array}
\label{ito4}%
\end{equation}
for all $0\leq t\leq s\leq T.$
\end{remark}

\subsection{Backward stochastic inequalities}

\hspace{\parindent}Based on \cite[Proposition 6.80]{pa-ra/14} and its proof we
adapt here the Pardoux--R\u{a}\c{s}canu's inequalities (6.92) and (6.94) from
\cite{pa-ra/14} to the case of our framework (namely the fact that the
solutions are defined using not an equality but a stochastic inequality, see
Definition \ref{definition_weak solution}).

\begin{proposition}
\label{an-prop-dz}Let $\left(  Y,Z\right)  \in S_{m}^{0}\times\Lambda_{m\times
k}^{0}$ and $a\geq0,$ $\gamma\in\mathbb{R}$ such that, for all $0\leq t\leq
s<\infty,$%
\[
{%
{\displaystyle\int_{t}^{s}}
}\left\vert Z_{r}\right\vert ^{2}dr+{%
{\displaystyle\int_{t}^{s}}
}dD_{r}\leq a\left\vert Y_{s}\right\vert ^{2}+a{%
{\displaystyle\int_{t}^{s}}
}\left(  dR_{r}+\left\vert Y_{r}\right\vert dN_{r}\right)  +\gamma%
{\displaystyle\int_{t}^{s}}
\,\langle Y_{r},Z_{r}dB_{r}\rangle,\quad\mathbb{P}\text{--a.s.,}%
\]
where $R,N$ and $D$ are increasing and continuous p.m.s.p. with $R_{0}%
=N_{0}=D_{0}=0.$

Then, for all $q>0$ and for all stopping times $0\leq\sigma\leq\theta<\infty,$
the following inequality hold:%
\begin{equation}%
\begin{array}
[c]{l}%
\mathbb{E}^{\mathcal{F}_{\sigma}}\Big(%
{\displaystyle\int_{\sigma}^{\theta}}
\left\vert Z_{r}\right\vert ^{2}dr\Big)^{q/2}+\mathbb{E}^{\mathcal{F}_{\sigma
}}\Big(%
{\displaystyle\int_{\sigma}^{\theta}}
dD_{r}\Big)^{q/2}\medskip\\
\leq C_{a,\gamma,q}\,\mathbb{E}^{\mathcal{F}_{\sigma}}\bigg[\sup
\nolimits_{r\in\left[  \sigma,\theta\right]  }\left\vert Y_{r}\right\vert
^{q}+\Big(%
{\displaystyle\int_{\sigma}^{\theta}}
dR_{r}\Big)^{q/2}+\Big(%
{\displaystyle\int_{\sigma}^{\theta}}
\left\vert Y_{r}\right\vert dN_{r}\Big)^{q/2}\bigg]\medskip\\
\leq2\,C_{a,\gamma,q}\,\mathbb{E}^{\mathcal{F}_{\sigma}}\bigg[\sup
\nolimits_{r\in\left[  \sigma,\theta\right]  }\left\vert Y_{r}\right\vert
^{q}+\Big(%
{\displaystyle\int_{\sigma}^{\theta}}
dR_{r}\Big)^{q/2}+\Big(%
{\displaystyle\int_{\sigma}^{\theta}}
dN_{r}\Big)^{q}\bigg],\quad\mathbb{P}\text{--a.s.,}%
\end{array}
\label{an4}%
\end{equation}
where $C_{a,\gamma,q}$ is a positive constant depending only on $a,\gamma$ and
$q.$
\end{proposition}

\begin{proof}
We follow the first part of the proof of \cite[Proposition 6.80]{pa-ra/14}.
Let the sequence of stopping times%
\begin{equation}
\theta_{n}=\theta\wedge\inf\left\{  s\geq\sigma:\sup\nolimits_{r\in\left[
\sigma,\sigma\vee s\right]  }\left\vert Y_{r}-Y_{\sigma}\right\vert +%
{\displaystyle\int_{\sigma}^{\sigma\vee s}}
\left\vert Z_{r}\right\vert ^{2}dr+%
{\displaystyle\int_{\sigma}^{\sigma\vee s}}
d\left(  D_{r}+R_{r}+N_{r}\right)  \geq n\right\}  . \label{an4-a}%
\end{equation}
We have, for $q>0,$%
\begin{equation}%
\begin{array}
[c]{l}%
\mathbb{E}^{\mathcal{F}_{\sigma}}\Big(%
{\displaystyle\int_{\sigma}^{\theta_{n}}}
\left\vert Z_{r}\right\vert ^{2}dr\Big)^{q/2}+\mathbb{E}^{\mathcal{F}_{\sigma
}}\Big(%
{\displaystyle\int_{\sigma}^{\theta_{n}}}
dD_{r}\Big)^{q/2}\medskip\\
\leq2\mathbb{E}^{\mathcal{F}_{\sigma}}\Big(%
{\displaystyle\int_{\sigma}^{\theta_{n}}}
\left\vert Z_{r}\right\vert ^{2}dr+%
{\displaystyle\int_{\sigma}^{\theta_{n}}}
dD_{r}\Big)^{q/2}\medskip\\
\leq C_{a,\gamma,q}^{\prime}\,\mathbb{E}^{\mathcal{F}_{\sigma}}\Big[\left\vert
Y_{\theta_{n}}\right\vert ^{q}+\Big(%
{\displaystyle\int_{\sigma}^{\theta_{n}}}
dR_{r}\Big)^{q/2}+\Big(%
{\displaystyle\int_{\sigma}^{\theta_{n}}}
\left\vert Y_{r}\right\vert dN_{r}\Big)^{q/2}+\Big|%
{\displaystyle\int_{\sigma}^{\theta_{n}}}
\,\langle Y_{r},Z_{r}dB_{r}\rangle\Big|^{q/2}\Big].
\end{array}
\label{an5}%
\end{equation}
By Burkholder--Davis--Gundy inequality we get%
\begin{align*}
C_{a,\gamma,q}^{\prime}\,\mathbb{E}^{\mathcal{F}_{\sigma}}\Big|%
{\displaystyle\int_{\sigma}^{\theta_{n}}}
\,\langle Y_{r},Z_{r}dB_{r}\rangle\Big|^{q/2}  &  \leq C_{a,\gamma,q}%
^{\prime\prime}\,\mathbb{E}^{\mathcal{F}_{\sigma}}\Big(%
{\displaystyle\int_{\sigma}^{\theta_{n}}}
\,\left\vert Y_{r}\right\vert ^{2}\left\vert Z_{r}\right\vert ^{2}%
dr\Big)^{q/4}\\
&  \leq C_{a,\gamma,q}^{\prime\prime}\,\mathbb{E}^{\mathcal{F}_{\sigma}}%
\sup\nolimits_{r\in\left[  \sigma,\theta_{n}\right]  }\left\vert
Y_{r}\right\vert ^{q/2}\Big(%
{\displaystyle\int_{\sigma}^{\theta_{n}}}
\,\left\vert Z_{r}\right\vert ^{2}dr\Big)^{q/4}\\
&  \leq\frac{1}{2}\,\left(  C_{a,\gamma,q}^{\prime\prime}\right)
^{2}\,\mathbb{E}^{\mathcal{F}_{\sigma}}\sup\nolimits_{r\in\left[
\sigma,\theta_{n}\right]  }\left\vert Y_{r}\right\vert ^{q}+\frac{1}%
{2}\,\mathbb{E}^{\mathcal{F}_{\sigma}}\Big(%
{\displaystyle\int_{\sigma}^{\theta_{n}}}
\left\vert Z_{r}\right\vert ^{2}dr\Big)^{q/2}%
\end{align*}
and consequently from (\ref{an5}) the following inequality holds%
\begin{equation}%
\begin{array}
[c]{l}%
\mathbb{E}^{\mathcal{F}_{\sigma}}\Big(%
{\displaystyle\int_{\sigma}^{\theta_{n}}}
\left\vert Z_{r}\right\vert ^{2}dr\Big)^{q/2}+\mathbb{E}^{\mathcal{F}_{\sigma
}}\Big(%
{\displaystyle\int_{\sigma}^{\theta_{n}}}
dD_{r}\Big)^{q/2}\medskip\\
\leq C_{a,\gamma,q}\,\mathbb{E}^{\mathcal{F}_{\sigma}}\Big[\sup\nolimits_{r\in
\left[  \sigma,\theta\right]  }\left\vert Y_{r}\right\vert ^{q}+\Big(%
{\displaystyle\int_{\sigma}^{\theta}}
dR_{r}\Big)^{q/2}+\Big(%
{\displaystyle\int_{\sigma}^{\theta}}
\left\vert Y_{r}\right\vert dN_{r}\Big)^{q/2}\Big].
\end{array}
\label{an5-aa}%
\end{equation}
Since%
\[
\Big(%
{\displaystyle\int_{\sigma}^{\theta}}
\left\vert Y_{r}\right\vert dN_{r}\Big)^{q/2}\leq\frac{1}{2}\,\sup
\nolimits_{r\in\left[  \sigma,\theta\right]  }\left\vert Y_{r}\right\vert
^{q}+\frac{1}{2}\,\Big(%
{\displaystyle\int_{\sigma}^{\theta}}
dN_{r}\Big)^{q},
\]
from (\ref{an5-aa}) we infer%
\begin{equation}%
\begin{array}
[c]{l}%
\mathbb{E}^{\mathcal{F}_{\sigma}}\Big(%
{\displaystyle\int_{\sigma}^{\theta_{n}}}
\left\vert Z_{r}\right\vert ^{2}dr\Big)^{q/2}+\mathbb{E}^{\mathcal{F}_{\sigma
}}\Big(%
{\displaystyle\int_{\sigma}^{\theta_{n}}}
dD_{r}\Big)^{q/2}\medskip\\
\leq C_{a,\gamma,q}\,\mathbb{E}^{\mathcal{F}_{\sigma}}\Big[\sup\nolimits_{r\in
\left[  \sigma,\theta\right]  }\left\vert Y_{r}\right\vert ^{q}+\Big(%
{\displaystyle\int_{\sigma}^{\theta}}
dR_{r}\Big)^{q/2}+\Big(%
{\displaystyle\int_{\sigma}^{\theta}}
dN_{r}\Big)^{q}\Big].
\end{array}
\label{an5-ab}%
\end{equation}
Consequently, by Fatou's Lemma, as $n\rightarrow\infty,$ inequality
(\ref{an4}) follows.\hfill
\end{proof}

\begin{proposition}
\label{an-prop-ydz}Let $\left(  Y,Z\right)  \in S_{m}^{0}\times\Lambda
_{m\times k}^{0}$ , $a\geq0,$ $\gamma\in\mathbb{R}$ and $1<q\leq p$ satisfying
for all $0\leq t\leq s<\infty:$%
\[%
\begin{array}
[c]{l}%
\displaystyle\left\vert Y_{t}\right\vert ^{q}+%
{\displaystyle\int_{t}^{s}}
\left\vert Y_{r}\right\vert ^{q-2}\mathbf{1}_{Y_{r}\neq0}{\Large \,}\left\vert
Z_{r}\right\vert ^{2}dr+{%
{\displaystyle\int_{t}^{s}}
\left\vert Y_{r}\right\vert ^{q-2}\,\mathbf{1}_{Y_{r}\neq0}\,dD_{r}}\medskip\\
\displaystyle\leq a\left\vert Y_{s}\right\vert ^{q}+a%
{\displaystyle\int_{t}^{s}}
\left[  \left\vert Y_{r}\right\vert ^{q-2}\mathbf{1}_{Y_{r}\neq0}%
\mathbf{1}_{q\geq2}dR_{r}+\left\vert Y_{r}\right\vert ^{q-1}dN_{r}\right]
+\gamma%
{\displaystyle\int_{t}^{s}}
\left\vert Y_{r}\right\vert ^{q-2}\,\langle Y_{r},Z_{r}dB_{r}\rangle
,\quad\mathbb{P}\text{--a.s.,}%
\end{array}
\]
where\footnote{\thinspace\ We use the convention: $\left\vert Y_{r}\right\vert
^{q-2}Y_{r}=\left\vert Y_{r}\right\vert ^{q-2}\,{\mathbf{1}_{Y_{r}\neq0}%
\,}Y_{r}\,,$ for any $q\geq1.$} $R,N$ and $D$ are increasing and continuous
p.m.s.p. with $R_{0}=N_{0}=D_{0}=0.$

If $\sigma$ and $\theta$ are two stopping times such that $0\leq\sigma
\leq\theta<\infty$ and%
\begin{equation}
\mathbb{E}\sup\nolimits_{r\in\left[  \sigma,\theta\right]  }\left\vert
Y_{r}\right\vert ^{p}<\infty, \label{an6-ab}%
\end{equation}
then, $\mathbb{P}$--a.s.,%
\begin{equation}%
\begin{array}
[c]{l}%
\mathbb{E}^{\mathcal{F}_{\sigma}}\Big[\sup\nolimits_{r\in\left[  \sigma
,\theta\right]  }\left\vert Y_{r}\right\vert ^{p}+\Big(%
{\displaystyle\int_{\sigma}^{\theta}}
\left\vert Y_{r}\right\vert ^{q-2}{\Large \,}\mathbf{1}_{Y_{r}\neq0}%
{\,}\left\vert Z_{r}\right\vert ^{2}dr\Big)^{p/q}+\Big(%
{\displaystyle\int_{\sigma}^{\theta}}
{\left\vert Y_{r}\right\vert ^{q-2}\,{\mathbf{1}_{Y_{r}\neq0}}\,dD}%
_{r}\Big)^{p/q}\Big]\medskip\\
\leq C_{p,q,a,\gamma}\,\mathbb{E}^{\mathcal{F}_{\sigma}}\Big[\left\vert
Y_{\theta}\right\vert ^{p}+\Big(%
{\displaystyle\int_{\sigma}^{\theta}}
\mathbf{1}_{q\geq2}{\,}dR_{r}\Big)^{p/2}+\Big(%
{\displaystyle\int_{\sigma}^{\theta}}
dN_{r}\Big)^{p}\Big].
\end{array}
\label{an5-b}%
\end{equation}
where $C_{p,q,a,\gamma}$ is a positive constant depending only on $p,q,a$ and
$\gamma.$
\end{proposition}

\begin{proof}
We follow the proof of \cite[Proposition 6.80]{pa-ra/14}. Let the stopping
time $\theta_{n}$ be defined by
\[
\theta_{n}=\theta\wedge\inf\left\{  s\geq\sigma:\sup\nolimits_{r\in\left[
\sigma,\sigma\vee s\right]  }\left\vert Y_{r}-Y_{\sigma}\right\vert +%
{\displaystyle\int_{\sigma}^{\sigma\vee s}}
\left\vert Z_{r}\right\vert ^{2}dr+%
{\displaystyle\int_{\sigma}^{\sigma\vee s}}
d\left(  D_{r}+R_{r}+N_{r}\right)  \geq n\right\}  .
\]
For any stopping time $\tau\in\left[  \sigma,\theta_{n}\right]  $ we have%
\begin{equation}%
\begin{array}
[c]{l}%
\left\vert Y_{\tau}\right\vert ^{q}+%
{\displaystyle\int_{\tau}^{\theta_{n}}}
\left\vert Y_{r}\right\vert ^{q-2}{\Large \,}\mathbf{1}_{Y_{r}\neq
0}{\Large \,}\left\vert Z_{r}\right\vert ^{2}dr+%
{\displaystyle\int_{\tau}^{\theta_{n}}}
{\left\vert Y_{r}\right\vert ^{q-2}{\Large \,}\mathbf{1}_{Y_{r}\neq
0}{\Large \,}dD}_{r}\medskip\\
\leq a\left\vert Y_{\theta_{n}}\right\vert ^{q}+a%
{\displaystyle\int_{\tau}^{\theta_{n}}}
\left(  \left\vert Y_{r}\right\vert ^{q-2}{\Large \,}\mathbf{1}_{Y_{r}\neq
0}{\Large \,}\mathbf{1}_{q\geq2}{\Large \,}dR_{r}+\left\vert Y_{r}\right\vert
^{q-1}dN_{r}\right)  +\gamma%
{\displaystyle\int_{\tau}^{\theta_{n}}}
\left\vert Y_{r}\right\vert ^{q-2}\,\langle Y_{r},Z_{r}dB_{r}\rangle.
\end{array}
\label{an6}%
\end{equation}
Remark that
\[
M_{s}=\int_{0}^{s}\mathbf{1}_{\left[  \sigma,\theta_{n}\right]  }\left(
r\right)  \,\left\vert Y_{r}\right\vert ^{q-2}\,\langle Y_{r},Z_{r}%
dB_{r}\rangle,\quad s\geq0
\]
is a martingale, since for any $T>0,$%
\begin{align*}
\mathbb{E}\Big(%
{\displaystyle\int_{0}^{T}}
\mathbf{1}_{\left[  \sigma,\theta_{n}\right]  }\left(  r\right)  \,\left\vert
Y_{r}\right\vert ^{2q-2}\left\vert Z_{r}\right\vert ^{2}dr\Big)^{1/2}  &
\leq\mathbb{E}\Big[\sup\nolimits_{r\in\left[  \sigma,\theta_{n}\right]
}\left\vert Y_{r}\right\vert ^{q-1}\Big(\int_{\sigma}^{\theta_{n}}\left\vert
Z_{r}\right\vert ^{2}dr\Big)^{1/2}\Big]\medskip\\
&  \leq\frac{q-1}{q}\,\mathbb{E}\sup\nolimits_{r\in\left[  \sigma,\theta
_{n}\right]  }\left\vert Y_{r}\right\vert ^{q}+\frac{1}{q}\,\mathbb{E}%
\Big(\int_{\sigma}^{\theta_{n}}\left\vert Z_{r}\right\vert ^{2}dr\Big)^{q/2}%
\medskip\\
&  \leq\frac{q-1}{q}\,\mathbb{E}\left(  \left\vert Y_{\sigma}\right\vert
+n\right)  ^{q}+\frac{1}{q}\,n^{q/2}<\infty.
\end{align*}
Therefore, from (\ref{an6}),%
\begin{equation}%
\begin{array}
[c]{l}%
\mathbb{E}^{\mathcal{F}_{\sigma}}\Big[\Big(%
{\displaystyle\int_{\sigma}^{\theta_{n}}}
\left\vert Y_{r}\right\vert ^{q-2}{\Large \,}\mathbf{1}_{Y_{r}\neq
0}{\Large \,}\left\vert Z_{r}\right\vert ^{2}dr\Big)^{p/q}+\Big(%
{\displaystyle\int_{\sigma}^{\theta_{n}}}
{\left\vert Y_{r}\right\vert ^{q-2}{\Large \,}\mathbf{1}_{Y_{r}\neq
0}{\Large \,}dD}_{r}\Big)^{p/q}\Big]\medskip\\
\leq C_{p,q,a}{{\Large \,}}\mathbb{E}^{\mathcal{F}_{\sigma}}\Big[\left\vert
Y_{\theta_{n}}\right\vert ^{p}+\Big(%
{\displaystyle\int_{\sigma}^{\theta_{n}}}
\left(  \left\vert Y_{r}\right\vert ^{q-2}{{\Large \,}}\mathbf{1}_{Y_{r}\neq
0}{{\Large \,}}\mathbf{1}_{q\geq2}{{\Large \,}}dR_{r}\right)  \Big)^{p/q}%
+\Big(%
{\displaystyle\int_{\sigma}^{\theta_{n}}}
\left\vert Y_{r}\right\vert ^{q-1}dN_{r}\Big)^{p/q}\Big]
\end{array}
\label{an6-a}%
\end{equation}
and%
\begin{align}
\mathbb{E}^{\mathcal{F}_{\sigma}}\sup\nolimits_{\tau\in\left[  \sigma
,\theta_{n}\right]  }\left\vert Y_{\tau}\right\vert ^{p}  &  \leq
C_{p,q,a,\gamma}^{\prime}{{\Large \,}}\bigg[\mathbb{E}^{\mathcal{F}_{\sigma}%
}\left\vert Y_{\theta_{n}}\right\vert ^{p}+\mathbb{E}^{\mathcal{F}_{\sigma}%
}\Big(%
{\displaystyle\int_{\sigma}^{\theta_{n}}}
\left\vert Y_{r}\right\vert ^{q-2}{{\Large \,}}\mathbf{1}_{Y_{r}\neq
0}{{\Large \,}}\mathbf{1}_{q\geq2}{{\Large \,}}dR_{r}\Big)^{p/q}%
\medskip\label{an7}\\
&  \quad+\mathbb{E}^{\mathcal{F}_{\sigma}}\Big(%
{\displaystyle\int_{\sigma}^{\theta_{n}}}
\left\vert Y_{r}\right\vert ^{q-1}dN_{r}\Big)^{p/q}+\mathbb{E}^{\mathcal{F}%
_{\sigma}}\sup\nolimits_{\tau\in\left[  \sigma,\theta_{n}\right]  }\left\vert
M_{\theta_{n}}-M_{\tau}\right\vert ^{p/q}\bigg]\medskip\nonumber\\
&  \leq C_{p,q,a,\gamma}^{\prime\prime}{{\Large \,}}\bigg[\mathbb{E}%
^{\mathcal{F}_{\sigma}}\left\vert Y_{\theta_{n}}\right\vert ^{p}%
+\mathbb{E}^{\mathcal{F}_{\sigma}}\Big(%
{\displaystyle\int_{\sigma}^{\theta_{n}}}
\left\vert Y_{r}\right\vert ^{q-2}{{\Large \,}}\mathbf{1}_{Y_{r}\neq
0}{{\Large \,}}\mathbf{1}_{q\geq2}{{\Large \,}}dR_{r}\Big)^{p/q}%
\medskip\nonumber\\
&  \quad+\mathbb{E}^{\mathcal{F}_{\sigma}}\Big(%
{\displaystyle\int_{\sigma}^{\theta_{n}}}
\left\vert Y_{r}\right\vert ^{q-1}dN_{r}\Big)^{p/q}+\mathbb{\mathbb{E}%
^{\mathcal{F}_{\sigma}}}\Big(%
{\displaystyle\int_{\sigma}^{\theta_{n}}}
\left\vert Y_{r}\right\vert ^{2q-2}\left\vert Z_{r}\right\vert ^{2}%
dr\Big)^{p/\left(  2q\right)  }\bigg].\nonumber
\end{align}
On the other hand, from (\ref{an6-a}),%
\begin{align*}
\displaystyle  &  C_{p,q,a,\gamma}^{\prime\prime}\,\mathbb{\mathbb{E}%
^{\mathcal{F}_{\sigma}}}\Big(\int_{\sigma}^{\theta_{n}}\left\vert
Y_{r}\right\vert ^{2q-2}\left\vert Z_{r}\right\vert ^{2}dr\Big)^{p/\left(
2q\right)  }\\
\displaystyle  &  \leq C_{p,q,a,\gamma}^{\prime\prime}\,\mathbb{\mathbb{E}%
^{\mathcal{F}_{\sigma}}}\Big[\sup\nolimits_{r\in\left[  \sigma,\theta
_{n}\right]  }\left\vert Y_{r}\right\vert ^{p/2}\Big(\int_{\sigma}^{\theta
_{n}}\mathbf{1}_{Y_{r}\neq0}{\Large \,}\left\vert Y_{r}\right\vert
^{q-2}\left\vert Z_{r}\right\vert ^{2}dr\Big)^{p/\left(  2q\right)  }\Big]\\
\displaystyle  &  \leq\frac{1}{2}\,\mathbb{\mathbb{E}^{\mathcal{F}_{\sigma}}%
}\sup\nolimits_{r\in\left[  \sigma,\theta_{n}\right]  }\left\vert
Y_{r}\right\vert ^{p}+\frac{\left(  C_{p,q,a,\gamma}^{\prime\prime}\right)
^{2}}{2}\,\mathbb{\mathbb{E}^{\mathcal{F}_{\sigma}}}\Big(\int_{\sigma}%
^{\theta_{n}}\mathbf{1}_{Y_{r}\neq0}{\Large \,}\left\vert Y_{r}\right\vert
^{q-2}\left\vert Z_{r}\right\vert ^{2}dr\Big)^{p/q}\\
\displaystyle  &  \leq\frac{1}{2}\,\mathbb{\mathbb{E}^{\mathcal{F}_{\sigma}}%
}\sup\nolimits_{r\in\left[  \sigma,\theta_{n}\right]  }\left\vert
Y_{r}\right\vert ^{p}+C_{p,q,a,\gamma}^{\prime\prime\prime}\,\mathbb{E}%
^{\mathcal{F}_{\tau}}\bigg[\left\vert Y_{\theta_{n}}\right\vert ^{p}+\Big(%
{\displaystyle\int_{\sigma}^{\theta_{n}}}
\left(  \left\vert Y_{r}\right\vert ^{q-2}\,\mathbf{1}_{Y_{r}\neq
0}\,\mathbf{1}_{q\geq2}\,dR_{r}\right)  \Big)^{p/q}\bigg]\\
&  \displaystyle\quad+C_{p,q,a,\gamma}^{\prime\prime\prime}\,\mathbb{E}%
^{\mathcal{F}_{\tau}}\Big(%
{\displaystyle\int_{\sigma}^{\theta_{n}}}
\left\vert Y_{r}\right\vert ^{q-1}dN_{r}\Big)^{p/q}.
\end{align*}
Using this last inequality in (\ref{an7}) we obtain%
\begin{equation}%
\begin{array}
[c]{l}%
\displaystyle\mathbb{E}^{\mathcal{F}_{\sigma}}\sup\nolimits_{\tau\in\left[
\sigma,\theta_{n}\right]  }\left\vert Y_{\tau}\right\vert ^{p}\medskip\\
\displaystyle\leq C_{p,q,a,\gamma}\,\mathbb{E}^{\mathcal{F}_{\sigma}%
}\bigg[\left\vert Y_{\theta_{n}}\right\vert ^{p}+\Big(%
{\displaystyle\int_{\sigma}^{\theta_{n}}}
\left\vert Y_{r}\right\vert ^{q-2}\,\mathbf{1}_{Y_{r}\neq0}\,\mathbf{1}%
_{q\geq2}\,dR_{r}\Big)^{p/q}+\Big(%
{\displaystyle\int_{\sigma}^{\theta_{n}}}
\left\vert Y_{r}\right\vert ^{q-1}dN_{r}\Big)^{p/q}\bigg].
\end{array}
\label{an8}%
\end{equation}
Now, by Young's inequality,%
\begin{align*}
&  C_{p,q,a,\gamma}\,\mathbb{E}^{\mathcal{F}_{\sigma}}\bigg[\Big(%
{\displaystyle\int_{\sigma}^{\theta_{n}}}
\left\vert Y_{r}\right\vert ^{q-2}\,\mathbf{1}_{Y_{r}\neq0}\,\mathbf{1}%
_{q\geq2}\,dR_{r}\Big)^{p/q}+\Big(%
{\displaystyle\int_{\sigma}^{\theta_{n}}}
\left\vert Y_{r}\right\vert ^{q-1}dN_{r}\Big)^{p/q}\bigg]\\
&  \leq C_{a,\gamma}\,\mathbb{E}^{\mathcal{F}_{\sigma}}\bigg[\Big(\sup
\nolimits_{r\in\left[  \sigma,\theta_{n}\right]  }\left(  \left\vert
Y_{r}\,\mathbf{1}_{Y_{r}\neq0}\right\vert ^{q-2}\,\mathbf{1}_{q\geq2}\right)
{\displaystyle\int_{\sigma}^{\theta_{n}}}
\mathbf{1}_{q\geq2}\,dR_{r}\Big)^{p/q}+\Big(\sup\nolimits_{r\in\left[
\sigma,\theta_{n}\right]  }\left\vert Y_{r}\right\vert ^{q-1}%
{\displaystyle\int_{\sigma}^{\theta_{n}}}
dN_{r}\Big)^{p/q}\bigg]\\
&  \leq\frac{1}{2}\,\mathbb{E}^{\mathcal{F}_{\sigma}}\sup\nolimits_{\tau
\in\left[  \sigma,\theta_{n}\right]  }\left\vert Y_{\tau}\right\vert ^{p}%
+\hat{C}_{p,q,a,\gamma}\,\mathbb{E}^{\mathcal{F}_{\sigma}}\Big(%
{\displaystyle\int_{\sigma}^{\theta_{n}}}
\mathbf{1}_{q\geq2}\,dR_{r}\Big)^{p/2}+\hat{C}_{p,q,a,\gamma}\,\mathbb{E}%
^{\mathcal{F}_{\sigma}}\Big(%
{\displaystyle\int_{\sigma}^{\theta_{n}}}
dN_{r}\Big)^{p},
\end{align*}
which yields, via (\ref{an8}),%
\begin{equation}
\mathbb{E}^{\mathcal{F}_{\sigma}}\sup\nolimits_{\tau\in\left[  \sigma
,\theta_{n}\right]  }\left\vert Y_{\tau}\right\vert ^{p}\leq C_{p,q,a,\gamma
}\,\mathbb{E}^{\mathcal{F}_{\sigma}}\bigg[\left\vert Y_{\theta_{n}}\right\vert
^{p}+\Big(%
{\displaystyle\int_{\sigma}^{\theta_{n}}}
\mathbf{1}_{q\geq2}\,dR_{r}\Big)^{p/2}+\Big(%
{\displaystyle\int_{\sigma}^{\theta_{n}}}
dN_{r}\Big)^{p}\bigg]. \label{an9}%
\end{equation}
Hence, form the last two inequalities, we deduce%
\begin{equation}%
\begin{array}
[c]{l}%
\mathbb{E}^{\mathcal{F}_{\sigma}}\bigg[\Big(%
{\displaystyle\int_{\sigma}^{\theta_{n}}}
\left\vert Y_{r}\right\vert ^{q-2}\,\mathbf{1}_{Y_{r}\neq0}\,\mathbf{1}%
_{q\geq2}\,dR_{r}\Big)^{p/q}+\Big(%
{\displaystyle\int_{\sigma}^{\theta_{n}}}
\left\vert Y_{r}\right\vert ^{q-1}dN_{r}\Big)^{p/q}\bigg]\medskip\\
\leq\tilde{C}_{p,q,a,\gamma}\,\mathbb{E}^{\mathcal{F}_{\sigma}}%
\bigg[\left\vert Y_{\theta_{n}}\right\vert ^{p}+\Big(%
{\displaystyle\int_{\sigma}^{\theta_{n}}}
\mathbf{1}_{q\geq2}\,dR_{r}\Big)^{p/2}+\Big(%
{\displaystyle\int_{\sigma}^{\theta_{n}}}
dN_{r}\Big)^{p}\bigg].
\end{array}
\label{an9-a}%
\end{equation}
By Beppo Levi monotone convergence theorem and by Lebesgue dominated
convergence theorem we deduce, from (\ref{an8}), (\ref{an9}) and
(\ref{an9-a}), inequalities (\ref{an5-a}) and (\ref{an5-b}).\hfill
\end{proof}

\begin{remark}
\label{an-prop-ydz_2}Passing to the limit in (\ref{an6-a}) and (\ref{an8}), as
$n\rightarrow\infty,$ we deduce (using Beppo Levi's monotone convergence
theorem and condition (\ref{an6-ab}) and Lebesgue dominated convergence
theorem) that the following inequality holds%
\begin{equation}%
\begin{array}
[c]{l}%
\displaystyle\mathbb{E}^{\mathcal{F}_{\sigma}}\Big[\sup\nolimits_{r\in\left[
\sigma,\theta\right]  }\left\vert Y_{r}\right\vert ^{p}+\Big(%
{\displaystyle\int_{\sigma}^{\theta}}
\left\vert Y_{r}\right\vert ^{q-2}{\Large \,}\mathbf{1}_{Y_{r}\neq0}%
{\,}\left\vert Z_{r}\right\vert ^{2}dr\Big)^{p/q}+\Big(%
{\displaystyle\int_{\sigma}^{\theta}}
{\left\vert Y_{r}\right\vert ^{q-2}\,{\mathbf{1}_{Y_{r}\neq0}}\,dD}%
_{r}\Big)^{p/q}\Big]\medskip\\
\displaystyle\leq C_{p,q,a,\gamma}\,\mathbb{E}^{\mathcal{F}_{\sigma}%
}\Big[\left\vert Y_{\theta}\right\vert ^{p}+\Big(%
{\displaystyle\int_{\sigma}^{\theta}}
\left\vert Y_{r}\right\vert ^{q-2}\mathbf{1}_{Y_{r}\neq0}\mathbf{1}_{q\geq
2}dR_{r}\Big)^{p/q}+\Big(%
{\displaystyle\int_{\sigma}^{\theta}}
\left\vert Y_{r}\right\vert ^{q-1}dN_{r}\Big)^{p/q}\Big],\quad\mathbb{P}%
\text{--a.s..}%
\end{array}
\label{an5-a}%
\end{equation}
Moreover (using again the same Beppo Levi theorem and Lebesgue theorem) we can
see that inequalities (\ref{an5-b}) and (\ref{an5-a}) are also true in the
case $0\leq\sigma\leq\theta\leq\infty.$

\end{remark}

\begin{proposition}
\label{an-prop-yd}Let $Y$ be a continuous stochastic process. Let $q\geq1$ and
$b,L\geq0$ such that%
\[
\mathbb{E}\sup\nolimits_{r\in\left[  0,T\right]  }\left\vert Y_{r}\right\vert
^{q}\leq L<\infty
\]
and, for all $0\leq t\leq T<\infty,$%
\[
\mathbb{E}^{\mathcal{F}_{t}}\bigg(\left\vert Y_{t}\right\vert ^{q}+%
{\displaystyle\int_{t}^{T}}
{dD}_{r}\bigg)\leq b\,\mathbb{E}^{\mathcal{F}_{t}}\bigg[\left\vert
Y_{T}\right\vert ^{q}+%
{\displaystyle\int_{0}^{T}}
\left(  \left\vert Y_{r}\right\vert ^{q-2}\,\mathbf{1}_{Y_{r}\neq
0}\,\mathbf{1}_{q\geq2}\,dR_{r}+\left\vert Y_{r}\right\vert ^{q-1}%
dN_{r}\right)  \bigg],\quad\mathbb{P}\text{--a.s.,}%
\]
where $R,N$ and $D$ are increasing and continuous p.m.s.p. $R_{0}=N_{0}%
=D_{0}=0.\medskip$

Then, for any $0<\alpha<1,$%
\begin{equation}%
\begin{array}
[c]{l}%
\displaystyle\mathbb{E}\sup\nolimits_{t\in\left[  0,T\right]  }\left\vert
Y_{t}\right\vert ^{\alpha q}+\mathbb{E}\Big(\int_{0}^{T}{dD}_{r}\Big)^{\alpha
}\medskip\\
\displaystyle\leq\dfrac{2b^{\alpha}}{1-\alpha}\,\bigg[\left(  \mathbb{E}%
\left\vert Y_{T}\right\vert ^{q}\right)  ^{\alpha}+L^{\frac{\alpha\left(
q-2\right)  }{q}}\bigg(\mathbb{E}\Big(%
{\displaystyle\int_{0}^{T}}
\,\mathbf{1}_{q\geq2}dR_{r}\Big)^{\frac{q}{2}}\bigg)^{\frac{2\alpha}{q}%
}+L^{\frac{\alpha\left(  q-1\right)  }{q}}\bigg(\mathbb{E}\Big(%
{\displaystyle\int_{0}^{T}}
\,dN_{r}\Big)^{q}\bigg)^{\frac{\alpha}{q}}\bigg]
\end{array}
\label{an-aa0a}%
\end{equation}
and, for any $\varepsilon,\delta>0,$ there exists $C_{\alpha,q,b,\varepsilon
,\delta}>0$ such that%
\begin{equation}%
\begin{array}
[c]{l}%
\displaystyle\mathbb{E}\sup\nolimits_{t\in\left[  0,T\right]  }\left\vert
Y_{t}\right\vert ^{\alpha q}+\mathbb{E}\Big(\int_{0}^{T}{dD}_{r}\Big)^{\alpha
}\medskip\\
\displaystyle\leq C_{\alpha,q,b,\varepsilon,\delta}\,\bigg[\left(
\mathbb{E}\left\vert Y_{T}\right\vert ^{q}\right)  ^{\alpha}+\bigg(\mathbb{E}%
\Big(%
{\displaystyle\int_{0}^{T}}
\,\mathbf{1}_{q\geq2}dR_{r}\Big)^{\frac{q}{2}+\varepsilon}\bigg)^{\frac{\alpha
q}{q+2\varepsilon}}+\bigg(\mathbb{E}\Big(%
{\displaystyle\int_{0}^{T}}
\,dN_{r}\Big)^{q+\delta}\bigg)^{\frac{\alpha q}{q+\delta}}\bigg].
\end{array}
\label{an-aa0b}%
\end{equation}

\end{proposition}

\begin{proof}
By \cite[Proposition 1.56, $(A_{3})$]{pa-ra/14}, the conclusions clearly hold
true in the case $q=1.\medskip$

Let $q>1$ and $0<\alpha<1.$

We remark first that%
\[%
\begin{array}
[c]{l}%
\displaystyle\mathbb{E}\sup\nolimits_{t\in\left[  0,T\right]  }\left\vert
Y_{t}\right\vert ^{\alpha q}\leq\mathbb{E}\sup\nolimits_{t\in\left[
0,T\right]  }\Big(\left\vert Y_{t}\right\vert ^{q}+%
{\displaystyle\int_{t}^{T}}
{dD}_{r}\Big)^{\alpha},\medskip\\
\displaystyle\mathbb{E}\Big(%
{\displaystyle\int_{0}^{T}}
{dD}_{r}\Big)^{\alpha}\leq\mathbb{E}\sup\nolimits_{t\in\left[  0,T\right]
}\Big(\left\vert Y_{t}\right\vert ^{q}+%
{\displaystyle\int_{t}^{T}}
{dD}_{r}\Big)^{\alpha},
\end{array}
\]
hence%
\begin{equation}
\mathbb{E}\sup\nolimits_{t\in\left[  0,T\right]  }\left\vert Y_{t}\right\vert
^{\alpha q}+\mathbb{E}\Big(%
{\displaystyle\int_{0}^{T}}
{dD}_{r}\Big)^{\alpha}\leq2\,\mathbb{E}\sup\nolimits_{t\in\left[  0,T\right]
}\Big(\left\vert Y_{t}\right\vert ^{q}+%
{\displaystyle\int_{t}^{T}}
{dD}_{r}\Big)^{\alpha}. \label{prop_4.5_1}%
\end{equation}
By \cite[Proposition 1.56, $(A_{3})$]{pa-ra/14} we have%
\begin{equation}%
\begin{array}
[c]{l}%
\displaystyle\mathbb{E}\sup\nolimits_{t\in\left[  0,T\right]  }\Big(\left\vert
Y_{t}\right\vert ^{q}+%
{\displaystyle\int_{t}^{T}}
{dD}_{r}\Big)^{\alpha}\medskip\\
\displaystyle\leq\frac{b^{\alpha}}{1-\alpha}\,\bigg[\mathbb{E}\left\vert
Y_{T}\right\vert ^{q}+\mathbb{E}%
{\displaystyle\int_{0}^{T}}
\left\vert Y_{r}\right\vert ^{q-2}\,\mathbf{1}_{Y_{r}\neq0}\,\mathbf{1}%
_{q\geq2}\,dR_{r}+\mathbb{E}%
{\displaystyle\int_{0}^{T}}
\left\vert Y_{r}\right\vert ^{q-1}dN_{r}\bigg]^{\alpha}\medskip\\
\displaystyle\leq\frac{b^{\alpha}}{1-\alpha}\,\big(\mathbb{E}\left\vert
Y_{T}\right\vert ^{q}\big)^{\alpha}+A^{\alpha}+B^{\alpha},
\end{array}
\label{prop_4.5}%
\end{equation}
where%
\begin{align*}
A^{\alpha}  &  =\frac{b^{\alpha}}{1-\alpha}\,\bigg(\mathbb{E}\Big(\sup
\nolimits_{r\in\left[  0,T\right]  }\left(  \left\vert Y_{r}\,\mathbf{1}%
_{Y_{r}\neq0}\right\vert ^{q-2}\,\mathbf{1}_{q\geq2}\right)  \cdot%
{\displaystyle\int_{0}^{T}}
\mathbf{1}_{q\geq2}\,dR_{r}\Big)\bigg)^{\alpha},\\
B^{\alpha}  &  =\frac{b^{\alpha}}{1-\alpha}\,\bigg(\mathbb{E}\Big(\sup
\nolimits_{r\in\left[  0,T\right]  }\left\vert Y_{r}\right\vert ^{q-1}%
{\displaystyle\int_{0}^{T}}
dN_{r}\Big)\bigg)^{\alpha}.
\end{align*}
By H\"{o}lder's inequality with $\beta=\frac{q}{q-1}>1$ and $\beta^{\prime}=q$
we have%
\begin{equation}%
\begin{array}
[c]{l}%
B^{\alpha}\leq\dfrac{b^{\alpha}}{1-\alpha}\,\left(  \mathbb{E}\sup
\nolimits_{r\in\left[  0,T\right]  }\left\vert Y_{r}\,\right\vert
^{\beta\left(  q-1\right)  }\,\right)  ^{\alpha/\beta}\cdot\bigg(\mathbb{E}%
\Big(%
{\displaystyle\int_{0}^{T}}
\,dN_{r}\Big)^{\beta^{\prime}}\bigg)^{\alpha/\beta^{\prime}}\medskip\\
\leq\dfrac{b^{\alpha}}{1-\alpha}\,L^{\alpha\left(  q-1\right)  /q}%
\,\bigg(\mathbb{E}\Big(%
{\displaystyle\int_{0}^{T}}
\,dN_{r}\Big)^{q}\bigg)^{\alpha/q}.
\end{array}
\label{an-aa1}%
\end{equation}
In addition, again by H\"{o}lder's inequality and Young's inequality, with
$\beta=\frac{\alpha q}{q-1}>1,$ $\beta^{\prime}=\frac{\alpha}{\alpha
-(q-1)/q}\,,$ $\gamma=\frac{\beta}{\alpha}=\frac{q}{q-1}>1$ and $\gamma
^{\prime}=q,$ we obtain the next inequality for any $\alpha$ such that
$\frac{q-1}{q}<\alpha<1:$%
\begin{equation}%
\begin{array}
[c]{l}%
B^{\alpha}\leq\dfrac{1}{8}\left(  \mathbb{E}\sup\nolimits_{r\in\left[
0,T\right]  }\left\vert Y_{r}\,\right\vert ^{\beta\left(  q-1\right)
}\,\right)  ^{\alpha\gamma/\beta}+C_{\alpha,\beta,\gamma,b}\bigg(\mathbb{E}%
\Big(%
{\displaystyle\int_{0}^{T}}
\,dN_{r}\Big)^{\beta^{\prime}}\bigg)^{\alpha\gamma^{\prime}/\beta^{\prime}%
}\medskip\\
=\dfrac{1}{8}\,\mathbb{E}\sup\nolimits_{r\in\left[  0,T\right]  }\left\vert
Y_{r}\,\right\vert ^{\alpha q}\,+C_{\alpha,q,b}\bigg(\mathbb{E}\Big(%
{\displaystyle\int_{0}^{T}}
\,dN_{r}\Big)^{\alpha q/\left(  \alpha q-q+1\right)  }\bigg)^{\alpha q-q+1}%
\end{array}
\label{an-aa2}%
\end{equation}
Of course,%
\begin{equation}
A^{\alpha}=\left\{
\begin{array}
[c]{ll}%
0, & \text{if }1<q<2,\medskip\\
\dfrac{b^{\alpha}}{1-\alpha}\,\bigg(\mathbb{E}%
{\displaystyle\int_{0}^{T}}
dR_{r}\bigg)^{\alpha}, & \text{if }q=2.
\end{array}
\right.  \label{an-aa2a}%
\end{equation}
If $q>2$, by H\"{o}lder's inequality with $\beta=\frac{q}{q-2}>1$ and
$\beta^{\prime}=\frac{q}{2}\,,$ we have%
\begin{equation}%
\begin{array}
[c]{l}%
A^{\alpha}\leq\dfrac{b^{\alpha}}{1-\alpha}\,\left(  \mathbb{E}\sup
\nolimits_{r\in\left[  0,T\right]  }\left\vert Y_{r}\,\right\vert
^{\beta\left(  q-2\right)  }\,\right)  ^{\alpha/\beta}\cdot\bigg(\mathbb{E}%
\Big(%
{\displaystyle\int_{0}^{T}}
\,dR_{r}\Big)^{\beta^{\prime}}\bigg)^{\alpha/\beta^{\prime}}\medskip\\
\leq\dfrac{b^{\alpha}}{1-\alpha}\,L^{\alpha\left(  q-2\right)  /q}%
\bigg(\mathbb{E}\Big(%
{\displaystyle\int_{0}^{T}}
\,dR_{r}\Big)^{q/2}\bigg)^{2\alpha/q}.
\end{array}
\label{an-aa3}%
\end{equation}
We see that inequality (\ref{an-aa3}) it is satisfied also in the case
$1<q\leq2.\medskip$

In addition, again by H\"{o}lder's inequality and Young's inequality, with
$\beta=\frac{\alpha q}{q-2}>1,$ $\beta^{\prime}=\frac{\alpha}{\alpha
-(q-2)/q}\,,$ $\gamma=\frac{\beta}{\alpha}=\frac{q}{q-2}>1$ and $\gamma
^{\prime}=\frac{q}{2}\,,$ we obtain the next inequality for any $\alpha$ such
that $\frac{q-2}{q}<\alpha<1:$%
\begin{equation}%
\begin{array}
[c]{l}%
A^{\alpha}\leq\dfrac{1}{8}\,\left(  \mathbb{E}\sup\nolimits_{r\in\left[
0,T\right]  }\left\vert Y_{r}\,\right\vert ^{\beta\left(  q-2\right)
}\,\right)  ^{\alpha\gamma/\beta}+C_{\alpha,\beta,\gamma,b}\bigg(\mathbb{E}%
\Big(%
{\displaystyle\int_{0}^{T}}
\,dR_{r}\Big)^{\beta^{\prime}}\bigg)^{\alpha\gamma^{\prime}/\beta^{\prime}%
}\medskip\\
=\dfrac{1}{8}\,\mathbb{E}\sup\nolimits_{r\in\left[  0,T\right]  }%
\big(\left\vert Y_{r}\,\right\vert ^{\alpha q}\,\big)+C_{\alpha,q,b}%
\bigg(\mathbb{E}\Big(%
{\displaystyle\int_{0}^{T}}
\,dR_{r}\Big)^{\alpha q/\left(  \alpha q-q+2\right)  }\bigg)^{\left(  \alpha
q-q+2\right)  /2}.
\end{array}
\label{an-aa4}%
\end{equation}
We see that inequality (\ref{an-aa4}) it is satisfied also in the case
$1<q\leq2$ and for any $\alpha$ such that $0<\alpha<1.\medskip$

Now it is clear that inequality (\ref{an-aa0a}) follows from inequalities
(\ref{prop_4.5_1}), (\ref{prop_4.5}), (\ref{an-aa1}), (\ref{an-aa2a}) and
(\ref{an-aa3}).$\medskip$

On the other hand, from inequalities (\ref{prop_4.5_1}), (\ref{prop_4.5}),
(\ref{an-aa2}), (\ref{an-aa2a}) and (\ref{an-aa4}) we deduce that, for any
$\alpha$ such that $\frac{q-1}{q}<\alpha<1,$%
\[%
\begin{array}
[c]{l}%
\displaystyle\mathbb{E}\sup\nolimits_{t\in\left[  0,T\right]  }\left\vert
Y_{t}\right\vert ^{\alpha q}+\mathbb{E}\Big(\int_{0}^{T}{dD}_{r}\Big)^{\alpha
}\medskip\\
\displaystyle\leq C_{\alpha,q,b}\,\bigg[\left(  \mathbb{E}\left\vert
Y_{T}\right\vert ^{q}\right)  ^{\alpha}+\bigg(\mathbb{E}\Big(%
{\displaystyle\int_{0}^{T}}
\,\mathbf{1}_{q\geq2}dR_{r}\Big)^{\frac{\alpha q}{\alpha q-q+2}}%
\bigg)^{\frac{\alpha q-q+2}{2}}+\bigg(\mathbb{E}\Big(%
{\displaystyle\int_{0}^{T}}
\,dN_{r}\Big)^{\frac{\alpha q}{\alpha q-q+1}}\bigg)^{\alpha q-q+1}\bigg].
\end{array}
\]
If $0<\alpha<1$ is arbitrary fixed, then the last inequality hold also for
$\alpha$ replaced by any $\bar{\alpha}$ such that%
\begin{equation}
\alpha\vee\frac{\mathbf{1}_{q\geq2}\,\left(  q-2\right)  \left(
q+2\varepsilon\right)  }{q\left(  q+2\varepsilon-2\right)  }\vee\frac{\left(
q-1\right)  \left(  q+\delta\right)  }{q\left(  q+\delta-1\right)  }%
<\bar{\alpha}<1. \label{alpha_bar}%
\end{equation}
By H\"{o}lder's inequality we have%
\begin{align*}
&  \mathbb{E}\sup\nolimits_{t\in\left[  0,T\right]  }\left\vert Y_{t}%
\right\vert ^{\alpha q}+\mathbb{E}\Big(\int_{0}^{T}{dD}_{r}\Big)^{\alpha}%
\leq\bigg(\mathbb{E}\sup\nolimits_{t\in\left[  0,T\right]  }\left\vert
Y_{t}\right\vert ^{\bar{\alpha}q}\bigg)^{\alpha/\bar{\alpha}}+\bigg(\mathbb{E}%
\Big(\int_{0}^{T}{dD}_{r}\Big)^{\bar{\alpha}}\bigg)^{\frac{\alpha}{\bar
{\alpha}}}\\
&  \leq2\left[  \mathbb{E}\sup\nolimits_{t\in\left[  0,T\right]  }\left\vert
Y_{t}\right\vert ^{\bar{\alpha}q}+\mathbb{E}\Big(\int_{0}^{T}{dD}%
_{r}\Big)^{\bar{\alpha}}\right]  ^{\frac{\alpha}{\bar{\alpha}}}\\
&  \leq C_{\bar{\alpha},q,b}\,\bigg[\left(  \mathbb{E}\left\vert
Y_{T}\right\vert ^{q}\right)  ^{\bar{\alpha}}+\bigg(\mathbb{E}\Big(%
{\displaystyle\int_{0}^{T}}
\,\mathbf{1}_{q\geq2}dR_{r}\Big)^{\frac{\bar{\alpha}q}{\bar{\alpha}q-q+2}%
}\bigg)^{\frac{\bar{\alpha}q-q+2}{2}}+\bigg(\mathbb{E}\Big(%
{\displaystyle\int_{0}^{T}}
\,dN_{r}\Big)^{\frac{\bar{\alpha}q}{\bar{\alpha}q-q+1}}\bigg)^{\bar{\alpha
}q-q+1}\bigg]^{\frac{\alpha}{\bar{\alpha}}}\\
&  \leq C_{\bar{\alpha},q,b}\bigg[\left(  \mathbb{E}\left\vert Y_{T}%
\right\vert ^{q}\right)  ^{\alpha}+\bigg(\mathbb{E}\Big(%
{\displaystyle\int_{0}^{T}}
\,\mathbf{1}_{q\geq2}dR_{r}\Big)^{\frac{\bar{\alpha}q}{\bar{\alpha}q-q+2}%
}\bigg)^{\frac{\bar{\alpha}q-q+2}{\bar{\alpha}q}\cdot\frac{\alpha q}{2}%
}+\bigg(\mathbb{E}\Big(%
{\displaystyle\int_{0}^{T}}
\,dN_{r}\Big)^{\frac{\bar{\alpha}q}{\bar{\alpha}q-q+1}}\bigg)^{\frac
{\bar{\alpha}q-q+1}{\bar{\alpha}q}\cdot\alpha q}\bigg]
\end{align*}
Using (\ref{alpha_bar}) we obtain%
\[
\frac{\bar{\alpha}q}{\bar{\alpha}q-q+2}\leq\frac{q}{2}+\varepsilon
\quad\text{and}\quad\frac{\bar{\alpha}q}{\bar{\alpha}q-q+1}\leq q+\delta
\]
and, by H\"{o}lder's inequality,%
\[%
\begin{array}
[c]{l}%
\displaystyle\bigg(\mathbb{E}\Big(%
{\displaystyle\int_{0}^{T}}
\mathbf{1}_{q\geq2}dR_{r}\Big)^{\frac{\bar{\alpha}q}{\bar{\alpha}q-q+2}%
}\bigg)^{\frac{\bar{\alpha}q-q+2}{\bar{\alpha}q}\,\frac{\alpha q}{2}}%
\leq\bigg(\mathbb{E}\Big(%
{\displaystyle\int_{0}^{T}}
\,\mathbf{1}_{q\geq2}dR_{r}\Big)^{\frac{q}{2}+\varepsilon}\bigg)^{\frac{\alpha
q}{q+2\varepsilon}}\medskip\\
\displaystyle\bigg(\mathbb{E}\Big(%
{\displaystyle\int_{0}^{T}}
dN_{r}\Big)^{\frac{\bar{\alpha}q}{\bar{\alpha}q-q+1}}\bigg)^{\frac{\bar
{\alpha}q-q+1}{\bar{\alpha}q}\,\alpha q}\leq\bigg(\mathbb{E}\Big(%
{\displaystyle\int_{0}^{T}}
\,dN_{r}\Big)^{q+\delta}\bigg)^{\frac{\alpha q}{q+\delta}}\,.
\end{array}
\]
Consequently inequality (\ref{an-aa0b}) holds for any $0<\alpha<1.$\hfill
\end{proof}

\begin{proposition}
\label{an-prop-dk} Let:

\begin{itemize}
\item $\left(  Y,Z\right)  \in S_{m}^{0}\times\Lambda_{m\times k}^{0}$ ;

\item $K\in S_{m}^{0}$ and $K_{\cdot}\in\mathrm{BV}_{\mathrm{loc}}\left(
\mathbb{R}_{+};\mathbb{R}^{m}\right)  ,$ $\mathbb{P}$--a.s.;

\item $D,R,N,\tilde{R}$ be some increasing continuous p.m.s.p. with
$D_{0}=R_{0}=N_{0}=0$;

\item $V$ be a bounded variation p.m.s.p. with $V_{0}=0;$

\item $\sigma$ and $\theta$ be two stopping times such that $0\leq\sigma
\leq\theta<\infty.\medskip$
\end{itemize}

\noindent\textbf{I.} If for all $0\leq t\leq s<\infty,\;\mathbb{P}$-a.s.%
\begin{equation}
\left\vert Y_{t}\right\vert ^{2}+{%
{\displaystyle\int_{t}^{s}}
}\left\vert Z_{r}\right\vert ^{2}dr+{%
{\displaystyle\int_{t}^{s}}
}dD_{r}\leq\left\vert Y_{s}\right\vert ^{2}+2{%
{\displaystyle\int_{t}^{s}}
}\left\langle Y_{r},dK_{r}\right\rangle -2%
{\displaystyle\int_{t}^{s}}
\langle Y_{r},Z_{r}dB_{r}\rangle, \label{an-10-a}%
\end{equation}
and for some $\lambda<1$%
\begin{equation}
{%
{\displaystyle\int_{t}^{s}}
}\left\langle Y_{r},dK_{r}\right\rangle \leq%
{\displaystyle\int_{t}^{s}}
\left(  dR_{r}+|Y_{r}|dN_{r}+|Y_{r}|^{2}dV_{r}\right)  +\dfrac{\lambda}{2}\,%
{\displaystyle\int_{t}^{s}}
\left\vert Z_{r}\right\vert ^{2}dr, \label{an-10-b}%
\end{equation}
then, for any $q>0$, there exists a positive constant $C_{q,\lambda}$ such
that $\mathbb{P}$--a.s.%
\begin{equation}%
\begin{array}
[c]{l}%
\mathbb{E}^{\mathcal{F}_{\sigma}}\Big(%
{\displaystyle\int_{\sigma}^{\theta}}
e^{2V_{r}}\left\vert Z_{r}\right\vert ^{2}ds\Big)^{q/2}+\mathbb{E}%
^{\mathcal{F}_{\sigma}}\Big(%
{\displaystyle\int_{\sigma}^{\theta}}
e^{2V_{r}}dD_{r}\Big)^{q/2}\medskip\\
\leq C_{q,\lambda}\mathbb{E}^{\mathcal{F}_{\sigma}}\left[  \sup\nolimits_{r\in
\left[  \sigma,\theta\right]  }\left\vert e^{V_{r}}Y_{r}\right\vert ^{q}+\Big(%
{\displaystyle\int_{\sigma}^{\theta}}
e^{2V_{r}}dR_{r}\Big)^{q/2}+\Big(%
{\displaystyle\int_{\sigma}^{\theta}}
e^{2V_{r}}\left\vert Y_{r}\right\vert dN_{r}\Big)^{q/2}\right]  \medskip\\
\leq2C_{q,\lambda}\mathbb{E}^{\mathcal{F}_{\sigma}}\left[  \sup\nolimits_{r\in
\left[  \sigma,\theta\right]  }\left\vert e^{V_{r}}Y_{r}\right\vert ^{q}+\Big(%
{\displaystyle\int_{\sigma}^{\theta}}
e^{2V_{r}}dR_{r}\Big)^{q/2}+\Big(%
{\displaystyle\int_{\sigma}^{\theta}}
e^{V_{r}}dN_{r}\Big)^{q}\right]  .
\end{array}
\label{an-11}%
\end{equation}
\noindent\textbf{II.} If $q>1,$%
\begin{equation}%
\begin{array}
[c]{rl}%
\left(  i\right)  & \displaystyle\left\vert Y_{t}\right\vert ^{q}+\dfrac{q}%
{2}\,n_{q}\,\int_{t}^{s}\left\vert Y_{r}\right\vert ^{q-2}\,\mathbf{1}%
_{Y_{r}\neq0}\,\left\vert Z_{r}\right\vert ^{2}dr+\int_{t}^{s}\left\vert
Y_{r}\right\vert ^{q-2}\,\mathbf{1}_{Y_{r}\neq0}\,dD_{r}\medskip\\
& \displaystyle\leq\left\vert Y_{s}\right\vert ^{q}+q\int_{t}^{s}\left\vert
Y_{r}\right\vert ^{q-2}\,\mathbf{1}_{Y_{r}\neq0}\,\left[  d\tilde{R}%
_{r}+\left\langle Y_{r},dK_{r}\right\rangle \right]  -q\int_{t}^{s}\left\vert
Y_{r}\right\vert ^{q-2}\,\mathbf{1}_{Y_{r}\neq0}\,\left\langle Y_{r}%
,Z_{r}dB_{r}\right\rangle ,\medskip\\
\left(  ii\right)  & \displaystyle\mathbb{E}\sup\nolimits_{r\in\left[
\sigma,\theta\right]  }e^{qV_{r}}\left\vert Y_{r}\right\vert ^{q}<\infty
\end{array}
\label{an-12}%
\end{equation}
and for some $\lambda<1$%
\begin{equation}
d\tilde{R}_{r}+\left\langle Y_{r},dK_{r}\right\rangle \leq\left(
\mathbf{1}_{q\geq2}\,dR_{r}+|Y_{r}|dN_{r}+|Y_{r}|^{2}dV_{r}\right)
+\dfrac{n_{q}}{2}\,\lambda\,\left\vert Z_{r}\right\vert ^{2}dt, \label{an-13}%
\end{equation}
then there exists some positive constant $C_{q,\lambda},$ $C_{q,\lambda
}^{\prime}$ such that $\mathbb{P}$--a.s.,%
\begin{equation}%
\begin{array}
[c]{l}%
\mathbb{E}^{\mathcal{F}_{\sigma}}\bigg[\sup\nolimits_{\tau\in\left[
\sigma,\theta\right]  }\left\vert e^{V_{r}}Y_{r}\right\vert ^{q}+%
{\displaystyle\int_{\sigma}^{\theta}}
e^{qV_{r}}\left\vert Y_{r}\right\vert ^{q-2}\,\mathbf{1}_{Y_{r}\neq
0}\,\left\vert Z_{r}\right\vert ^{2}ds+%
{\displaystyle\int_{\sigma}^{\theta}}
e^{qV_{r}}\left\vert Y_{r}\right\vert ^{q-2}\,\mathbf{1}_{Y_{r}\neq0}%
\,dD_{r}\bigg]\medskip\\
\leq C_{q,\lambda}\,\mathbb{E}^{\mathcal{F}_{\sigma}}\bigg[\left\vert
e^{V_{\theta}}Y_{\theta}\right\vert ^{q}+\Big(%
{\displaystyle\int_{\sigma}^{\theta}}
e^{qV_{r}}\left\vert Y_{r}\right\vert ^{q-2}\,\mathbf{1}_{Y_{r}\neq
0}\,\mathbf{1}_{q\geq2}\,dR_{r}\Big)+\Big(%
{\displaystyle\int_{\sigma}^{\theta}}
e^{qV_{r}}\left\vert Y_{r}\right\vert ^{q-1}dN_{r}\Big)\bigg]\medskip\\
\leq C_{q,\lambda}^{\prime}\,\mathbb{E}^{\mathcal{F}_{\sigma}}\bigg[\left\vert
e^{V_{\theta}}Y_{\theta}\right\vert ^{q}+\Big(%
{\displaystyle\int_{\sigma}^{\theta}}
e^{2V_{r}}\,\mathbf{1}_{q\geq2}\,dR_{r}\Big)^{q/2}+\Big(%
{\displaystyle\int_{\sigma}^{\theta}}
e^{V_{r}}dN_{r}\Big)^{q}\bigg].
\end{array}
\label{an-14}%
\end{equation}

\end{proposition}

\begin{proof}
Using inequalities (\ref{an-10-a}) and (\ref{an-10-b}) we obtain, for all
$0\leq t\leq s<\infty,$%
\[
\left\vert Y_{t}\right\vert ^{2}+\left(  1-\lambda\right)  {%
{\displaystyle\int_{t}^{s}}
}\left\vert Z_{r}\right\vert ^{2}dr+{%
{\displaystyle\int_{t}^{s}}
}dD_{r}\leq\left\vert Y_{s}\right\vert ^{2}+{%
{\displaystyle\int_{t}^{s}}
}\left[  \left(  2dR_{r}+2|Y_{r}|dN_{r}\right)  +|Y_{r}|^{2}d\left(
2V_{r}\right)  \right]  -2%
{\displaystyle\int_{t}^{s}}
\,\langle Y_{r},Z_{r}dB_{r}\rangle,
\]
which yields, applying \cite[Proposition 6.69]{pa-ra/14} (or \cite[Lemma
12]{ma-ra/07}),%
\begin{align*}
&  \left\vert e^{V_{t}}Y_{t}\right\vert ^{2}+\left(  1-\lambda\right)  {%
{\displaystyle\int_{t}^{s}}
}\left\vert e^{V_{r}}Z_{r}\right\vert ^{2}dr+{%
{\displaystyle\int_{t}^{s}}
}e^{2V_{r}}dD_{r}\\
&  \leq\left\vert e^{V_{s}}Y_{s}\right\vert ^{2}+2{%
{\displaystyle\int_{t}^{s}}
}\left[  e^{2V_{r}}dR_{r}+|e^{V_{r}}Y_{r}|e^{V_{r}}dN_{r}\right]  -2%
{\displaystyle\int_{t}^{s}}
\,\langle e^{V_{r}}Y_{r},e^{V_{r}}Z_{r}dB_{r}\rangle.
\end{align*}
Inequality (\ref{an-11}) follows now by Proposition \ref{an-prop-dz}%
.$\medskip$

In the same manner, using (\ref{an-12}), (\ref{an-13}) and \cite[Proposition
6.69]{pa-ra/14}, we infer%
\[%
\begin{array}
[c]{l}%
\left\vert e^{V_{t}}Y_{t}\right\vert ^{q}+\dfrac{q}{2}\,n_{q}\,\left(
1-\lambda\right)
{\displaystyle\int_{t}^{s}}
\left\vert e^{V_{r}}Y_{r}\right\vert ^{q-2}\,\mathbf{1}_{Y_{r}\neq
0}\,\left\vert e^{V_{r}}Z_{r}\right\vert ^{2}dr+%
{\displaystyle\int_{t}^{s}}
\left\vert e^{V_{r}}Y_{r}\right\vert ^{q-2}\,\mathbf{1}_{Y_{r}\neq
0}\,e^{2V_{r}}dD_{r}\medskip\\
\leq\left\vert e^{V_{s}}Y_{s}\right\vert ^{q}+q%
{\displaystyle\int_{t}^{s}}
\Big[\left\vert e^{V_{r}}Y_{r}\right\vert ^{q-2}\,\mathbf{1}_{Y_{r}\neq
0}\,\mathbf{1}_{q\geq2}\,e^{2V_{r}}dR_{r}+|e^{V_{r}}Y_{r}|^{q-1}e^{V_{r}%
}dN_{r}\Big]\medskip\\
\quad-q%
{\displaystyle\int_{t}^{s}}
\left\vert e^{V_{r}}Y_{r}\right\vert ^{q-2}\,\mathbf{1}_{Y_{r}\neq
0}\,\left\langle e^{V_{r}}Y_{r},e^{V_{r}}Z_{r}dB_{r}\right\rangle .
\end{array}
\]
Inequalities from (\ref{an-14}) follow now by Proposition \ref{an-prop-ydz}
and Remark \ref{an-prop-ydz_2}.\hfill$\medskip$
\end{proof}

With a similar approach we deduce the next results.

\begin{proposition}
[{\cite[Proposition 6.80]{pa-ra/14}}]\label{Appendix_result 1}Let $\left(
Y,Z\right)  \in S_{m}^{0}\times\Lambda_{m\times k}^{0}$ satisfying%
\[
Y_{t}=Y_{T}+\int_{t}^{T}dK_{s}-\int_{t}^{T}Z_{s}dB_{s},\;0\leq t\leq
T,\quad\mathbb{P}\text{--a.s.},
\]
where $K\in S_{m}^{0}$ and $K_{\cdot}\in\mathrm{BV}_{\mathrm{loc}}\left(
\mathbb{R}_{+};\mathbb{R}^{m}\right)  $, $\mathbb{P}$--a.s..

Let $\tau$ and $\sigma$ be two stopping times such that $0\leq\tau\leq
\sigma<\infty$. Assume that there exists three increasing and continuous
p.m.s.p. $D,R,N$ with $D_{0}=R_{0}=N_{0}=0$ and a bounded variation p.m.s.p.
$V$ with $V_{0}=0$ such that for, $\lambda<1,$%
\[
dD_{t}+\left\langle Y_{t},dK_{t}\right\rangle \leq dR_{t}+|Y_{t}|dN_{t}%
+|Y_{t}|^{2}dV_{t}+\dfrac{\lambda}{2}\,\left\vert Z_{t}\right\vert ^{2}dt.
\]
Then, for any $q>0$, there exists a positive constant $C_{q,\lambda}$ such
that $\mathbb{P}$--a.s.%
\[%
\begin{array}
[c]{l}%
\displaystyle\mathbb{E}^{\mathcal{F}_{\tau}}\Big(\int_{\tau}^{\sigma}%
e^{2V_{s}}dD_{s}\Big)^{q/2}+\mathbb{E}^{\mathcal{F}_{\tau}}\Big(\int_{\tau
}^{\sigma}e^{2V_{s}}\left\vert Z_{s}\right\vert ^{2}ds\Big)^{q/2}\medskip\\
\multicolumn{1}{r}{\displaystyle\leq C_{q,\lambda}\,\mathbb{E}^{\mathcal{F}%
_{\tau}}\left[  \sup\nolimits_{s\in\left[  \tau,\sigma\right]  }\left\vert
e^{V_{s}}Y_{s}\right\vert ^{q}+\Big(\int_{\tau}^{\sigma}e^{2V_{s}}%
dR_{s}\Big)^{q/2}+\Big(\int_{\tau}^{\sigma}e^{V_{s}}dN_{s}\Big)^{q}\right]  .}%
\end{array}
\]
Moreover, if $p>1$ and%
\[%
\begin{array}
[c]{l}%
\displaystyle dD_{t}+\left\langle Y_{t},dK_{t}\right\rangle \leq\left(
\mathbf{1}_{p\geq2}dR_{t}+|Y_{t}|dN_{t}+|Y_{t}|^{2}dV_{t}\right)
+\dfrac{n_{p}}{2}\lambda\left\vert Z_{t}\right\vert ^{2}dt,\medskip\\
\displaystyle\mathbb{E}\sup\nolimits_{s\in\left[  \tau,\sigma\right]
}e^{pV_{s}}\left\vert Y_{s}\right\vert ^{p}<\infty,
\end{array}
\]
then there exists a positive constant $C_{p,\lambda}$ such that $\mathbb{P}%
$--a.s.,%
\begin{equation}%
\begin{array}
[c]{l}%
\displaystyle\mathbb{E}^{\mathcal{F}_{\tau}}\Big(\sup\nolimits_{s\in\left[
\tau,\sigma\right]  }\left\vert e^{V_{s}}Y_{s}\right\vert ^{p}\Big)+\mathbb{E}%
^{\mathcal{F}_{\tau}}\Big(\int_{\tau}^{\sigma}e^{2V_{s}}dD_{s}\Big)^{p/2}%
+\mathbb{E}^{\mathcal{F}_{\tau}}\Big(\int_{\tau}^{\sigma}e^{2V_{s}}\left\vert
Z_{s}\right\vert ^{2}ds\Big)^{p/2}\medskip\\
\displaystyle\leq C_{p,\lambda}\,\mathbb{E}^{\mathcal{F}_{\tau}}\left[
\left\vert e^{V_{\sigma}}Y_{\sigma}\right\vert ^{p}+\Big(\int_{\tau}^{\sigma
}e^{2V_{s}}\mathbf{1}_{p\geq2}dR_{s}\Big)^{p/2}+\Big(\int_{\tau}^{\sigma
}e^{V_{s}}dN_{s}\Big)^{p}\right]  .
\end{array}
\label{an3a}%
\end{equation}

\end{proposition}

Based mainly on this previous result one can prove:

\begin{proposition}
[{\cite[Corollary 6.81]{pa-ra/14}}]\label{Appendix_result 2}Let $\left(
Y,Z\right)  \in S_{m}^{0}\times\Lambda_{m\times k}^{0}$ satisfying%
\[
Y_{t}=Y_{T}+\int_{t}^{T}dK_{s}-\int_{t}^{T}Z_{s}dB_{s},\;0\leq t\leq
T,\quad\mathbb{P}\text{--a.s.,}%
\]
where $K\in S_{m}^{0}$ and $K_{\cdot}\in\mathrm{BV}_{\mathrm{loc}}\left(
\mathbb{R}_{+};\mathbb{R}^{m}\right)  ,\;\mathbb{P}$--a.s..

Let $\tau$ and $\sigma$ be two stopping times such that $0\leq\tau\leq
\sigma<\infty$. Assume that there exists two increasing and continuous
p.m.s.p. $D,N$ with $N_{0}=0$ and a bounded variation p.m.s.p. $V$ with
$V_{0}=0$ such that, for $\lambda<1,$%
\[%
\begin{array}
[c]{l}%
dD_{t}+\left\langle Y_{t},dK_{t}\right\rangle \leq|Y_{t}|dN_{t}+|Y_{t}%
|^{2}dV_{t}\,,\medskip\\
\mathbb{E}\sup\nolimits_{s\in\left[  \tau,\sigma\right]  }\left\vert e^{V_{s}%
}Y_{s}\right\vert <\infty.
\end{array}
\]
Then%
\[
e^{V_{\tau}}\left\vert Y_{\tau}\right\vert \leq\mathbb{E}^{\mathcal{F}_{\tau}%
}e^{V_{\sigma}}\left\vert Y_{\sigma}\right\vert +\mathbb{E}^{\mathcal{F}%
_{\tau}}\int_{\tau}^{\sigma}e^{V_{s}}dN_{s}%
\]
and, for all $0<a<1,$%
\[%
\begin{array}
[c]{l}%
\displaystyle\sup\nolimits_{s\in\left[  \tau,\sigma\right]  }\left[
\mathbb{E}\left(  e^{V_{s}}\left\vert Y_{s}\right\vert \right)  \right]
^{a}+\mathbb{E}\Big(\sup\nolimits_{s\in\left[  \tau,\sigma\right]  }\left\vert
e^{V_{s}}Y_{s}\right\vert ^{a}\Big)+\mathbb{E}\Big(\int_{\tau}^{\sigma
}e^{2V_{s}}\left\vert Z_{s}\right\vert ^{2}ds\Big)^{a/2}+\mathbb{E}%
\Big(\int_{\tau}^{\sigma}e^{2V_{s}}dD_{s}\Big)^{a/2}\smallskip\\
\displaystyle\leq C_{a}\,\Big(\mathbb{E}\left(  e^{V_{\sigma}}\left\vert
Y_{\sigma}\right\vert \right)  \Big)^{a}+C_{a}\,\Big(\mathbb{E}\int_{\tau
}^{\sigma}e^{V_{s}}dN_{s}\Big)^{a}%
\end{array}
\]

\end{proposition}

\subsection{Smoothing approximations}

\begin{lemma}
\label{L1-approx}Let $\varepsilon>0$ and let $Q:\Omega\times\mathbb{R}%
_{+}\rightarrow\mathbb{R}_{+}$ be a strictly increasing continuous stochastic
process such that $Q_{0}=0$ and $\lim_{t\rightarrow\infty}Q_{t}=\infty,$ and
let $G:\Omega\times\mathbb{R}_{+}\rightarrow\mathbb{R}^{m}$ be a measurable
stochastic process such that $\sup_{t\in\mathbb{R}_{+}}\left\vert
G_{t}\right\vert <\infty,$ $\mathbb{P}$--a.s..$\smallskip$

Define%
\begin{equation}
G_{t}^{\varepsilon}=\frac{1}{Q_{\varepsilon}}\,%
{\displaystyle\int_{t\vee\varepsilon}^{\infty}}
e^{-\frac{Q_{r}-Q_{t\vee\varepsilon}}{Q_{\varepsilon}}}G_{r}\,dQ_{r}\,.
\label{def_G_epsilon}%
\end{equation}
Then $G^{\varepsilon}:\Omega\times\mathbb{R}_{+}\rightarrow\mathbb{R}^{m}$ are
continuous stochastic processes and, $\mathbb{P}$--a.s.,%
\begin{equation}%
\begin{array}
[c]{ll}%
\left(  a\right)  & \left\vert G_{t}^{\varepsilon}\right\vert \leq
\sup\nolimits_{r\geq0}\left\vert G_{r}\right\vert ,\quad\text{for all }%
t\geq0;\medskip\\
\left(  b\right)  & \lim\nolimits_{\varepsilon\rightarrow0}G_{t}^{\varepsilon
}=G_{t}\,,\;\;\text{a.e.}\mathbb{\;}t\geq0;\medskip\\
\left(  c\right)  & \left\vert G_{t}^{\varepsilon}-G_{t}\right\vert \leq
\exp\big(2-1/\sqrt{Q_{\varepsilon}}\big)\,\sup\nolimits_{r\geq0}\left\vert
G_{r}\right\vert \medskip\\
& \quad+\sup\nolimits_{r\geq0}\left\{  \left\vert G_{r}-G_{t}\right\vert
:0\leq Q_{r}-Q_{t}\leq\sqrt{Q_{\varepsilon}}\vee Q_{\varepsilon}\right\}
,\quad\text{for all }t\geq0.
\end{array}
\label{a1}%
\end{equation}
Moreover, if $G$ is a continuous stochastic process, then, for all $T>0,$%
\begin{equation}
\lim\nolimits_{\varepsilon\rightarrow0}\left(  \sup\nolimits_{s\in\left[
0,T\right]  }\left\vert G_{s}^{\varepsilon}-G_{s}\right\vert \right)
=0,\quad\mathbb{P}\text{--a.s..} \label{a2}%
\end{equation}

\end{lemma}

\begin{proof}
$\left(  b\right)  $ Let $n\in\mathbb{N}^{\ast}.$ We can assume that
$0<\varepsilon<t.$%
\begin{align*}
\left\vert G_{t}^{\varepsilon}-G_{t}\right\vert  &  \leq\dfrac{1}%
{Q_{\varepsilon}}\,\int_{t}^{\infty}e^{-\frac{Q_{r}-Q_{t}}{Q_{\varepsilon}}%
}|G_{r}-G_{t}|dQ_{r}\\
&  =\int_{0}^{\infty}e^{-s}\left\vert G_{Q^{-1}\left(  Q_{t}+sQ_{\varepsilon
}\right)  }-G_{Q^{-1}\left(  Q_{t}\right)  }\right\vert ds\\
&  \leq%
{\displaystyle\int_{0}^{n}}
\left\vert G_{Q^{-1}\left(  Q_{t}+sQ_{\varepsilon}\right)  }-G_{Q^{-1}\left(
Q_{t}\right)  }\right\vert ds+2\sup\nolimits_{r\geq0}\left\vert G_{r}%
\right\vert \int_{n}^{\infty}e^{-s}ds.
\end{align*}
Since%
\[
\lim\nolimits_{\varepsilon\rightarrow0}%
{\displaystyle\int_{0}^{n}}
\left\vert G_{Q^{-1}\left(  Q_{t}+sQ_{\varepsilon}\right)  }-G_{Q^{-1}\left(
Q_{t}\right)  }\right\vert ds=0,\quad\text{a.e. }t\in\left[  0,n\right]  ,
\]
we have for all $n\in\mathbb{N}^{\ast}$%
\[
\limsup\nolimits_{\varepsilon\rightarrow0}\left\vert G_{t}^{\varepsilon}%
-G_{t}\right\vert \leq2e^{-n}\sup\nolimits_{r\geq0}\left\vert G_{r}\right\vert
,\quad\text{a.e. }t\in\left(  0,T\right)  .
\]
which yields $\left(  b\right)  .\medskip$

\noindent$\left(  c\right)  $ Let $t_{\varepsilon}=Q^{-1}\left(  Q_{t}%
+\sqrt{Q_{\varepsilon}}\right)  .$ We have%
\[%
\begin{array}
[c]{l}%
\displaystyle\left\vert G_{t}^{\varepsilon}-G_{t}\right\vert \leq\dfrac
{1}{Q_{\varepsilon}}\,\int_{t\vee\varepsilon}^{\infty}e^{-\frac{Q_{r}%
-Q_{t\vee\varepsilon}}{Q_{\varepsilon}}}|G_{r}-G_{t}|dQ_{r}\medskip\\
\displaystyle\leq\sup\nolimits_{r\in\left[  t\vee\varepsilon,t_{\varepsilon
}\vee\varepsilon\right]  }|G_{r}-G_{t}|\,\dfrac{1}{Q_{\varepsilon}}%
\,\int_{t\vee\varepsilon}^{t_{\varepsilon}\vee\varepsilon}e^{-\frac
{Q_{r}-Q_{t\vee\varepsilon}}{Q_{\varepsilon}}}dQ_{r}+2\sup\nolimits_{s\geq
0}|G_{s}|\,\dfrac{1}{Q_{\varepsilon}}\,\int_{t_{\varepsilon}\vee\varepsilon
}^{\infty}e^{-\frac{Q_{r}-Q_{t\vee\varepsilon}}{Q_{\varepsilon}}}%
dQ_{r}\medskip\\
\displaystyle\leq\sup\nolimits_{r\in\left[  t\vee\varepsilon,t_{\varepsilon
}\vee\varepsilon\right]  }|G_{r}-G_{t}|\,\int_{0}^{\frac{Q_{t_{\varepsilon
}\vee\varepsilon}-Q_{t\vee\varepsilon}}{Q_{\varepsilon}}}\left(
-e^{-s}\right)  ds+2\sup\nolimits_{s\geq0}|G_{s}|\,\dfrac{1}{Q_{\varepsilon}%
}\,\int_{t_{\varepsilon}}^{\infty}e^{-\frac{Q_{r}-Q_{t\vee\varepsilon}%
}{Q_{\varepsilon}}}dQ_{r}\medskip\\
\displaystyle\leq\sup\nolimits_{r\in\left[  t\vee\varepsilon,t_{\varepsilon
}\vee\varepsilon\right]  }|G_{r}-G_{t}|+2e^{-\frac{Q_{t_{\varepsilon}%
}-Q_{t\vee\varepsilon}}{Q_{\varepsilon}}}\sup\nolimits_{s\geq0}|G_{s}|\,.
\end{array}
\]
Since%
\[
\frac{Q_{t_{\varepsilon}}-Q_{t\vee\varepsilon}}{Q_{\varepsilon}}=\frac
{\sqrt{Q_{\varepsilon}}}{Q_{\varepsilon}}+\frac{Q_{t}-Q_{t\vee\varepsilon}%
}{Q_{\varepsilon}}\geq\frac{1}{\sqrt{Q_{\varepsilon}}}-1,
\]
inequality (\ref{a1}$-c$) follows.$\medskip$

Clearly, (\ref{a2}) follows from (\ref{a1}$-c$).\hfill
\end{proof}

\begin{remark}
\label{R1-approx}Let $\varepsilon>0$ and let $Q:\Omega\times\mathbb{R}%
\rightarrow\mathbb{R}$ be a strictly increasing continuous stochastic process
such that $Q_{0}=0$ and $\lim_{t\rightarrow\infty}Q_{t}=\infty,$
$\lim_{t\rightarrow-\infty}Q_{t}=-\infty,$ and let $G:\Omega\times
\mathbb{R}\rightarrow\mathbb{R}^{m}$ be a measurable stochastic process such
that $\sup_{t\in\mathbb{R}_{+}}\left\vert G_{t}\right\vert <\infty,$
$\mathbb{P}$--a.s..

Then similar boundedness and convergence results as in the previous Lemma
\ref{L1-approx} hold true for $G^{i,\varepsilon}:\Omega\times\mathbb{R}%
\rightarrow\mathbb{R}^{m},$ $i=\overline{1,4}\,,$ defined by%
\begin{align*}
G_{t}^{1,\varepsilon}  &  =\frac{1}{Q_{\varepsilon}}\,%
{\displaystyle\int_{t}^{\infty}}
G_{r}\,e^{-\frac{Q_{r}-Q_{t}}{Q_{\varepsilon}}}dQ_{r}\,,\quad t\in
\mathbb{R},\\
G_{t}^{2,\varepsilon}  &  =\frac{1}{Q_{\varepsilon}}\,\int_{-\infty}^{t}%
G_{r}\,e^{-\frac{Q_{t}-Q_{r}}{Q_{\varepsilon}}}dQ_{r}\,,\quad t\in
\mathbb{R},\\
G_{t}^{3,\varepsilon}  &  =e^{-\frac{Q_{t}}{Q_{\varepsilon}}}G_{0}+\frac
{1}{Q_{\varepsilon}}\,%
{\displaystyle\int_{0}^{t}}
G_{r}\,e^{-\frac{Q_{t}-Q_{r}}{Q_{\varepsilon}}}dQ_{r}\\
&  =\frac{1}{Q_{\varepsilon}}\,\int_{-\infty}^{t}\left[  \mathbf{1}%
_{(-\infty,0)}\left(  r\right)  G_{0}+\mathbf{1}_{[0,\infty)}\left(  r\right)
G_{r}\right]  e^{-\frac{Q_{t}-Q_{r}}{Q_{\varepsilon}}}dQ_{r}\,,\quad t\geq0,\\
G_{t}^{4,\varepsilon}  &  =\mathbf{1}_{[0,\varepsilon)}\left(  t\right)
G_{0}+\mathbf{1}_{[\varepsilon,\infty)}\left(  t\right)  \,\frac
{1}{Q_{\varepsilon}}\,%
{\displaystyle\int_{0}^{t}}
G_{r}\,e^{-\frac{Q_{t}-Q_{r\vee\varepsilon}}{Q_{\varepsilon}}}dQ_{r}\,,\quad
t\geq0.
\end{align*}

\end{remark}

\begin{corollary}
\label{C1-approx}Let the assumptions of Lemma \ref{L1-approx} be satisfied and
$\varphi:\mathbb{R}^{m}\rightarrow(-\infty,+\infty]$ be a proper convex lower
semicontinuous function such that $\int_{0}^{\infty}\left\vert \varphi\left(
G_{u}\right)  \right\vert dQ_{u}<\infty,$ $\mathbb{P}$--a.s..

Then, for any $0\leq\alpha\leq\beta,$%
\[
\lim\nolimits_{\varepsilon\rightarrow0}\int_{\alpha}^{\beta}\varphi\left(
G_{r}^{\varepsilon}\right)  dQ_{r}=\int_{\alpha}^{\beta}\varphi\left(
G_{r}\right)  dQ_{r}\,,
\]
where $G^{\varepsilon}$ is given by (\ref{def_G_epsilon}).

Moreover, if $\displaystyle\mathbb{E}\int_{0}^{\infty}\left\vert
\varphi\left(  G_{u}\right)  \right\vert dQ_{u}<\infty,$ then, for any
stopping times $0\leq\sigma\leq\theta,$%
\[
\lim\nolimits_{\varepsilon\rightarrow0}\mathbb{E}\int_{\sigma}^{\theta}%
\varphi\left(  G_{r}^{\varepsilon}\right)  dQ_{r}=\mathbb{E}\int_{\sigma
}^{\theta}\varphi\left(  G_{r}\right)  dQ_{r}\,.
\]

\end{corollary}

\begin{proof}
We have%
\begin{align*}
\int_{\alpha}^{\beta}\varphi\left(  G_{r}^{\varepsilon}\right)  dQ_{r}  &
\leq\int_{\alpha}^{\beta}\left(  \frac{1}{Q_{\varepsilon}}\,%
{\displaystyle\int_{r\vee\varepsilon}^{\infty}}
e^{-\frac{Q_{u}-Q_{r\vee\varepsilon}}{Q_{\varepsilon}}}\varphi\left(
G_{u}\right)  dQ_{u}\right)  dQ_{r}\\
&  =\int_{0}^{\infty}\varphi\left(  G_{u}\right)  \left(  \int_{0}^{\infty
}\mathbf{1}_{\left[  \alpha,\beta\right]  }\left(  r\right)  \,\mathbf{1}%
_{[r\vee\varepsilon,\infty)}\left(  u\right)  \,\frac{1}{Q_{\varepsilon}%
}\,e^{-\frac{Q_{u}-Q_{r\vee\varepsilon}}{Q_{\varepsilon}}}dQ_{r}\right)
dQ_{u}\\
&  =\int_{0}^{\infty}\varphi\left(  G_{u}\right)  \,\mathbf{1}_{[\varepsilon
,\infty)}\left(  u\right)  \left(  \frac{1}{Q_{\varepsilon}}\,\int_{0}%
^{u}\mathbf{1}_{\left[  \alpha,\beta\right]  }\left(  r\right)  \,e^{-\frac
{Q_{u}-Q_{r\vee\varepsilon}}{Q_{\varepsilon}}}dQ_{r}\right)  dQ_{u}\,,
\end{align*}
since $\mathbf{1}_{[r\vee\varepsilon,\infty)}\left(  u\right)  =\mathbf{1}%
_{[0,u]}\left(  r\right)  $ $\mathbf{1}_{[\varepsilon,\infty)}\left(
u\right)  .$

We obtain%
\begin{equation}%
\begin{array}
[c]{l}%
\displaystyle\frac{1}{Q_{\varepsilon}}\,%
{\displaystyle\int_{0}^{u}}
\mathbf{1}_{\left[  \alpha,\beta\right]  }\left(  r\right)  \,e^{-\frac
{Q_{u}-Q_{r}}{Q_{\varepsilon}}}dQ_{r}\leq\frac{1}{Q_{\varepsilon}}\,%
{\displaystyle\int_{0}^{u}}
\mathbf{1}_{\left[  \alpha,\beta\right]  }\left(  r\right)  \,e^{-\frac
{Q_{u}-Q_{r\vee\varepsilon}}{Q_{\varepsilon}}}dQ_{r}\medskip\\
\displaystyle=\frac{1}{Q_{\varepsilon}}\,%
{\displaystyle\int_{0}^{u}}
\mathbf{1}_{\left[  \alpha,\beta\right]  }\left(  r\right)  \left[
\mathbf{1}_{[0,\varepsilon)}\left(  r\right)  \,e^{-\frac{Q_{u}-Q_{\varepsilon
}}{Q_{\varepsilon}}}+\mathbf{1}_{[\varepsilon,\infty)}\left(  r\right)
\,e^{-\frac{Q_{u}-Q_{r}}{Q_{\varepsilon}}}\right]  dQ_{r}\medskip\\
\displaystyle\leq\frac{Q_{u\wedge\varepsilon}}{Q_{\varepsilon}}\,e^{1-\frac
{Q_{u}}{Q_{\varepsilon}}}+\frac{1}{Q_{\varepsilon}}\,%
{\displaystyle\int_{0}^{u}}
\mathbf{1}_{\left[  \alpha,\beta\right]  }\left(  r\right)  \,e^{-\frac
{Q_{u}-Q_{r}}{Q_{\varepsilon}}}dQ_{r}\,,
\end{array}
\label{a3}%
\end{equation}
since%
\[%
{\displaystyle\int_{0}^{u}}
\mathbf{1}_{\left[  \alpha,\beta\right]  }\left(  r\right)  \,\mathbf{1}%
_{[0,\varepsilon)}\left(  r\right)  \,dQ_{r}\leq Q_{u\wedge\varepsilon}\,.
\]
By Remark \ref{R1-approx} (with the extension $Q_{r}=r,$ for $r<0$),%
\[
\lim\nolimits_{\varepsilon\rightarrow0}\frac{1}{Q_{\varepsilon}}\,%
{\displaystyle\int_{0}^{u}}
\mathbf{1}_{\left[  \alpha,\beta\right]  }\left(  r\right)  \,e^{-\frac
{Q_{u}-Q_{r}}{Q_{\varepsilon}}}dQ_{r}=\mathbf{1}_{\left[  \alpha,\beta\right]
}\left(  u\right)  ,\quad\text{a.e. }u\geq0,
\]
hence, by (\ref{a3}),%
\[
\lim\nolimits_{\varepsilon\rightarrow0}\frac{1}{Q_{\varepsilon}}\,%
{\displaystyle\int_{0}^{u}}
\mathbf{1}_{\left[  \alpha,\beta\right]  }\left(  r\right)  \,e^{-\frac
{Q_{u}-Q_{r\vee\varepsilon}}{Q_{\varepsilon}}}dQ_{r}=\mathbf{1}_{\left[
\alpha,\beta\right]  }\left(  u\right)  ,\quad\text{a.e. }u\geq0.
\]
On the other hand, since%
\[
0\leq\frac{1}{Q_{\varepsilon}}\,\int_{0}^{u}\mathbf{1}_{\left[  \alpha
,\beta\right]  }\left(  r\right)  \,e^{-\frac{Q_{u}-Q_{r\vee\varepsilon}%
}{Q_{\varepsilon}}}dQ_{r}\leq e+1,
\]
by Fatou's Lemma and by the Lebesgue dominated convergence theorem, we infer%
\[
\int_{\alpha}^{\beta}\varphi\left(  G_{r}\right)  dQ_{r}\leq\liminf
\nolimits_{\varepsilon\rightarrow0}\int_{\alpha}^{\beta}\varphi\left(
G_{r}^{\varepsilon}\right)  dQ_{r}\leq\limsup\nolimits_{\varepsilon
\rightarrow0}\int_{\alpha}^{\beta}\varphi\left(  G_{r}^{\varepsilon}\right)
dQ_{r}\leq\int_{\alpha}^{\beta}\varphi\left(  G_{r}\right)  dQ_{r}\,.
\]
The second assertion of this corollary follows in the same manner.\hfill
\end{proof}

\begin{proposition}
\label{p1-approx}Let $Q:\Omega\times\left[  0,T\right]  \rightarrow
\mathbb{R}_{+}$ be a strictly increasing continuous stochastic process such
that $Q_{0}=0.$

Let $\tau:\Omega\rightarrow\left[  0,\infty\right]  $\ be a stopping time and
$\eta:\Omega\rightarrow\mathbb{R}^{m}$\ a $\mathcal{F}_{\tau}$--measurable
random variable such that $\mathbb{E}\left\vert \eta\right\vert ^{p}<\infty,$
if $p>1,$\ and $\left(  \xi,\zeta\right)  \in S_{m}^{p}\times\Lambda_{m\times
k}^{p}\left(  0,\mathbb{\infty}\right)  $\ the unique pair associated to
$\eta$\ given by the martingale representation formula (see
\cite[\textit{Corollary 2.44}]{pa-ra/14}):%
\[
\left\{
\begin{array}
[c]{l}%
\xi_{t}=\eta-\displaystyle{\int_{t}^{\infty}}\zeta_{s}dB_{s},\;t\geq
0,\quad\mathbb{P}\text{--\textit{a.s.,}}\medskip\\
\xi_{t}=\mathbb{E}^{\mathcal{F}_{t}}\eta=\mathbb{E}^{\mathcal{F}_{t\wedge\tau
}}\eta\quad\text{\textit{and}}\quad\zeta_{t}=\mathbf{1}_{\left[
0,\tau\right]  }\left(  t\right)  \zeta_{t}%
\end{array}
\right.
\]
(or equivalently, $\xi_{t}=\eta-\int_{t\wedge\tau}^{\tau}\zeta_{s}%
dB_{s},\;t\geq0,\;\mathbb{P}$--a.s.).$\smallskip$

Let $U\in S_{m}^{p}\,,$ with $p>1,$ be such that
\[%
\begin{array}
[c]{ll}%
\left(  a\right)  & \mathbb{E}\sup\nolimits_{t\geq0}\left\vert U_{t}%
\right\vert ^{p}<\infty,\medskip\\
\left(  b\right)  & \lim_{t\rightarrow\infty}\mathbb{E}\left\vert U_{t}%
-\xi_{t}\right\vert ^{p}=0.
\end{array}
\]
Define%
\begin{equation}
U_{t}^{\varepsilon}=\frac{1}{Q_{\varepsilon}}\,%
{\displaystyle\int_{t\vee\varepsilon}^{\infty}}
e^{-\frac{Q_{r}-Q_{t\vee\varepsilon}}{Q_{\varepsilon}}}\,U_{r}\,dQ_{r}%
\quad\text{and}\quad M_{t}^{\varepsilon}=\mathbb{E}^{\mathcal{F}_{t}%
}\big(U_{t}^{\varepsilon}\big)\,,\quad t\geq0. \label{a5}%
\end{equation}
Then:$\medskip$

\noindent\textbf{I.}%
\begin{equation}%
\begin{array}
[c]{rl}%
\left(  j\right)  & \left\vert M_{t}^{\varepsilon}\right\vert \leq
\mathbb{E}^{\mathcal{F}_{t}}\sup\nolimits_{r\geq0}\left\vert U_{r}\right\vert
,\quad\mathbb{P}\text{--a.s., for all }t\geq0,\medskip\\
\left(  jj\right)  & \mathbb{E}\sup\nolimits_{t\geq0}\left\vert M_{t}%
^{\varepsilon}\right\vert ^{p}\leq C_{p}\,\mathbb{E}\sup\nolimits_{r\geq
0}\left\vert U_{r}\right\vert ^{p}.
\end{array}
\label{a6}%
\end{equation}
Also, for any $t\geq0,$%
\begin{align}
\left\vert M_{t}^{\varepsilon}-U_{t}\right\vert  &  \leq\mathbb{E}%
^{\mathcal{F}_{t}}\Big[\big(2-1/\sqrt{Q_{\varepsilon}}\big)\,\sup
\nolimits_{r\geq0}\left\vert U_{r}\right\vert \label{a6-1}\\
&  \quad+\sup\nolimits_{r\geq0}\left\{  \left\vert U_{r}-U_{t}\right\vert
:0\leq Q_{r}-Q_{t}\leq\sqrt{Q_{\varepsilon}}\vee Q_{\varepsilon}\right\}
\Big]\nonumber
\end{align}
which yields%
\begin{equation}%
\begin{array}
[c]{rl}%
\left(  jjj\right)  & \lim\nolimits_{\varepsilon\rightarrow0}M_{t}%
^{\varepsilon}=U_{t}\,,\quad\mathbb{P}\text{--a.s.,}\quad\text{for all }%
t\geq0;\medskip\\
\left(  jv\right)  & \lim\nolimits_{\varepsilon\rightarrow0}\mathbb{E}%
\sup\nolimits_{t\in\left[  0,T\right]  }\left\vert M_{t}^{\varepsilon}%
-U_{t}\right\vert ^{p}=0,\quad\text{for all }T>0.
\end{array}
\label{a6-2}%
\end{equation}
\noindent\textbf{II. }$M^{\varepsilon}$ is the unique solution of the BSDE:
\begin{equation}
\left\{
\begin{array}
[c]{l}%
M_{t}^{\varepsilon}=M_{T}^{\varepsilon}+\dfrac{1}{Q_{\varepsilon}}\,{%
{\displaystyle\int_{t}^{T}}
}1_{[\varepsilon,\infty)}\left(  r\right)  \left(  U_{r}-M_{r}^{\varepsilon
}\right)  dQ_{r}-{%
{\displaystyle\int_{t}^{T}}
}R_{r}^{\varepsilon}\,dB_{r}\,,\quad\text{for all }T>0,\;t\in\left[
0,T\right]  ,\medskip\\
\lim\nolimits_{t\rightarrow\infty}\mathbb{E}\left\vert M_{t}^{\varepsilon}%
-\xi_{t}\right\vert ^{p}=0.
\end{array}
\right.  \label{a4}%
\end{equation}
Moreover,%
\begin{equation}
\lim\nolimits_{t\rightarrow\infty}\mathbb{E}\sup\nolimits_{s\geq t}\left\vert
U_{t}-\xi_{t}\right\vert ^{p}=0\quad\Longrightarrow\quad\lim
\nolimits_{t\rightarrow\infty}\mathbb{E}\Big(\sup\nolimits_{s\geq t}\left\vert
M_{s}^{\varepsilon}-\xi_{s}\right\vert ^{p}\Big)=0. \label{a8}%
\end{equation}
\noindent\textbf{III. }Let $\varphi:\mathbb{R}^{m}\rightarrow(-\infty
,+\infty]$ be a proper convex lower semicontinuous function such that
\[
\mathbb{E}%
{\displaystyle\int_{0}^{\infty}}
\left\vert \varphi\left(  U_{r}\right)  \right\vert dQ_{r}<\infty.
\]
Let $0\leq s\leq t$ and the stopping times $s^{\ast}=Q_{s}^{-1},$ $t^{\ast
}=Q_{t}^{-1}\,,$ where $Q_{\cdot}^{-1}$ is the inverse of the function
$r\mapsto Q_{r}:[0,\infty)\rightarrow\lbrack0,\infty).$

Then%
\[
\lim\nolimits_{\varepsilon\rightarrow0}\mathbb{E}\int_{s^{\ast}}^{t^{\ast}%
}\varphi\left(  M_{r}^{\varepsilon}\right)  dQ_{r}=\mathbb{E}\int_{s^{\ast}%
}^{t^{\ast}}\varphi\left(  U_{r}\right)  dQ_{r}\,.
\]
Moreover, if $g:\mathbb{R}^{m}\times\mathbb{R}^{n}\rightarrow\mathbb{R}_{+}$
is a continuous function, $D:\Omega\times\mathbb{R}_{+}\rightarrow
\mathbb{R}^{n}$ is a continuous stochastic process such that for all $R>0$
\[
\mathbb{E}%
{\displaystyle\int_{0}^{R}}
\left\vert \varphi\left(  U_{r}\right)  \right\vert \sup\nolimits_{\theta
\in\left[  0,r\right]  }g\left(  U_{\theta},D_{\theta}\right)  dQ_{r}%
+\mathbb{E}%
{\displaystyle\int_{0}^{R}}
\left\vert \varphi\left(  U_{r}\right)  \right\vert \sup\nolimits_{\theta
\in\left[  0,r\right]  }\sup\nolimits_{0<\varepsilon\leq1}g\left(  M_{\theta
}^{\varepsilon},D_{\theta}\right)  dQ_{r}<\infty,
\]
then, for all $0\leq T\leq\infty,$%
\begin{equation}%
\begin{array}
[c]{rl}%
\left(  v\right)  & \mathbb{E}%
{\displaystyle\int_{s^{\ast}\wedge T}^{t^{\ast}\wedge T}}
g\left(  M_{r}^{\varepsilon},D_{r}\right)  \varphi\left(  M_{r}^{\varepsilon
}\right)  dQ_{r}\leq\mathbb{E}%
{\displaystyle\int_{s^{\ast}\wedge T}^{t^{\ast}\wedge T}}
g\left(  M_{r}^{\varepsilon},D_{r}\right)  \varphi\left(  U_{r}^{\varepsilon
}\right)  dQ_{r}\,,\medskip\\
\left(  vj\right)  & \lim\nolimits_{\varepsilon\rightarrow0}\mathbb{E}%
{\displaystyle\int_{s^{\ast}\wedge T}^{t^{\ast}\wedge T}}
g\left(  M_{r}^{\varepsilon},D_{r}\right)  \varphi\left(  M_{r}^{\varepsilon
}\right)  dQ_{r}=\mathbb{E}%
{\displaystyle\int_{s^{\ast}\wedge T}^{t^{\ast}\wedge T}}
g\left(  U_{r},D_{r}\right)  \varphi\left(  U_{r}\right)  dQ_{r}\,.
\end{array}
\label{a9}%
\end{equation}

\end{proposition}

\begin{proof}
By Doob's inequality (see \cite[Theorem 1.60]{pa-ra/14}) and (\ref{a6}$-j$) we
get estimate (\ref{a6}$-jj$).$\medskip$

Clearly%
\[
\left\vert M_{t}^{\varepsilon}-U_{t}\right\vert \leq\mathbb{E}^{\mathcal{F}%
_{t}}\left\vert U_{t}^{\varepsilon}-U_{t}\right\vert \leq\mathbb{E}%
^{\mathcal{F}_{t}}\sup\nolimits_{r\in\left[  0,T\right]  }\left\vert
U_{t}^{\varepsilon}-U_{t}\right\vert
\]
and conclusions (\ref{a6-1}) and (\ref{a6-2}) hold by Lemma \ref{L1-approx}
and Doob's inequality.$\medskip$

Let us to prove (\ref{a4}). By the martingale representation theorem we have%
\begin{align*}
\dfrac{1}{Q_{\varepsilon}}\,%
{\displaystyle\int_{\varepsilon}^{\infty}}
e^{-\frac{Q_{r}}{Q_{\varepsilon}}}\,U_{r}\,dQ_{r}  &  =\mathbb{E}%
^{\mathcal{F}_{t}}\dfrac{1}{Q_{\varepsilon}}%
{\displaystyle\int_{\varepsilon}^{\infty}}
e^{-\frac{Q_{r}}{Q_{\varepsilon}}}\,U_{r}\,dQ_{r}+%
{\displaystyle\int_{t}^{\infty}}
\tilde{R}_{r}^{\varepsilon}\,dB_{r}\\
&  =e^{-\frac{Q_{t\vee\varepsilon}}{Q_{\varepsilon}}}M_{t}^{\varepsilon
}+\mathbb{E}^{\mathcal{F}_{t}}\dfrac{1}{Q_{\varepsilon}}%
{\displaystyle\int_{\varepsilon}^{t\vee\varepsilon}}
e^{-\frac{Q_{r}}{Q_{\varepsilon}}}\,U_{r}\,dQ_{r}+%
{\displaystyle\int_{t}^{\infty}}
\tilde{R}_{r}^{\varepsilon}\,dB_{r}\,,
\end{align*}
which yields%
\begin{equation}
e^{-\frac{Q_{t\vee\varepsilon}}{Q_{\varepsilon}}}M_{t}^{\varepsilon}=\dfrac
{1}{Q_{\varepsilon}}\,%
{\displaystyle\int_{t}^{\infty}}
1_{[\varepsilon,\infty)}\left(  r\right)  \,e^{-\frac{Q_{r}}{Q_{\varepsilon}}%
}\,U_{r}\,dQ_{r}-%
{\displaystyle\int_{t}^{\infty}}
\tilde{R}_{r}^{\varepsilon}\,dB_{r}\,. \label{a7-a}%
\end{equation}
Now by It\^{o}'s formula%
\[%
\begin{array}
[c]{l}%
\displaystyle M_{t}^{\varepsilon}=M_{T}^{\varepsilon}-\int_{t}^{T}d\left[
e^{\frac{Q_{r\vee\varepsilon}}{Q_{\varepsilon}}}\left(  e^{-\frac
{Q_{r\vee\varepsilon}}{Q_{\varepsilon}}}M_{r}^{\varepsilon}\right)  \right]
\medskip\\
\displaystyle=M_{T}^{\varepsilon}-\dfrac{1}{Q_{\varepsilon}}\,\int_{t}%
^{T}1_{[\varepsilon,\infty)}\left(  r\right)  e^{\frac{Q_{r\vee\varepsilon}%
}{Q_{\varepsilon}}}\left(  e^{-\frac{Q_{r\vee\varepsilon}}{Q_{\varepsilon}}%
}M_{r}^{\varepsilon}\right)  dQ_{r}-\int_{t}^{T}e^{\frac{Q_{r\vee\varepsilon}%
}{Q_{\varepsilon}}}d\left(  e^{-\frac{Q_{r\vee\varepsilon}}{Q_{\varepsilon}}%
}M_{r}^{\varepsilon}\right)  \medskip\\
\displaystyle=M_{T}^{\varepsilon}-\dfrac{1}{Q_{\varepsilon}}\,\int_{t}%
^{T}1_{[\varepsilon,\infty)}\left(  r\right)  \,M_{r}^{\varepsilon}%
\,dQ_{r}+\dfrac{1}{Q_{\varepsilon}}\,\int_{t}^{T}e^{\frac{Q_{r\vee\varepsilon
}}{Q_{\varepsilon}}}1_{[\varepsilon,\infty)}\left(  r\right)  \,e^{-\frac
{Q_{r\vee\varepsilon}}{Q_{\varepsilon}}}U_{r}\,dQ_{r}-%
{\displaystyle\int_{t}^{T}}
e^{\frac{Q_{r\vee\varepsilon}}{Q_{\varepsilon}}}\tilde{R}_{r}^{\varepsilon
}\,dB_{r}\medskip\\
\displaystyle=M_{T}^{\varepsilon}+\dfrac{1}{Q_{\varepsilon}}\,\int_{t}%
^{T}1_{[\varepsilon,\infty)}\left(  r\right)  \,\left(  U_{r}-M_{r}%
^{\varepsilon}\right)  dQ_{r}-%
{\displaystyle\int_{t}^{T}}
R_{r}^{\varepsilon}\,dB_{r}\,,
\end{array}
\]
where $R_{r}^{\varepsilon}%
\xlongequal{\hspace{-4pt}{\rm def}\hspace{-4pt}}e^{\frac{Q_{r\vee\varepsilon}%
}{Q_{\varepsilon}}}\tilde{R}_{r}^{\varepsilon}\,.\medskip$

The convergence result from (\ref{a4}) is obtained as follows:%
\begin{align*}
\left\vert M_{t}^{\varepsilon}-\xi_{t}\right\vert  &  =\Big|\mathbb{E}%
^{\mathcal{F}_{t}}\dfrac{1}{Q_{\varepsilon}}\,%
{\displaystyle\int_{t\vee\varepsilon}^{\infty}}
e^{-\frac{Q_{r}-Q_{t\vee\varepsilon}}{Q_{\varepsilon}}}\left(  U_{r}-\xi
_{r}\right)  dQ_{r}+\mathbb{E}^{\mathcal{F}_{t}}\dfrac{1}{Q_{\varepsilon}}\,%
{\displaystyle\int_{t\vee\varepsilon}^{\infty}}
e^{-\frac{Q_{r}-Q_{t\vee\varepsilon}}{Q_{\varepsilon}}}\,\left(  \xi_{r}%
-\xi_{t}\right)  \,dQ_{r}\Big|\\
&  \leq\Big|\mathbb{E}^{\mathcal{F}_{t}}%
{\displaystyle\int_{0}^{\infty}}
e^{-s}\left(  U_{Q^{-1}\left(  sQ_{\varepsilon}+Q_{t\vee\varepsilon}\right)
}-\xi_{Q^{-1}\left(  sQ_{\varepsilon}+Q_{t\vee\varepsilon}\right)  }\right)
ds\Big|+\mathbb{E}^{\mathcal{F}_{t}}\sup\nolimits_{r\geq t}\left\vert \xi
_{r}-\xi_{t}\right\vert
\end{align*}
By Jensen's inequality and, after that, by Burkholder--Davis--Gundy inequality
(see \cite[Corollary 2.9]{pa-ra/14}) we have%
\[
\left(  \mathbb{E}^{\mathcal{F}_{t}}\sup\nolimits_{r\geq t}\left\vert \xi
_{r}-\xi_{t}\right\vert \right)  ^{p}\leq\left(  \mathbb{E}^{\mathcal{F}_{t}%
}\sup\nolimits_{r\geq t}\left\vert {\int_{t}^{r}}\zeta_{s}\,dB_{s}\right\vert
\right)  ^{p}\leq C_{p}\,\mathbb{E}^{\mathcal{F}_{t}}\left(  \int_{t}^{\infty
}\left\vert \zeta_{s}\right\vert ^{2}ds\right)  ^{p/2}.
\]
Hence (by Jensen's and Holder's inequalities)%
\[%
\begin{array}
[c]{l}%
\displaystyle\mathbb{E}\left\vert M_{t}^{\varepsilon}-\xi_{t}\right\vert
^{p}\medskip\\
\displaystyle\leq2^{p-1}\mathbb{E}\left(  \mathbb{E}^{\mathcal{F}_{t}}%
{\displaystyle\int_{0}^{\infty}}
e^{-s}\left\vert U_{Q^{-1}\left(  sQ_{\varepsilon}+Q_{t\vee\varepsilon
}\right)  }-\xi_{Q^{-1}\left(  sQ_{\varepsilon}+Q_{t\vee\varepsilon}\right)
}\right\vert ds\right)  ^{p}+2^{p-1}\,\mathbb{E}\left(  \mathbb{E}%
^{\mathcal{F}_{t}}\sup\nolimits_{r\geq t}\left\vert \xi_{r}-\xi_{t}\right\vert
\right)  ^{p}\medskip\\
\displaystyle\leq2^{p-1}%
{\displaystyle\int_{0}^{\infty}}
e^{-s}\,\mathbb{E}\left\vert U_{Q^{-1}\left(  sQ_{\varepsilon}+Q_{t\vee
\varepsilon}\right)  }-\xi_{Q^{-1}\left(  sQ_{\varepsilon}+Q_{t\vee
\varepsilon}\right)  }\right\vert ^{p}ds+C_{p}\,\mathbb{E}\left(  \int
_{t}^{\infty}\left\vert \zeta_{s}\right\vert ^{2}ds\right)  ^{p/2}%
\end{array}
\]
and, using the Lebesgue dominated convergence theorem, we get%
\[
\lim\nolimits_{t\rightarrow\infty}\mathbb{E}\left\vert M_{t}^{\varepsilon}%
-\xi_{t}\right\vert ^{p}=0.
\]
In order to prove (\ref{a8}), we see that, for $\varepsilon<T\leq t$ and
$1<q<p,$%
\[
\left\vert M_{t}^{\varepsilon}-\xi_{t}\right\vert ^{p}\leq2^{p-1}%
\,\mathbb{E}^{\mathcal{F}_{t}}\sup\nolimits_{r\geq T}\left\vert U_{r}-\xi
_{r}\right\vert ^{p}+2^{p-1}\left(  \mathbb{E}^{\mathcal{F}_{t}}%
\sup\nolimits_{r\geq t}\left\vert \xi_{r}-\xi_{t}\right\vert ^{q}\right)
^{p/q}%
\]
and consequently (by Burkholder--Davis--Gundy and Doob's inequality)%
\begin{align*}
\mathbb{E}\sup\nolimits_{t\geq T}\left\vert M_{t}^{\varepsilon}-\xi
_{t}\right\vert ^{p}  &  \leq2^{p-1}\,\mathbb{E}\sup\nolimits_{r\geq
T}\left\vert U_{r}-\xi_{r}\right\vert ^{p}+2^{p-1}\,\mathbb{E}\sup
\nolimits_{t\geq T}\left[  \mathbb{E}^{\mathcal{F}_{t}}\sup\nolimits_{r\geq
t}\left\vert {\int_{t}^{r}}\zeta_{s}\,dB_{s}\right\vert ^{q}\right]  ^{p/q}\\
&  \leq2^{p-1}\,\mathbb{E}\sup\nolimits_{r\geq T}\left\vert U_{r}-\xi
_{r}\right\vert ^{p}+C_{p,q}\,\mathbb{E}\sup\nolimits_{t\geq T}\left[
\mathbb{E}^{\mathcal{F}_{t}}\left(  \int_{T}^{\infty}\left\vert \zeta
_{s}\right\vert ^{2}ds\right)  ^{q/2}\right]  ^{p/q}\\
&  \leq C_{p}\,\mathbb{E}\sup\nolimits_{r\geq T}\left\vert U_{r}-\xi
_{r}\right\vert ^{p}+C_{p,q}^{\prime}\,\mathbb{E}\left(  \int_{T}^{\infty
}\left\vert \zeta_{s}\right\vert ^{2}ds\right)  ^{p/2}%
\end{align*}
and (\ref{a8}) follows.$\medskip$

Inequality (\ref{a9}$-v$) follows since, using the notation $r^{\ast}%
=Q_{r}^{-1}\,,$%
\[%
\begin{array}
[c]{l}%
\displaystyle\mathbb{E}%
{\displaystyle\int_{s^{\ast}\wedge T}^{t^{\ast}\wedge T}}
g\left(  M_{r}^{\varepsilon},D_{r}\right)  \varphi\left(  M_{r}^{\varepsilon
}\right)  dQ_{r}=\mathbb{E}\int_{s^{\ast}}^{t^{\ast}}\mathbf{1}_{\left[
0,T\right]  }\left(  r\right)  g\left(  M_{r}^{\varepsilon},D_{r}\right)
\varphi\left(  \mathbb{E}^{\mathcal{F}_{r}}\left(  U_{r}^{\varepsilon}\right)
\right)  dQ_{r}\medskip\\
\displaystyle\leq\mathbb{E}\int_{s^{\ast}}^{t^{\ast}}\mathbb{E}^{\mathcal{F}%
_{r}}\left[  \mathbf{1}_{\left[  0,T\right]  }\left(  r\right)  g\left(
M_{r}^{\varepsilon},D_{r}\right)  \varphi\left(  U_{r}^{\varepsilon}\right)
\right]  dQ_{r}=\mathbb{E}\int_{s}^{t}\mathbb{E}^{\mathcal{F}_{r^{\ast}}%
}\left[  \mathbf{1}_{\left[  0,T\right]  }\left(  r^{\ast}\right)  g\left(
M_{r^{\ast}}^{\varepsilon},D_{r^{\ast}}\right)  \varphi\left(  U_{r^{\ast}%
}^{\varepsilon}\right)  \right]  dr\medskip\\
\displaystyle=\int_{s}^{t}\mathbb{E}\left[  \mathbf{1}_{\left[  0,T\right]
}\left(  r^{\ast}\right)  g\left(  M_{r^{\ast}}^{\varepsilon},D_{r^{\ast}%
}\right)  \varphi\left(  U_{r^{\ast}}^{\varepsilon}\right)  \right]
dr=\mathbb{E}\int_{s}^{t}\mathbf{1}_{\left[  0,T\right]  }\left(  r^{\ast
}\right)  g\left(  M_{r^{\ast}}^{\varepsilon},D_{r^{\ast}}\right)
\varphi\left(  U_{r^{\ast}}^{\varepsilon}\right)  dr\medskip\\
\displaystyle=\mathbb{E}\int_{s^{\ast}}^{t^{\ast}}\mathbf{1}_{\left[
0,T\right]  }\left(  r\right)  g\left(  M_{r}^{\varepsilon},D_{r}\right)
\varphi\left(  U_{r}^{\varepsilon}\right)  dQ_{r}=\mathbb{E}%
{\displaystyle\int_{s^{\ast}\wedge T}^{t^{\ast}\wedge T}}
g\left(  M_{r}^{\varepsilon},D_{r}\right)  \varphi\left(  U_{r}^{\varepsilon
}\right)  dQ_{r}\,.
\end{array}
\]
As in the proof of Corollary \ref{C1-approx} we have%
\[%
\begin{array}
[c]{l}%
\displaystyle\mathbb{E}%
{\displaystyle\int_{s^{\ast}\wedge T}^{t^{\ast}\wedge T}}
g\left(  M_{r}^{\varepsilon},D_{r}\right)  \varphi\left(  U_{r}^{\varepsilon
}\right)  dQ_{r}\medskip\\
\displaystyle\leq\mathbb{E}\int_{0}^{\infty}\varphi\left(  U_{\theta}\right)
\,\mathbf{1}_{[\varepsilon,\infty)}\left(  \theta\right)  \left(  \frac
{1}{Q_{\varepsilon}}\,\int_{0}^{\theta}\mathbf{1}_{\left[  s^{\ast}\wedge
T,t^{\ast}\wedge T\right]  }\left(  r\right)  \,g\left(  M_{r}^{\varepsilon
},D_{r}\right)  \,e^{-\frac{Q_{\theta}-Q_{r\vee\varepsilon}}{Q_{\varepsilon}}%
}\,dQ_{r}\right)  dQ_{\theta}\medskip\\
\displaystyle\leq\mathbb{E}\int_{0}^{\infty}\varphi\left(  U_{\theta}\right)
\,\mathbf{1}_{[\varepsilon,\infty)}\left(  \theta\right)  \cdot\sup
\nolimits_{r\in\left[  0,\theta\right]  }\left\vert g\left(  M_{r}%
^{\varepsilon},D_{r}\right)  -g\left(  U_{r},D_{r}\right)  \right\vert
dQ_{\theta}\medskip\\
\displaystyle\quad+\mathbb{E}\int_{0}^{\infty}\varphi\left(  U_{\theta
}\right)  \,\mathbf{1}_{[\varepsilon,\infty)}\left(  \theta\right)  \left(
\frac{1}{Q_{\varepsilon}}\,\int_{0}^{\theta}\mathbf{1}_{\left[  s^{\ast}\wedge
T,t^{\ast}\wedge T\right]  }\left(  r\right)  \,g\left(  U_{r},D_{r}\right)
\,e^{-\frac{Q_{\theta}-Q_{r\vee\varepsilon}}{Q_{\varepsilon}}}\,dQ_{r}\right)
dQ_{\theta}\,.
\end{array}
\]
Now, by Fatou's Lemma and Remark \ref{R1-approx} (with the extension
$Q_{r}=r,$ for $r<0$), we have%
\[%
\begin{array}
[c]{l}%
\displaystyle\mathbb{E}%
{\displaystyle\int_{s^{\ast}\wedge T}^{t^{\ast}\wedge T}}
g\left(  U_{r},D_{r}\right)  \varphi\left(  U_{r}\right)  dQ_{r}\leq
\liminf_{\varepsilon\rightarrow0_{+}}\mathbb{E}%
{\displaystyle\int_{s^{\ast}\wedge T}^{t^{\ast}\wedge T}}
g\left(  M_{r}^{\varepsilon},D_{r}\right)  \varphi\left(  M_{r}^{\varepsilon
}\right)  dQ_{r}\medskip\\
\displaystyle\leq\liminf\nolimits_{\varepsilon\rightarrow0_{+}}\mathbb{E}%
{\displaystyle\int_{s^{\ast}\wedge T}^{t^{\ast}\wedge T}}
g\left(  M_{r}^{\varepsilon},D_{r}\right)  \varphi\left(  U_{r}^{\varepsilon
}\right)  dQ_{r}\leq\limsup\nolimits_{\varepsilon\rightarrow0_{+}}\mathbb{E}%
{\displaystyle\int_{s^{\ast}\wedge T}^{t^{\ast}\wedge T}}
g\left(  M_{r}^{\varepsilon},D_{r}\right)  \varphi\left(  U_{r}^{\varepsilon
}\right)  dQ_{r}\medskip\\
\displaystyle\leq\limsup\nolimits_{\varepsilon\rightarrow0_{+}}\mathbb{E}%
\int_{0}^{\infty}\varphi\left(  U_{\theta}\right)  \Big(\mathbf{1}%
_{[\varepsilon,\infty)}\left(  \theta\right)  \,\frac{1}{Q_{\varepsilon}%
}\,\int_{0}^{\theta}\mathbf{1}_{\left[  s^{\ast}\wedge T,t^{\ast}\wedge
T\right]  }\left(  r\right)  \,g\left(  U_{r},D_{r}\right)  \,e^{-\frac
{Q_{\theta}-Q_{r\vee\varepsilon}}{Q_{\varepsilon}}}dQ_{r}\Big)dQ_{\theta
}\medskip\\
\displaystyle=\mathbb{E}%
{\displaystyle\int_{s^{\ast}\wedge T}^{t^{\ast}\wedge T}}
g\left(  U_{\theta},D_{\theta}\right)  \varphi\left(  U_{\theta}\right)
dQ_{\theta}%
\end{array}
\]
and convergence (\ref{a9}$-vj$) follows.\hfill
\end{proof}

\subsection{Mollifier approximation\label{Annex-MA}}

\hspace{\parindent}Let $F:\Omega\times\mathbb{R}_{+}\times\mathbb{R}^{m}%
\times\mathbb{R}^{m\times k}\rightarrow\mathbb{R}^{m}$\ and $G:\Omega
\times\mathbb{R}_{+}\times\mathbb{R}^{m}\rightarrow\mathbb{R}^{m}$ be such
that assumptions $\left(  \mathrm{A}_{5}\right)  $ and $\left(  \mathrm{A}%
_{6}\right)  $ are satisfied.$\medskip$

\noindent Let $\rho\in C_{0}^{\infty}\left(  \mathbb{R}^{m};\mathbb{R}%
_{+}\right)  $ such that $\rho\left(  y\right)  =0$ if $\left\vert
y\right\vert \geq1$ and $\int_{\mathbb{R}^{m}}\rho\left(  y\right)
dy=1.\medskip$

Let $\kappa>0$ be such that%
\[
\int_{\overline{B\left(  0,1\right)  }}\left\vert \nabla\rho\left(  v\right)
\right\vert dv\leq\kappa\quad\text{and}\quad\left\vert \nabla_{y}\rho\left(
y\right)  \right\vert \leq\kappa\,\mathbf{1}_{\overline{B\left(  0,1\right)
}}\left(  y\right)  ,\quad\text{for all }y\in\mathbb{R}^{m}.
\]
Define, for $0<\varepsilon\leq1,$%
\begin{equation}%
\begin{array}
[c]{l}%
\displaystyle F_{\varepsilon}\left(  t,y,z\right)  =\int_{\overline{B\left(
0,1\right)  }}F\left(  t,y-\varepsilon u,\beta_{\varepsilon}\left(  z\right)
\right)  \,\mathbf{1}_{\left[  0,1\right]  }\left(  \varepsilon\left\vert
F\left(  t,y-\varepsilon u,0\right)  \right\vert \right)  \,\rho\left(
u\right)  du\medskip\\
\displaystyle=\int_{\mathbb{R}^{m}}F\left(  t,y-\varepsilon u,\beta
_{\varepsilon}\left(  z\right)  \right)  \,\mathbf{1}_{\left[  0,1\right]
}\left(  \varepsilon\left\vert F\left(  t,y-\varepsilon u,0\right)
\right\vert \right)  \,\rho\left(  u\right)  du\medskip\\
\displaystyle=\frac{1}{\varepsilon^{m}}\,\int_{\mathbb{R}^{m}}F\left(
t,u,\beta_{\varepsilon}\left(  z\right)  \right)  \,\mathbf{1}_{\left[
0,1\right]  }\left(  \varepsilon\left\vert F\left(  t,u,0\right)  \right\vert
\right)  \,\rho\left(  \dfrac{y-u}{\varepsilon}\right)  du\,,
\end{array}
\label{sde9c}%
\end{equation}
where%
\[
\beta_{\varepsilon}\left(  z\right)  =\frac{z}{1\vee\left(  \varepsilon
\left\vert z\right\vert \right)  }=\Pr\nolimits_{\overline{B\left(
0,1/\varepsilon\right)  }}\left(  z\right)  .
\]
For all $z,\hat{z}\in\mathbb{R}^{m\times k}$ and all $\varepsilon,\delta>0$ we
have%
\[%
\begin{array}
[c]{l}%
\displaystyle\left\vert \beta_{\varepsilon}\left(  z\right)  \right\vert
\leq\left\vert z\right\vert \wedge\frac{1}{\varepsilon}\,,\medskip\\
\displaystyle\left\vert \beta_{\varepsilon}\left(  z\right)  -\beta
_{\varepsilon}\left(  \hat{z}\right)  \right\vert \leq\left\vert z-\hat
{z}\right\vert ,\medskip\\
\displaystyle\left\vert \beta_{\varepsilon}\left(  z\right)  -\beta_{\delta
}\left(  z\right)  \right\vert \leq\mathbf{1}_{[\frac{1}{\varepsilon}%
\wedge\frac{1}{\delta},\infty)}\left(  \left\vert z\right\vert \right)
\,\mathbf{1}_{\varepsilon\neq\delta}\cdot\left\vert z\right\vert .
\end{array}
\]
Clearly, from the assumptions satisfied by $F,$ we have, for all
$y,u\in\mathbb{R}^{m}$ with $\left\vert u\right\vert \leq1$ and for all
$z\in\mathbb{R}^{m\times k},$%
\[
\left\vert F\left(  t,y-\varepsilon u,\beta_{\varepsilon}\left(  z\right)
\right)  \right\vert \leq\ell_{t}\,\left\vert z\right\vert +F_{\left\vert
y\right\vert +1}^{\#}\left(  t\right)
\]
and consequently%
\begin{equation}
\left\vert F_{\varepsilon}\left(  t,y,z\right)  \right\vert \leq\ell
_{t}\,\left\vert z\right\vert +F_{\left\vert y\right\vert +1}^{\#}\left(
t\right)  \quad\text{and}\quad\left\vert F_{\varepsilon}\left(  t,0,0\right)
\right\vert \leq F_{1}^{\#}\left(  t\right)  . \label{ma-4}%
\end{equation}
The mollifier approximation $F_{\varepsilon}$ of $F$ satisfies the following
properties:%
\begin{equation}%
\begin{array}
[c]{ll}%
\left(  a\right)  & \left\vert F_{\varepsilon}\left(  t,y,z\right)
\right\vert \leq\ell_{t}\,\beta_{\varepsilon}\left(  z\right)  +\dfrac
{1}{\varepsilon}\leq\dfrac{1}{\varepsilon}\,\left(  1+\ell_{t}\right)
,\medskip\\
\left(  b\right)  & \left\vert F_{\varepsilon}\left(  t,y,z\right)
-F_{\varepsilon}\left(  t,y,\hat{z}\right)  \right\vert \leq\ell
_{t}\,\left\vert z-\hat{z}\right\vert \medskip\\
\left(  c\right)  & \left\vert F_{\varepsilon}\left(  t,y,z\right)
-F_{\varepsilon}\left(  t,\hat{y},z\right)  \right\vert \leq\dfrac{\kappa
}{\varepsilon}\,\Big[\ell_{t}\,\left\vert \beta_{\varepsilon}\left(  z\right)
\right\vert +\dfrac{1}{\varepsilon}\Big]\,\left\vert y-\hat{y}\right\vert
\leq\dfrac{\kappa\left(  1+\ell_{t}\right)  }{\varepsilon^{2}}\,\left\vert
y-\hat{y}\right\vert .
\end{array}
\label{ma-1}%
\end{equation}
Indeed,%
\[%
\begin{array}
[c]{l}%
\displaystyle\left\vert F_{\varepsilon}\left(  t,y,z\right)  -F_{\varepsilon
}\left(  t,\hat{y},z\right)  \right\vert \medskip\\
\displaystyle\leq\frac{1}{\varepsilon^{m}}\,\int_{\mathbb{R}^{m}}\left\vert
F\left(  t,u,\beta_{\varepsilon}\left(  z\right)  \right)  \right\vert
\,\mathbf{1}_{\left[  0,1\right]  }\left(  \varepsilon\left\vert F\left(
t,u,0\right)  \right\vert \right)  \,\Big|\rho\Big(\dfrac{y-u}{\varepsilon
}\Big)-\rho\Big(\dfrac{\hat{y}-u}{\varepsilon}\Big)\Big|du\medskip\\
\displaystyle\leq\frac{1}{\varepsilon^{m+1}}\,\left\vert y-\hat{y}\right\vert
\int_{\mathbb{R}^{m}}\left[  \ell_{t}\,\left\vert \beta_{\varepsilon}\left(
z\right)  \right\vert +\left\vert F\left(  t,u,0\right)  \right\vert \right]
\mathbf{1}_{\left[  0,1\right]  }\left(  \varepsilon\left\vert F\left(
t,u,0\right)  \right\vert \right)  \Big(\int_{0}^{1}\Big|\nabla\rho
\Big(\dfrac{y-u}{\varepsilon}+\theta\,\frac{\hat{y}-y}{\varepsilon
}\Big)\Big|d\theta\Big)du\medskip\\
\displaystyle=\frac{1}{\varepsilon}\,\left\vert y-\hat{y}\right\vert \int
_{0}^{1}\Big(\int_{\mathbb{R}^{m}}\left[  \ell_{t}\,\left\vert \beta
_{\varepsilon}\left(  z\right)  \right\vert +\left\vert F\left(
t,y+\theta\left(  \hat{y}-y\right)  -\varepsilon v,0\right)  \right\vert
\right]  \,\mathbf{1}_{\left[  0,1\right]  }\left(  \varepsilon\left\vert
F\left(  t,y+\theta\left(  \hat{y}-y\right)  -\varepsilon v,0\right)
\right\vert \right)  \medskip\\
\displaystyle\hfill\cdot\left\vert \nabla\rho\left(  v\right)  \right\vert
dv\Big)d\theta\medskip\\
\displaystyle\leq\frac{1}{\varepsilon}\,\left\vert y-\hat{y}\right\vert
\Big[\ell_{t}\,\left\vert \beta_{\varepsilon}\left(  z\right)  \right\vert
+\frac{1}{\varepsilon}\Big]\int_{0}^{1}\Big(\int_{\overline{B\left(
0,1\right)  }}\left\vert \nabla\rho\left(  v\right)  \right\vert
dv\Big)d\theta\leq\frac{\kappa}{\varepsilon}\,\Big[\ell_{t}\,\left\vert
\beta_{\varepsilon}\left(  z\right)  \right\vert +\frac{1}{\varepsilon
}\Big]\left\vert y-\hat{y}\right\vert
\end{array}
\]
since $0\leq\left(  a+b\right)  \,\mathbf{1}_{\left[  0,1\right]  }\left(
\varepsilon b\right)  \leq a+\dfrac{1}{\varepsilon}\,,$ for all $a,b\geq
0.\bigskip$

We also have, for all $y,\hat{y}\in\mathbb{R}^{m}$ with $\left\vert \hat
{y}\right\vert \leq\rho$ and for all $z\in\mathbb{R}^{m\times k},$%
\begin{equation}%
\begin{array}
[c]{l}%
\left\langle y-\hat{y},F_{\varepsilon}\left(  t,y,z\right)  \right\rangle
\leq\mu_{t}\,\left\vert y-\hat{y}\right\vert ^{2}+\left\vert y-\hat
{y}\right\vert \left[  F_{\rho+1}^{\#}\left(  t\right)  +\ell_{t}\,\left\vert
z\right\vert \right]  \medskip\\
\leq\left\vert y-\hat{y}\right\vert \,F_{\rho+1}^{\#}\left(  t\right)
+\Big(\mu_{t}+\dfrac{1}{2n_{p}\lambda}\,\ell_{t}^{2}\,\mathbf{1}_{z\neq
0}\Big)^{+}\left\vert y-\hat{y}\right\vert ^{2}+\dfrac{n_{p}\lambda}%
{2}\,\left\vert z\right\vert ^{2},\quad\text{for all }\lambda>0,
\end{array}
\label{ma-2}%
\end{equation}
where $p>1,$ $n_{p}=\left(  p-1\right)  \wedge1.\medskip$

Indeed, by taking%
\[
\alpha_{\varepsilon}\left(  t,y\right)  =%
{\displaystyle\int_{\overline{B\left(  0,1\right)  }}}
\mathbf{1}_{\left[  0,1\right]  }\left(  \varepsilon\left\vert F\left(
t,y-\varepsilon u,0\right)  \right\vert \right)  \,\rho\left(  u\right)  du,
\]
we have $0\leq\alpha_{\varepsilon}\left(  t,y\right)  \leq1$ and%
\begin{align*}
&  \left\langle y-\hat{y},F_{\varepsilon}\left(  t,y,z\right)  \right\rangle
\\
&  =%
{\displaystyle\int_{\overline{B\left(  0,1\right)  }}}
\left\langle y-\hat{y},F\left(  t,y-\varepsilon u,\beta_{\varepsilon}\left(
z\right)  \right)  -F\left(  t,\hat{y}-\varepsilon u,\beta_{\varepsilon
}\left(  z\right)  \right)  \right\rangle \mathbf{1}_{\left[  0,1\right]
}\left(  \varepsilon\left\vert F\left(  t,y-\varepsilon u,0\right)
\right\vert \right)  \rho\left(  u\right)  du\medskip\\
&  \quad+%
{\displaystyle\int_{\overline{B\left(  0,1\right)  }}}
\left\langle y-\hat{y},F\left(  t,\hat{y}-\varepsilon u,\beta_{\varepsilon
}\left(  z\right)  \right)  -F\left(  t,\hat{y}-\varepsilon u,0\right)
\right\rangle \mathbf{1}_{\left[  0,1\right]  }\left(  \varepsilon\left\vert
F\left(  t,y-\varepsilon u,0\right)  \right\vert \right)  \rho\left(
u\right)  du\\
&  \quad+%
{\displaystyle\int_{\overline{B\left(  0,1\right)  }}}
\left\langle y-\hat{y},F\left(  t,\hat{y}-\varepsilon u,0\right)
\right\rangle \mathbf{1}_{\left[  0,1\right]  }\left(  \varepsilon\left\vert
F\left(  t,y-\varepsilon u,0\right)  \right\vert \right)  \rho\left(
u\right)  du\\
&  \leq\left[  \mu_{t}\left\vert y-\hat{y}\right\vert ^{2}+\left\vert
y-\hat{y}\right\vert \ell_{t}\left\vert \beta_{\varepsilon}\left(  z\right)
\right\vert \right]  \alpha_{\varepsilon}\left(  t,y\right)  +\left\vert
y-\hat{y}\right\vert F_{\rho+1}^{\#}\left(  t\right)  .
\end{align*}
Moreover, for all $y,\hat{y}\in\mathbb{R}^{m},\lambda\in\mathbb{R}_{+}^{\ast}$
such that $\left\vert y\right\vert \leq\rho,\left\vert \hat{y}\right\vert
\leq\rho:$%
\begin{equation}%
\begin{array}
[c]{ll}%
\left(  a\right)  & \left\langle y-\hat{y},F_{\varepsilon}\left(
t,y,z\right)  -F_{\varepsilon}\left(  t,\hat{y},z\right)  \right\rangle
\medskip\\
& \leq\mu_{t}^{+}\left\vert y-\hat{y}\right\vert ^{2}+\left\vert y-\hat
{y}\right\vert \left[  F_{\rho+1}^{\#}\left(  t\right)  +\ell_{t}\,\left\vert
z\right\vert \right]  \,\mathbf{1}_{[\frac{1}{\varepsilon},\infty)}(F_{\rho
+1}^{\#}\left(  t\right)  )\medskip\\
\left(  b\right)  & \left\langle y-\hat{y},F_{\varepsilon}\left(
t,y,z\right)  -F_{\varepsilon}\left(  t,\hat{y},\hat{z}\right)  \right\rangle
\medskip\\
& \leq\left\vert y-\hat{y}\right\vert \left[  F_{\rho+1}^{\#}\left(  t\right)
+\ell_{t}\,\left\vert \hat{z}\right\vert \right]  \,\mathbf{1}_{[\frac
{1}{\varepsilon},\infty)}(F_{\rho+1}^{\#}\left(  t\right)  )\medskip\\
& \quad+\Big(\mu_{t}^{+}+\dfrac{1}{2n_{p}\lambda}\,\ell_{t}^{2}\,\mathbf{1}%
_{z\neq\hat{z}}\Big)\left\vert y-\hat{y}\right\vert ^{2}+\dfrac{n_{p}\lambda
}{2}\,\left\vert z-\hat{z}\right\vert ^{2}\medskip\\
\left(  c\right)  & \left\langle y-\hat{y},F_{\varepsilon}\left(
t,y,z\right)  -F_{\delta}\left(  t,\hat{y},\hat{z}\right)  \right\rangle
\medskip\\
& \leq\left\vert \varepsilon-\delta\right\vert \left[  \mu_{t}^{+}\left\vert
\varepsilon-\delta\right\vert +2F_{\rho+1}^{\#}\left(  t\right)  +2\ell
_{t}\left\vert z\right\vert \right]  \medskip\\
& \quad+\left\vert y-\hat{y}\right\vert \Big[2\mu_{t}^{+}\,\left\vert
\varepsilon-\delta\right\vert +\ell_{t}\,\left\vert \hat{z}\right\vert
\,\mathbf{1}_{[\frac{1}{\varepsilon}\wedge\frac{1}{\delta},\infty)}\left(
\left\vert \hat{z}\right\vert \right)  \,\mathbf{1}_{\varepsilon\neq\delta
}\medskip\\
& \quad\quad\quad\quad\quad\quad+(F_{\rho+1}^{\#}\left(  t\right)  +\ell
_{t}\,\left\vert \hat{z}\right\vert )\,\mathbf{1}_{[\frac{1}{\varepsilon
}\wedge\frac{1}{\delta},\infty)}(F_{\rho+1}^{\#}\left(  t\right)
)\Big]\medskip\\
& \quad+\Big(\mu_{t}^{+}+\dfrac{1}{2n_{p}\lambda}\,\ell_{t}^{2}\,\mathbf{1}%
_{z\neq\hat{z}}\Big)\left\vert y-\hat{y}\right\vert ^{2}+\dfrac{n_{p}\lambda
}{2}\,\left\vert z-\hat{z}\right\vert ^{2}.
\end{array}
\label{ma-3}%
\end{equation}
Obviously, it is sufficient to prove (\ref{ma-3}$-c$).

We have%
\[%
\begin{array}
[c]{l}%
\displaystyle\left\langle y-\hat{y},F_{\varepsilon}\left(  t,y,z\right)
-F_{\delta}\left(  t,\hat{y},\hat{z}\right)  \right\rangle \medskip\\
\displaystyle\leq%
{\displaystyle\int_{\overline{B\left(  0,1\right)  }}}
\left\langle y-\varepsilon u-\left(  \hat{y}-\delta u\right)  +\left(
\varepsilon-\delta\right)  u,F\left(  t,y-\varepsilon u,\beta_{\varepsilon
}\left(  z\right)  \right)  -F\left(  t,\hat{y}-\delta u,\beta_{\varepsilon
}\left(  z\right)  \right)  \right\rangle \medskip\\
\displaystyle\hfill\cdot\mathbf{1}_{\left[  0,1\right]  }\left(
\varepsilon\left\vert F\left(  t,y-\varepsilon u,0\right)  \right\vert
\right)  \rho\left(  u\right)  du\medskip\\
\displaystyle\quad+%
{\displaystyle\int_{\overline{B\left(  0,1\right)  }}}
\left\langle y-\hat{y},F\left(  t,\hat{y}-\delta u,\beta_{\varepsilon}\left(
z\right)  \right)  -F\left(  t,\hat{y}-\delta u,\beta_{\delta}\left(  \hat
{z}\right)  \right)  \right\rangle \,\mathbf{1}_{\left[  0,1\right]  }\left(
\varepsilon\left\vert F\left(  t,y-\varepsilon u,0\right)  \right\vert
\right)  \rho\left(  u\right)  du\medskip\\
\displaystyle\quad+%
{\displaystyle\int_{\overline{B\left(  0,1\right)  }}}
\left\langle y-\hat{y},F\left(  t,\hat{y}-\delta u,\beta_{\delta}\left(
\hat{z}\right)  \right)  \right\rangle \medskip\\
\displaystyle\hfill\cdot\left[  \mathbf{1}_{\left[  0,1\right]  }\left(
\varepsilon\left\vert F\left(  t,y-\varepsilon u,0\right)  \right\vert
\right)  -\mathbf{1}_{\left[  0,1\right]  }\left(  \delta\left\vert F\left(
t,\hat{y}-\delta u,0\right)  \right\vert \right)  \right]  \rho\left(
u\right)  du\medskip\medskip\\
\displaystyle\leq\mu_{t}%
{\displaystyle\int_{\overline{B\left(  0,1\right)  }}}
\left\vert y-\varepsilon u-\left(  \hat{y}-\delta u\right)  \right\vert
^{2}\,\mathbf{1}_{\left[  0,1\right]  }\left(  \varepsilon\left\vert F\left(
t,y-\varepsilon u,0\right)  \right\vert \right)  \rho\left(  u\right)
du+2\left\vert \varepsilon-\delta\right\vert \left[  F_{\rho+1}^{\#}\left(
t\right)  +\ell_{t}\,\left\vert \beta_{\varepsilon}\left(  z\right)
\right\vert \right]  \medskip\\
\displaystyle\quad+\left\vert y-\hat{y}\right\vert \,\ell_{t}\,\left\vert
\beta_{\varepsilon}\left(  z\right)  -\beta_{\delta}\left(  \hat{z}\right)
\right\vert \,\alpha_{\varepsilon}\left(  t,y\right)  +\left\vert y-\hat
{y}\right\vert \left[  F_{\rho+1}^{\#}\left(  t\right)  +\ell_{t}\,\left\vert
\beta_{\delta}\left(  \hat{z}\right)  \right\vert \right]  \,\mathbf{1}%
_{[\frac{1}{\varepsilon}\wedge\frac{1}{\delta},\infty)}(F_{\rho+1}^{\#}\left(
t\right)  ).
\end{array}
\]
But%
\begin{align*}
\mu_{t}\left\vert y-\varepsilon u-\left(  \hat{y}-\delta u\right)  \right\vert
^{2}  &  \leq\mu_{t}^{+}\,\left\vert y-\varepsilon u-\left(  \hat{y}-\delta
u\right)  \right\vert ^{2}\\[2pt]
&  \leq\mu_{t}^{+}\,\left\vert y-\hat{y}\right\vert ^{2}+2\mu_{t}%
^{+}\,\left\vert y-\hat{y}\right\vert \left\vert \varepsilon-\delta\right\vert
+\mu_{t}^{+}\,\left\vert \varepsilon-\delta\right\vert ^{2}%
\end{align*}
and, using the properties of $\beta_{\varepsilon}\,,$ inequality
(\ref{ma-3}$-c$) follows.

\begin{remark}
The function $G$ will be approximate in the same manner. For $0<\varepsilon
\leq1:$%
\begin{equation}
G_{\varepsilon}\left(  t,y\right)  =%
{\displaystyle\int_{\overline{B\left(  0,1\right)  }}}
G\left(  t,y-\varepsilon u\right)  \,\mathbf{1}_{\left[  0,1\right]  }\left(
\varepsilon\left\vert G\left(  t,y-\varepsilon u\right)  \right\vert \right)
\,\rho\left(  u\right)  du. \label{ma-5}%
\end{equation}
Inequalities (\ref{ma-1})--(\ref{ma-3}) are similarly obtained for $G,$ with
$z=\hat{z}=0$ and $\ell=0.$
\end{remark}

\addcontentsline{toc}{section}{References}

\end{document}